\documentclass[11pt]{book}
\title{Foundations of Algebraic Theories\\ and 
Higher Dimensional Categories
\\ \quad
\\
}
\author{Soichiro Fujii}
\date{December 7, 2018}
\usepackage[a4paper,width=145mm,top=30mm,bottom=30mm]{geometry}
\usepackage[pagebackref=true]{hyperref}
\usepackage{graphicx}
\usepackage{mathrsfs}
\usepackage{amsfonts, amssymb, amsthm}
\usepackage{mathtools}
\usepackage{comment}
\usepackage{bussproofs}
\usepackage{thmtools}
\usepackage{tikz}
\usetikzlibrary{arrows.meta,decorations.pathreplacing,decorations.markings,shapes,calc}
\usetikzlibrary{matrix,arrows,decorations.pathmorphing,backgrounds,decorations.markings,positioning}
\tikzset{2cell/.style={-implies,double,double equal sign distance,shorten 
>=2pt, shorten <=3pt}}
\tikzset{2cellshort/.style={-implies,double,double equal sign distance,shorten 
>=4pt, shorten <=5pt}}
\tikzset{2cellr/.style={implies-,double,double equal sign distance,shorten 
>=3pt, shorten <=2pt}}
\tikzset{3cell/.style={-implies,double,double distance=2.5pt,shorten >=2pt, 
shorten <=3pt}}
\tikzset{labelsize/.style={font=\scriptsize}}
\tikzset{string/.style={very thick}}
\tikzset{
  pto/.style={->,postaction={decorate},
    decoration={
        markings,
        mark=at position 0.5 with {\arrow{|}}}
  },
}
\newcommand{\tzsquare}[3]{
\draw[2cell] (#1,#2)++(0,0.25) to node[auto,labelsize]{#3} ++(0,-0.5);}

\newcommand{\tzsquareup}[3]{
\draw[2cell] (#1,#2)++(0,-0.25) to node[auto,swap,labelsize]{#3} ++(0,0.5);}
\newcommand{\tzsquareupswap}[3]{
\draw[2cell] (#1,#2)++(0,-0.25) to node[auto,labelsize]{#3} ++(0,0.5);}
\numberwithin{equation}{chapter}

\declaretheoremstyle[
spaceabove=6pt, spacebelow=6pt,
headfont=\normalfont\itshape,
notefont=\mdseries, notebraces={(}{)},
bodyfont=\normalfont,
postheadspace=0.5em,
qed=$\square$
]{myproofstyle}

\declaretheorem[style=plain,numberwithin=chapter,name=Theorem]{thm}

\declaretheorem[style=plain,sibling=thm,name=Proposition]{proposition}
\declaretheorem[style=plain,sibling=thm,name=Corollary]{corollary}

\declaretheorem[style=definition,qed=$\blacksquare$,sibling=thm,name=Definition]{definition}
\declaretheorem[style=definition,qed=$\blacksquare$,sibling=thm,name=Example]{example}
\declaretheorem[style=definition,qed=$\blacksquare$,sibling=thm,name=Axiom]{axiom}
\declaretheorem[style=definition,qed=$\blacksquare$,sibling=thm,name=Convention]{convention}

\mathchardef\mhyphen="2D

\newcommand{\CAT}{{\mathbf{CAT}}}
\newcommand{\Cat}{\mathbf{Cat}}
\newcommand{\Set}{\mathbf{Set}}
\newcommand{\SET}{\mathbf{SET}}
\newcommand{\tCAT}{\mathscr{C\!A\!T}}
\newcommand{\tCat}{{\mathscr{C}\!at}}

\newcommand{\MonCATls}{{\mathscr{M}\!on\mathscr{C\!A\!T}^{\mathrm{ls}}_{\mathrm{lax}}}}
\newcommand{\SymMonCATls}{{\mathscr{S}\!ym\mathscr{M}\!on
\mathscr{C\!A\!T}^{\mathrm{ls}}_{\mathrm{lax}}}}
\newcommand{\MonCAT}{{\mathscr{M}\!on\mathscr{C\!A\!T}_{\mathrm{lax}}}}
\newcommand{\MonCATst}{{\mathscr{M}\!on\mathscr{C\!A\!T}_{\mathrm{strong}}}}
\newcommand{\MonCATol}{{\mathscr{M}\!on\mathscr{C\!A\!T}_{\mathrm{oplax}}}}
\newcommand{\twoCAT}{\underline{2\mhyphen\mathscr{C\!A\!T}}}

\newcommand{\Ab}{\mathbf{Ab}}

\newcommand{\Str}{\mathrm{Str}}
\newcommand{\Sem}{\mathrm{Sem}}

\newcommand{\univ}[1]{\mathcal{#1}}

\newcommand{\PROF}{{\mathscr{P\!R\!O\!F}}}

\newcommand{\PPROF}{{\mathbb{PROF}}}

\newcommand{\FPow}{{\mathscr{F\!P}\!ow}}
\newcommand{\FProd}{{\mathscr{F\!P}\!rod}}

\newcommand{\tcat}[1]{\mathscr{#1}}

\newcommand{\monoid}[1]{\mathsf{#1}}
\newcommand{\End}{\monoid{End}}
\newcommand{\Endcl}[1]{\monoid{End}(#1)}
\newcommand{\Endopset}[1]{\monoid{End}_\Set(#1)}
\newcommand{\Endopab}[1]{\monoid{End}_\Ab(#1)}

\newcommand{\Th}[1]{{\mathbf{Th}(#1)}}
\newcommand{\Mon}[1]{{\mathbf{Mon}(#1)}}
\newcommand{\FinSet}{\mathbf{FinSet}}

\newcommand{\enrich}[2]{{\langle #1, #2\rangle}}

\newcommand{\Simp}{\Delta_{a}}
\newcommand{\ord}[1]{\mathbf{#1}}

\newcommand{\dcat}[1]{\mathbb{#1}}

\newcommand{\isoa}{\mathfrak{a}}
\newcommand{\isor}{\mathfrak{r}}
\newcommand{\isol}{\mathfrak{l}}

\newcommand{\Cpo}{\mathbf{Cpo}}

\newcommand{\F}{\mathbf{F}}
\newcommand{\Pcat}{\mathbf{P}}
\newcommand{\Ncat}{\mathbf{N}}
\newcommand{\NN}{\mathbb{N}}
\newcommand{\pres}[2]{{{\langle\, #1\, |\, #2 \,\rangle}}}
\newcommand{\interp}[1]{{[\![#1]\!]}}
\newcommand{\interpp}[1]{{[\![#1]\!]'}}
\newcommand{\group}{\mathrm{Grp}}
\newcommand{\mon}{\mathrm{Mon}}
\newcommand{\cmon}{\mathrm{CMon}}
\newcommand{\clo}{\mathrm{Clo}}

\newcommand{\Mod}[2]{{\mathbf{Mod}(#1,#2)}}

\newcommand{\olAct}[1]{{\mathscr{A}\!ct_{\mathrm{oplax}}(#1)}}
\newcommand{\olActl}[1]{{\mathscr{A}\!ct^{\mathrm{l}}_{\mathrm{oplax}}(#1)}}
\newcommand{\Enrich}[1]{{\mathscr{E}\!nrich(#1)}}
\newcommand{\Enrichtwo}[1]{{\mathscr{E}\!nrich'(#1)}}
\newcommand{\Enrichr}[1]{{\mathscr{E}\!nrich^{\mathrm{r}}(#1)}}
\newcommand{\MtMod}[1]{{\mathscr{M}\!\mathscr{M}\!od(#1)}}
\newcommand{\MtModfib}[2]{{\mathbf{MMod}_{#1}(#2)}}

\newcommand{\MtTH}{{\mathscr{M}\!\mathscr{T\!H}}}

\makeatletter
\DeclareRobustCommand{\rvdots}{%
  \vbox{
    \baselineskip4\p@\lineskiplimit\z@
    \kern-\p@
    \hbox{.}\hbox{.}\hbox{.}
  }}
\makeatother

\makeatletter
\newcommand{\pto}{}% just for safety
\newcommand{\pgets}{}% just for safety
\DeclareRobustCommand{\pto}{\mathrel{\mathpalette\p@to@gets\to}}
\DeclareRobustCommand{\pgets}{\mathrel{\mathpalette\p@to@gets\gets}}
\newcommand{\p@to@gets}[2]{%
  \ooalign{\hidewidth$\m@th#1\mapstochar\mkern5mu$\hidewidth\cr$\m@th#1\longrightarrow$\cr}%
}
\makeatother

\newcommand{\ptensor}{\odot}
\newcommand{\ptensorrev}{\mathbin{\ptensor^\mathrm{rev}}}

\newcommand{\G}{\mathbb{G}}

\newcommand{\N}{\mathbb{N}}

\newcommand{\id}[1]{{\mathrm{id}_{#1}}}
\newcommand{\ob}[1]{{\mathrm{ob}(#1)}}

\newcommand{\sfam}[1]{\mathcal{#1}}

\newcommand{\op}{{\mathrm{op}}}
\newcommand{\co}{{\mathrm{co}}}
\newcommand{\coop}{{\mathrm{coop}}}

\newcommand{\ar}[1]{{\mathrm{ar}_{#1}}}

\newcommand{\name}[1]{{\lceil #1\rceil}}

\newcommand{\cat}[1]{\mathcal{#1}}
\newcommand{\smcat}[1]{\mathbb{#1}}
\newcommand{\cocomma}[1]{{\mathbf{gph}(#1)}}

\newcommand{\reg}{\mathrm{R}}
\newcommand{\streg}{\mathrm{SR}}

\newcommand{\DayI}{\widehat{I}}
\newcommand{\Dayo}{\mathbin{\widehat{\otimes}}}
\newcommand{\Lan}{\mathrm{Lan}}
\newcommand{\Ran}{\mathrm{Ran}}

\newcommand{\normalemb}{{K_\mathrm{N}}}
\newcommand{\discemb}{{K_\mathrm{D}}}

\newcommand{\enGph}[1]{{{#1}\mhyphen\mathbf{Gph}}}
\newcommand{\enCat}[1]{{{#1}\mhyphen\mathbf{Cat}}}

\newcommand{\entCAT}[1]{{{#1}\mhyphen\tCAT}}

\newcommand{\Mnd}[1]{{\mathbf{Mnd}(#1)}}

\newcommand{\defemph}[1]{\textbf{#1}}

\newcommand{\Contr}[1]{{\mathbf{Contr}(#1)}}
\newcommand{\OC}[1]{{#1}\mhyphen{\mathbf{OC}}}
\newcommand{\Operad}[1]{{{#1}\mhyphen\mathbf{Opd}}}
\newcommand{\arity}[1]{{\mathrm{ar}_{#1}}}
\newcommand{\colim}{\mathrm{colim}}

\begin{document}

\begin{comment}
普遍代数は様々な代数構造を，等式理論または抽象クローンとして表現することで統一的にとらえる理論である．
数学やその関連分野においては至る所に代数構造が現れるが，それらに対応するために
いくつかの普遍代数の変種が提案されてきた．
対称オペラド，非対称オペラド，一般化オペラド，モナドなどがその例である．
これらの普遍代数の変種は，代数理論の概念と呼ばれる．
代数理論の概念はどれも代数構造を記述するための基礎理論を提供するという基本目的を
共有しているものの，その達成のためにはそれぞれの代数理論の概念によって多様な方法を用いており，
我々の知る限り上で挙げた例を全て含むような代数構造の概念のための一般的な枠組みは
知られていなかった．
こうした枠組みは，代数理論の概念の本質的な構造を明らかにすることで，そのより深い理解を可能にすると同時に，
異なる代数理論の概念を比較するための統一的な方法も与えると考えられる．
本論文の第一部では，上で挙げた例を全て含む代数理論の概念のための統一的な枠組みを展開する．
そのために重要となるのは，それぞれの代数理論の概念はモノイダル圏と同一視でき，
理論はモノイダル圏のモノイド対象と対応する，という観察である．
さらに我々は，理論のモデルの定義の背景にある概念として，メタモデルと呼ばれる圏論的構造を導入する．
メタモデルは標準的な理論のモデルの定義を含むだけでなく，非標準的な定義，例えば
対称オペラドの次数付き代数や，日野，小林，蓮尾，Jacobsにより導入された$\Set$上のモナドの相対代数
といったものも例として含んでいる．
我々はさらに代数理論の概念の間の射を，モノイダル版のprofunctorとして導入する．
任意の強モノイダル関手はそのような射の随伴対を生成し，異なる代数理論の概念におけるモデルの圏の間の
同型を与える統一的な手法を提供する．
その他，一般的な構造-意味論随伴に関する結果やモデルの圏の二重圏的な普遍性も示される．

本論文の第二部では，代数的構造の一般論の研究から特定の代数的構造である高次元圏についての研究
へと話題を転じる．
高次元圏はトポロジー，数理物理学や理論計算機科学といった様々な分野において現れる．
一方で高次元圏の構造は非常に複雑であり，その定義ですら困難を伴うことが知られている．
既知のいくつかの高次元圏の定義のうちで，我々はBataninにより提案されLeinsterにより改良された定義
を研究の対象とする．
BataninとLeinsterのアプローチは高次元圏をある一般化オペラドのモデルとして定義するものであり，
従ってそれは本論文の第一部で展開される統一的な枠組みの範疇に属する．
彼らの定義はまたvan den Berg，Garner，LumsdaineによりMartin-L{\"o}fの内包的型理論
における型が持つ高次元の構造を記述する目的にも用いられた．
我々はBataninとLeinsterの定義においてextensive圏の概念が中心的な役割を果たしている
ことを示す．
また，この結果を用いて彼らの定義を任意の局所表示可能extensive圏によって
enrichできるように一般化する．
\end{comment}
\frontmatter
\begin{titlepage}
    \begin{center}
        \vspace*{1cm}
 
        \Huge
        Foundations of Algebraic Theories\\ and 
        Higher Dimensional Categories
        
        \vspace{10cm}
        \LARGE{Soichiro Fujii}
 
        \vfill
 \LARGE{A Doctor Thesis
 Submitted to\\
 the Graduate School of the University of Tokyo\\
 on December 7, 2018\\
 in Partial Fulfillment of the Requirements\\
 for the Degree of Doctor of Information Science and
 Technology\\
 in Computer Science}
 
    \end{center}
\end{titlepage}
\chapter*{Abstract}
Universal algebra uniformly captures various algebraic structures, 
by expressing them  
as equational theories or abstract clones. The ubiquity of 
algebraic structures in mathematics and related fields 
has given rise to several variants of 
universal algebra, such as symmetric operads, 
non-symmetric operads, generalised operads, and monads.
These variants of universal algebra are called
{notions of algebraic theory}.
Although notions of algebraic theory share the basic aim of providing 
a background theory to describe algebraic structures,
they use various techniques to achieve this goal and, 
to the best of our knowledge,
no general framework for notions of algebraic theory 
which includes all of the examples above was known.
Such a framework would lead to a better understanding of 
notions of algebraic theory by revealing their essential structure,
and provide a uniform way to compare different notions of algebraic theory.
In the first part of this thesis, we develop a unified framework 
for notions of algebraic theory which includes all of the above examples.
Our key observation is that each notion of algebraic theory can be 
identified with a monoidal category, in such a way that theories 
correspond to monoid objects therein. 
We introduce a categorical structure called {metamodel}, which 
underlies the definition of models of theories.
The notion of metamodel subsumes not only the standard definitions of 
models but also non-standard ones, such as 
graded algebras of symmetric operads and 
relative algebras of monads on $\Set$ introduced by 
Hino, Kobayashi, Hasuo and Jacobs.
We also consider morphisms between notions of algebraic theory, which are
a monoidal version of profunctors.
Every strong monoidal functor gives rise to an adjoint pair of such morphisms,
and provides a uniform way to establish isomorphisms between categories of 
models in different notions of algebraic theory.
A general structure-semantics adjointness result and a double categorical
universal property of categories of models are also shown.

In the second part of this thesis, we shift from the general study of 
algebraic structures, and focus on a particular algebraic 
structure: higher dimensional categories.
Higher dimensional categories arise in such diverse fields as 
topology, mathematical physics 
and theoretical computer science. 
On the other hand, the structure of higher dimensional categories is 
quite complex and even their definition is known to be 
subtle.
Among several existing definitions of higher dimensional categories,
we choose to look at the one proposed by Batanin and later 
refined by Leinster.
In Batanin and Leinster's approach, 
higher dimensional categories are defined as models 
of a certain generalised operad, hence it falls within the
unified framework developed in the first part of this thesis.
Batanin and Leinster's definition has also been used 
by van den Berg, Garner and Lumsdaine to 
describe the higher dimensional 
structures of types in Martin-L{\"o}f intensional type theory.
We show that the notion of extensive category plays a central role
in Batanin and Leinster's definition.
Using this, we generalise their definition by 
allowing enrichment over any locally presentable extensive category.

\chapter*{Acknowledgement}
First, I wish to express my sincere gratitude to my supervisor, Masami Hagiya,
for his support and valuable advice.
I also thank Ichiro Hasuo, who has been my supervisor until March 
2017, for his continuous encouragement.
I am grateful to Shin-ya Katsumata for having weekly meetings with me, 
in which we have discussed various topics.

I thank Thomas Cottrell and John Power of the University of Bath for the
pleasant collaboration which forms the second part of this thesis.
I have learned a lot from hours of critical  discussions with them.

I would like to thank my family and friends, as always,
for their warm support and cheerful encouragement.

The financial support during my doctor course 
by the IST-RA program of the Graduate School of Information Science and 
Technology, the University of Tokyo, is gratefully acknowledged. 

\tableofcontents % 目次
%\listoffigures % 図目次
%\listoftables % 表目次
%\lstlistoflistings % ソースコード目次
%-------------------
\mainmatter %% 本文
\chapter{Introduction}
\section{Algebraic structures in mathematics and computer science}
\label{sec_alg_str_math_cs}
Algebras permeate both pure and applied mathematics.
Important types of algebras, such as vector spaces, groups and rings, arise 
naturally in many branches of mathematical sciences and 
it would not be an exaggeration to say that algebraic structures 
are one of the most universal and fundamental structures in mathematics.

In computer science, too, concepts related to algebraic structures play an 
essential role.
For example, in programming language theory, we can find relationship to 
algebraic structures via the study of \emph{computational effects}.
Let us start with an explanation of computational effects.
Computer programs may roughly be thought of as mathematical functions, mapping 
an input to the result of computation.
However, this understanding is too crude and in reality programs often show 
non-functional behaviours;
for example, if a program interacts with the memory of the computer,
then an input to the program alone might not suffice to determine its output
(one has to know the initial state of the memory as well).
Such non-functional behaviours of programs are called {computational 
effects}.
It has been known since the work by Moggi~\cite{Moggi_comp_lambda} 
that computational effects can be modelled uniformly using the notion of 
\emph{monad}.
As we shall see later, a monad can be thought of as a specification of 
a type of algebras.

More recently, another approach to computational effects has been 
proposed by Plotkin and Power \cite{Plotkin_Power}.
In this approach, computational effects are modelled by
\emph{Lawvere theories} \cite{Lawvere_thesis} instead of monads;
a Lawvere theory can also be thought of as a specification of a type of 
algebras, akin to \emph{equational 
theory} in universal algebra.
Constructions on Lawvere theories originally developed in the study of 
algebraic structures, such as tensors and sums of Lawvere 
theories~\cite{Freyd}, have been shown to be capable of modelling combinations 
of computational effects~\cite{Hyland_Plotkin_Power},
and the resulting Lawvere theories can verify equivalences of programs 
which are crucial in program optimisation \cite{Kammar_Plotkin}.

As another example of algebraic structures arising in computer science,
one can point out a deep connection of higher dimensional categories and
the Martin-L{\"o}f (intensional) type theory~\cite{Martin-Lof}.
Higher dimensional categories may be thought of as particularly intricate types 
of algebras, defined by a number of complex operations and 
equations.
Their importance was first recognised in homotopy theory~\cite{Grothendieck},
because they naturally arise as 
higher dimensional versions of the fundamental groupoids of topological spaces.
It has been shown that equality types in the Martin-L{\"o}f type theory
endow a weak $\omega$-category structure to each 
type~\cite{Hofmann_Streicher,vandenBerg_Garner,Lumsdaine}.
This observation has led the researchers to seek more profound connections of 
type theory and homotopy theory, bridged by higher category theory,
culminating in the introduction and recent intensive study of 
\emph{homotopy type theory}~\cite{HoTT}.

\medskip

This thesis studies foundational issues around algebraic structures.
In the first part of the thesis, we investigate metamathematical aspects of
algebraic structures, by developing a unified framework for notions of 
algebraic theory.
In the second part, we focus on a particular type of algebras,
weak $n$-dimensional categories for each natural number $n$,
and generalise a known definition.
We now turn to more detailed outlines of these parts of the thesis.

\section{Unifying notions of algebraic theory}
A type of algebras, such as groups, is normally specified by a family of 
operations and a family of equational axioms. 
We call such a specification of a type of algebras an \emph{algebraic theory},
and call a background theory for a type of algebraic theories a \emph{notion of 
algebraic theory}. 
In order to capture various types of algebras, a variety of notions of 
algebraic theory have been introduced.
Examples include {universal algebra} \cite{Birkhoff_abst_alg},
symmetric and non-symmetric operads \cite{May_loop},
generalised operads (also called clubs)
\cite{Burroni_T_cats,Kelly_club_data_type,Hermida_representable,Leinster_book}, 
PROPs  and PROs \cite{MacLane_cat_alg},
and monads \cite{Eilenberg_Moore,Linton_equational};
we shall review these notions of algebraic theory in 
Chapter~\ref{chap_notions_ex}.

Notions of algebraic theory all aim to provide a means to define algebras, but 
they attain this goal in quite distinct manners.
The diversity of the existing notions of algebraic theory 
leaves one wonder what, if any, is a formal core or essence shared by them.
Our main aim in the first part of this thesis is to provide an answer to this 
question, by developing a unified 
framework for notions of algebraic theory.

The starting point of our approach is quite simple. 
We identify a notion of algebraic theory with an (arbitrary) monoidal category,
and algebraic theories in a notion of algebraic theory with monoid objects in
the corresponding monoidal category.
As we shall review in Section~\ref{subsec:alg_thy_as_monoids},
it has been observed that each type of algebraic theories we have 
listed above can be characterised as monoid objects in a suitable monoidal 
category.
From now on let us adopt the terminology to be introduced in 
Chapter~\ref{chap_framework}:
we call a monoidal category a \emph{metatheory} and a monoid object therein
a \emph{theory}, to remind ourselves of our intention. 

In order to formalise the semantical aspect of notions of algebraic theory---by
which we mean definitions of \emph{models} (= {algebras}) of an algebraic 
theory, their homomorphisms, and so on---we introduce the concept of 
\emph{metamodel}.
Metamodels are a certain categorical structure defined relative to   
a metatheory $\cat{M}$ and a category $\cat{C}$, 
and are meant to capture a 
\emph{notion of model} of an algebraic theory, i.e., what it means to take a 
model of a theory in $\cat{M}$ 
in the category $\cat{C}$.
A model of an algebraic theory is always given relative to some notion of 
model, even though usually it is not recognised explicitly.
We shall say more about the idea of notions of model at the beginning of 
Section~\ref{subsec:enrichment}.
A metamodel of a metatheory $\cat{M}$ in a category $\cat{C}$
generalise both an $\cat{M}$-category (as in enriched category theory) having 
the same set of objects as $\cat{C}$,
and a {(left) oplax action} of $\cat{M}$ on $\cat{C}$.
Indeed, as we shall see in Sections~\ref{subsec:enrichment} and 
\ref{subsec:oplax_action}, it has been observed that enrichments
(which we introduce as a slight generalisation of $\cat{M}$-categories)
and oplax actions can account for the standard semantics of the known
notions of algebraic theory.
Our concept of metamodel provides a unified account of the semantical aspects 
of notions of algebraic theory.

Metamodels of a fixed metatheory $\cat{M}$ naturally form a 
2-category $\MtMod{\cat{M}}$,
and we shall see that theories in $\cat{M}$ can be identified with
certain metamodels of $\cat{M}$ in the terminal category $1$.
This way we obtain a fully faithful 2-functor 
from the category $\Th{\cat{M}}$ of theories in $\cat{M}$ (which is identical 
to the category of monoid objects in $\cat{M}$)
to $\MtMod{\cat{M}}$.
A metamodel $\Phi$ of $\cat{M}$ in $\cat{C}$ provides a definition of model of 
a monoid object in $\cat{M}$ as an object of $\cat{C}$ with additional 
structure,
hence if we fix a metamodel $(\cat{C},\Phi)$ and a theory $\monoid{T}$,
we obtain the 
category of models $\Mod{\monoid{T}}{(\cat{C},\Phi)}$ equipped with 
the forgetful functor 
$U\colon\Mod{\monoid{T}}{(\cat{C},\Phi)}\longrightarrow\cat{C}$.
By exploiting the 2-category $\MtMod{\cat{M}}$, the 
construction $\Mod{-}{-}$ of categories of models may be expressed as 
the following composition 
\begin{equation}
\label{eqn:Mod_as_hom}
\begin{tikzpicture}[baseline=-\the\dimexpr\fontdimen22\textfont2\relax ]
      \node (1) at (0,1.5)  {$\Th{\cat{M}}^\op\times \MtMod{\cat{M}}$};
      \node (2) at (0,0)  {$\MtMod{\cat{M}}^\op\times \MtMod{\cat{M}}$};
      \node (3) at (0,-1.5) {$\tCAT$,};
      \draw[->] (1) to node[auto,labelsize]{inclusion} (2);
      \draw[->] (2) to node[auto,labelsize] {$\MtMod{\cat{M}}(-,-)$} (3);
\end{tikzpicture}
\end{equation}
where $\MtMod{\cat{M}}(-,-)$ is the hom-2-functor
and $\tCAT$ is a 2-category of categories.

We also introduce morphisms (and 2-cells) between metatheories
(Section \ref{subsec:morphism_metatheory}).
Such morphisms are a monoidal version of \emph{profunctors}.
The principal motivation of the introduction of morphisms of metatheories is 
to compare different notions of algebraic theory,
and indeed our morphisms of metatheories induce 2-functors between the 
corresponding 2-categories of metamodels.
Analogously to the well-known fact for profunctors
that any functor induces an adjoint pair of profunctors,
we see that any strong monoidal functor $F$ induces an adjoint pair 
$F_\ast\dashv F^\ast$ of morphisms of metatheories.
Therefore, whenever we have a strong monoidal functor 
$F\colon \cat{M}\longrightarrow\cat{N}$ between metamodels,
we obtain a 2-adjunction
\begin{equation}
\label{eqn:2-adjunction}
\begin{tikzpicture}[baseline=-\the\dimexpr\fontdimen22\textfont2\relax ]
      \node (L) at (0,0)  {$\MtMod{\cat{M}}$};
      \node (R) at (4.5,0)  {$\MtMod{\cat{N}}$.};
      \draw[->,transform canvas={yshift=5pt}]  (L) to node[auto,labelsize] 
      {$\MtMod{F_\ast}$} (R);
      \draw[<-,transform canvas={yshift=-5pt}]  (L) to 
      node[auto,swap,labelsize] {$\MtMod{F^\ast}$} 
      (R);
      \node[rotate=-90,labelsize] at (2.25,0)  {$\dashv$};
\end{tikzpicture} 
\end{equation}
Now, the strong monoidal $F$ also induces a functor 
\[
\Th{F}\colon\Th{\cat{M}}\longrightarrow\Th{\cat{N}},
\]
which is in fact a restriction of $\MtMod{F_\ast}$.
This implies that, immediately from the description (\ref{eqn:Mod_as_hom})
of categories of models and the 2-adjointness (\ref{eqn:2-adjunction}),
for any $\monoid{T}\in\Th{\cat{M}}$ and $(\cat{C},\Phi)\in\MtMod{\cat{N}}$,
we have a canonical isomorphism of categories
\begin{equation}
\label{eqn:iso_models}
\Mod{\Th{F}(\monoid{T})}{(\cat{C},\Phi)}\cong
\Mod{\monoid{T}}{\MtMod{F^\ast}(\cat{C},\Phi)}.
\end{equation}
In fact, as we shall see, the action of $\MtMod{-}$ on morphisms of metatheories
preserves the ``underlying categories'' of metamodels.
So $\MtMod{F^\ast}(\cat{C},\Phi)$ is also a metamodel of $\cat{M}$ in $\cat{C}$,
and we have an isomorphism of categories \emph{over} $\cat{C}$
(that is, the isomorphism (\ref{eqn:iso_models}) commutes with the forgetful 
functors).

The above argument gives a unified conceptual account for a range of known 
results on the compatibility of semantics of notions of algebraic theory.
For example, it is known that any Lawvere theory $\monoid{T}$
induces a monad $\monoid{T}'$ on $\Set$
in a way such that the models of $\monoid{T}$ and $\monoid{T'}$ in $\Set$
(with respect to the standard notions of model) coincide;
this result follows from the existence of a natural strong monoidal functor 
between the metatheories corresponding to Lawvere theories and monads on $\Set$,
together with the simple observation that the induced 2-functor between
the 2-categories of metamodels preserves the standard metamodel.
This and other examples will be treated in Section~\ref{sec:comparing}.

\medskip

In Chapter~\ref{chap:str_sem} we study \emph{structure-semantics adjunctions} 
within our framework. 
If we fix a metatheory $\cat{M}$ and a metamodel $(\cat{C},\Phi)$ of $\cat{M}$,
we obtain a functor 
\begin{equation}
\label{eqn:semantics_from_Th}
\Th{\cat{M}}^\op\longrightarrow \CAT/\cat{C}
\end{equation}
by mapping a theory $\monoid{T}$ in $\cat{M}$  to the category of models 
$\Mod{\monoid{T}}{(\cat{C},\Phi)}$ equipped with the forgetful functor 
into $\cat{C}$.
The functor (\ref{eqn:semantics_from_Th}) is sometimes called the 
\emph{semantics functor},
and it has been observed for many notions of algebraic theory that 
this functor (or an appropriate variant of it) admits a left adjoint
called the \emph{structure functor} 
\cite{Lawvere_thesis,Linton_equational,Linton_outline,Dubuc_Kan,Street_FTM,Avery_thesis}.
The idea behind the structure functor is as follows.
One can understand a functor $V\colon\cat{A}\longrightarrow\cat{C}$
into $\cat{C}$ as specifying an additional structure (in a broad sense) on 
objects in $\cat{C}$, by viewing $\cat{A}$ as the category of $\cat{C}$-objects 
equipped with that structure, and $V$ as the forgetful functor.
The structure functor then maps $V$ to the best approximation of that structure
by theories in $\cat{M}$.
Indeed, if (\ref{eqn:semantics_from_Th}) is fully faithful (though
this is not always the case), then the structure functor reconstructs the 
theory from
its category of models.

We cannot get a left adjoint to the functor 
(\ref{eqn:semantics_from_Th}) for an arbitrary metatheory $\cat{M}$ and its 
metamodel $(\cat{C},\Phi)$.
In order to get general structure-semantics adjunctions,
we extend the category $\Th{\cat{M}}$ of theories in $\cat{M}$
to the category $\Th{\widehat{\cat{M}}}$ of theories in the metatheory 
$\widehat{\cat{M}}=[\cat{M}^\op,\SET]$ equipped with the \emph{convolution 
monoidal structure}~\cite{Day_thesis}.
We show in Theorem~\ref{thm:str_sem_small} that 
the structure-semantics adjunction
\[
\begin{tikzpicture}[baseline=-\the\dimexpr\fontdimen22\textfont2\relax ]
      \node(11) at (0,0) 
      {$\Th{\widehat{\cat{M}}}^\op$};
      \node(22) at (4,0) {$\CAT/\cat{C}$};
  
      \draw [->,transform canvas={yshift=5pt}]  (22) to node 
      [auto,swap,labelsize]{$\Str$} (11);
      \draw [->,transform canvas={yshift=-5pt}]  (11) to node 
      [auto,swap,labelsize]{$\Sem$} (22);
      \path (11) to node[midway](m){} (22); 

      \node at (m) [labelsize,rotate=90] {$\vdash$};
\end{tikzpicture}
\] 
exists for any metatheory $\cat{M}$ and its metamodel $(\cat{C},\Phi)$.

\medskip

We conclude the first part of this thesis in Chapter~\ref{chap:double_lim},
by giving a universal characterisation of categories of models in our 
framework.
It is known that the Eilenberg--Moore categories (= categories of models) of 
monads can be characterised by a 2-categorical universal property in the 
2-category $\tCAT$ of categories~\cite{Street_FTM}.
We show in Theorem~\ref{thm:double_categorical_univ_property}
that our category of models admit a similar universal characterisation,
but instead of inside the 2-category $\tCAT$, inside the \emph{pseudo double 
category} $\PPROF$ of categories, functors, profunctors and natural 
transformations.
The notion of pseudo double category, as well as $\PPROF$ itself, 
was introduced by Grandis and Par{\'e} \cite{GP1}.
In the same paper they also introduced the notion of \emph{double limit},
a suitable limit notion in (pseudo) double categories.
The double categorical universal property that our categories of models 
enjoy can also be formulated in terms of double limits;
see Corollary~\ref{cor:Mod_as_dbl_lim}.

\section{Higher dimensional category theory}
Higher dimensional category theory is a relatively young field.
It studies higher dimensional generalisations of categories,
such as \emph{strict $n$-categories} and \emph{weak $n$-categories}
for $n\in\N\cup\{\omega\}$; in this thesis we shall only consider 
the case where $n\in\N$.

Let us start with the description of the simpler strict $n$-categories.
A strict $n$-category has 0-cells, which we draw as
\[\bullet,\]
1-cells lying between pairs of 0-cells 
\[\bullet\longrightarrow\bullet,\]
2-cells lying between pairs of parallel 1-cells
\[\begin{tikzpicture}[baseline=-\the\dimexpr\fontdimen22\textfont2\relax ]
      \node(21) at (0,0) {$\bullet$};
      \node(22) at (1.5,0) {$\bullet,$};

      \draw [->,bend left=30]  (21) to node (2u) {} (22);      
      \draw [->,bend right=30] (21) to node (2b) {} (22); 
      
      \draw [->] (2u) to (2b);
\end{tikzpicture}
\]
and so on up to $n$-cells lying between pairs of parallel $(n-1)$-cells.
There are also various identity cells and composition operations of cells,
which are required to satisfy a number of equations.
One way to make this informal description of strict $n$-category precise 
without too much complication is to define it by induction on $n$.
That is, an $(n+1)$-category $\cat{A}$ may be given by a set $\ob{\cat{A}}$
of 0-cells (or objects), and for each pair $A,B\in\ob{\cat{A}}$ of 0-cells,
an $n$-category $\cat{A}(A,B)$,
together with a family of operations ($n$-functors) 
$j_A\colon 1\longrightarrow\cat{A}(A,A)$ 
and $M_{A,B,C}\colon \cat{A}(B,C)\times\cat{A}(A,B)\longrightarrow\cat{A}(A,C)$,
subject to the category axioms.
Using the notion of enriched category~\cite{Kelly:enriched},
we may give a succinct 
inductive definition of the category $\enCat{n}$ of small strict 
$n$-categories and (strict) $n$-functors as follows:
\begin{equation}
\label{eqn:induction_enCat}
\enCat{0}=\Set,\qquad\enCat{(n+1)}=\enCat{(\enCat{n})}.
\end{equation}
Here, the construction $\enCat{(-)}$ maps any monoidal category $\cat{V}$
to the category $\enCat{\cat{V}}$ of all small $\cat{V}$-categories
and $\cat{V}$-functors.
In the above definition, we always use the cartesian monoidal structure,
the category $\enCat{\cat{V}}$ having all finite products whenever $\cat{V}$ 
does.

The more general \emph{weak} $n$-categories may be obtained by modifying
the definition of strict $n$-category, replacing equational axioms by
coherent equivalences.
For $n=0$ and $1$ there is no difference between the strict and weak notions,
0-categories being sets and 1-categories being ordinary categories.
Weak 2-categories are known as \emph{bicategories}~\cite{Benabou_bicat}.
In a bicategory, the compositions $(h\circ g)\circ f$ and $h\circ(g\circ f)$
of 1-cells may not be equal;
instead there must be a designated invertible 2-cell $\alpha_{f,g,h}\colon
(h\circ g)\circ f\longrightarrow h\circ (g\circ f)$,
and these 2-cells are required to satisfy some \emph{coherence} axioms,
such as the pentagon axiom asserting the commutativity of
the diagram
\begin{equation}
\label{eqn:pentagon}
\begin{tikzpicture}[baseline=-\the\dimexpr\fontdimen22\textfont2\relax ]
      \node (T) at (0,1.5)  {$((k\circ h)\circ g)\circ f$};
      \node (L) at (-3,0.2)  {$(k\circ (h\circ g))\circ f$};
      \node (R) at (3,0.2)  {$(k\circ h)\circ (g\circ f)$};
      \node (BL) at (-2,-1.5) {$k\circ ((h\circ g)\circ f)$};
      \node (BR) at (2,-1.5) {$k\circ(h\circ(g\circ f)).$};
      \draw[->] (T) to node[auto,swap,labelsize] {$\alpha_{g,h,k}\circ f$} (L);
      \draw[->] (L) to node[auto,swap,labelsize] {$\alpha_{f,h\circ g, k}$} 
      (BL);
      \draw[->] (BL) to node[auto,swap,labelsize] {$k\circ\alpha_{f,g,h}$} (BR);
      \draw[->] (T) to node[auto,labelsize] {$\alpha_{f,g,k\circ h}$} (R);
      \draw[->] (R) to node[auto,labelsize] {$\alpha_{g\circ f,h,k}$} (BR);
\end{tikzpicture}
\end{equation}
Weak 3-categories are known as 
\emph{tricategories}~\cite{Gordon_Power_Street_tricat}. 
In a tricategory we also have 2-cells like $\alpha_{f,g,h}$,
which are now required to be only equivalences rather than isomorphisms;
instead of the commutativity of the diagram (\ref{eqn:pentagon})
there is a designated invertible 3-cell (say, $\pi_{f,g,h,k}$) filling that 
diagram,
and these 3-cells must satisfy their own coherence axioms.

Although weak $n$-categories are fundamental, arising in topology (as the 
fundamental 
$n$-groupoid of a topological space~\cite{Grothendieck})
as well as in computer science (as the structure of a type in Martin-L{\"o}f
intensional type theory \cite{vandenBerg_Garner,Lumsdaine}),
they are quite complex structure.
Various authors have proposed definitions of weak $n$-category (see e.g., 
\cite{Leinster_survey}).
Among them, we shall focus on the one proposed by Batanin~\cite{Batanin_98}
and later modified by Leinster~\cite{Leinster_book}; we 
remark that it is their definition which is used in 
\cite{vandenBerg_Garner,Lumsdaine} to capture the structure of a type
in Martin-L{\"o}f type theory.

Let us describe Leinster's approach, since that is what we shall consider
in this thesis.
Leinster defines weak $n$-categories as \emph{$n$-graphs} with
additional algebraic structure, where
an $n$-graph consists of 0-cells, 1-cells lying between pairs of 0-cells,
2-cells lying between pairs of parallel 1-cells, and so on up to $n$-cells 
lying between pairs of parallel $(n-1)$-cells (and no operations).
Using the notion of enriched graph~\cite{Wolff}, the category $\enGph{n}$ of 
$n$-graphs can be given inductively as follows:
\begin{equation}
\label{eqn:induction_enGph}
\enGph{0}=\Set,\qquad\enGph{(n+1)}=\enGph{(\enGph{n})}.
\end{equation}
It is easily shown by induction that the canonical forgetful functor
$U^{(n)}\colon\enCat{n}\longrightarrow\enGph{n}$ has a left adjoint $F^{(n)}$, 
and the adjunction $F^{(n)}\dashv U^{(n)}$ generates a monad $\monoid{T}^{(n)}$
on $\enGph{n}$, the \emph{free strict $n$-category monad}.\footnote{The functor 
$U^{(n)}$ is in fact monadic, so Eilenberg--Moore algebras of 
$\monoid{T}^{(n)}$ are precisely strict $n$-categories.}
The monad $\monoid{T}^{(n)}$ is in fact \emph{cartesian}, and it is known that
any cartesian monad $\monoid{S}$ on a category $\cat{C}$
with finite limits defines a 
notion of algebraic theory (in the sense of the previous section), 
that of \emph{$\monoid{S}$-operads}.
An $\monoid{S}$-operad naturally takes models in the category $\cat{C}$;
thus in the current case, $\monoid{T}^{(n)}$-operads takes models in 
$\enGph{n}$.
Leinster then introduces the notion of \emph{contraction} on 
$\monoid{T}^{(n)}$-operads,
and defines a $\monoid{T}^{(n)}$-operad $\monoid{L}^{(n)}$ 
as the \emph{initial operad with a contraction}.
Finally, weak $n$-categories are defined to be models of $\monoid{L}^{(n)}$.
%See \cite{Leinster_book} for more details.

Leinster's definition of weak $n$-category 
starts from the category $\Set$ of sets,
in the sense that the key inductive definitions (\ref{eqn:induction_enCat})
and (\ref{eqn:induction_enGph}) have the base cases $\Set$.
Necessarily, certain properties of $\Set$ must be used to carry out the 
definition, but it has not been clear precisely which properties are used,
because many propositions in \cite{Leinster_book} 
are proved by set-theoretic manipulation.
Our main goal in the second part of this thesis is to clarify this.
The conclusion we get is that, among many properties that the category $\Set$
enjoys, 
\emph{extensivity}~\cite{Carboni_Lack_Walters,Cockett} and \emph{local 
presentability}~\cite{Gabriel_Ulmer,Adamek_Rosicky} are enough to 
carry out the definition of weak $n$-category.
We show this by generalising Leinster's definition, 
starting from an arbitrary extensive and locally presentable category $\cat{V}$
(again in the sense that we modify the base cases of (\ref{eqn:induction_enCat})
and (\ref{eqn:induction_enGph}), replacing $\Set$ by $\cat{V}$).
We call the resulting ``enriched'' weak $n$-categories
\emph{weak $n$-dimensional $\cat{V}$-categories}.
Examples of categories $\cat{V}$ of interest other than $\Set$ 
satisfying both extensivity and local presentability include the category 
$\omega$-$\Cpo$ of posets with 
sups of $\omega$-chains,
$\omega$-$\Cpo$-bicategories (weak 2-dimensional $\omega$-$\Cpo$-categories)
being used in the work \cite{Power_Tanaka} axiomatising 
binders \cite{FPT99}.

\medskip
In Chapter~\ref{chap:extensive},
we prepare for our main development by showing several properties
of extensive categories.
In particular, we show that if $\cat{V}$ is extensive, then so are 
$\enGph{\cat{V}}$
and $\enCat{\cat{V}}$ (for the latter category to make sense, we also have to 
assume that $\cat{V}$ has finite products),
thus illuminating the implicit induction in Leinster's approach.

Using properties on extensive categories shown in Chapter~\ref{chap:extensive},
in Chapter~\ref{chap:free_strict_n_cat_monad} we prove that 
even when we start from an arbitrary extensive category $\cat{V}$ with finite 
limits, we obtain an adjunction between the category 
$\enGph{\cat{V}}^{(n)}$ of $n$-dimensional
$\cat{V}$-graphs (enriching $\enGph{n}$) and the category 
$\enCat{\cat{V}}^{(n)}$ of strict $n$-dimensional $\cat{V}$-categories
(enriching $\enCat{n}$).
We moreover show that the resulting monad $\monoid{T}^{(n)}$ on 
$\enGph{\cat{V}}^{(n)}$ is cartesian.
This allows us to consider $\monoid{T}^{(n)}$-operads.

In Chapter~\ref{chap:def_weak_n_V_cat}, we first
generalise Leinster's notion of contraction to the enriched case. 
Leinster's original definition of contraction was couched in purely set 
theoretic 
terms, so we adapt Garner's conceptual reformulation \cite{Garner_homotopy} of 
it (with homotopy theoretic background~\cite{Garner_understanding}).
This way we may give a meaning to the phrase \emph{$\monoid{T}^{(n)}$-operad 
with a contraction} for an arbitrary extensive category 
$\cat{V}$ with finite limits.
Finally, to show the existence of the \emph{initial} such,
we assume that our  $\cat{V}$ is locally presentable as well.
Under this additional assumption we prove that the initial 
$\monoid{T}^{(n)}$-operad with a contraction $\monoid{L}^{(n)}$ exists,
and we define weak $n$-dimensional $\cat{V}$-categories
to be models of $\monoid{L}^{(n)}$.

\section{Set theoretic conventions}
\label{sec:foundational_convention}
As is typical in category theory, in this thesis we will occasionally
have to consider sets larger than those one usually encounters 
in other areas of mathematics and computer science.
In order to deal with them, we shall assume 
the existence of a few universes.
Roughly speaking, a universe $\univ{U}$ is a set with a sufficiently strong
closure property so that one can perform a range of set theoretic operations 
on elements in $\univ{U}$ without having to worry about the 
resulting set popping out of $\univ{U}$.
For example, if a group $G$ is an element of $\univ{U}$
(that is, the tuple consisting of the underlying set, the unit element, the 
inverse operation and the multiplication operation of $G$, is in $\univ{U}$),
so are all subgroups of $G$, quotient groups of $G$, powers of $G$ by
elements of $\univ{U}$, etc.
Note, however, that the set of {all} groups in $\univ{U}$ is 
{not} in $\univ{U}$. 

Although we will never refer to the details of the definition
of universe in this thesis, we state it here for the sake of completeness.
\begin{definition}[{\cite[Definition~1.1.1]{KS:CS}}]
A set $\univ{U}$ is called a \defemph{universe} if the following hold:
\begin{itemize}
\item if $x\in\univ{U}$ and $y\in x$, then $y\in\univ{U}$;
\item if $x\in\univ{U}$, then $\{x\}\in\univ{U}$;
\item if $x\in\univ{U}$, then $\mathcal{P}(x)=\{y\mid y\subseteq 
x\}\in\univ{U}$;
\item if $I\in\univ{U}$ and $(x_i)_{i\in I}$ is an $I$-indexed family of 
elements of $\univ{U}$, then $\bigcup_{i\in I} x_i\in\univ{U}$;
\item $\N\in\univ{U}$, where $\N=\{0,1,\dots\}$ and for 
all $n\in\N$, $n=\{0,1,\dots, n-1\}$.
\qedhere
\end{itemize}
\end{definition}

The following axiom of universes is often assumed in 
addition to ZFC in the literature.

\begin{axiom}
For each set $x$, there exists a universe $\univ{U}$ 
such that $x\in\univ{U}$.
\end{axiom}

In fact, in this thesis we will only need three universes $\univ{U}_1$,
$\univ{U}_2$ and $\univ{U}_3$ with $\univ{U}_1\in\univ{U}_2\in\univ{U}_3$. 
We now fix these universes once and for all.

Let $\univ{U}$ be a universe.
We define several size-regulating conditions 
on sets and other mathematical structures
in reference to $\univ{U}$.

\begin{itemize}
\item A set is said to be 
\defemph{in $\univ{U}$} if 
it is an element of $\univ{U}$.
\end{itemize}
%Similarly, we say that a mathematical structure 
%(a group, a topological space, a category, etc.)
%is \defemph{in $\univ{U}$}, if it, regarded as a suitable tuple consisting
%of its underlying set and additional structure, is an element of $\univ{U}$.
%Provided that we already know the notion of isomorphism between 
%mathematical structures of a given type, 
%we say that a mathematical structure of that type is 
%\defemph{$\univ{U}$-small} if it is isomorphic to one in $\univ{U}$.

In this thesis, a category is always assumed to have sets of objects
and of morphisms (rather than \emph{proper classes} of them).
We say that a category $\cat{C}$ is
\begin{itemize}
\item \defemph{in $\univ{U}$}
if the tuple $(\ob{\cat{C}},(\cat{C}(A,B))_{A,B\in\ob{\cat{C}}},
(\id{C}\in\cat{C}(C,C))_{C\in\ob{\cat{C}}},(\circ_{A,B,C}\colon$
$\cat{C}(B,C)\times\cat{C}(A,B)\longrightarrow\cat{C}(A,C))_{A,B,C\in
\ob{\cat{C}}})$, consisting of the data for $\cat{C}$,
is an element of $\univ{U}$;
\item \defemph{locally in $\univ{U}$}
if for each $A,B\in\ob{\cat{C}}$, 
the hom-set $\cat{C}(A,B)$ is in $\cat{U}$.
\end{itemize}
We also write $C\in\cat{C}$ for $C\in\ob{\cat{C}}$.

We extend these definitions to other mathematical structures.
For example, a group is said to be \defemph{in $\univ{U}$} if it is an element
of $\univ{U}$, 
a 2-category is \defemph{locally in $\univ{U}$}  
if all its hom-categories are in $\univ{U}$, and so on.

\medskip

Recall the universes $\univ{U}_1$, $\univ{U}_2$ and $\univ{U}_3$
we have fixed above.
\begin{convention}
\label{conv:size}
A set or other mathematical structure (group, category, etc.) 
is said to be:
\begin{itemize}
\item \defemph{small} if it is in $\univ{U}_1$;
\item \defemph{large} if it is in $\univ{U}_2$;
\item \defemph{huge} if it is in $\univ{U}_3$.
\end{itemize}
Sets and other mathematical structures are often assumed to be small by 
default, even when we do not say so explicitly.

A category (or a 2-category) is said to be:
\begin{itemize}
\item \defemph{locally small} if it is large and locally in $\univ{U}_1$;
\item \defemph{locally large} if it is huge and locally in $\univ{U}_2$.\qedhere
\end{itemize}
\end{convention}
In the following, we mainly talk about the 
size-regulating conditions using the terms \emph{small}, \emph{large} 
and \emph{huge},
avoiding direct references to the universes $\univ{U}_1$, $\univ{U}_2$
and $\univ{U}_3$.

We shall use the following basic (2-)categories throughout this thesis.
\begin{itemize}
\item $\Set$, the (large) category of all small sets and functions.
\item $\SET$, the (huge) category of all large sets and functions.
\item $\Cat$, the (large) category of all small categories and functors.
\item $\CAT$, the (huge) category of all large categories and functors.
\item $\tCat$, the (large) 2-category of all small categories, functors and 
natural 
transformations.
\item $\tCAT$, the (huge) 2-category of all large categories, functors and 
natural 
transformations.
\item $\twoCAT$, the 2-category of all huge 2-categories,
2-functors and 2-natural transformations.
\end{itemize}

\section{2-categorical notions}
\label{sec:2-cat_notions}
In order to fix the terminology, we define various 2-categorical notions here.

A 2-functor $F\colon\tcat{A}\longrightarrow\tcat{B}$ is called:
\begin{itemize}
\item \defemph{fully faithful} iff for each $A,A'\in\tcat{A}$,
$F_{A,A'}\colon\tcat{A}(A,A')\longrightarrow\tcat{B}(FA,FA')$ is an isomorphism 
of categories;
\item \defemph{locally an equivalence} iff for each $A,A'\in\tcat{A}$,
$F_{A,A'}\colon\tcat{A}(A,A')\longrightarrow\tcat{B}(FA,FA')$ is an equivalence 
of categories;
\item \defemph{locally fully faithful} iff for each $A,A'\in\tcat{A}$,
$F_{A,A'}\colon\tcat{A}(A,A')\longrightarrow\tcat{B}(FA,FA')$ is fully faithful;
\item \defemph{locally faithful} iff for each $A,A'\in\tcat{A}$,
$F_{A,A'}\colon\tcat{A}(A,A')\longrightarrow\tcat{B}(FA,FA')$ is faithful;
\item \defemph{bijective on objects} iff 
$\ob{F}\colon\ob{\tcat{A}}\longrightarrow\ob{\tcat{B}}$ is a bijection;
\item \defemph{essentially surjective (on objects)} iff for each $B\in\tcat{B}$,
there exists $A\in\tcat{A}$ and an isomorphism $FA\cong B$ in $\tcat{B}$;
\item an \defemph{isomorphism} iff it is bijective on objects and fully 
faithful;
\item an \defemph{equivalence} iff it is essentially surjective and fully 
faithful.
\end{itemize}

For a 2-category $\tcat{B}$, let
\begin{itemize}
\item $\tcat{B}^\op$ be the 2-category obtained by reversing 1-cells: 
$\tcat{B}^\op(A,B)=\tcat{B}(B,A)$;
\item $\tcat{B}^\co$ be the 2-category obtained by reversing 2-cells:
$\tcat{B}^\co(A,B)=\tcat{B}(A,B)^\op$;
\item $\tcat{B}^\coop$ be the 2-category obtained by reversing both 1-cells and 
2-cells: $\tcat{B}^\coop(A,B)=\tcat{B}(B,A)^\op$.
\end{itemize} 
We adopt the same notation for bicategories as well.

\part{A unified framework for notions of algebraic theory}

\chapter{Notions of algebraic theory}
\label{chap_notions_ex}
In almost every field of pure and applied mathematics, 
\emph{algebras} (in a broad sense) arise quite naturally 
in one way or another.
An algebra, typically, is a set equipped with 
a family of operations.
So for example the symmetric group of order five $\mathfrak{S}_5$
and the ring of integers $\mathbb{Z}$ are both algebras.
Structural similarities between important algebras have led to 
the introduction and study of various \emph{types of algebras},
such as 
monoids, groups, rings, vector spaces, lattices, 
Boolean algebras, and Heyting algebras.
A type of algebras is normally specified by a family of operations
and a family of equational axioms. 
We shall use the term \emph{algebraic theory} to refer to 
a specification of a type of algebras.

Subsequently, 
various authors have set out to develop \emph{notions of algebraic theory}.
A notion of algebraic theory is a background theory for 
a certain type of algebraic theories.
The most famous classical example of notions of algebraic theory
is Birkhoff's \emph{universal algebra}~\cite{Birkhoff_abst_alg}. 
%reviewed in Section~\ref{sec:univ_alg}.

There are several 
motivations behind the introduction of notions of algebraic theory.
First, by working at this level of generality,
one can prove theorems for various types of algebras
once and for all; for instance, the homomorphism theorems
in universal algebra (see e.g., \cite[Section II.6]{Burris_Sankappanavar}) 
generalise the homomorphism 
theorems for groups
to monoids, rings, lattices, etc.
Second, novel notions of algebraic theory have sometimes 
been proposed in order to 
set up powerful languages expressive enough to capture
interesting but intricate types of algebras.
This applies to (the topological versions of) 
symmetric and non-symmetric operads, used to define
up-to-homotopy topological commutative monoids and monoids~\cite{May_loop},
and to globular operads, by which a definition of weak $\omega$-category
is given~\cite{Batanin_98,Leinster_book}.

In this chapter we shall review several known notions of algebraic theory,
in order to provide motivation and background knowledge for our unified 
framework for notions of algebraic theory developed from 
Chapter~\ref{chap_framework} on.
The contents of this chapter are well-known to the specialists.

\section{Universal algebra}
\label{sec:univ_alg}
Universal algebra~\cite{Birkhoff_abst_alg} deals with types of algebras
defined by finitary operations and equations between them.
As a running example, let us consider 
\emph{groups}.
A group can be defined as a set $G$
equipped with an element $e^G\in G$ (the unit), and 
two functions $i^G\colon G\longrightarrow G$ (the inverse) and 
$m^G\colon G\times G\longrightarrow G$ (the multiplication),
satisfying the following axioms:
\begin{itemize}
\item for all $g_1\in G$, $m^G(g_1,e^G)=g_1$ (the right unit axiom);
\item for all $g_1\in G$, $m^G(g_1,i^G(g_1))=e^G$ (the right inverse axiom);
\item for all $g_1,g_2,g_3\in G$, $m^G(m^G(g_1,g_2),g_3)=m^G(g_1,m^G(g_2,g_3))$ 
(the associativity axiom).
\end{itemize}
(From these three axioms it follows that 
for all $g_1\in G$, $m^G(e^G,g_1)=g_1$ (the left unit axiom) and 
$m^G(i^G(g_1),g_1)=e^G$ (the left inverse axiom).)
This definition of group turns out to be an instance of 
the notion of \emph{presentation of an equational theory},
one of the most fundamental notions in universal algebra introduced below.

First we introduce the notion of \emph{graded set}, which provides a convenient 
language for clean development of universal algebra.
\begin{definition}
\label{def:graded_set}
\begin{enumerate}
\item An \defemph{($\NN$-)graded set} $\Gamma$ is a family 
$\Gamma=(\Gamma_n)_{n\in\NN}$ of sets
indexed by natural numbers $\NN=\{0,1,2,\dots\}$.
By an \defemph{element of $\Gamma$} we mean an element of the 
set $\coprod_{n\in\NN} \Gamma_n=\{\,(n,\gamma)\mid n\in\NN, \gamma\in 
\Gamma_n\,\}$.
We write $x\in\Gamma$ if $x$ is an element of $\Gamma$.
\item If $\Gamma=(\Gamma_n)_{n\in\NN}$ and $\Gamma'=(\Gamma'_n)_{n\in\NN}$
are graded sets, then a \defemph{morphism of graded sets} $f\colon 
\Gamma\longrightarrow
\Gamma'$ is a family of functions $f=(f_n\colon \Gamma_n\longrightarrow 
\Gamma'_n)_{n\in\NN}$.\qedhere
\end{enumerate}
\end{definition}

We can routinely extend the basic notions of set theory to graded sets.
For example, we say that a graded set $\Gamma'$ is a \defemph{graded subset}
of a graded set $\Gamma$ (written as $\Gamma'\subseteq \Gamma$)
if for each $n\in\NN$, $\Gamma'_n$ is a subset of $\Gamma_n$.
Given arbitrary graded sets $\Gamma$ and $\Gamma'$, their 
\defemph{cartesian product} (written as $\Gamma\times\Gamma'$) is defined by 
$(\Gamma\times\Gamma')_n=\Gamma_n\times\Gamma'_n$ for each $n\in\NN$.
An \defemph{equivalence relation} on a graded set $\Gamma$ is 
a graded subset $R\subseteq \Gamma\times\Gamma$ such that each $R_n\subseteq
\Gamma_n\times\Gamma_n$ is an equivalence relation on the set $\Gamma_n$.
Given such an equivalence relation $R$ on $\Gamma$, we can form the 
\defemph{quotient graded set} $\Gamma/R$ by 
setting $(\Gamma/R)_n=\Gamma_n/R_n$, the quotient set of $\Gamma_n$
with respect to $R_n$.
These notions will be used below.

A graded set can be seen as a \emph{(functional) signature}.
That is, we can regard a graded set $\Sigma$ as the signature whose set of 
$n$-ary functional symbols is given by $\Sigma_n$ for each $n\in\NN$.
We often use the symbol $\Sigma$ to denote a graded set when we want to
emphasise this aspect of graded sets, as in the following definition. 

\begin{definition}
\label{def:Sigma_alg}
Let $\Sigma$ be a graded set.
\begin{enumerate}
\item A \defemph{$\Sigma$-algebra} is a set $A$ equipped with,
for each $n\in\NN$ and $\sigma\in\Sigma_n$,
a function $\interp{\sigma}^A\colon A^n\longrightarrow A$
called the \defemph{interpretation of $\sigma$}.
We write such a $\Sigma$-algebra as 
$(A,(\interp{\sigma}^A)_{n\in\NN,\sigma\in\Sigma_n})$
or simply $(A,\interp{-}^A)$.
We often omit the superscript in $\interp{-}^A$.
\item If $(A,\interp{-}^A)$
and $(B,\interp{-}^{B})$
are $\Sigma$-algebras,
then a \defemph{$\Sigma$-homomorphism} from 
$(A,\interp{-}^A)$ to 
$(B,\interp{-}^{B})$
is a function $f\colon A\longrightarrow B$ such that for any $n\in\NN$,
$\sigma\in\Sigma_n$ and $a_1,\dots,a_n\in A$,
\[
f(\interp{\sigma}^A(a_1,\dots,a_n))=\interp{\sigma}^B(f(a_1),\dots,f(a_n))
\]
holds (that is, the diagram
\[
\begin{tikzpicture}[baseline=-\the\dimexpr\fontdimen22\textfont2\relax ]
      \node (TL) at (0,2)  {$A^n$};
      \node (TR) at (3,2)  {$B^n$};
      \node (BL) at (0,0) {$A$};
      \node (BR) at (3,0) {$B$};
      \draw[->] (TL) to node[auto,labelsize](T) {$f^n$} (TR);
      \draw[->]  (TR) to node[auto,labelsize] {$\interp{\sigma}^{B}$} (BR);
      \draw[->]  (TL) to node[auto,swap,labelsize] {$\interp{\sigma}^A$} (BL);
      \draw[->] (BL) to node[auto,labelsize](B) {$f$} (BR);
\end{tikzpicture} 
\]
commutes).\qedhere
\end{enumerate}
\end{definition}

As an example, let us consider the graded set $\Sigma^\group$
defined as $\Sigma^\group_0=\{e\}$, $\Sigma^\group_1=\{i\}$,
$\Sigma^\group_2=\{m\}$ and $\Sigma^\group_n=\emptyset$ for all $n\geq 3$.
Then the structure of a group is given by
that of a $\Sigma^\group$-algebra.
Note that to give an element $e^G\in G$ is equivalent to give a 
function $\interp{e}\colon 1\longrightarrow G$ where $1$ is 
a singleton set, and that for any set $G$, $G^0$ is a singleton set.
Also, between groups, the notions of group homomorphism and 
$\Sigma^\group$-homomorphism coincide.

However, not all $\Sigma^\group$-algebras are groups;
for a $\Sigma^\group$-algebra to be a group, the interpretations
must satisfy the group axioms.
Notice that all group axioms are {equations} between
certain expressions built from
variables and operations.
This is the fundamental feature shared by all algebraic structures
expressible in universal algebra.
The following notion of $\Sigma$-term defines {``expressions
built from variables and operations''}
relative to arbitrary graded sets $\Sigma$.

\begin{definition}
\label{def:Sigma_term}
Let $\Sigma$ be a graded set.
The graded set $T(\Sigma)=(T(\Sigma)_n)_{n\in\NN}$ of \defemph{$\Sigma$-terms}
is defined inductively as follows.
\begin{enumerate}
\item For each $n\in\NN$ and  $i\in\{1,\dots,n\}$, 
\[
x_i^{(n)}\in T(\Sigma)_n.
\]
We sometimes omit the superscript and write $x_i$ for $x_i^{(n)}$.
\item For each $n,k\in\NN$, $\sigma\in\Sigma_k$ and $t_1,\dots,t_k\in 
T(\Sigma)_n$,
\[
\sigma(t_1,\dots,t_k)\in T(\Sigma)_n.
\]
When $k=0$, we usually omit the parentheses in $\sigma()$
and write instead as $\sigma$.\qedhere
\end{enumerate}
\end{definition}

An immediate application of the inductive nature of the above
definition of $\Sigma$-terms is the canonical extension of the interpretation 
function $\interp{-}$
of a $\Sigma$-algebra from $\Sigma$ to $T(\Sigma)$.

\begin{definition}
\label{def:interp_Sigma_term}
Let $\Sigma$ be a graded set and 
$(A,\interp{-})$ be a 
$\Sigma$-algebra.
We define the \defemph{interpretation} $\interpp{-}$ of $\Sigma$-terms
recursively as follows.
\begin{enumerate}
\item For each $n\in\NN$ and $i\in\{1,\dots,n\}$,
\[
\interpp{x^{(n)}_i}\colon A^n\longrightarrow A
\]
is the $i$-th projection $(a_1,\dots,a_n)\longmapsto a_i$.
\item For each $n,k\in\NN$, $\sigma\in\Sigma_k$ and $t_1\dots,t_k\in 
T(\Sigma)_n$,
\[
\interpp{\sigma(t_1,\dots,t_k)}\colon A^n\longrightarrow A
\] 
maps $(a_1,\dots,a_n)\in A^n$ to 
$\interp{\sigma}(\interpp{t_1}(a_1,\dots,a_n),\dots,
\interpp{t_k}(a_1,\dots,a_n))$;
that is, $\interpp{\sigma(t_1,\dots,t_k)}$ is the following composite:
\[
\begin{tikzpicture}[baseline=-\the\dimexpr\fontdimen22\textfont2\relax ]
      \node (1) at (0,0)  {$A^n$};
      \node (2) at (3,0)  {$A^k$};
      \node (3) at (5,0) {$A.$};
      \draw[->] (1) to node[auto,labelsize]{$\langle 
      \interpp{t_1},\dots,\interpp{t_k}\rangle$} (2);
      \draw[->] (2) to node[auto,labelsize] {$\interp{\sigma}$} (3);
\end{tikzpicture} 
\]
\end{enumerate}
Note that for any $n\in\NN$ and $\sigma\in\Sigma_n$,
$\interp{\sigma}=\interpp{\sigma(x^{(n)}_1,\dots,x^{(n)}_n)}$.
Henceforth, for any $\Sigma$-term $t$
we simply write $\interp{t}$ for $\interpp{t}$ defined above.
\end{definition}

\begin{definition}
\label{def:Sigma_eq}
Let $\Sigma$ be a graded set.
An element of the graded set $T(\Sigma)\times T(\Sigma)$ is called a 
\defemph{$\Sigma$-equation}.
We write a $\Sigma$-equation $(n,(t,s))\in T(\Sigma)\times T(\Sigma)$ (that is,
$n\in\NN$ and $t,s\in T(\Sigma)_n$) as $t\approx_n s$ or $t\approx s$.
\end{definition}

\begin{definition}
\label{defn:univ_alg_pres_eq_thy}
A \defemph{presentation of an equational theory} $\pres{\Sigma}{E}$ 
is a pair consisting of:
\begin{itemize}
\item a graded set $\Sigma$ of \defemph{basic operations}, 
and
\item a graded set $E\subseteq T(\Sigma)\times T(\Sigma)$
of \defemph{equational axioms}.\qedhere
\end{itemize}
\end{definition}

\begin{definition}
\label{defn:univ_alg_model}
Let $\pres{\Sigma}{E}$ be a presentation of an equational theory.
\begin{enumerate}
\item A \defemph{model of $\pres{\Sigma}{E}$} is a $\Sigma$-algebra
$(A,\interp{-})$
such that for any $t\approx_n s\in E$,
$\interp{t}=\interp{s}$ holds.
\item A \defemph{homomorphism}
between models of $\pres{\Sigma}{E}$ is just
a $\Sigma$-homomorphism between the corresponding
$\Sigma$-algebras.\qedhere
\end{enumerate}
\end{definition}

Consider the presentation of an equational theory 
$\pres{\Sigma^\group}{E^\group}$, where 
\[E^\group_1=\{\,m(x^{(1)}_1,e)\approx x^{(1)}_1, \quad
m(x^{(1)}_1,i(x^{(1)}_1))\approx e\,\},
\]
\[ E^\group_3=\{\,m(m(x^{(3)}_1,x^{(3)}_2),x^{(3)}_3)\approx
m(x^{(3)}_1,m(x^{(3)}_2,x^{(3)}_3))\,\}\]
and $E^\group_n=\emptyset$ for all $n\in\NN\setminus \{1,3\}$.
Clearly, groups are the same as models of
$\pres{\Sigma^\group}{E^\group}$.
Many other types of algebras---indeed all examples we have mentioned in the 
first paragraph of this chapter---can be written as models of $\pres{\Sigma}{E}$
for a suitable choice of the presentation of an equational theory
$\pres{\Sigma}{E}$;
see any introduction to universal algebra 
(e.g.,~\cite{Burris_Sankappanavar})
for details.

\medskip

We conclude this section by reviewing the machinery of 
\emph{equational logic}, which enables us to investigate consequences of 
equational axioms without referring to their models.
We assume that the reader is familiar with the basics of 
mathematical logic, such as 
substitution of a term $t$ for a variable $x$ in a term $s$ 
(written as $s[x\mapsto t]$), simultaneous substitutions
(written as $s[x_1\mapsto t_1, \dots, x_k\mapsto t_k]$), and 
the notion of proof (tree) and its definition by inference rules.

\begin{definition}
\label{def:eq_logic}
Let $\pres{\Sigma}{E}$ be a presentation of an equational theory.
\begin{enumerate}
\item Define the set of \defemph{$\pres{\Sigma}{E}$-proofs} inductively
by the following inference rules.
Every $\pres{\Sigma}{E}$-proof is a finite rooted tree whose vertices 
are labelled by $\Sigma$-equations.
\begin{center}
\AxiomC{}
\LeftLabel{(\sc{Ax})\ }
\RightLabel{\ (if $t\approx_n s \in E$)}
\UnaryInfC{$t\approx_n s$}
\DisplayProof

\bottomAlignProof
\AxiomC{}
\LeftLabel{(\sc{Refl})\ }
\UnaryInfC{$t\approx_n t$}
\DisplayProof
\quad
\bottomAlignProof
\AxiomC{$t\approx_n s $}
\LeftLabel{(\sc{Sym})\ }
\UnaryInfC{$s\approx_n t$}
\DisplayProof
\quad
\bottomAlignProof
\AxiomC{$t\approx_n s$}
\AxiomC{$s\approx_n u$}
\LeftLabel{(\sc{Trans})\ }
\BinaryInfC{$t\approx_n u$}
\DisplayProof

\bottomAlignProof
\AxiomC{$s\approx_k s'$}
\AxiomC{$t_1\approx_n t'_1$}
\AxiomC{$\cdots$}
\AxiomC{$t_k\approx_n t'_k$}
\LeftLabel{(\sc{Cong})\ }
\QuaternaryInfC{$s[x^{(k)}_1 \mapsto t_1, \dots, x^{(k)}_k\mapsto t_k]\approx_n 
s'[x^{(k)}_1 \mapsto t'_1, \dots, x^{(k)}_k\mapsto t'_k]$}
\DisplayProof
\end{center}
\item A $\Sigma$-equation 
$t\approx_n s\in T(\Sigma)\times T(\Sigma)$ is called an
\defemph{equational theorem of $\pres{\Sigma}{E}$} if there exists 
a $\pres{\Sigma}{E}$-proof whose root is labelled by $t\approx_n s$.
We write \[
\pres{\Sigma}{E}\vdash t\approx_n s\]
to mean that $t\approx_n s$ is 
an equational theorem of $\pres{\Sigma}{E}$, and denote by 
$\overline{E}\subseteq T(\Sigma)\times T(\Sigma)$ the 
graded set of all equational theorems of $\pres{\Sigma}{E}$.
\qedhere
\end{enumerate}
\end{definition}

Equational logic is known to be both \emph{sound} and \emph{complete},
in the following sense.

\begin{definition}
\label{def:semantical_consequence_rel}
\begin{enumerate}
\item Let $\Sigma$ be a graded set and 
$(A,\interp{-})$ be a $\Sigma$-algebra.
For any $\Sigma$-equation $t\approx_n s\in T(\Sigma)\times T(\Sigma)$, we write 
\[
(A,\interp{-})\vDash t\approx_n s
\]
to mean $\interp{t}=\interp{s}$.
\item Let $\pres{\Sigma}{E}$ be a presentation of an equational theory.
For any $\Sigma$-equation $t\approx_n s\in T(\Sigma)\times T(\Sigma)$, we write
\[
\pres{\Sigma}{E}\vDash t\approx_n s
\] 
to mean that for any model 
$(A,\interp{-})$ of $\pres{\Sigma}{E}$,
$(A,\interp{-})\vDash t\approx_n s$.\qedhere
\end{enumerate}
\end{definition}

\begin{thm}
\label{thm:eq_logic_sound_complete}
Let $\pres{\Sigma}{E}$ be a presentation of an equational theory.
\begin{enumerate}
\item (Soundness) Let $t\approx_n s\in T(\Sigma)\times T(\Sigma)$.
If $\pres{\Sigma}{E}\vdash t\approx_n s$ then $\pres{\Sigma}{E}\vDash 
t\approx_n s$.
\item (Completeness) Let $t\approx_n s\in T(\Sigma)\times T(\Sigma)$.
If $\pres{\Sigma}{E}\vDash t\approx_n s$ then $\pres{\Sigma}{E}\vdash 
t\approx_n s$.
\end{enumerate}
\end{thm}
\begin{proof}
The soundness theorem is proved by a straightforward induction over
$\pres{\Sigma}{E}$-proofs.
For the completeness theorem, see 
e.g.,~\cite[Corollary~1.5]{Johnstone_notes} or 
\cite[Section~II.14]{Burris_Sankappanavar}.
\end{proof}

\section{Clones}
\label{sec:clone}
The central notion we have introduced in the previous
section is that of 
\emph{presentation of an equational theory} 
(Definition~\ref{defn:univ_alg_pres_eq_thy}),
whose main purpose is to define its 
\emph{models} (Definition~\ref{defn:univ_alg_model}).
It can happen, however, that
two different presentations of equational theories 
define the ``same'' models, sometimes 
in a quite superficial manner.

For example, consider the following presentation of an equational theory
$\pres{\Sigma^{\group'}}{E^{\group'}}$:
\[
\Sigma^{\group'}=\Sigma^\group,
\]
\begin{multline*}
E^{\group'}_1=\{\,m(x^{(1)}_1,e)\approx x^{(1)}_1,\quad
m(e,x^{(1)}_1)\approx x^{(1)}_1, \\
m(x^{(1)}_1,i(x^{(1)}_1))\approx e,\quad
m(i(x^{(1)}_1),x^{(1)}_1)\approx e\,\},
\end{multline*}
\[
E^{\group'}_n=E^{\group}_n\quad\text{ for all }n\in\NN\setminus\{1\}.
\]
It is a classical fact that a group can be defined either 
as a model of $\pres{\Sigma^\group}{E^\group}$ or
as a model of $\pres{\Sigma^{\group'}}{E^{\group'}}$.
Indeed, we may add arbitrary {equational theorems} of 
$\pres{\Sigma^\group}{E^\group}$, such as
$i(i(x_1))\approx x_1$, $i(m(x_1,x_2))\approx m(i(x_2),i(x_1))$
and $x_1\approx x_1$,
as additional equational axioms and still obtain the groups as the models.

As another example, let us consider the presentation of an equational 
theory $\pres{\Sigma^{\group''}}{E^{\group''}}$ defined as:
\[
\Sigma^{\group''}_0=\{e,e'\}, \quad
\Sigma^{\group''}_n=\Sigma^\group_n \quad\text{ for all }n\in\NN\setminus\{0\},
\]
\[
E^{\group''}_0=\{e\approx e'\}, \quad E^{\group''}_n=E^\group_n\quad
\text{ for all }n\in\NN\setminus\{0\}.
\]
To make a set $A$ into a model of $\pres{\Sigma^{\group''}}{E^{\group''}}$,
formally we have to specify two elements $\interp{e}$ and $\interp{e'}$
of $A$, albeit they are forced to be equal and play the role of unit
with respect to the group structure determined by $\interp{m}$.
We cannot quite say that models of $\pres{\Sigma^{\group''}}{E^{\group''}}$
are \emph{equal} to models of $\pres{\Sigma^\group}{E^\group}$,
since their data differ;
however, it should be intuitively clear that there is no point
in distinguishing them.
(In precise mathematical terms, our claim of the ``sameness''
amounts to the existence of an isomorphism of categories between 
the categories of 
models of $\pres{\Sigma^\group}{E^\group}$
and of models of $\pres{\Sigma^{\group''}}{E^{\group''}}$
preserving the underlying sets of models, i.e., 
commuting with the forgetful functors into $\Set$.)

\medskip

A presentation of an equational theory has much freedom in choices both of 
basic operations and of equational axioms.
It is really a \emph{presentation}.
In fact, there is a notion which may be thought of as
an \emph{equational theory} itself,
something that a presentation of an equational theory presents;
it is called an \emph{(abstract) clone}.

\begin{definition}
\label{def:clone}
A \defemph{clone} $\monoid{T}$ consists of:
\begin{description}
\item[(CD1)] a graded set $T=(T_n)_{n\in\NN}$;
\item[(CD2)] for each $n\in\NN$ and $i\in\{1,\dots,n\}$, an element
\[
p^{(n)}_i\in T_n;
\]
\item[(CD3)] for each $k,n\in\NN$, a function
\[
\circ^{(n)}_k\colon T_k\times (T_n)^k\longrightarrow T_n
\]
whose action on an element $(\phi,\theta_1,\dots,\theta_k)\in T_k\times (T_n)^k$
we write as $\phi\circ^{(n)}_k(\theta_1,\dots,\theta_k)$
or simply as $\phi\circ(\theta_1,\dots,\theta_k)$;
\end{description}
satisfying the following equations:
\begin{description}
\item[(CA1)] for each $k,n\in\NN$, $j\in\{1,\dots,k\}$ and 
$\theta_1,\dots,\theta_k\in T_n$, 
\[
p^{(k)}_j\circ^{(n)}_k(\theta_1,\dots,\theta_k) = \theta_j;
\]
\item[(CA2)] for each $n\in\NN$, $\theta\in T_n$,
\[
\theta\circ^{(n)}_n (p^{(n)}_1,\dots,p^{(n)}_n) = \theta;
\]
\item[(CA3)] for each $l,k,n\in\NN$, $\psi\in T_l$, $\phi_1,\dots,\phi_l\in
T_k$, $\theta_1,\dots,\theta_k\in T_n$, 
\begin{multline*}
\psi\circ^{(k)}_l\big(\phi_1\circ^{(n)}_k(\theta_1,\dots,\theta_k),\ 
\dots,\ \phi_l\circ^{(n)}_k(\theta_1,\dots,\theta_k)\big)
\\=
\big(\psi\circ^{(k)}_l(\phi_1,\dots,\phi_l)\big)\circ^{(n)}_k 
(\theta_1,\dots,\theta_k).
\end{multline*}
\end{description}
Such a clone is written as 
$\monoid{T}=(T,(p^{(i)}_n)_{n\in\NN,i\in\{1,\dots,n\}},
(\circ^{(n)}_k)_{k,n\in\NN})$
or simply $(T,p,\circ)$.
\end{definition}

To understand the definition of clone,
it is helpful to draw some pictures known as \emph{string diagrams}
(cf.~\cite{Curien_operad,Leinster_book}).
Given a clone $\monoid{T}=(T,p,\circ)$,
let us denote an element $\theta$ of $T_n$
by a triangle with $n$ ``{input wires}'' and a single ``{output wire}'':
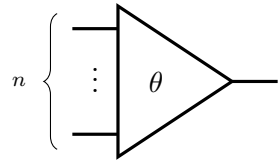
\begin{equation*}
%\label{eqn:triangle_theta}
\begin{tikzpicture}[baseline=-\the\dimexpr\fontdimen22\textfont2\relax ]
      \coordinate (TL) at (0,1);
      \coordinate (BL) at (0,-1);
      \coordinate (R) at (1.5,0);
      \path[draw,string] (TL)--(BL)--(R)--cycle;
      \node at (0.5,0) {$\theta$};
      \path[draw,string] (TL)+(0,-0.3)-- +(-0.6,-0.3);
      \node at (-0.3,0) {$\rvdots$};
      \path[draw,string] (BL)+(0,0.3)-- +(-0.6,0.3);
      \draw [decorate,decoration={brace,amplitude=5pt,mirror}]
      (TL)++(-0.8,-0.1) -- +(0,-1.8) node [midway,xshift=-0.5cm,labelsize] 
      {$n$};
      \path[draw,string] (R)-- +(0.6,0);
\end{tikzpicture}
\end{equation*}
The element $p^{(n)}_i$ in (CD2) may also be denoted by
\begin{equation}
\label{eqn:clone_p}
\begin{tikzpicture}[baseline=-\the\dimexpr\fontdimen22\textfont2\relax ]
      \coordinate (t) at (0,1);
      \coordinate (s) at (0,0.4);
      \coordinate (t3) at (0,0);
      \coordinate (t4) at (0,-0.4);
      \coordinate (t5) at (0,-1);
      \node [circle,draw,inner sep=0.15em,fill] at (0.7,-1) {};
      \node [circle,draw,inner sep=0.15em,fill] at (0.7,-0.4) {};  
      \node [circle,draw,inner sep=0.15em,fill] at (0.7,0.4) {};
      \node [circle,draw,inner sep=0.15em,fill] at (0.7,1) {};
      \draw [string] (0,1) -- (0.7,1);
      \node at (0.3,0.7) {$\rvdots$};
      \draw [string] (0,0.4) -- (0.7,0.4);
      \draw [string] (0,0) -- (2,0);
      \draw [string] (0,-1) -- (0.7,-1);
      \node at (0.3,-0.7) {$\rvdots$};
      \draw [string] (0,-0.4) -- (0.7,-0.4);
%      \node [labelsize, left of=t, node distance=0.4cm] {$1$};
%      \node [labelsize, left of=s, node distance=0.4cm] {$i-1$};
      \node [labelsize, left of=t3, node distance=0.42cm] {$(i\text{-th})$};
%      \node [labelsize, left of=t4, node distance=0.4cm] {$i+1$};
%      \node [labelsize, left of=t5, node distance=0.4cm] {$n$};
      \draw [decorate,decoration={brace,amplitude=5pt,mirror}]
      (-0.8,1.1) -- +(0,-2.2) node [midway,xshift=-0.5cm,labelsize] 
      {$n$};
\end{tikzpicture}
\end{equation}
and $\phi\circ^{(n)}_k(\theta_1,\dots,\theta_k)$ in (CD3) by
\begin{equation}
\label{eqn:clone_composition}
\begin{tikzpicture} [baseline={([yshift=-0.5ex]current bounding box.center)}]
%[baseline=-\the\dimexpr\fontdimen22\textfont2\relax ]
      \coordinate (TL) at (0,1.5);
      \coordinate (BL) at (0,-1.5);
      \coordinate (R) at (1.5,0);
      \path[draw,string] (TL)--(BL)--(R)--cycle;
      \node at (0.5,0) {$\phi$};
      \path[draw,string] (TL)+(0,-0.3)-- +(-0.6,-0.3);
      \node at (-0.3,0) {$\rvdots$};
      \path[draw,string] (BL)+(0,0.3)-- +(-0.6,0.3);
      \path[draw,string] (R)-- +(0.6,0);
  \begin{scope}[shift={(-2.1,1.2)}]
      \coordinate (TL) at (0,0.8);
      \coordinate (BL) at (0,-0.8);
      \coordinate (R) at (1.5,0);
      \path[draw,string] (TL)--(BL)--(R)--cycle;
      \node at (0.5,0) {$\theta_1$};
      \node at (0.5,-1.2) {$\rvdots$};
      \node at (-0.3,0) {$\rvdots$};
  \end{scope}
  \begin{scope}[shift={(-2.1,-1.2)}]
      \coordinate (TL) at (0,0.8);
      \coordinate (BL) at (0,-0.8);
      \coordinate (R) at (1.5,0);
      \path[draw,string] (TL)--(BL)--(R)--cycle;
      \node at (0.5,0) {$\theta_{k}$};
      \node at (-0.3,0) {$\rvdots$};
  \end{scope}
  \begin{scope}[shift={(-2.1,0)}]
      \node [circle,draw,inner sep=0.15em,fill] at (-1.3,0.5) {};
      \node [circle,draw,inner sep=0.15em,fill] at (-1.3,-0.5) {};
      \draw [string] (-2,0.5) -- (-1.4,0.5) ..controls (-1,0.5) and  (-0.9,1.7) 
      .. (-0.6,1.7)--(0,1.7);
      \draw [string] (-2,0.5) -- (-1.4,0.5) ..controls (-1,0.5) and  
      (-0.9,-0.7) .. (-0.6,-0.7)--(0,-0.7);
      \draw [string] (-2,-0.5) -- (-1.4,-0.5) ..controls (-1,-0.5) and  
      (-0.9,0.7) .. (-0.6,0.7)--(0,0.7);
      \draw [string] (-2,-0.5) -- (-1.4,-0.5) ..controls (-1,-0.5) and  
      (-0.9,-1.7) .. (-0.6,-1.7)--(0,-1.7);
      \node at (-1.7,0) {$\rvdots$};
      \draw [decorate,decoration={brace,amplitude=5pt}]
      (-2.2,-0.7) -- (-2.2,0.7) node [midway,xshift=-0.5cm,labelsize] 
      {$n$};
  \end{scope}
\end{tikzpicture}\ .
\end{equation}
Then the axioms (CA1)--(CA3) simply assert natural equations 
between the resulting ``circuits''.
For instance, (CA2) for $n=3$ reads:
\begin{equation*}
\begin{tikzpicture}[baseline=-\the\dimexpr\fontdimen22\textfont2\relax ]
      \draw [string] (0,0.3) -- (0.6,0.3) ..controls (0.9,0.3) and  (1.2,1.5) 
      .. (1.5,1.5)--(3.1,1.5)..controls (3.4,1.5) and  (3.7,0.3) .. 
      (4,0.3)--(4.6,0.3);
      \draw [string] (0,0) -- (0.6,0) ..controls (0.9,0) and  (1.2,1.2) .. 
      (1.5,1.2)--(2,1.2);
      \draw [string] (0,-0.3) -- (0.6,-0.3) ..controls (0.9,-0.3) and  
      (1.2,0.9) .. (1.5,0.9)--(2,0.9);
      \draw [string] (0,0.3) -- (2,0.3);
      \draw [string] (0,0) -- (4.6,0);
      \draw [string] (0,-0.3) -- (2,-0.3);
      \draw [string] (0,-0.3) -- (0.6,-0.3) ..controls (0.9,-0.3) and  
      (1.2,-1.5) .. (1.5,-1.5)--(3.1,-1.5)..controls (3.4,-1.5) and  (3.7,-0.3) 
      .. (4,-0.3)--(4.6,-0.3);
      \draw [string] (0,0) -- (0.6,0) ..controls (0.9,0) and  (1.2,-1.2) .. 
      (1.5,-1.2)--(2,-1.2);
      \draw [string] (0,0.3) -- (0.6,0.3) ..controls (0.9,0.3) and  (1.2,-0.9) 
      .. (1.5,-0.9)--(2,-0.9);
      \node [circle,draw,inner sep=0.15em,fill] at (2,1.2) {};
      \node [circle,draw,inner sep=0.15em,fill] at (2,0.9) {};  
      \node [circle,draw,inner sep=0.15em,fill] at (2,0.3) {};
      \node [circle,draw,inner sep=0.15em,fill] at (2,-0.3) {};  
      \node [circle,draw,inner sep=0.15em,fill] at (2,-0.9) {};
      \node [circle,draw,inner sep=0.15em,fill] at (2,-1.2) {}; 
      \node [circle,draw,inner sep=0.15em,fill] at (0.7,0.3) {};
      \node [circle,draw,inner sep=0.15em,fill] at (0.7,0) {};
      \node [circle,draw,inner sep=0.15em,fill] at (0.7,-0.3) {};
      \path[draw,string] (4.6,0.7)--(4.6,-0.7)--(6.1,0)--cycle;
      \node at (5.1,0) {$\theta$};
      \draw [string] (6.1,0) -- (6.7,0);
\end{tikzpicture}
\ =\ 
\begin{tikzpicture}[baseline=-\the\dimexpr\fontdimen22\textfont2\relax ]
      \draw [string] (0,0.3) -- (0.6,0.3);
      \draw [string] (0,0) -- (0.6,0);
      \draw [string] (0,-0.3) -- (0.6,-0.3);

      \path[draw,string] (0.6,0.7)--(0.6,-0.7)--(2.1,0)--cycle;
      \node at (1.1,0) {$\theta$};
      \draw [string] (2.1,0) -- (2.7,0);
\end{tikzpicture}\ .
\end{equation*}

Next we define {models} of a clone.
We first need a few preliminary definitions.

\begin{definition}
\label{def:endo_clone}
Let $A$ be a set.
Define the clone $\Endcl{A}=(\enrich{A}{A},p,\circ)$
%of \defemph{endo-multimorphisms on $A$}
as follows:
\begin{description}
\item[(CD1)] for each $n\in\NN$, let $\enrich{A}{A}_n$ be the set of all
functions from $A^n$ to $A$; 
\item[(CD2)] for each $n\in\NN$ and $i\in\{1,\dots,n\}$,
let $p^{(n)}_i$ be the $i$-th projection 
$A^n\longrightarrow A$, $(a_1,\dots,a_n)\longmapsto a_i$;
\item[(CD3)] for each $k,n\in\NN$, $g\colon A^k\longrightarrow A$ and
$f_1,\dots,f_k\colon A^n\longrightarrow A$, 
let $g\circ^{(n)}_k(f_1,\dots,f_k)$ be the function
$(a_1,\dots,a_n)\longmapsto g(f_1(a_1,\dots,a_n),\dots,f_k(a_1,\dots,a_n))$,
that is, the following composite:
\[
\begin{tikzpicture}[baseline=-\the\dimexpr\fontdimen22\textfont2\relax ]
      \node (1) at (0,0)  {$A^n$};
      \node (2) at (3,0)  {$A^k$};
      \node (3) at (5,0) {$A.$};
      \draw[->] (1) to node[auto,labelsize]{$\langle f_1,\dots,f_k\rangle$} (2);
      \draw[->] (2) to node[auto,labelsize] {$g$} (3);
\end{tikzpicture} 
\]
\end{description}
It is straightforward to check the axioms (CA1)--(CA3).
\end{definition}

\begin{definition}
\label{def:clone_hom}
Let $\monoid{T}=(T,p,\circ)$ and $\monoid{T'}=(T',p',\circ')$ be 
clones.
A \defemph{clone homomorphism} from $\monoid{T}$ to $\monoid{T'}$
is a morphism of graded sets (Definition~\ref{def:graded_set}) 
$h\colon T\longrightarrow T'$ which preserves the structure of clones;
precisely,
\begin{itemize}
\item for each $n\in\NN$ and $i\in\{1,\dots,n\}$, $h_n(p^{(n)}_i)=p'^{(n)}_i$;
\item for each $k,n\in\NN$, $\phi\in T_k$ and $\theta_1,\dots,\theta_k\in T_n$,
\[
h_n\big(\phi\circ^{(n)}_k(\theta_1,\dots,\theta_k)\big)=
h_k(\phi)\circ'^{(n)}_k \big(h_n(\theta_1),\dots,h_n(\theta_k)\big).\qedhere
\]
\end{itemize}
\end{definition}

\begin{definition}
\label{def:clone_model}
Let $\monoid{T}$ be a clone.
A \defemph{model of $\monoid{T}$} consists of a set $A$ and 
a clone homomorphism $\alpha\colon \monoid{T}\longrightarrow \Endcl{A}$.
\end{definition}

Let us then define the notion of homomorphism between models.
First we extend the definition of the graded set $\enrich{A}{A}$ 
introduced in Definition~\ref{def:endo_clone}.

\begin{definition}
\begin{enumerate}
\item Let $A$ and $B$ be sets. The graded set $\enrich{A}{B}$ is defined by 
setting,
for each $n\in\NN$, $\enrich{A}{B}_n$ be the set of all functions from $A^n$
to $B$.
\item Let $A,A'$ and $B$ be sets and $f\colon A'\longrightarrow A$ be a 
function.
The morphism of graded sets 
$\enrich{f}{B}\colon\enrich{A}{B}\longrightarrow\enrich{A'}{B}$
is defined by setting, for each $n\in\NN$, 
$\enrich{f}{B}_n\colon\enrich{A}{B}_n\longrightarrow\enrich{A'}{B}_n$ be the 
precomposition by $f^n\colon (A')^n\longrightarrow A^n$;
that is, $h\longmapsto h\circ f^n$. 
\item Let $A,B$ and $B'$ be sets and $g\colon B\longrightarrow B'$ be a 
function.
The morphism of graded sets 
$\enrich{A}{g}\colon\enrich{A}{B}\longrightarrow\enrich{A}{B'}$
is defined by setting, for each $n\in\NN$, 
$\enrich{A}{g}_n\colon\enrich{A}{B}_n\longrightarrow\enrich{A}{B'}_n$ be the 
postcomposition by $g\colon B\longrightarrow B'$;
that is, $h\longmapsto g\circ h$.\qedhere
\end{enumerate}
\end{definition}

\begin{definition}
\label{def:clone_mod_hom}
Let $\monoid{T}$ be a clone, and $(A,\alpha)$ and $(B,\beta)$ be models of 
$\monoid{T}$.
A \defemph{homomorphism} from $(A,\alpha)$ to $(B,\beta)$ is a function 
$f\colon A\longrightarrow B$ making the following diagram of 
morphisms of graded sets commute:
\[
\begin{tikzpicture}[baseline=-\the\dimexpr\fontdimen22\textfont2\relax ]
      \node (TL) at (0,2)  {$T$};
      \node (TR) at (3,2)  {$\enrich{A}{A}$};
      \node (BL) at (0,0) {$\enrich{B}{B}$};
      \node (BR) at (3,0) {$\enrich{A}{B}.$};
      \draw[->] (TL) to node[auto,labelsize](T) {$\alpha$} (TR);
      \draw[->]  (TR) to node[auto,labelsize] {$\enrich{A}{f}$} (BR);
      \draw[->]  (TL) to node[auto,swap,labelsize] {$\beta$} (BL);
      \draw[->] (BL) to node[auto,swap,labelsize](B) {$\enrich{f}{B}$} (BR);
\end{tikzpicture} \qedhere
\]
\end{definition}

\medskip

Now let us turn to the relation between
presentations of equational theories 
(Definition~\ref{defn:univ_alg_pres_eq_thy})
and clones.
We start with the observation that the graded set $T(\Sigma)$ of $\Sigma$-terms
(Definition~\ref{def:Sigma_term})
has a canonical clone structure, given as follows:
\begin{description}
\item[(CD2)] for each $n\in\NN$ and $i\in\{1,\dots,n\}$, let $p^{(n)}_i$ be 
$x^{(n)}_i\in T(\Sigma)_n$;
\item[(CD3)] for each $k,n\in\NN$, $s\in T(\Sigma)_k$ and $t_1,\dots,t_k\in 
T(\Sigma)_n$, let $s\circ^{(n)}_k(t_1,\dots,t_k)$ be
$s[x^{(k)}_1\mapsto t_1,\dots,x^{(k)}_k\mapsto t_k]\in T(\Sigma)_n$.
\end{description}
We denote the resulting clone by $\monoid{T}(\Sigma)$.
In fact, this clone is characterised by the following universal property.

\begin{proposition}
\label{prop:free_clone}
Let $\Sigma$ be a graded set.
The clone $\monoid{T}(\Sigma)$ is the free clone generated from the 
graded set $\Sigma$. 
That is, the morphism of graded sets $\eta_\Sigma\colon \Sigma\longrightarrow
T(\Sigma)$, defined by 
$(\eta_\Sigma)_n(\sigma)=\sigma(x^{(n)}_1,\dots,x^{(n)}_n)$ for each $n\in\NN$ 
and $\sigma\in\Sigma_n$,
satisfies the following property:
given any clone $\monoid{S}=(S,p,\circ)$ and any morphism of 
graded sets $f\colon \Sigma\longrightarrow S$, there exists a 
unique clone homomorphism $g\colon \monoid{T}(\Sigma)\longrightarrow \monoid{S}$
such that $g\circ \eta_\Sigma=f$.
\[
\begin{tikzpicture}[baseline=-\the\dimexpr\fontdimen22\textfont2\relax ]
      \node (TL) at (0,1)  {$\Sigma$};
      \node (TR) at (3,1)  {${T}(\Sigma)$};
      \node (BR) at (3,-1) {$S$};
      \node (B) at (1.5,-1.8) {{(graded sets)}};
      \draw[->] (TL) to node[auto,labelsize](T) {$\eta_\Sigma$} (TR);
      \draw[->]  (TR) to node[auto,labelsize] {$g$} (BR);
      \draw[->]  (TL) to node[auto,swap,labelsize] {$f$} (BR);
\end{tikzpicture} 
\qquad
\begin{tikzpicture}[baseline=-\the\dimexpr\fontdimen22\textfont2\relax ]
      \node (TR) at (3,1)  {$\monoid{T}(\Sigma)$};
      \node (BR) at (3,-1) {$\monoid{S}$};
      \node (B) at (3,-1.8) {(clones)};
      \draw[->,dashed]  (TR) to node[auto,labelsize] {$g$} (BR);
\end{tikzpicture} 
\]
\end{proposition}
\begin{proof}
The clone homomorphism $g$ may be defined by using the inductive nature of 
the definition of $T(\Sigma)$, as follows:
\begin{enumerate}
\item for each $n\in\NN$ and $i\in\{1,\dots,n\}$, let
\[
g_n(x^{(n)}_i)=p^{(n)}_i;
\]
\item for each $k,n\in\NN$, $\sigma\in\Sigma_k$ and $t_1,\dots,t_k\in 
T(\Sigma)_n$, let 
\[
g_n(\sigma(t_1,\dots,t_k))=f_k(\sigma)\circ^{(n)}_k(g_n(t_1),\dots,g_n(t_k)).
\]
\end{enumerate}
To check that $g$ is indeed a clone homomorphism, it suffices to show 
for each $s\in T(\Sigma)_k$ and $t_1,\dots,t_k\in T(\Sigma)_n$,
\[
g_n(s[x^{(k)}_1\mapsto t_1,\dots,x^{(k)}_k\mapsto t_k])=
g_k(s)\circ^{(n)}_k(g_n(t_1),\dots, g_n(t_k));
\]
this can be shown by induction on $s$.
The uniqueness of $g$ is clear.
\end{proof}
The construction given in Definition~\ref{def:interp_Sigma_term}
is a special case of the above; let $\monoid{S}$ be $\Endcl{A}$.

Recall from Definition~\ref{def:eq_logic} the 
graded set $\overline{E}\subseteq T(\Sigma)\times T(\Sigma)$ 
of equational theorems of a presentation of an equational theory 
$\pres{\Sigma}{E}$.
By the rules ({\sc{Refl}}), ({\sc{Sym}}) and ({\sc{Trans}}), $\overline{E}$
is an equivalence relation on $T(\Sigma)$. 
Hence we may consider the quotient graded set $T(\Sigma)/\overline{E}$.
By the rule ({\sc{Cong}}), the clone operations on $T(\Sigma)$
induce well-defined operations on $T(\Sigma)/\overline{E}$;
in particular, we can define $\circ^{(n)}_k$ on $T(\Sigma)/\overline{E}$ by
\[
[\phi]_{\overline{E}}\circ^{(n)}_k([\theta_1]_{\overline{E}},\dots,
[\theta_k]_{\overline{E}})= 
[\phi(\theta_1,\dots,\theta_k)]_{\overline{E}}.
\]
This makes the graded set $T(\Sigma)/\overline{E}$ into a clone;
the clone axioms for $T(\Sigma)/{\overline{E}}$ may be immediately checked 
by noticing the existence of a surjective morphism of graded sets 
$q\colon T(\Sigma)\longrightarrow T(\Sigma)/\overline{E}$ (given by 
$\theta\longmapsto [\theta]_{\overline{E}}$) preserving 
the clone operations.
The resulting clone is denoted by $\monoid{T}^\pres{\Sigma}{E}$.
It is also characterised by a universal property.
\begin{proposition}
\label{prop:quotient_clone}
Let $\pres{\Sigma}{E}$ be a presentation of an equational theory.
The clone homomorphism $q\colon \monoid{T}(\Sigma)\longrightarrow 
\monoid{T}^\pres{\Sigma}{E}$, defined by $q_n(\theta)=[\theta]_{\overline{E}}$
for each $n\in\NN$ and $\theta\in T(\Sigma)_n$, satisfies the following 
property: given any clone $\monoid{S}=(S,p,\circ)$ and a clone homomorphism 
$g\colon \monoid{T}(\Sigma)\longrightarrow \monoid{S}$ such that for any 
$t\approx_n s\in E$, $g_n(t)=g_n(s)$, there exists a unique
clone homomorphism $h\colon 
\monoid{T}^\pres{\Sigma}{E}\longrightarrow\monoid{S}$
such that $h\circ q=g$.
\[
\begin{tikzpicture}[baseline=-\the\dimexpr\fontdimen22\textfont2\relax ]
      \node (TL) at (0,1)  {$\monoid{T}(\Sigma)$};
      \node (TR) at (3,1)  {$\monoid{T}^\pres{\Sigma}{E}$};
      \node (BR) at (3,-1) {$\monoid{S}$};
      \draw[->] (TL) to node[auto,labelsize](T) {$q$} (TR);
      \draw[->,dashed]  (TR) to node[auto,labelsize] {$h$} (BR);
      \draw[->]  (TL) to node[auto,swap,labelsize] {$g$} (BR);
\end{tikzpicture} 
\]
\end{proposition}
\begin{proof}
The clone homomorphism $h$ is given by 
$h_n([\theta]_{\overline{E}})=g_n(\theta)$;
this is shown to be well-defined by induction on $\pres{\Sigma}{E}$-proofs
(see Definition~\ref{def:eq_logic}).
The uniqueness of $h$ is immediate from the surjectivity of $q$.
\end{proof}

We can now show that for any presentation of an equational theory 
$\pres{\Sigma}{E}$,
to give a model of $\pres{\Sigma}{E}$ is equivalent to give a model of the clone
$\monoid{T}^\pres{\Sigma}{E}$.
A model of the clone $\monoid{T}^\pres{\Sigma}{E}$ 
(Definition~\ref{def:clone_model}) can be---by 
Proposition~\ref{prop:quotient_clone}---equivalently given as a suitable clone 
homomorphism out of 
$\monoid{T}(\Sigma)$; this in turn is---by 
Proposition~\ref{prop:free_clone}---equivalently given as a suitable morphism 
of graded sets 
out of $\Sigma$,
which is nothing but a model of the presentation of an equational theory 
$\pres{\Sigma}{E}$ (Definition~\ref{defn:univ_alg_model}).

We also remark that every clone is isomorphic to a clone of the form 
$\monoid{T}^\pres{\Sigma}{E}$
for some presentation of an equational theory $\pres{\Sigma}{E}$.

\medskip

The inference rules of equational logic we have given in
Definition~\ref{def:eq_logic} can be understood as 
the inductive definition of the congruence relation 
$\overline{E}\subseteq T(\Sigma)\times T(\Sigma)$  
on the clone $\monoid{T}(\Sigma)$
generated by $E\subseteq T(\Sigma)\times T(\Sigma)$.
The notion of clone therefore provides conceptual understanding 
of equational logic.

We conclude that
the classical universal algebra based on presentations of equational 
theories may be replaced by the theory of clones to a certain extent.
Given a presentation of an equational theory $\pres{\Sigma}{E}$,
the clone $\monoid{T}^\pres{\Sigma}{E}$ it presents can be obtained 
by letting $T^\pres{\Sigma}{E}$ to be the graded set of \emph{$\Sigma$-terms 
modulo equational theorems of $\pres{\Sigma}{E}$}.

We remark that the well-known notion of \emph{Lawvere theory} 
\cite{Lawvere_thesis} is essentially equivalent to that of clones;
see e.g., \cite{Taylor_clone}.
In this thesis, we shall only deal with clones,
leaving the translation of the results to Lawvere theories or universal algebra
to the interested reader.

\section{Non-symmetric operads}
Non-symmetric operads~\cite{May_loop} may be seen as a variant
of clones.
Compared to clones, non-symmetric operads are less expressive (for example,
the groups cannot be captured by non-symmetric operads), but 
their models can be taken in wider contexts than for clones (we will
introduce a notion of model of a non-symmetric operad using abelian 
groups and their tensor products).

Before giving the definition of non-symmetric operad,
we shall introduce the corresponding notion of presentation.
Let $\Sigma$ be a graded set. We say that
a $\Sigma$-term (Definition~\ref{def:Sigma_term}) $t\in T(\Sigma)_n$ is  
\defemph{strongly regular} if in $t$ each of the variables 
$x^{(n)}_1,\dots,x^{(n)}_n$
appears precisely once, and in this order (from left to right).
For example, consider the graded set $\Sigma^\mon$ defined by 
$\Sigma^\mon_0=\{e\}$, $\Sigma^\mon_2=\{m\}$ and $\Sigma^\mon_n=\emptyset$
for all $n\in\NN\setminus\{0,2\}$.
Among the $\Sigma^\mon$-terms,
\[
m(x^{(1)}_1,e)\in T(\Sigma)_1,\quad
m(x^{(2)}_1,x^{(2)}_2)\in T(\Sigma)_2,
\quad
m(m(x^{(3)}_1,x^{(3)}_2),x^{(3)}_3)\in T(\Sigma)_3 
\]
are strongly regular, but 
\[
m(x^{(1)}_1,x^{(1)}_1)\in T(\Sigma)_1,\quad
x^{(2)}_1\in T(\Sigma)_2,\quad
m(x^{(2)}_2,x^{(2)}_1)\in T(\Sigma)_2
\]
are not.
The following definition
introduces the same notion inductively.
\begin{definition}
\label{def:streg_Sigma_term}
Let $\Sigma$ be a graded set.
The graded set $T_\streg(\Sigma)=(T_\streg(\Sigma)_n)_{n\in\NN}$ of 
\defemph{strongly regular $\Sigma$-terms} is defined inductively as follows.
\begin{enumerate}
\item $x^{(1)}_1\in T_\streg(\Sigma)_1$.
\item For each $k,n_1,\dots,n_k\in\NN$, $\sigma\in\Sigma_k$ and 
$t_1\in T_\streg(\Sigma)_{n_1},\dots,t_k\in T_\streg(\Sigma)_{n_k}$,
writing $n_1+\dots+n_k=n$,
\begin{multline*}
\sigma(t_1[x^{(n_1)}_1\mapsto x^{(n)}_1,\dots,x^{(n_1)}_{n_1}\mapsto 
x^{(n)}_{n_1}],
\dots,\\
t_k[x^{(n_k)}_1\mapsto x^{(n)}_{n_1+\dots+n_{k-1}+1},\dots,
x^{(n_k)}_{n_k}\mapsto x^{(n)}_{n_1+\dots+n_{k-1}+n_k}])
\in T_\streg(\Sigma)_{n}.
\end{multline*}
When $k=0$, we usually write $\sigma$ instead of $\sigma()$.\qedhere
\end{enumerate}
\end{definition}

\begin{definition}[cf.~Definition~\ref{def:Sigma_eq}]
Let $\Sigma$ be a graded set. 
An element of the graded set $T_\streg(\Sigma)\times T_\streg(\Sigma)$ is called
a \defemph{strongly regular $\Sigma$-equation}.
We write a strongly regular $\Sigma$-equation $(n,(t,s))\in 
T_\streg(\Sigma)\times T_\streg(\Sigma)$ as $t\approx_n s$ or $t\approx s$.
\end{definition}

\begin{definition}[cf.~Definition~\ref{defn:univ_alg_pres_eq_thy}]
A \defemph{strongly regular presentation of an equational theory}
$\pres{\Sigma}{E}$ is a pair consisting of:
\begin{itemize}
\item a graded set $\Sigma$ of \defemph{basic operations}, and
\item a graded set $E\subseteq T_\streg(\Sigma)\times T_\streg(\Sigma)$ of 
(strongly regular) \defemph{equational axioms}.\qedhere
\end{itemize}
\end{definition}

The notion of model of a strongly regular presentation of an equational theory
may be defined just as in Definition~\ref{defn:univ_alg_model},
since any strongly regular presentation of an equational theory can be seen as 
a presentation of an equational theory.

As an example of strongly regular presentations of equational theories, 
consider $\pres{\Sigma^\mon}{E^\mon}$ defined as follows:
\[
E^\mon_1=\{\,m(x^{(1)}_1,e)\approx x^{(1)}_1, \quad m(e,x^{(1)}_1)\approx 
x^{(1)}_1 \,\},
\]
\[
E^\mon_3=\{\,m(m(x^{(3)}_1,x^{(3)}_2),x^{(3)}_3)\approx
m(x^{(3)}_1,m(x^{(3)}_2,x^{(3)}_3))\,\},
\]
$E^\mon_n=\emptyset$ for all $n\in\NN\setminus\{1,3\}$.
Models of $\pres{\Sigma^\mon}{E^\mon}$ are precisely \emph{monoids}.

\medskip

In order to appreciate the value of strongly regular presentations 
of equational theories (and of non-symmetric operads),
let us now introduce another notion of model.
This notion of model is based on abelian groups, in contrast to 
the one introduced in Definition~\ref{defn:univ_alg_model}
based on sets.

Let $\pres{\Sigma}{E}$ be a strongly regular presentation of an equational 
theory.
Define an \defemph{interpretation} of $\Sigma$ on an abelian group
$A$ to be a function 
$\interp{-}$ which for each $n\in\NN$
and $\sigma\in\Sigma_n$, assigns a group homomorphism
$\interp{\sigma}\colon A^{\otimes n}\longrightarrow A$,
where $A^{\otimes n}$ is the tensor product of $n$-many copies of $A$
($A^{\otimes 0}$ is the additive abelian group $\mathbb{Z}$, the unit
for tensor).
Given such an interpretation $\interp{-}$ of $\Sigma$,
we can extend it to $T_\streg(\Sigma)$, following the inductive definition
of $T_\streg(\Sigma)$ in Definition~\ref{def:streg_Sigma_term}
(cf.~Definition~\ref{def:interp_Sigma_term}):
\begin{enumerate}
\item Let $\interp{x^{(1)}_1}\colon A\longrightarrow A$ be the 
identity homomorphism.
\item For $k,n_1,\dots,n_k\in\NN$, $\sigma\in\Sigma_k$
and $t_1\in T_\streg(\Sigma)_{n_1},\dots,t_k\in T_\streg(\Sigma)_{n_k}$, let 
\[
\interp{\sigma(t_1,\dots,t_k)}\colon A^{\otimes (n_1+\dots+n_k)}\longrightarrow
A
\]
(we omit the substitutions in $t_i$) 
to be the composite:
\[
\begin{tikzpicture}[baseline=-\the\dimexpr\fontdimen22\textfont2\relax ]
      \node (1) at (-1,0)  {$A^{\otimes (n_1+\dots+n_k)}$};
      \node (2) at (3,0)  {$A^{\otimes k}$};
      \node (3) at (5,0) {$A.$};
      \draw[->] (1) to node[auto,labelsize]{$\interp{t_1}\otimes\dots\otimes 
      \interp{t_k}$} (2);
      \draw[->] (2) to node[auto,labelsize] {$\interp{\sigma}$} (3);
\end{tikzpicture} 
\]
\end{enumerate}
(Note that, in contrast, we cannot extend $\interp{-}$ to $T(\Sigma)$
in any natural way.
For example, tensor products do not have an analogue of projections
for cartesian products.)
We define an \defemph{abelian group model of $\pres{\Sigma}{E}$}
to be an abelian group $A$ together with an interpretation $\interp{-}$
of $\Sigma$ on $A$ such that for any $t\approx_n s\in E$, 
$\interp{t}=\interp{s}$
(cf.~Definition~\ref{defn:univ_alg_model}, which may be called 
a \emph{set model of $\pres{\Sigma}{E}$}).
The abelian group models of $\pres{\Sigma^\mon}{E^\mon}$ are 
precisely the \emph{rings (with $1$)}. 

\medskip

We now turn to the definition of non-symmetric operad:

\begin{definition}[cf.~Definition~\ref{def:clone}]
\label{def:non_symm_op}
A \defemph{non-symmetric operad} $\monoid{T}$ consists of:
\begin{description}
\item[(ND1)] a graded set $T=(T_n)_{n\in\NN}$;
\item[(ND2)] an element $\id{}\in T_1$;
\item[(ND3)] for each $k,n_1,\dots,n_k\in \NN$, a function
(we omit the sub- and superscripts)
\[
\circ\colon T_k\times T_{n_1}\times\dots\times T_{n_k}\longrightarrow 
T_{n_1+\dots+ n_k}
\]
whose action we write as $(\phi,\theta_1,\dots,\theta_k)\longmapsto \phi\circ 
(\theta_1,\dots,\theta_k)$
\end{description}
satisfying the following equations:
\begin{description}
\item[(NA1)] for each $n\in\NN$ and $\theta\in T_n$,
\[
\id{}\circ (\theta)=\theta;
\]
\item[(NA2)] for each $n\in\NN$ and $\theta\in T_n$,
\[
\theta\circ (\id{},\dots,\id{}) = \theta;
\] 
\item[(NA3)] for each $l,\ k_1,\dots,k_l,\ 
n_{1,1},\dots,n_{1,k_1},\dots,n_{l,1},
\dots,n_{l,k_l}\in\NN$, \ $\psi\in T_l$, \ $\phi_1\in T_{k_1}, \dots, \phi_l\in 
T_{k_l}$, \ $\theta_{1,1}\in T_{n_{1,1}},\dots,\theta_{1,k_1}\in T_{n_{1,k_1}},
\dots, \theta_{l,1}\in T_{n_{l,1}},\dots,$ $\theta_{l,k_l}\in T_{n_{l, k_l}}$,
\begin{multline*}
\psi\circ\big(\phi_1\circ (\theta_{1,1},\dots,\theta_{1,k_1}),\ \dots\ ,
\phi_l\circ (\theta_{l,1},\dots,\theta_{l,k_l}) \big)\\
=
\big(\psi\circ 
(\phi_1,\dots,\phi_l)\big)\circ(\theta_{1,1},\dots,\theta_{1,k_1},\ \dots,\ 
\theta_{l,1},\dots,\theta_{l,k_l}).
\end{multline*}
\end{description}
Such a non-symmetric operad is written as $\monoid{T}=(T,\id{},\circ)$.
\end{definition}

We can understand the above definition by string diagrams.
Compared to the case of clones, this time we use a rather restricted 
class of diagrams; we no longer allow
the permuting, copying and discarding facilities,
previously drawn as follows:
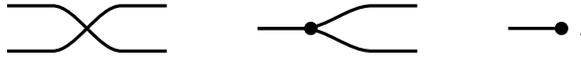
\begin{equation*}
\begin{tikzpicture}[baseline=-\the\dimexpr\fontdimen22\textfont2\relax ]
      \draw [string] (0,0.3) -- (0.6,0.3) ..controls (0.9,0.3) and  (1.2,-0.3) 
      .. (1.5,-0.3)--(2.1,-0.3);
      \draw [string] (0,-0.3) -- (0.6,-0.3) ..controls (0.9,-0.3) and  
      (1.2,0.3) .. (1.5,0.3)--(2.1,0.3);
\end{tikzpicture}
\qquad\quad
\begin{tikzpicture}[baseline=-\the\dimexpr\fontdimen22\textfont2\relax ]
      \draw [string] (0,0) -- (0.6,0) ..controls (0.9,0) and  (1.2,-0.3) 
      .. (1.5,-0.3)--(2.1,-0.3);
      \draw [string] (0,0) -- (0.6,0) ..controls (0.9,0) and  
      (1.2,0.3) .. (1.5,0.3)--(2.1,0.3);
      \node [circle,draw,inner sep=0.15em,fill] at (0.7,0) {};
\end{tikzpicture}
\qquad\quad
\begin{tikzpicture}[baseline=-\the\dimexpr\fontdimen22\textfont2\relax ]
      \draw [string] (0,0) -- (0.7,0);
      \node [circle,draw,inner sep=0.15em,fill] at (0.7,0) {};
\end{tikzpicture}\ .
\end{equation*}
Without these components, we cannot draw the picture 
(\ref{eqn:clone_composition})
for composition in clones.
The natural alternative would be the picture
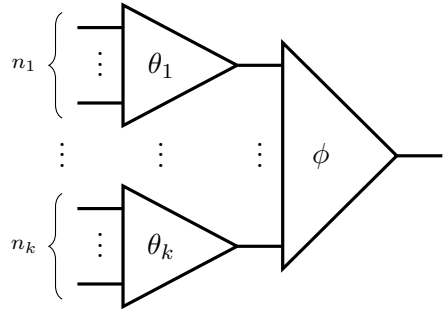
\[
\begin{tikzpicture} [baseline={([yshift=-0.5ex]current bounding box.center)}]
%[baseline=-\the\dimexpr\fontdimen22\textfont2\relax ]
      \coordinate (TL) at (0,1.5);
      \coordinate (BL) at (0,-1.5);
      \coordinate (R) at (1.5,0);
      \path[draw,string] (TL)--(BL)--(R)--cycle;
      \node at (0.5,0) {$\phi$};
      \path[draw,string] (TL)+(0,-0.3)-- +(-0.6,-0.3);
      \node at (-0.3,0) {$\rvdots$};
      \path[draw,string] (BL)+(0,0.3)-- +(-0.6,0.3);
      \path[draw,string] (R)-- +(0.6,0);
  \begin{scope}[shift={(-2.1,1.2)}]
      \coordinate (TL) at (0,0.8);
      \coordinate (BL) at (0,-0.8);
      \coordinate (R) at (1.5,0);
      \path[draw,string] (TL)--(BL)--(R)--cycle;
      \node at (0.5,0) {$\theta_1$};
      \node at (0.5,-1.2) {$\rvdots$};
      \node at (-0.3,0) {$\rvdots$};
  \end{scope}
  \begin{scope}[shift={(-2.1,-1.2)}]
      \coordinate (TL) at (0,0.8);
      \coordinate (BL) at (0,-0.8);
      \coordinate (R) at (1.5,0);
      \path[draw,string] (TL)--(BL)--(R)--cycle;
      \node at (0.5,0) {$\theta_{k}$};
      \node at (-0.3,0) {$\rvdots$};
  \end{scope}
  \begin{scope}[shift={(-2.1,0)}]
      \draw [string] (-0.6,1.7)--(0,1.7);
      \draw [string] (-0.6,-0.7)--(0,-0.7);
      \draw [string] (-0.6,0.7)--(0,0.7);
      \draw [string] (-0.6,-1.7)--(0,-1.7);
      \node at (-0.8,0) {$\rvdots$};
      \draw [decorate,decoration={brace,amplitude=5pt}]
      (-0.8,-1.9) -- (-0.8,-0.5) node [midway,xshift=-0.5cm,labelsize] 
      {$n_k$};
      \draw [decorate,decoration={brace,amplitude=5pt}]
      (-0.8,0.5) -- (-0.8,1.9) node [midway,xshift=-0.5cm,labelsize] 
      {$n_1$};
  \end{scope}
\end{tikzpicture}
\]
and this is our interpretation of (ND3).
The element $\id{}$ in (ND2) is represented by the 
only diagram of the form (\ref{eqn:clone_p}) which we can still draw,
namely,
\[
\begin{tikzpicture}[baseline=-\the\dimexpr\fontdimen22\textfont2\relax ]
      \draw [string] (0,0) -- (2.1,0);
\end{tikzpicture}\ .
\]
The axioms (NA1)--(NA3) may be understood in the light of these diagrams.

\medskip

Let us move on to the definition of models of a non-symmetric operad.
As with strongly regular presentations of equational theories,
non-symmetric operads also admits both notions of model,
one based on sets and the other based on abelian groups (and a lot more,
as we shall see later).

\begin{definition}[cf.~Definition~\ref{def:endo_clone}]
Let $A$ be a set. 
Define the non-symmetric operad $\Endopset{A}=(\enrich{A}{A},\id{},\circ)$ 
as follows:
\begin{description}
\item[(ND1)] for each $n\in\NN$, let $\enrich{A}{A}_n$
be the set of all functions from $A^n$ to $A$;
\item[(ND2)] the element $\id{}\in\enrich{A}{A}$ is the identity function on 
$A$;
\item[(ND3)] for each $k,n_1,\dots,n_k\in\NN$, $g\colon A^k\longrightarrow A$,
$f_1\colon A^{n_1}\longrightarrow A,\dots,f_k\colon A^{n_k}\longrightarrow A$,
let $g\circ (f_1,\dots,f_k)$ be the following composite: 
\[
\begin{tikzpicture}[baseline=-\the\dimexpr\fontdimen22\textfont2\relax ]
      \node (1) at (-1,0)  {$A^{n_1+\dots+n_k}$};
      \node (2) at (3,0)  {$A^{k}$};
      \node (3) at (5,0) {$A.$};
      \draw[->] (1) to node[auto,labelsize]{$f_1\times\dots\times 
      f_k$} (2);
      \draw[->] (2) to node[auto,labelsize] {$g$} (3);
\end{tikzpicture} \qedhere
\]
\end{description}
\end{definition}

\begin{definition}
Let $A$ be an abelian group. 
Define the non-symmetric operad $\Endopab{A}=(\enrich{A}{A},\id{},\circ)$ 
as follows:
\begin{description}
\item[(ND1)] for each $n\in\NN$, let $\enrich{A}{A}_n$
be the set of all group homomorphisms from $A^{\otimes n}$ to $A$;
\item[(ND2)] the element $\id{}\in\enrich{A}{A}$ is the identity homomorphism 
on $A$;
\item[(ND3)] for each $k,n_1,\dots,n_k\in\NN$, $g\colon A^{\otimes k} 
\longrightarrow A$,
$f_1\colon A^{\otimes n_1}\longrightarrow A,\dots,
f_k\colon A^{\otimes n_k}\longrightarrow A$,
let $g\circ (f_1,\dots,f_k)$ be the following composite: 
\[
\begin{tikzpicture}[baseline=-\the\dimexpr\fontdimen22\textfont2\relax ]
      \node (1) at (-1,0)  {$A^{\otimes (n_1+\dots+n_k)}$};
      \node (2) at (3,0)  {$A^{\otimes k}$};
      \node (3) at (5,0) {$A.$};
      \draw[->] (1) to node[auto,labelsize]{$f_1\otimes\dots\otimes 
      f_k$} (2);
      \draw[->] (2) to node[auto,labelsize] {$g$} (3);
\end{tikzpicture} \qedhere
\]
\end{description}
\end{definition}
We define the notion of \defemph{non-symmetric operad homomorphism} 
between non-symmetric operads 
just in the same way as that of clone homomorphism 
(Definition~\ref{def:clone_hom}).

\begin{definition}[cf.~Definition~\ref{def:clone_model}]
\label{def:ns_operad_model}
Let $\monoid{T}$ be a non-symmetric operad.
\begin{enumerate}
\item A \defemph{set model of $\monoid{T}$} consists of a set $A$
and a non-symmetric operad homomorphism $\monoid{T}\longrightarrow 
\Endopset{A}$.
\item An \defemph{abelian group model of $\monoid{T}$} consists of an abelian 
group $A$
and a non-symmetric operad homomorphism $\monoid{T}\longrightarrow \Endopab{A}$.
\end{enumerate}
Homomorphisms between set or abelian group models of $\monoid{T}$
are defined just as in Definition~\ref{def:clone_mod_hom}.
\end{definition}

The relationship between strongly regular presentations of 
equational theories and non-symmetric operads is completely parallel to  
the one between presentations of equational theories and clones:
for each graded set $\Sigma$, 
the graded set $T_\streg(\Sigma)$ has the structure of non-symmetric operad,
and is moreover the free such generated by $\Sigma$, 
there is a version of equational logic which can be seen as giving an 
inductive definition of the congruence relation 
for non-symmetric operad, and so on.
\begin{comment}
The abstract soundness and completeness theorems 
(cf.~Definition~\ref{thm:abst_sound_complete})
hold for non-symmetric operads without fuss, and 
as a consequence, so does the soundness theorem.
However, it might be worth noting that the completeness result
with respect to set-models (cf.~Definition~\ref{thm:eq_logic_sound_complete})
fails for non-symmetric operads;
see \cite{Leinster_operad_counterexample}.
\end{comment}

\section{Symmetric operads}
\label{sec:sym_operad}
Symmetric operads~\cite{May_loop} are an intermediate notion of algebraic theory
which lie between clones and 
non-symmetric operads, in terms of expressive power as well as 
in terms of range of notions of models.

Let us first discuss the corresponding presentation.
Given a graded set $\Sigma$,  
a $\Sigma$-term (Definition~\ref{def:Sigma_term}) $t\in T(\Sigma)_n$ is called 
\defemph{regular} if in $t$ each of the variables $x^{(n)}_1,\dots,x^{(n)}_n$
appears precisely once.
By way of illustration, consider the graded set $\Sigma^\mon$. 
Every strongly regular $\Sigma^\mon$-term is regular, 
and 
\[
m(x^{(2)}_2,x^{(2)}_1)\in T(\Sigma^\mon)_2
\]
is an example of $\Sigma^\mon$-terms which are regular but not strongly regular.

For any graded set $\Sigma$, let us denote by 
$T_\reg(\Sigma)\subseteq T(\Sigma)$
the graded set of all regular $\Sigma$-terms.
The following definitions should now be straightforward.

\begin{definition}[cf.~Definition~\ref{def:Sigma_eq}]
\label{def:reg_Sigma_eq}
Let $\Sigma$ be a graded set. An element of the graded set 
$T_\reg(\Sigma)\times T_\reg(\Sigma)$ is called a \defemph{regular 
$\Sigma$-equation}.
We write a regular $\Sigma$-equation $(n,(t,s))\in T_\reg(\Sigma)\times 
T_\reg(\Sigma)$ as $t\approx_n s$ or $t\approx s$.
\end{definition}

\begin{definition}[cf.~Definition~\ref{defn:univ_alg_pres_eq_thy}]
A \defemph{regular presentation of an equational theory} $\pres{\Sigma}{E}$
is a pair consisting of:
\begin{itemize}
\item a graded set $\Sigma$ of \defemph{basic operations}, and 
\item a graded set $E\subseteq T_\reg(\Sigma)\times T_\reg(\Sigma)$ of (regular)
\defemph{equational axioms}.\qedhere
\end{itemize} 
\end{definition}
The notion of model of a regular presentation of an equational theory
(based on sets) may be defined just as in Definition~\ref{defn:univ_alg_model},
since any regular presentation of an equational theory can be seen as 
a presentation of an equational theory.

As an example of regular (but not strongly regular) presentations of equational 
theories, 
consider
$\pres{\Sigma^\mon}{E^\cmon}$ defined as follows:
\[
E^\cmon_1=\{\,m(x^{(1)}_1,e)\approx x^{(1)}_1, \quad m(e,x^{(1)}_1)\approx 
x^{(1)}_1 \,\},
\]
\[
E^\cmon_2=\{\,m(x^{(2)}_1,x^{(2)}_2)\approx m(x^{(2)}_2,x^{(2)}_1)\,\},
\]
\[
E^\cmon_3=\{\,m(m(x^{(3)}_1,x^{(3)}_2),x^{(3)}_3)\approx
m(x^{(3)}_1,m(x^{(3)}_2,x^{(3)}_3))\,\},
\]
$E^\cmon_n=\emptyset$ for all $n\in\NN\setminus\{1,2,3\}$.
Models of $\pres{\Sigma^\mon}{E^\cmon}$
are precisely \emph{commutative monoids}.

We can also define the notion of model of a regular presentation of an 
equational 
theory based on abelian groups; we omit the details here.

\medskip

Let us turn to the definition of symmetric operad.
In order to define symmetric operads, we have to give preliminary 
definitions concerning symmetric groups. 

For each natural number $n$, the \defemph{symmetric group of order $n$},
written as $\mathfrak{S}_n$, is defined as the set of all bijections
on the set $[n]=\{1,\dots, n\}$ together with the multiplication $\cdot$
given by composition of functions: 
for $u,v\colon [n]\longrightarrow[n]$,
their multiplication $v\cdot u$ is the composite
\[
\begin{tikzpicture}[baseline=-\the\dimexpr\fontdimen22\textfont2\relax ]
      \node (1) at (0,0)  {$[n]$};
      \node (2) at (2,0)  {$[n]$};
      \node (3) at (4,0) {$[n].$};
      \draw[->] (1) to node[auto,labelsize]{$u$} (2);
      \draw[->] (2) to node[auto,labelsize] {$v$} (3);
\end{tikzpicture} 
\]
The identity function on $[n]$ is written as $e_n\in\mathfrak{S}_n$.

We may visualise elements of $\mathfrak{S}_n$ by string diagrams.
For example, the element $u\in\mathfrak{S}_3$ defined as $u(1)=2$,
$u(2)=3$ and $u(3)=1$ may be drawn as follows:
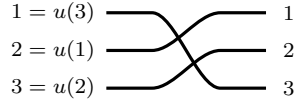
\begin{equation*}
\begin{tikzpicture}[baseline=-\the\dimexpr\fontdimen22\textfont2\relax ]
      \coordinate (s1) at (0,0.5);
      \coordinate (s2) at (0,0);
      \coordinate (s3) at (0,-0.5);
      \coordinate (t1) at (2.1,0.5);
      \coordinate (t2) at (2.1,0);
      \coordinate (t3) at (2.1,-0.5);
      \draw [string] (0,0.5) -- (0.6,0.5) ..controls (0.9,0.5) and  (1.2,-0.5) 
      .. (1.5,-0.5)--(2.1,-0.5);
      \draw [string] (0,0) -- (0.6,0) ..controls (0.9,0) and  (1.2,0.5) .. 
      (1.5,0.5)--(2.1,0.5);
      \draw [string] (0,-0.5) -- (0.6,-0.5) ..controls (0.9,-0.5) and  
      (1.2,0) .. (1.5,0)--(2.1,0);
      \node [labelsize, left of=s1, node distance=0.7cm] {$1=u(3)$};
      \node [labelsize, left of=s2, node distance=0.7cm] {$2=u(1)$};
      \node [labelsize, left of=s3, node distance=0.7cm] {$3=u(2)$};
      \node [labelsize, right of=t1, node distance=0.3cm] {$1$};
      \node [labelsize, right of=t2, node distance=0.3cm] {$2$};
      \node [labelsize, right of=t3, node distance=0.3cm] {$3$};
\end{tikzpicture}
\end{equation*}
The composition $v\cdot u$ of $u$ with $v\in\mathfrak{S}_3$
such that $v(1)=2$, $v(2)=1$ and $v(3)=3$ is then drawn as:
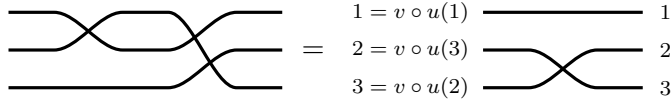
\begin{equation*}
\begin{tikzpicture}[baseline=-\the\dimexpr\fontdimen22\textfont2\relax ]
      \draw [string] (-1.5,0) -- (-0.9,0) ..controls (-0.6,0) and (-0.3,0.5) 
      ..(0,0.5) -- (0.6,0.5) ..controls (0.9,0.5) and  (1.2,-0.5) 
      .. (1.5,-0.5)--(2.1,-0.5);
      \draw [string] (-1.5,0.5) -- (-0.9,0.5) ..controls (-0.6,0.5) and 
      (-0.3,0) ..(0,0) -- (0.6,0) ..controls (0.9,0) and  (1.2,0.5) .. 
      (1.5,0.5)--(2.1,0.5);
      \draw [string] (-1.5,-0.5) -- (0.6,-0.5) ..controls (0.9,-0.5) and  
      (1.2,0) .. (1.5,0)--(2.1,0);
\end{tikzpicture}
\ =\ 
\begin{tikzpicture}[baseline=-\the\dimexpr\fontdimen22\textfont2\relax ]
      \coordinate (s1) at (0,0.5);
      \coordinate (s2) at (0,0);
      \coordinate (s3) at (0,-0.5);
      \coordinate (t1) at (2.1,0.5);
      \coordinate (t2) at (2.1,0);
      \coordinate (t3) at (2.1,-0.5);
      \draw [string] (0,0.5) -- (2.1,0.5);
      \draw [string] (0,0) -- (0.6,0) ..controls (0.9,0) and  (1.2,-0.5) .. 
      (1.5,-0.5)--(2.1,-0.5);
      \draw [string] (0,-0.5) -- (0.6,-0.5) ..controls (0.9,-0.5) and  (1.2,0) 
      ..  (1.5,0)--(2.1,0);
      \node [labelsize, left of=s1, node distance=0.95cm] {$1=v\circ u(1)$};
      \node [labelsize, left of=s2, node distance=0.95cm] {$2=v\circ u(3)$};
      \node [labelsize, left of=s3, node distance=0.95cm] {$3=v\circ u(2)$};
      \node [labelsize, right of=t1, node distance=0.3cm] {$1$};
      \node [labelsize, right of=t2, node distance=0.3cm] {$2$};
      \node [labelsize, right of=t3, node distance=0.3cm] {$3$};
\end{tikzpicture}
\end{equation*}

For each $k,n_1,\dots,n_k\in\NN$, there is a canonical group homomorphism
\[
\oplus\colon \mathfrak{S}_{n_1}\times\dots\times 
\mathfrak{S}_{n_k}\longrightarrow\mathfrak{S}_{n_1+\dots+n_k}
\]
which, in terms of string diagrams, just ``stacks the diagrams vertically'';
we view the set $[n_1+\dots+n_k]$ as consisting of $k$ blocks,
and perform permutation {inside} each block. 
Formally, $\oplus$
maps $(u_1,\dots,u_k)\in \mathfrak{S}_{n_1}\times\dots\times 
\mathfrak{S}_{n_k}$ to  
$u_1\oplus \cdots\oplus u_k\in \mathfrak{S}_{n_1+\dots+n_k}$,
the bijection on $[n_1+\dots+n_k]$ 
mapping an element $j\in[n_1+\dots+n_k]$ 
with $j=n_1+\dots+n_{i-1}+j'$ for $1\leq i\leq k$ and $1\leq j'\leq n_i$ 
(the $j'$-th element in the $i$-th block) to
\[
(u_1\oplus \dots\oplus u_k)(j)=n_1+\dots+n_{i-1}+u_i(j')
\]
(the $u_i(j')$-th element in the $i$-th block).

Still letting $k,n_1,\dots,n_k$ be arbitrary natural numbers, we 
have another function (not a group homomorphism in general)
\[
(-)_{n_1,\dots,n_k}\colon
\mathfrak{S}_k\longrightarrow \mathfrak{S}_{n_1+\dots+n_k}.
\]
We again view the set $[n_1+\dots+n_k]$
as consisting of $k$ blocks, but this time we permute these blocks. 
As an example, take $k=3$, $n_1=3$, $n_2=2$ and $n_3=2$,
and consider $v\in\mathfrak{S}_3$ defined above:
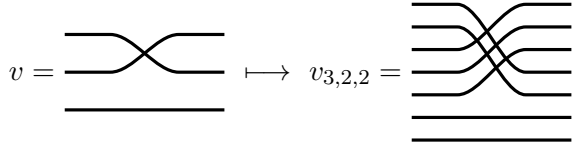
\begin{equation*}
v=
\begin{tikzpicture}[baseline=-\the\dimexpr\fontdimen22\textfont2\relax ]
      \draw [string] (0,0.5) -- (0.6,0.5) ..controls (0.9,0.5) and  (1.2,0) 
      .. (1.5,0)--(2.1,0);
      \draw [string] (0,0) -- (0.6,0) ..controls (0.9,0) and  (1.2,0.5) .. 
      (1.5,0.5)--(2.1,0.5);
      \draw [string] (0,-0.5) -- (2.1,-0.5);
\end{tikzpicture}
\ \longmapsto\ 
v_{3,2,2}=
\begin{tikzpicture}[baseline=-\the\dimexpr\fontdimen22\textfont2\relax ]
      \draw [string] (0,0.9) -- (0.6,0.9) ..controls (0.9,0.9) and  (1.2,0) 
      .. (1.5,0)--(2.1,0);
      \draw [string] (0,0.6) -- (0.6,0.6) ..controls (0.9,0.6) and  (1.2,-0.3) 
      .. (1.5,-0.3)--(2.1,-0.3);
      \draw [string] (0,0.3) -- (0.6,0.3) ..controls (0.9,0.3) and  (1.2,0.9) 
      ..(1.5,0.9)--(2.1,0.9);
      \draw [string] (0,0) -- (0.6,0) ..controls (0.9,0) and  (1.2,0.6) .. 
      (1.5,0.6)--(2.1,0.6);
      \draw [string] (0,-0.3) -- (0.6,-0.3) ..controls (0.9,-0.3) and  
      (1.2,0.3) .. (1.5,0.3)--(2.1,0.3);
      \draw [string] (0,-0.6) -- (2.1,-0.6);
      \draw [string] (0,-0.9) -- (2.1,-0.9);
\end{tikzpicture}
\end{equation*}
Formally, given any $v\in\mathfrak{S}_k$, the bijection $v_{n_1,\dots,n_k}\in
\mathfrak{S}_{n_1+\dots+n_k}$ maps an element $j\in[n_1+\dots+n_k]$
with $j=n_1+\dots+n_{i-1}+j'$ for $1\leq i\leq k$ and $1\leq j'\leq n_i$
(the $j'$-th element in the $i$-th block) to 
\[
v_{n_1,\dots,n_k}(j)=n_{v^{-1}(1)}+\dots+n_{v^{-1}(v(i)-1)}+j'
\]
(the $j'$-th element in the $v(i)$-th block).

\begin{definition}[cf.~Definition~\ref{def:clone}]
\label{def:symm_op}
A \defemph{symmetric operad} $\monoid{T}$ consists of:
\begin{description}
\item[(SD1)] a graded set $T=(T_n)_{n\in\NN}$;
\item[(SD2)] an element $\id{}\in T_1$;
\item[(SD3)] for each $k,n_1,\dots,n_k\in\NN$, a function
(we omit the sub- and superscripts)
\[
\circ\colon T_k\times T_{n_1}\times\dots\times T_{n_k}\longrightarrow
T_{n_1+\dots+n_k}
\]
whose action we write as $(\phi,\theta_1,\dots,\theta_k)\longmapsto
\phi\circ(\theta_1,\dots,\theta_k)$;
\item[(SD4)] for each $n\in\NN$, a function
\[
\bullet\colon \mathfrak{S}_n\times T_n\longrightarrow T_n
\]
\end{description}
satisfying the following equations:
\begin{description}
\item[(SA1)] for each $n\in\NN$ and $\theta\in T_n$,
\[
\id{}\circ (\theta)=\theta;
\]
\item[(SA2)] for each $n\in\NN$ and $\theta\in T_n$,
\[
\theta\circ (\id{},\dots,\id{}) = \theta;
\] 
\item[(SA3)] for each $l,\ k_1,\dots,k_l,\ 
n_{1,1},\dots,n_{1,k_1},\dots,n_{l,1},
\dots,n_{l,k_l}\in\NN$, \ $\psi\in T_l$, \ $\phi_1\in T_{k_1}, \dots, \phi_l\in 
T_{k_l}$, \ $\theta_{1,1}\in T_{n_{1,1}},\dots,\theta_{1,k_1}\in T_{n_{1,k_1}},
\dots, \theta_{l,1}\in T_{n_{l,1}},\dots, \theta_{l,k_l}\in T_{n_{l, k_l}}$,
\begin{multline*}
\psi\circ\big(\phi_1\circ (\theta_{1,1},\dots,\theta_{1,k_1}),\ \dots\ ,
\phi_l\circ (\theta_{l,1},\dots,\theta_{l,k_l}) \big)\\
=
\big(\psi\circ 
(\phi_1,\dots,\phi_l)\big)\circ(\theta_{1,1},\dots,\theta_{1,k_1},\ \dots\ ,
\theta_{l,1},\dots,\theta_{l,k_l});
\end{multline*}
\item[(SA4)] for each $n\in\NN$, the function $\bullet\colon 
\mathfrak{S}_n\times T_n\longrightarrow T_n$ is a left group action, 
that is, for each $\theta\in T_n$ and $u,v\in \mathfrak{S}_n$, 
\[
e_n\bullet \theta =\theta, \quad
(v\cdot u)\bullet \theta=v\bullet(u\bullet\theta);
\]
\item[(SA5)] for each $k,n_1,\dots,n_k\in\NN$, $\phi\in T_k$, $\theta_1\in 
T_{n_1},\dots,\theta_k\in T_{n_k}$ and $u_1\in\mathfrak{S}_{n_1},\dots,
u_k\in\mathfrak{S}_{n_k}$,
\[
\phi\circ(u_1\bullet\theta_1,\dots,u_k\bullet\theta_k)
=(u_1\oplus \dots\oplus u_k)\bullet\big(\phi\circ(\theta_1,\dots,\theta_k)\big);
\]
\item[(SA6)] for each $k,n_1,\dots,n_k\in\NN$, $\phi\in T_k$, $\theta_1\in 
T_{n_1},\dots,\theta_k\in T_{n_k}$ and $v\in\mathfrak{S}_k$,
\[
(v\bullet \phi)\circ(\theta_{v^{-1}(1)},\dots,\theta_{v^{-1}(k)})
=v_{n_1,\dots,n_k}\bullet \big(\phi\circ(\theta_1,\dots,\theta_k)\big).\qedhere
\]
\end{description}
\end{definition}

In terms of string diagrams,
symmetric operads correspond to the intermediate class of the diagrams
in which we can use the component
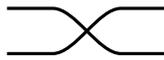
\begin{equation*}
\begin{tikzpicture}[baseline=-\the\dimexpr\fontdimen22\textfont2\relax ]
      \draw [string] (0,0.3) -- (0.6,0.3) ..controls (0.9,0.3) and  (1.2,-0.3) 
      .. (1.5,-0.3)--(2.1,-0.3);
      \draw [string] (0,-0.3) -- (0.6,-0.3) ..controls (0.9,-0.3) and  
      (1.2,0.3) .. (1.5,0.3)--(2.1,0.3);
\end{tikzpicture}
\end{equation*}
for permutation, but not 
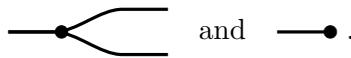
\begin{equation*}
\begin{tikzpicture}[baseline=-\the\dimexpr\fontdimen22\textfont2\relax ]
      \draw [string] (0,0) -- (0.6,0) ..controls (0.9,0) and  (1.2,-0.3) 
      .. (1.5,-0.3)--(2.1,-0.3);
      \draw [string] (0,0) -- (0.6,0) ..controls (0.9,0) and  
      (1.2,0.3) .. (1.5,0.3)--(2.1,0.3);
      \node [circle,draw,inner sep=0.15em,fill] at (0.7,0) {};
\end{tikzpicture}
\quad\text{and}\quad
\begin{tikzpicture}[baseline=-\the\dimexpr\fontdimen22\textfont2\relax ]
      \draw [string] (0,0) -- (0.7,0);
      \node [circle,draw,inner sep=0.15em,fill] at (0.7,0) {};
\end{tikzpicture}\ .
\end{equation*}

Once again, there is a completely parallel story for symmetric operads
as the ones for clones and non-symmetric operads.
Indeed, Curien~\cite{Curien_operad} and Hyland~\cite{Hyland_elts}
have developed a unified framework for 
clones, symmetric operads and non-symmetric operads.

\section{Monads}
Monads are introduced in category theory, and 
the language of categories is the best way to present them.
Hence from now on we shall assume the reader is familiar with the basics of 
category theory \cite{MacLane_CWM}.

The definition of monad is quite simple, so we begin with it.
\begin{definition}
\label{def:monad}
Let $\cat{C}$ be a large category. 
\begin{enumerate}
\item 
A \defemph{monad on $\cat{C}$} consists of:
\begin{itemize}
\item a functor $T\colon \cat{C}\longrightarrow \cat{C}$;
\item a natural transformation $\eta\colon \id{\cat{C}}\longrightarrow T$ 
(called the \defemph{unit});
\item a natural transformation $\mu\colon T\circ T\longrightarrow T$
(called the \defemph{multiplication}),
\end{itemize}
making the following diagrams commute:
\begin{equation*}
%\label{eqn:action_model}
\begin{tikzpicture}[baseline=-\the\dimexpr\fontdimen22\textfont2\relax ]
      \node (TL) at (0,2)  {$\id{\cat{C}}\circ T$};
      \node (TR) at (2,2)  {$T\circ T$};
      \node (BR) at (2,0) {$T$};
      \draw[->] (TL) to node[auto,labelsize](T) {$\eta\circ T$} (TR);
      \draw[->] (TR) to node[auto,labelsize](R) {$\mu$} (BR);
      \draw[->]  (TL) to node[auto,labelsize,swap] (L) {$\id{T}$} (BR);
\end{tikzpicture} 
\qquad
\begin{tikzpicture}[baseline=-\the\dimexpr\fontdimen22\textfont2\relax ]
      \node (TL) at (0,2)  {$T\circ \id{\cat{C}}$};
      \node (TR) at (2,2)  {$T\circ T$};
      \node (BR) at (2,0) {$T$};
      \draw[->] (TL) to node[auto,labelsize](T) {$T\circ \eta$} (TR);
      \draw[->] (TR) to node[auto,labelsize](R) {$\mu$} (BR);
      \draw[->]  (TL) to node[auto,labelsize,swap] (L) {$\id{T}$} (BR);
\end{tikzpicture} 
\qquad
\begin{tikzpicture}[baseline=-\the\dimexpr\fontdimen22\textfont2\relax ]
      \node (TL) at (0,2)  {$T\circ T\circ T$};
      \node (TR) at (3,2)  {$T\circ T$};
      \node (BL) at (0,0) {$T\circ T$};
      \node (BR) at (3,0) {$T$};
      \draw[->] (TL) to node[auto,labelsize](T) {$\mu\circ T$} (TR);
      \draw[->]  (TR) to node[auto,labelsize] {$\mu$} (BR);
      \draw[->]  (TL) to node[auto,swap,labelsize] {$T\circ \mu$} (BL);
      \draw[->] (BL) to node[auto,labelsize](B) {$\mu$} (BR);
\end{tikzpicture} 
\end{equation*}
\item Let $\monoid{T}=(T,\eta,\mu)$ and $\monoid{T'}=(T',\eta',\mu')$
be monads on $\cat{C}$.
A \defemph{morphism of monads on $\cat{C}$} from $\monoid{T}$ to $\monoid{T'}$
is a natural transformation $\alpha\colon T\longrightarrow T'$
which commutes with the units and multiplications. 
\end{enumerate}
We denote the category of monads on $\cat{C}$ by $\Mnd{\cat{C}}$.
\end{definition}
Next we introduce
models of a monad, usually called \emph{Eilenberg--Moore 
algebras}.

\begin{definition}[\cite{Eilenberg_Moore}]
\label{def:EM_alg}
Let $\cat{C}$ be a large category and $\monoid{T}=(T,\eta,\mu)$ be a monad on 
$\cat{C}$.
\begin{enumerate}
\item 
An \defemph{Eilenberg--Moore algebra of $\monoid{T}$}
consists of:
\begin{itemize}
\item an object $C\in\cat{C}$;
\item a morphism $\gamma\colon TC\longrightarrow C$ in $\cat{C}$,
\end{itemize}
making the following diagrams commute:
\begin{equation*}
\begin{tikzpicture}[baseline=-\the\dimexpr\fontdimen22\textfont2\relax ]
      \node (TL) at (0,2)  {$C$};
      \node (TR) at (2,2)  {$TC$};
      \node (BR) at (2,0)  {$C$};
      \draw[->] (TL) to node[auto,labelsize](T) {$\eta_C$} (TR);
      \draw[->] (TR) to node[auto,labelsize](R) {$\gamma$} (BR);
      \draw[->]  (TL) to node[auto,labelsize,swap] (L) {$\id{C}$} (BR);
\end{tikzpicture} 
\qquad
\begin{tikzpicture}[baseline=-\the\dimexpr\fontdimen22\textfont2\relax ]
      \node (TL) at (0,2)  {$TTC$};
      \node (TR) at (2,2)  {$TC$};
      \node (BL) at (0,0) {$TC$};
      \node (BR) at (2,0) {$C.$};
      \draw[->] (TL) to node[auto,labelsize](T) {$\mu_C$} (TR);
      \draw[->]  (TR) to node[auto,labelsize] {$\gamma$} (BR);
      \draw[->]  (TL) to node[auto,swap,labelsize] {$T\gamma$} (BL);
      \draw[->] (BL) to node[auto,labelsize](B) {$\gamma$} (BR);
\end{tikzpicture} 
\end{equation*}
\item Let $(C,\gamma)$ and $(C',\gamma')$ be Eilenberg--Moore algebras of 
$\monoid{T}$.
A \defemph{homomorphism} from $(C,\gamma)$ to $(C',\gamma')$ 
is a morphism $f\colon C\longrightarrow C'$ in $\cat{C}$
making the following diagram commute:
\[
\begin{tikzpicture}[baseline=-\the\dimexpr\fontdimen22\textfont2\relax ]
      \node (TL) at (0,2)  {$TC$};
      \node (TR) at (2,2)  {$TC'$};
      \node (BL) at (0,0) {$C$};
      \node (BR) at (2,0) {$C'.$};
      \draw[->] (TL) to node[auto,labelsize](T) {$Tf$} (TR);
      \draw[->]  (TR) to node[auto,labelsize] {$\gamma'$} (BR);
      \draw[->]  (TL) to node[auto,swap,labelsize] {$\gamma$} (BL);
      \draw[->] (BL) to node[auto,labelsize](B) {$f$} (BR);
\end{tikzpicture}
\]
\end{enumerate}
The category of all Eilenberg--Moore algebras of $\monoid{T}$ and their
homomorphism is celled the \defemph{Eilenberg--Moore category of $\monoid{T}$},
and is denoted by $\cat{C}^\monoid{T}$.
\end{definition}
Excellent introductions to monads abound (see e.g., 
\cite[Chapter~VI]{MacLane_CWM}). 
Here, we simply remark that monads typically arise from free constructions.
For example, there is a monad $\monoid{T}$ on the category $\Set$ of (small) 
sets which maps any (small) set $X$ to the underlying set of the free group 
generated by $X$ (with the canonical unit and multiplication),
and the Eilenberg--Moore algebras of $\monoid{T}$ are precisely groups.

We interpret that for each large category $\cat{C}$,
the monads on $\cat{C}$ form a single notion of algebraic theory;
hence in this section we have actually introduced 
a family of notions of algebraic theory, one for each large category. 
This is in contrast to the previous sections 
(Sections~\ref{sec:clone}--\ref{sec:sym_operad}),
where a single notion of algebraic theory was introduced in each section.

\section{Generalised operads}
\label{sec:S-operad}
Just like monads are a family of notions of algebraic theory 
parameterised by a large category, the term
\emph{generalised operads} 
\cite{Burroni_T_cats,Kelly_club_data_type,Hermida_representable,Leinster_book}
also refer to a family of notions of algebraic 
theory, this time parameterised by a large category with finite limits and a 
\emph{cartesian monad} thereon.
We start with the definition of cartesian monad.

\begin{definition}
\begin{enumerate}
\item Let $\cat{C}$ and $\cat{D}$ be categories, and $F,G\colon 
\cat{C}\longrightarrow \cat{D}$ be functors.
A natural transformation $\alpha\colon F\longrightarrow G$
is called \defemph{cartesian} if and only if all naturality squares of $\alpha$
are pullback squares;
that is, if and only if 
for any morphism $f\colon C\longrightarrow C'$ in $\cat{C}$, 
the square 
\[
\begin{tikzpicture}[baseline=-\the\dimexpr\fontdimen22\textfont2\relax ]
      \node (TL) at (0,2)  {$FC$};
      \node (TR) at (2,2)  {$GC$};
      \node (BL) at (0,0) {$FC'$};
      \node (BR) at (2,0) {$GC'$};
      \draw[->] (TL) to node[auto,labelsize](T) {$\alpha_C$} (TR);
      \draw[->]  (TR) to node[auto,labelsize] {$Gf$} (BR);
      \draw[->]  (TL) to node[auto,swap,labelsize] {$Ff$} (BL);
      \draw[->] (BL) to node[auto,labelsize](B) {$\alpha_{C'}$} (BR);
\end{tikzpicture}
\]
is a pullback of $Gf$ and $\alpha_{C'}$.
\item Let $\cat{C}$ be a category with all pullbacks.
A monad $\monoid{S}=(S,\eta,\mu)$ on $\cat{C}$ is called 
\defemph{cartesian} if and only if the functor $S$ preserves pullbacks,
and $\eta$ and $\mu$ are cartesian.\qedhere
\end{enumerate}
\end{definition}

For each cartesian monad $\monoid{S}=(S,\eta,\mu)$ on a large category $\cat{C}$
with all finite limits
we now introduce \emph{$\monoid{S}$-operads}, which
form a single notion of algebraic theory.

The crucial observation is that under this assumption,
the slice category $\cat{C}/S1$ (where $1$ is the terminal object of $\cat{C}$) 
acquires a canonical monoidal  
structure (see \cite[Chapter~VII]{MacLane_CWM} or 
\cite[Section~1.1]{Kelly:enriched} for the definition of monoidal category).
We write an object of $\cat{C}/S1$ either as $p\colon P\longrightarrow S1$
or $(P,p)$.
\begin{itemize}
\item The unit object is given by $I=(1,\eta_1\colon 1\longrightarrow S1)$.
\item Given a pair of objects $p\colon P\longrightarrow S1$ and $q\colon 
Q\longrightarrow S1$ in $\cat{C}/S1$,
first form the pullback
\begin{equation}
\label{eqn:pb_over_S1}
\begin{tikzpicture}[baseline=-\the\dimexpr\fontdimen22\textfont2\relax ]
      \node(0) at (0,1) {$(Q,q)\ast P$};
      \node(1) at (2,1) {$SP$};
      \node(2) at (0,-1) {$Q$};
      \node(3) at (2,-1) {$S1,$};
      
      \draw [->] 
            (0) to node (t)[auto,labelsize] {$\pi_2$} 
            (1);
      \draw [->] 
            (1) to node (r)[auto,labelsize] {$S!$} 
            (3);
      \draw [->] 
            (0) to node (l)[auto,swap,labelsize] {$\pi_1$} 
            (2);
      \draw [->] 
            (2) to node (b)[auto,swap,labelsize] {$q$} 
            (3);  

     \draw (0.2,0.4) -- (0.6,0.4) -- (0.6,0.8);
\end{tikzpicture}
\end{equation}
where $!\colon P\longrightarrow 1$ is the unique morphism to the terminal 
object.
The monoidal product $(Q,q)\otimes(P,p)\in\cat{C}/S1$
is $((Q,q)\ast P, \mu_1\circ Sp\circ \pi_2)$:
\[
\begin{tikzpicture}[baseline=-\the\dimexpr\fontdimen22\textfont2\relax ]
      \node(0) at (-0.5,0) {$(Q,q)\ast P$};
      \node(1) at (2,0) {$SP$};
      \node(2) at (4,0) {$SS1$};
      \node(3) at (6,0) {$S1.$};
      
      \draw [->] 
            (0) to node (t)[auto,labelsize] {$\pi_2$} 
            (1);
      \draw [->] 
            (1) to node (r)[auto,labelsize] {$Sp$} 
            (2);
      \draw [->] 
            (2) to node (l)[auto,labelsize] {$\mu_1$} 
            (3); 
\end{tikzpicture}
\]
\end{itemize}
We remark that this monoidal category arises as a restriction of Burroni's 
bicategory of $\monoid{S}$-spans~\cite{Burroni_T_cats}.

\begin{definition}
\label{def:S_operad}
Let $\cat{C}$ be a large category with all finite limits 
and $\monoid{S}=(S,\eta,\mu)$ a cartesian monad on $\cat{C}$.
\begin{enumerate}
\item An \defemph{$\monoid{S}$-operad} is a monoid object in the monoidal 
category
$(\cat{C}/S1,I,\otimes)$ introduced above; see Definition~\ref{def:monoid_obj}
for the definition of monoid object in a monoidal category.
\item A \defemph{morphism of $\monoid{S}$-operads} is a homomorphism 
of monoid objects; see Definition~\ref{def:monoid_obj} again.
\end{enumerate}
We denote the category of $\monoid{S}$-operads by $\Operad{\monoid{S}}$;
by definition it is identical to the category $\Mon{\cat{C}/S1}$ of monoid
objects in $\cat{C}/S1$.
\end{definition}
We normally write an $\monoid{S}$-operad as $\monoid{T}=((\arity{T}\colon 
T\longrightarrow S1),e,m)$, where $e\colon 1\longrightarrow T$ and $m\colon 
(T,\arity{T})\ast T\longrightarrow T$ are morphisms in $\cat{C}$.
The reason for the notation $\ar{T}$ is that often the object $S1$ in $\cat{C}$
may be interpreted as the object of arities, $T$ as the object of all
(derived) operations of the algebraic theory expressed by $\monoid{T}$, 
and $\ar{T}$ as assigning the arity to each operation.
Sometimes we also write an $\monoid{S}$-operad simply as $\monoid{T}=(T,e,m)$,
and in this case $T$ refers to an object of $\cat{C}/S1$ (rather than 
$\cat{C}$).

\begin{example}[{\cite[Example~4.2.7]{Leinster_book}}]
\label{ex:non_symm_op_as_gen_op}
If we let $\cat{C}=\Set$ and $\monoid{S}$ be the free monoid monad
(which is cartesian), 
then $\monoid{S}$-operads are equivalent to {non-symmetric operads}.
The arities are the natural numbers: $S1\cong\mathbb{N}$.

In more detail, the data of an $\monoid{S}$-operad in this case
consist of a set $T$, and functions 
$\arity{T}\colon T\longrightarrow \N$, $e\colon 1\longrightarrow T$
and $m\colon (T,\arity{T})\ast T\longrightarrow T$. 
Unravelling this, we obtain a graded set $(T_n)_{n\in\N}$, 
an element $\id{}\in T_1$ and 
a family of functions $(m_{k,n_1,\dots n_k}\colon T_k\times T_{n_1}\times\dots
\times T_{n_k}
\longrightarrow 
T_{n_1+\dots +n_k})_{k,n_1,\dots,n_k\in\N}$,
agreeing with Definition~\ref{def:non_symm_op}.
Note that indeed $T_n$ may be interpreted as the set of all (derived) 
operations of arity $n$.
\end{example}

\begin{example}
If we set $\cat{C}=\enGph{n}$, the category of $n$-graphs for 
$n\in\NN\cup\{\omega\}$ and $\monoid{S}$ be the free strict $n$-category monad, 
then 
$\monoid{S}$-operads are called \emph{$n$-globular operads};
see \cite[Chapter~8]{Leinster_book} for illustrations.
These generalised operads have been used to give a definition of 
weak $n$-categories, and they (and their generalisations) will
play a central role in the second part of this thesis.
\end{example}

Next we define models of an $\monoid{S}$-operad.
For this, we first show that the monoidal category $\cat{C}/S1$ has a canonical
pseudo action on $\cat{C}$.
Pseudo actions of monoidal categories are a category version of actions of 
monoids.
The precise definition of pseudo action is a variant of 
Definition~\ref{def:oplax_action}, obtained by replacing the term ``natural 
transformation'' there by ``natural isomorphism''.
The functor 
\[
\ast\colon(\cat{C}/S1)\times\cat{C}\longrightarrow\cat{C}
\]
defining this pseudo action is given by mapping $((Q,q),P)\in 
(\cat{C}/S1)\times\cat{C}$ to
$(Q,q)\ast P\in\cat{C}$ defined as the pullback (\ref{eqn:pb_over_S1}).
\begin{definition}
\label{def:S-operad_model}
Let $\cat{C}$ be a large category with finite limits, $\monoid{S}=(S,\eta,\mu)$
a cartesian monad on $\cat{C}$, and $\monoid{T}=(T,e,m)$ an 
$\monoid{S}$-operad.
\begin{enumerate}
\item A \defemph{model of $\monoid{T}$} consists of:
\begin{itemize}
\item an object $C\in\cat{C}$;
\item a morphism $\gamma\colon T\ast C\longrightarrow C$ in $\cat{C}$,
\end{itemize}
making the following diagrams commute:
\begin{equation*}
\begin{tikzpicture}[baseline=-\the\dimexpr\fontdimen22\textfont2\relax ]
      \node (TL) at (0,2)  {$I\ast C$};
      \node (TR) at (2,2)  {$T\ast C$};
%      \node (BL) at (0,0) {$T$};
      \node (BR) at (2,0) {${C}$};
      \draw[->]  (TL) to node[auto,labelsize] {$e\ast C$} (TR);
      \draw[->]  (TR) to node[auto,labelsize] {$\gamma$} (BR);
      \draw[->] (TL) to node[auto,swap,labelsize](B) {$\cong$} (BR);
\end{tikzpicture} 
\qquad
\begin{tikzpicture}[baseline=-\the\dimexpr\fontdimen22\textfont2\relax ]
      \node (TL) at (0,2)  {$(T\otimes T)\ast C$};
      \node (TR) at (5,2)  {$T\ast C$};
      \node (BL) at (0,0) {$T\ast (T\ast C)$};
      \node (BM) at (3,0) {$T\ast C$};
      \node (BR) at (5,0) {${C},$};
      \draw[->]  (TL) to node[auto,swap,labelsize] {$\cong$} (BL);
      \draw[->]  (TL) to node[auto,labelsize] {$m\ast C$} (TR);
      \draw[->]  (BL) to node[auto,labelsize] {$T\ast \gamma$} (BM);
      \draw[->]  (BM) to node[auto,labelsize] {$\gamma$} (BR);
      \draw[->]  (TR) to node[auto,labelsize](B) {$\gamma$} (BR);
\end{tikzpicture} 
\end{equation*}
where the arrows labelled with $\cong$ refer to the isomorphisms 
provided by the pseudo action.
\item Let $(C,\gamma)$ and $(C',\gamma')$ be models of $\monoid{T}$. A 
\defemph{homomorphism from $(C,\gamma)$ to 
$(C',\gamma')$}
is a morphism $f\colon C\longrightarrow C'$ in $\cat{C}$ making the 
following diagram commute:
\[
\begin{tikzpicture}[baseline=-\the\dimexpr\fontdimen22\textfont2\relax ]
      \node (TL) at (0,2)  {$T\ast C$};
      \node (TR) at (3,2)  {$T\ast C'$};
      \node (BL) at (0,0) {$C$};
      \node (BR) at (3,0) {$C'.$};
      \draw[->]  (TL) to node[auto,swap,labelsize] {$\gamma$} (BL);
      \draw[->]  (TL) to node[auto,labelsize] {$T\ast f$} (TR);
      \draw[->]  (BL) to node[auto,labelsize] {${f}$} (BR);
      \draw[->]  (TR) to node[auto,labelsize](B) {$\gamma'$} (BR);
\end{tikzpicture} \qedhere
\] 
\end{enumerate}
\end{definition}

We remark that in the setting of Example~\ref{ex:non_symm_op_as_gen_op},
the models defined by the above definition coincide with the set models of  
Definition~\ref{def:ns_operad_model}.

\section{Other examples}
\label{sec:other_ex}
As our principal aim in the first part of this thesis 
is to develop a formal framework and not to study a variety of concrete 
examples of notions of algebraic theory in detail, we briefly mention 
other examples of notions of algebraic theory and conclude the chapter.

First, there are PROPs and PROs \cite{MacLane_cat_alg},
which are the ``many-in, many-out'' versions of symmetric operads and 
non-symmetric operads respectively. 
In contrast to operations in a symmetric or non-symmetric operad,
which we have drawn in string diagrams as 
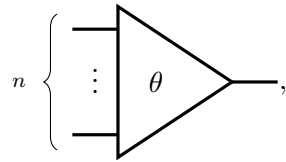
\begin{equation*}
%\label{eqn:triangle_theta}
\begin{tikzpicture}[baseline=-\the\dimexpr\fontdimen22\textfont2\relax ]
      \coordinate (TL) at (0,1);
      \coordinate (BL) at (0,-1);
      \coordinate (R) at (1.5,0);
      \path[draw,string] (TL)--(BL)--(R)--cycle;
      \node at (0.5,0) {$\theta$};
      \path[draw,string] (TL)+(0,-0.3)-- +(-0.6,-0.3);
      \node at (-0.3,0) {$\rvdots$};
      \path[draw,string] (BL)+(0,0.3)-- +(-0.6,0.3);
      \draw [decorate,decoration={brace,amplitude=5pt,mirror}]
      (TL)++(-0.8,-0.1) -- +(0,-1.8) node [midway,xshift=-0.5cm,labelsize] 
      {$n$};
      \path[draw,string] (R)-- +(0.6,0);
\end{tikzpicture},
\end{equation*}
operations in a PROP or PRO may be drawn as 
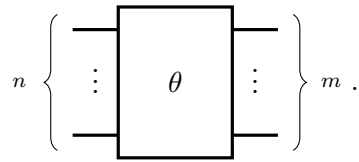
\begin{equation*}
%\label{eqn:triangle_theta}
\begin{tikzpicture}[baseline=-\the\dimexpr\fontdimen22\textfont2\relax ]
      \coordinate (TL) at (0,1);
      \coordinate (BL) at (0,-1);
      \coordinate (R) at (1.5,0);
      \coordinate (TR) at (1.5,1);
      \coordinate (BR) at (1.5,-1);
      \path[draw,string] (TL)--(BL)--(BR)--(TR)--cycle;
      \node at (0.75,0) {$\theta$};
      \path[draw,string] (TL)+(0,-0.3)-- +(-0.6,-0.3);
      \node at (-0.3,0) {$\rvdots$};
      \path[draw,string] (BL)+(0,0.3)-- +(-0.6,0.3);
      \draw [decorate,decoration={brace,amplitude=5pt,mirror}]
      (TL)++(-0.8,-0.1) -- +(0,-1.8) node [midway,xshift=-0.5cm,labelsize] 
      {$n$};
      \path[draw,string] (TR)+(0,-0.3)-- +(0.6,-0.3);
      \node at (1.8,0) {$\rvdots$};
      \path[draw,string] (BR)+(0,0.3)-- +(0.6,0.3);
      \draw [decorate,decoration={brace,amplitude=5pt}]
      (TR)++(0.8,-0.1) -- +(0,-1.8) node [midway,xshift=0.5cm,labelsize] 
      {$m$};
\end{tikzpicture}.
\end{equation*}

Another class of examples 
would be the multi-sorted versions of clones, symmetric and 
non-symmetric operads, known as \emph{multicategories}.
They are included in the work by Curien~\cite{Curien_operad} and 
Hyland~\cite{Hyland_elts}.
%We also have the multi-sorted versions of PROPs and PROs, known as 
%\emph{polycategories}.

Finally we mention enriched algebraic theories, such as 
enriched Lawvere theories 
\cite{Power_enriched_Lawvere}, the enriched versions of symmetric and 
non-symmetric operads \cite{May_loop,Kelly_operad},
and enriched monads \cite{Dubuc_Kan}.

We expect that these examples can also be incorporated into our framework
without much difficulty (for the enriched algebraic theories, we would have to 
develop the enriched version of our framework), but will not treat them further 
in this thesis.

\chapter{The framework}
\label{chap_framework}
%The aim of this chapter is to develop a unified framework for 
%various notions of algebraic theory. 
%We have seen several examples of notions of algebraic theory
%in the previous chapter. 

%There are several levels in the study of algebra.
%\begin{enumerate}
%\item At the first level, we have \emph{algebras}.
%\item At the second level, we have \emph{algebraic theories}; an
%algebraic theory is a specification of a type of algebras.
%\item At the third level, we have \emph{notions of algebraic theory};
%a notion of algebraic theory is a background theory for a type
%of algebraic theories.
%\end{enumerate}

In the previous chapter we have seen several examples of notions of algebraic 
theory,
in which the corresponding types of {algebraic theories} are called
under various names, such as \emph{clones}, \emph{non-symmetric operads} and 
\emph{monads (on $\cat{C}$)}.
Being a background theory for a type of algebraic theories,
each notion of algebraic theory has definitions of algebraic theory, 
of model of an algebraic theory
and of homomorphism between models.
Nevertheless, different notions of algebraic theory take 
different approaches to define these concepts, and the 
resulting definitions (say, of algebraic theory) can look quite remote.

The aim of this chapter is to provide a 
\emph{unified framework for notions of algebraic theory}
which includes all of the notions
of algebraic theory reviewed in the main body of the previous chapter
(Sections~\ref{sec:clone}--\ref{sec:S-operad}) 
as instances.
To the best of our knowledge, this is the first framework for notions of 
algebraic theory attaining such generality. 
Due to the diversity of notions of algebraic theory
we aim to capture, we take a very simple approach.
The basic idea is that we identify \emph{notions of algebraic theory}
with (large) \emph{monoidal categories},
and \emph{algebraic theories} with \emph{monoid objects} therein.
We also give definitions of models of an algebraic theory 
and of their homomorphisms
(relative to a notion of model).
Further consequences of this framework 
will be investigated in the subsequent chapters.
%In Chapter~\ref{chap:str_sem}
%we unify and extend the classical \emph{structure-semantics adjunctions}.
%In Chapter~\ref{chap:double_lim} we unify the constructions of various 
%categories of models as a certain double limits, generalising the 
%characterisation by Street of Eilenberg--Moore categories as 2-categorical 
%limits.

In Section~\ref{sec:framework_prelude}, we 
motivate our framework by reformulating the notions of algebraic theories
reviewed in the previous chapter using the structure of monoidal category.
We expect that the contents of this section are mostly known to the 
specialists, and try to refer to related papers that
have come to our attention.
The main body of our framework, developed from 
Section~\ref{sec:framework_basic} on, is our original contribution.

\section[Prelude]{Prelude: monoidal categorical perspectives on notions of 
algebraic theory}
\label{sec:framework_prelude}
In this section we motivate our framework by illuminating the 
key role that certain monoidal categories play in 
both syntax and semantics of various notions of algebraic theory.

\subsection{Algebraic theories as monoid objects}
\label{subsec:alg_thy_as_monoids}
We begin with the observation that algebraic theories in each of the notions
of algebraic theory reviewed in Chapter~\ref{chap_notions_ex}
may be understood as monoid objects in a certain monoidal category.

See \cite[Chapter~VII]{MacLane_CWM} or \cite[Section~1.1]{Kelly:enriched}
for the definition of monoidal category.
We normally write the unit object of a monoidal category as $I$
and the monoidal product as $\otimes$.
We will denote the coherent structural isomorphisms (obtained from 
associativity, and left and right unit isomorphisms)
by arrows labelled with $\cong$ (see below).

\begin{definition}
\label{def:monoid_obj}
Let $\cat{M}=(\cat{M},I,\otimes)$ be a large monoidal category.
\begin{enumerate}
\item 
A \defemph{monoid object in $\cat{M}$} (or simply a \defemph{monoid in 
$\cat{M}$}) is a triple $\monoid{T}=(T,e,m)$ 
consisting of
an object $T$ in $\cat{M}$, and morphisms $e\colon I\longrightarrow T$
and $m\colon T\otimes T\longrightarrow T$ in $\cat{M}$,
such that the following diagrams commute:
\begin{equation*}
%\label{eqn:action_model}
\begin{tikzpicture}[baseline=-\the\dimexpr\fontdimen22\textfont2\relax ]
      \node (TL) at (0,1)  {$I\otimes T$};
      \node (TR) at (2,1)  {$T\otimes T$};
      \node (BR) at (2,-1) {$T$};
      \draw[->] (TL) to node[auto,labelsize](T) {$e\otimes T$} (TR);
      \draw[->] (TR) to node[auto,labelsize](R) {$m$} (BR);
      \draw[->]  (TL) to node[auto,labelsize,swap] (L) {$\cong$} (BR);
\end{tikzpicture} 
\qquad
\begin{tikzpicture}[baseline=-\the\dimexpr\fontdimen22\textfont2\relax ]
      \node (TL) at (0,1)  {$T\otimes I$};
      \node (TR) at (2,1)  {$T\otimes T$};
      \node (BR) at (2,-1) {$T$};
      \draw[->] (TL) to node[auto,labelsize](T) {$T\otimes e$} (TR);
      \draw[->] (TR) to node[auto,labelsize](R) {$m$} (BR);
      \draw[->]  (TL) to node[auto,labelsize,swap] (L) {$\cong$} (BR);
\end{tikzpicture} 
\end{equation*}
\begin{equation*}
\begin{tikzpicture}[baseline=-\the\dimexpr\fontdimen22\textfont2\relax ]
      \node (TL) at (0,2)  {$(T\otimes T)\otimes T$};
      \node (TR) at (5,2)  {$T\otimes T$};
      \node (BL) at (0,0) {$T\otimes (T\otimes T)$};
      \node (BM) at (3,0) {$T\otimes T$};
      \node (BR) at (5,0) {$T$};
      \draw[->] (TL) to node[auto,labelsize](T) {$m\otimes T$} (TR);
      \draw[->]  (TR) to node[auto,labelsize] {$m$} (BR);
      \draw[->]  (TL) to node[auto,swap,labelsize] {$\cong$} (BL);
      \draw[->]  (BL) to node[auto,labelsize] {$T\otimes m$} (BM);
      \draw[->] (BM) to node[auto,labelsize](B) {$m$} (BR);
\end{tikzpicture} 
\end{equation*}
(Recall that the arrows labelled with $\cong$ are the 
suitable instances of structural isomorphisms 
of $\cat{M}$.)
\item Let $\monoid{T}=(T,e,m)$ and $\monoid{T'}=(T',e',m')$ be monoid objects in
$\cat{M}$. 
A \defemph{homomorphism} from $\monoid{T}$ to $\monoid{T'}$ is a morphism
$f\colon T\longrightarrow T'$ in $\cat{M}$ such that the following diagrams
commute:
\begin{equation*}
\begin{tikzpicture}[baseline=-\the\dimexpr\fontdimen22\textfont2\relax ]
      \node (TL) at (1,2)  {$I$};
%      \node (TR) at (2,2)  {$T\otimes T$};
      \node (BL) at (0,0) {$T$};
      \node (BR) at (2,0) {$T'$};
      \draw[->]  (TL) to node[auto,swap,labelsize] {$e$} (BL);
      \draw[->]  (BL) to node[auto,labelsize] {$f$} (BR);
      \draw[->] (TL) to node[auto,labelsize](B) {$e'$} (BR);
\end{tikzpicture} 
\qquad
\begin{tikzpicture}[baseline=-\the\dimexpr\fontdimen22\textfont2\relax ]
      \node (TL) at (0,2)  {$T\otimes T$};
      \node (TR) at (3,2)  {$T'\otimes T'$};
      \node (BL) at (0,0) {$T$};
      \node (BR) at (3,0) {$T'$};
      \draw[->]  (TL) to node[auto,swap,labelsize] {$m$} (BL);
      \draw[->]  (TL) to node[auto,labelsize] {$f\otimes f$} (TR);
      \draw[->]  (BL) to node[auto,labelsize] {$f$} (BR);
      \draw[->]  (TR) to node[auto,labelsize](B) {$m'$} (BR);
\end{tikzpicture} 
\end{equation*}
\end{enumerate}
The category of all monoid objects in $\cat{M}$ and homomorphisms 
is denoted by $\Mon{\cat{M}}$.
\end{definition}

\subsubsection{Clones as monoid objects in $[\F,\Set]$}
Clones (Definition~\ref{def:clone}) may be identified with monoid objects
in a certain monoidal category.
We first describe the underlying category.

\begin{definition}
Let $\F$ be the category defined as follows:
\begin{itemize}
\item The set of objects is $\ob{\F}=\{\,[n]\mid n\in\NN\,\}$,
where for each natural number $n\in\NN$, $[n]$ is 
defined to be the $n$-element set $\{1,\dots,n\}$.
\item A morphism is any function between these sets.\qedhere
\end{itemize}
\end{definition}
So the category $\F$ is a skeleton of the category $\FinSet$
of all (small) finite sets and functions.
The underlying category of the monoidal category for clones 
is the category $[\F,\Set]$ of 
all functors from $\F$ to $\Set$ and natural transformations.
For $X\in[\F,\Set]$ and $[n]\in\F$,
we write the set $X([n])$ as $X_n$.

We view an object $X\in[\F,\Set]$ as a functional signature,
just as we viewed a graded set as a functional signature in 
Section~\ref{sec:univ_alg}. 
However, objects in $[\F,\Set]$ have richer structure than
graded sets, namely the action of morphisms in $\F$.
We can understand this additional structure as certain
basic operations on function symbols
in the signature.
For instance, given a morphism $u\colon [3]\longrightarrow [4]$ in $\F$
with $u(1)=4, u(2)=2, u(3)=4$ and 
an element $\theta\in X_3$, 
the element $X_u(\theta)\in X_4$ may be drawn as
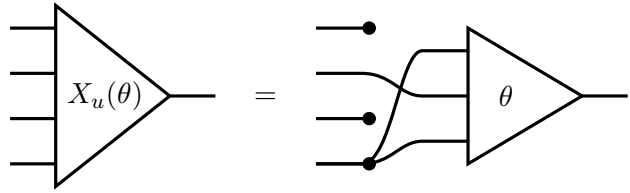
\[
\begin{tikzpicture}[baseline=-\the\dimexpr\fontdimen22\textfont2\relax ]
      \coordinate (TL) at (0,1.2);
      \coordinate (BL) at (0,-1.2);
      \coordinate (R) at (1.5,0);
      \path[draw,string] (TL)--(BL)--(R)--cycle;
      \node at (0.65,0) {$X_u(\theta)$};
            \path[draw,string] (R)-- +(0.6,0);
      \draw [string] (0,0.9) -- (-0.6,0.9);
      \draw [string] (0,0.3) -- (-0.6,0.3);
      \draw [string] (0,-0.3) -- (-0.6,-0.3);
      \draw [string] (0,-0.9) -- (-0.6,-0.9);
\end{tikzpicture}
\quad=\quad
\begin{tikzpicture}[baseline=-\the\dimexpr\fontdimen22\textfont2\relax ]
      \coordinate (TL) at (0,0.9);
      \coordinate (BL) at (0,-0.9);
      \coordinate (R) at (1.5,0);
      \path[draw,string] (TL)--(BL)--(R)--cycle;
      \node at (0.5,0) {$\theta$};
            \path[draw,string] (R)-- +(0.6,0);
      \node [circle,draw,inner sep=0.15em,fill] at (-1.3,-0.9) {};
      \node [circle,draw,inner sep=0.15em,fill] at (-1.3,0.9) {};
      \node [circle,draw,inner sep=0.15em,fill] at (-1.3,-0.3) {};
      \draw [string] (-2,0.3) -- (-1.4,0.3) ..controls (-1,0.3) and  (-0.9,0) 
      .. (-0.6,0)--(0,0);
      \draw [string] (-2,-0.9) -- (-1.4,-0.9) ..controls (-1,-0.9) and  
      (-0.9,0.6) .. (-0.6,0.6)--(0,0.6);
      \draw [string] (-2,-0.9) -- (-1.4,-0.9) ..controls (-1,-0.9) and  
      (-0.9,-0.6) .. (-0.6,-0.6)--(0,-0.6);      
      \draw [string] (-2,0.9) -- (-1.3,0.9);
      \draw [string] (-2,-0.3) -- (-1.3,-0.3);
\end{tikzpicture}
\]
(we are using the string diagram notation introduced in 
Section~\ref{sec:univ_alg}).
In the symbolic notation, 
\[
(X_u(\theta))(x_1,x_2,x_3,x_4)=\theta(x_4,x_2,x_4).
\]

The monoidal structure on the category $[\F,\Set]$
we shall consider is known as the \emph{substitution monoidal 
structure}.
\begin{definition}[\cite{Kelly_Power_coeq,FPT99}]
The \defemph{substitution monoidal structure} on $[\F,\Set]$ is defined as 
follows:
\begin{itemize}
\item The unit object is $I=\F([1],-)\colon \F\longrightarrow\Set$.
\item Given $X,Y\colon \F\longrightarrow\Set$, their monoidal product 
$Y\otimes X\colon \F\longrightarrow\Set$ maps $[n]\in\F$ to 
\begin{equation}
\label{eqn:subst_tensor_F}
(Y\otimes X)_n=\int^{[k]\in\F} Y_k\times (X_n)^k.
\end{equation}
\end{itemize}
\end{definition}
The integral sign with a superscript in 
(\ref{eqn:subst_tensor_F}) stands for a \emph{coend}
(dually, we will denote an \emph{end} by the integral sign with a subscript);
see \cite[Section~IX.~6]{MacLane_CWM}.
By definition, this coend
is a suitable quotient of the set 
\[
\coprod_{[k]\in\F} Y_k\times (X_n)^k,
\]
whose element we may draw as (\ref{eqn:clone_composition}),
assuming $\phi\in Y_k$ and $\theta_1,\dots,\theta_k\in X_n$;
the idea is that $\otimes$ performs a ``sequential composition''
of signatures.
Note that symbolically this indeed amounts to a (simultaneous) 
\emph{substitution}.

\medskip

We claim that clones are essentially the same as 
monoids in $[\F,\Set]$
(with respect to the substitution monoidal structure). 
A monoid in $[\F,\Set]$ consists of
\begin{itemize}
\item a functor $T\colon \F\longrightarrow \Set$;
\item a natural transformation $e\colon I\longrightarrow T$;
\item a natural transformation $m\colon T\otimes T\longrightarrow T$
\end{itemize}
satisfying the monoid axioms.
By the Yoneda lemma, $e$ corresponds to an element $\overline{e}\in T_1$,
and by the universality of coends, $m$ corresponds to
a natural transformation\footnote{In more detail, the relevant naturality
here may also be phrased as 
``natural in $[n]\in\F$ and \emph{extranatural} in $[k]\in\F$''; 
see \cite[Section~IX.~4]{MacLane_CWM}.
Following \cite{Kelly:enriched}, in this thesis we shall not distinguish 
(terminologically)
extranaturality from naturality, using the latter term for both.}
\[
(\overline{m}_{n,k}\colon T_k\times (T_n)^k\longrightarrow T_n)_{n,k\in \NN}.
\]
Hence given a monoid $(T,e,m)$ in $[\F,\Set]$, we can construct a 
clone with the underlying graded set $(T_n)_{n\in\NN}$ by setting 
$p^{(n)}_i=T_\name{i}(\overline{e})$ (here, 
$\name{i}\colon[1]\longrightarrow[n]$
is the morphism in $\F$ defined as $\name{i}(1)=i$) and 
$\circ^{(n)}_k=\overline{m}_{n,k}$.
Conversely, given a clone $(T,p,\circ)$,
we can construct a monoid in $[\F,\Set]$ as follows.
First we extend the graded set $T$ 
to a functor $T\colon \F\longrightarrow\Set$
by setting, for any $u\colon [m]\longrightarrow [n]$ in $\F$,
\[
T_u (\theta)= \theta \circ^{(n)}_m (p^{(m)}_{u(1)},\dots, p^{(m)}_{u(m)}).
\]
Then we may set $\overline{e}=p^{(1)}_1$ and $\overline{m}_{n,k}=\circ^{(n)}_k$.

\begin{proposition}[cf.~\cite{Curien_operad,Hyland_elts}]
The above constructions establish an isomorphism of categories 
between the category of clones and $\Mon{[\F,\Set]}$.
\end{proposition}

\subsubsection{Symmetric operads as monoid objects in $[\Pcat,\Set]$}
Symmetric operads (Definition~\ref{def:symm_op}) can be similarly
seen as monoids.
The main difference from the case of clones is that,
instead of the category $\F$, we use the following category.

\begin{definition}
Let $\Pcat$ be the category defined as follows:
\begin{itemize}
\item The set of objects is the same as $\F$: $\ob{\Pcat}=\{\,[n]\mid 
n\in\NN\,\}$ where $[n]=\{1,\dots,n\}$.
\item A morphism is any \emph{bijective} function.\qedhere
\end{itemize}
\end{definition}
So $\Pcat$ is the subcategory of $\F$ consisting of all isomorphisms.
For any $[n]\in\Pcat$, the monoid $\Pcat([n],[n])$ of endomorphisms on $[n]$
is isomorphic to the symmetric group $\mathfrak{S}_n$.

Symmetric operads are
monoids in a monoidal category whose underlying category is 
the functor category $[\Pcat,\Set]$.
We again interpret $[\Pcat,\Set]$ as a category of functional signatures, but 
this time a signature 
$X\in[\Pcat,\Set]$
is only equipped with action of morphisms in $\Pcat$.
In terms of string diagrams, this amounts to restricting 
the class of diagrams by 
prohibiting the use of
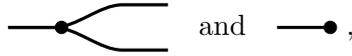
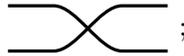
\begin{equation*}
\begin{tikzpicture}[baseline=-\the\dimexpr\fontdimen22\textfont2\relax ]
      \draw [string] (0,0) -- (0.6,0) ..controls (0.9,0) and  (1.2,-0.3) 
      .. (1.5,-0.3)--(2.1,-0.3);
      \draw [string] (0,0) -- (0.6,0) ..controls (0.9,0) and  
      (1.2,0.3) .. (1.5,0.3)--(2.1,0.3);
      \node [circle,draw,inner sep=0.15em,fill] at (0.7,0) {};
\end{tikzpicture}
\quad\text{and}\quad
\begin{tikzpicture}[baseline=-\the\dimexpr\fontdimen22\textfont2\relax ]
      \draw [string] (0,0) -- (0.7,0);
      \node [circle,draw,inner sep=0.15em,fill] at (0.7,0) {};
\end{tikzpicture}\ ,
\end{equation*}
but not 
\[
\begin{tikzpicture}[baseline=-\the\dimexpr\fontdimen22\textfont2\relax ]
      \draw [string] (0,0.3) -- (0.6,0.3) ..controls (0.9,0.3) and  (1.2,-0.3) 
      .. (1.5,-0.3)--(2.1,-0.3);
      \draw [string] (0,-0.3) -- (0.6,-0.3) ..controls (0.9,-0.3) and  
      (1.2,0.3) .. (1.5,0.3)--(2.1,0.3);
\end{tikzpicture}\ ;
\]
in terms of symbolic representations, we are
restricting $\Sigma$-terms to regular $\Sigma$-terms.

\medskip

The monoidal structure on $[\Pcat,\Set]$ we shall use is 
also called the substitution monoidal structure.
\begin{definition}[\cite{Kelly_operad}]
\label{def:subst_monoidal_P}
The \defemph{substitution monoidal structure} on $[\Pcat,\Set]$ is defined as 
follows:
\begin{itemize}
\item The unit object is $I=\Pcat([1],-)\colon\Pcat\longrightarrow\Set$.
\item Given $X,Y\colon\Pcat\longrightarrow\Set$, their monoidal product 
$Y\otimes X\colon \Pcat\longrightarrow\Set$ maps $[n]\in\Pcat$ to
\begin{equation*}
(Y\otimes X)_n
=\int^{[k]\in\Pcat} 
Y_k\times 
(X^{\circledast k})_n,
\end{equation*}
where 
\[
(X^{\circledast k})_n=\int^{[n_1],\dots,[n_k]\in\Pcat} 
\Pcat([n_1+\dots+n_k],[n])\times 
X_{n_1}\times\dots\times X_{n_k}.\qedhere
\]
\end{itemize}
\end{definition}

\begin{proposition}[cf.~\cite{Curien_operad,Hyland_elts}]
The category of symmetric operads is isomorphic to $\Mon{[\Pcat,\Set]}$.
\end{proposition}

\subsubsection{Non-symmetric operads as monoid objects in $[\Ncat,\Set]$}
For non-symmetric operads (Definition~\ref{def:non_symm_op}), 
we use the following category.

\begin{definition}
Let $\Ncat$ be the category defined as follows:
\begin{itemize}
\item The set of objects is the same as $\F$ and $\Pcat$.
\item There are only identity morphisms in $\Ncat$.\qedhere
\end{itemize}
\end{definition}
$\Ncat$ is the discrete category with the same objects as $\F$ and $\Pcat$.

We consider the functor category $[\Ncat,\Set]$,
which is nothing but the category of graded sets and their morphisms 
(Definition~\ref{def:graded_set}).

\begin{definition}
The \defemph{substitution monoidal structure} on $[\Ncat,\Set]$ is defined as 
follows:
\begin{itemize}
\item The unit object is $I=\Ncat([1],-)\colon \Ncat\longrightarrow \Set$.
\item Given $X,Y\colon \Ncat\longrightarrow\Set$, their monoidal product
$Y\otimes X\colon\Ncat\longrightarrow\Set$ maps $[n]\in\Ncat$ to 
\[
(Y\otimes X)_n=\coprod_{[k]\in\Ncat} 
Y_k\times\bigg(\coprod_{\substack{[n_1],\dots,[n_k]\in\Ncat\\n_1+\dots+n_k=n}}
X_{n_1}\times\dots\times X_{n_k}\bigg).\qedhere
\]
\end{itemize}
\end{definition}

\begin{proposition}[cf.~\cite{Curien_operad,Hyland_elts}]
The category of non-symmetric operads is isomorphic to $\Mon{[\Ncat,\Set]}$.
\end{proposition}

\medskip

For unified studies of various substitution monoidal structures,
see \cite{Tanaka_Power_psd_dist,Fiore_et_al_rel_psdmnd}, as well as the 
aforementioned \cite{Curien_operad,Hyland_elts}.

\subsubsection{Monads on $\cat{C}$ as monoid objects in $[\cat{C},\cat{C}]$}
Monads on a large category $\cat{C}$ (Definition~\ref{def:monad}) are also 
monoid objects, this time rather immediately from the definition.
\begin{definition}
Let $\cat{C}$ be a large category.
Define the monoidal category 
$[\cat{C},\cat{C}]=([\cat{C},\cat{C}],\id{\cat{C}},\circ)$ of endofunctors on 
$\cat{C}$ as follows:
\begin{itemize}
\item The underlying category is the category $[\cat{C},\cat{C}]$ of all
functors $\cat{C}\longrightarrow\cat{C}$ and natural transformations.
\item The unit object is the identity functor $\id{C}$ on $\cat{C}$.
\item The monoidal product is given by composition of functors.\qedhere
\end{itemize}
\end{definition}
The category $\Mnd{\cat{C}}$ of monads on $\cat{C}$ is clearly identical to 
$\Mon{[\cat{C},\cat{C}]}$.

\subsubsection{$\monoid{S}$-operads as monoid objects in $\cat{C}/S1$}
Finally, we recall that 
generalised operads ($\monoid{S}$-operads for a cartesian monad 
$\monoid{S}$ on a large category $\cat{C}$ with finite limits; 
Definition~\ref{def:S_operad})
were introduced as monoid objects in the first place.

\subsection{Notions of model as enrichments}
\label{subsec:enrichment}
In this section and next, we shall rephrase definitions of \emph{model} of 
an algebraic theory via monoidal categorical structures.

We start with a discussion on \emph{notions of model}.
An important feature of several notions of algebraic 
theory---most notably clones, symmetric operads and non-symmetric 
operads---is that we may consider 
models of an algebraic theory in more than one category.
For example, it is known that 
models of a clone can be taken in any category with 
finite products~\cite{Lawvere_thesis} (or even with finite powers).
We may phrase this fact by saying that clones admit many {notions of 
model}, one for each category with finite products.

Informally, a \emph{notion of model} for a notion of algebraic theory 
is a definition of model of an algebraic theory in that notion of algebraic 
theory.
Hence whenever we consider actual \emph{models} of an algebraic theory, we 
must specify in advance a notion of model with respect to which the models are 
taken.
Our framework emphasises the inevitable fact that
\emph{models are always relative to notions of model},
by treating notions of model as independent mathematical structures.

But how can we formalise such notions of model?
Below we show that the standard notions of model for clones, symmetric operads 
and non-symmetric operads can be captured by a categorical structure
which we call \emph{enrichment}.
Recall that we identify notions of algebraic theory
with large monoidal categories.

\begin{definition}
\label{def:enrichment}
Let $\cat{M}=(\cat{M},I,\otimes)$ be a large monoidal category.
An \defemph{enrichment over $\cat{M}$} consists of:
\begin{itemize}
\item  a large category $\cat{C}$;
\item a functor 
$\enrich{-}{-}\colon\cat{C}^\op\times\cat{C}\longrightarrow\cat{M}$;
\item a natural transformation $(j_C\colon 
I\longrightarrow\enrich{C}{C})_{C\in\cat{C}}$;
\item a natural transformation 
$(M_{A,B,C}\colon\enrich{B}{C}\otimes\enrich{A}{B}
\longrightarrow\enrich{A}{C})_{A,B,C\in\cat{C}}$,
\end{itemize}
making the following diagrams commute for all $A,B,C,D\in\cat{C}$:
\begin{equation*}
\begin{tikzpicture}[baseline=-\the\dimexpr\fontdimen22\textfont2\relax ]
      \node (TL) at (0,1)  {$I\otimes\enrich{A}{B}$};
      \node (TR) at (4,1)  {$\enrich{B}{B}\otimes\enrich{A}{B}$};
      \node (BR) at (4,-1) {$\enrich{A}{B}$};
      \draw[->] (TL) to node[auto,labelsize](T) {$j_B\otimes\enrich{A}{B}$} 
      (TR);
      \draw[->] (TR) to node[auto,labelsize](R) {$M_{A,B,B}$} (BR);
      \draw[->]  (TL) to node[auto,labelsize,swap] (L) {$\cong$} (BR);
\end{tikzpicture} 
\quad
\begin{tikzpicture}[baseline=-\the\dimexpr\fontdimen22\textfont2\relax ]
      \node (TL) at (0,1)  {$\enrich{A}{B}\otimes I$};
      \node (TR) at (4,1)  {$\enrich{A}{B}\otimes\enrich{A}{A}$};
      \node (BR) at (4,-1) {$\enrich{A}{B}$};
      \draw[->] (TL) to node[auto,labelsize](T) {$\enrich{A}{B}\otimes j_A$} 
      (TR);
      \draw[->] (TR) to node[auto,labelsize](R) {$M_{A,A,B}$} (BR);
      \draw[->]  (TL) to node[auto,labelsize,swap] (L) {$\cong$} (BR);
\end{tikzpicture} 
\end{equation*}
\begin{equation*}
\begin{tikzpicture}[baseline=-\the\dimexpr\fontdimen22\textfont2\relax ]
      \node (TL) at (0,1)  {$(\enrich{C}{D}\otimes \enrich{B}{C})\otimes 
      \enrich{A}{B}$};
      \node (TR) at (9,1)  {$\enrich{B}{D}\otimes\enrich{A}{B}$};
      \node (BL) at (0,-1) 
      {$\enrich{C}{D}\otimes(\enrich{B}{C}\otimes\enrich{A}{B})$};
      \node (BM) at (6,-1) {$\enrich{C}{D}\otimes\enrich{A}{C}$};
      \node (BR) at (9,-1) {$\enrich{A}{D}.$};
      \draw[->] (TL) to node[auto,labelsize](T) 
      {$M_{B,C,D}\otimes\enrich{A}{B}$} 
      (TR);
      \draw[->]  (TR) to node[auto,labelsize] {$M_{A,B,D}$} (BR);
      \draw[->]  (TL) to node[auto,swap,labelsize] {$\cong$} (BL);
      \draw[->]  (BL) to node[auto,labelsize] {$\enrich{C}{D}\otimes 
      M_{A,B,C}$} (BM);
      \draw[->] (BM) to node[auto,labelsize](B) {$M_{A,C,D}$} (BR);
\end{tikzpicture} 
\end{equation*}
We say that $(\cat{C},\enrich{-}{-},j,M)$ is an enrichment over $\cat{M}$,
or that $(\enrich{-}{-},j,M)$ is an enrichment of $\cat{C}$ over $\cat{M}$.
\end{definition}

An enrichment over $\cat{M}$ is 
not the same as a (large) \emph{$\cat{M}$-category} in enriched category 
theory~\cite{Kelly:enriched}.
It is rather a triple consisting of a large category $\cat{C}$, 
a large $\cat{M}$-category $\cat{D}$,
and an identity-on-objects functor $J\colon \cat{C}\longrightarrow\cat{D}_0$,
where $\cat{D}_0$ is the underlying category of $\cat{D}$.

In detail, given an enrichment $(\enrich{-}{-},j,M)$ of $\cat{C}$ in $\cat{M}$,
we may define the $\cat{M}$-category $\cat{D}$ with $\ob{\cat{D}}=\ob{\cat{C}}$
using the data $(\enrich{-}{-},j,M)$ of the enrichment 
(that is, $\cat{D}(A,B)=\enrich{A}{B}$ and so on).
The identity-on-objects functor $J\colon \cat{C}\longrightarrow\cat{D}_0$
may be defined by mapping a morphism $f\colon A\longrightarrow B$
in $\cat{C}$ to
the composite $\enrich{A}{f}\circ j_A$, or equivalently, $\enrich{f}{B}\circ 
j_B$:
\[
\begin{tikzpicture}[baseline=-\the\dimexpr\fontdimen22\textfont2\relax ]
      \node (TL) at (0,2)  {$I$};
      \node (TR) at (3,2)  {$\enrich{B}{B}$};
      \node (BL) at (0,0) {$\enrich{A}{A}$};
      \node (BR) at (3,0) {$\enrich{A}{B}.$};
      \draw[->]  (TL) to node[auto,swap,labelsize] {$j_A$} (BL);
      \draw[->]  (TL) to node[auto,labelsize] {$j_B$} (TR);
      \draw[->]  (BL) to node[auto,labelsize] {$\enrich{A}{f}$} (BR);
      \draw[->]  (TR) to node[auto,labelsize](B) {$\enrich{f}{B}$} (BR);
\end{tikzpicture} 
\]
We say that an enrichment is \defemph{normal} if the corresponding 
identity-on-objects functor $J$ is an isomorphism of categories.
We shall return to the relationship to enriched category theory
at the end of this section.

\begin{comment}
Alternatively, we can describe an enrichment over $\cat{M}$
as a genuine enriched category, but over a base different
from $\cat{M}$.\footnote{We thank Paul-Andr{\'e} Melli{\`e}s for
pointing this out.}
It is well-known that 
the functor $(-)_0=\cat{M}(I,-)\colon\cat{M}\longrightarrow\Set$
becomes a lax monoidal functor in a canonical way.
This makes the comma category $\widetilde{\cat{M}}=\Set\downarrow(-)_0$
(whose objects are triples $(E\in\Set,X\in\cat{M},u\colon E\longrightarrow
X_0\text{ in }\Set)$)
into a monoidal category,
and an enrichment over $\cat{M}$
is precisely an $\widetilde{\cat{M}}$-category.
\end{comment}

\medskip

From an enrichment, we now derive a definition of model 
of an algebraic theory.
First observe that, given an enrichment $\enrich{-}{-}=(\enrich{-}{-},j,M)$ of 
a large category $\cat{C}$ over a large monoidal category $\cat{M}$
and an object $C\in\cat{C}$,
we have a monoid object $\End_{\enrich{-}{-}}(C)=(\enrich{C}{C},j_C,M_{C,C,C})$
in $\cat{M}$; that these data define a monoid object 
may be seen immediately from Definition~\ref{def:enrichment}.
Because we identify algebraic theories
with monoid objects,
we give a definition of model of a monoid object $\monoid{T}$ in $\cat{M}$.

\begin{definition}
\label{def:enrich_model}
Let $\cat{M}=(\cat{M},I,\otimes)$ be a large monoidal category, 
$\monoid{T}=(T,e,m)$ be a monoid object in $\cat{M}$, 
$\cat{C}$ be a large category, and $\enrich{-}{-}=(\enrich{-}{-},j,M)$ be an
enrichment of $\cat{C}$ over $\cat{M}$.
\begin{enumerate}
\item A \defemph{model of $\monoid{T}$ in $\cat{C}$ with respect to 
$\enrich{-}{-}$}
is a pair $(C,\chi)$ consisting of an object $C$ of $\cat{C}$
and a monoid homomorphism $\chi\colon 
\monoid{T}\longrightarrow\End_\enrich{-}{-}(C)$; that is, 
a morphism $\chi\colon T\longrightarrow \enrich{C}{C}$
in $\cat{M}$ making the following diagrams commute:
\begin{equation*}
\begin{tikzpicture}[baseline=-\the\dimexpr\fontdimen22\textfont2\relax ]
      \node (TL) at (0,2)  {$I$};
      \node (TR) at (2,2)  {$T$};
%      \node (BL) at (0,0) {$T$};
      \node (BR) at (2,0) {$\enrich{C}{C}$};
      \draw[->]  (TL) to node[auto,labelsize] {$e$} (TR);
      \draw[->]  (TR) to node[auto,labelsize] {$\chi$} (BR);
      \draw[->] (TL) to node[auto,swap,labelsize](B) {$j_C$} (BR);
\end{tikzpicture} 
\qquad
\begin{tikzpicture}[baseline=-\the\dimexpr\fontdimen22\textfont2\relax ]
      \node (TL) at (0,2)  {$T\otimes T$};
      \node (TR) at (4,2)  {$T$};
      \node (BL) at (0,0) {$\enrich{C}{C}\otimes\enrich{C}{C}$};
      \node (BR) at (4,0) {$\enrich{C}{C}.$};
      \draw[->]  (TL) to node[auto,swap,labelsize] {$\chi\otimes\chi$} (BL);
      \draw[->]  (TL) to node[auto,labelsize] {$m$} (TR);
      \draw[->]  (BL) to node[auto,labelsize] {$M_{C,C,C}$} (BR);
      \draw[->]  (TR) to node[auto,labelsize](B) {$\chi$} (BR);
\end{tikzpicture} 
\end{equation*}
\item Let $(C,\chi)$ and $(C',\chi')$ be models of $\monoid{T}$ in $\cat{C}$
with respect to $\enrich{-}{-}$. A \defemph{homomorphism from $(C,\chi)$ to 
$(C',\chi')$}
is a morphism $f\colon C\longrightarrow C'$ in $\cat{C}$ making the 
following diagram commute:
\[
\begin{tikzpicture}[baseline=-\the\dimexpr\fontdimen22\textfont2\relax ]
      \node (TL) at (0,2)  {$T$};
      \node (TR) at (3,2)  {$\enrich{C'}{C'}$};
      \node (BL) at (0,0) {$\enrich{C}{C}$};
      \node (BR) at (3,0) {$\enrich{C}{C'}.$};
      \draw[->]  (TL) to node[auto,swap,labelsize] {$\chi$} (BL);
      \draw[->]  (TL) to node[auto,labelsize] {$\chi'$} (TR);
      \draw[->]  (BL) to node[auto,labelsize] {$\enrich{C}{f}$} (BR);
      \draw[->]  (TR) to node[auto,labelsize](B) {$\enrich{f}{C'}$} (BR);
\end{tikzpicture} 
\] 
\end{enumerate}
We denote the (large) category of models of $\monoid{T}$ in $\cat{C}$ with 
respect to 
$\enrich{-}{-}$ by $\Mod{\monoid{T}}{(\cat{C},\enrich{-}{-})}$.
\end{definition}

The above definitions of model and homomorphism are reminiscent of 
ones for clones (Definitions~\ref{def:clone_model} and \ref{def:clone_mod_hom}),
symmetric operads and non-symmetric operads
(Definition~\ref{def:ns_operad_model}).
Indeed, we can restore the standard notions of model
for these notions of algebraic theory via suitable enrichments.

\begin{example}
\label{ex:clone_enrichment}
Recall that clones may be identified with 
monoids in $[\F,\Set]$ with the substitution monoidal structure.
Let $\cat{C}$ be a locally small\footnote{Recall that by 
Convention~\ref{conv:size}, ``locally small'' implies ``large''.} 
category with all finite powers.
We have an enrichment of $\cat{C}$ over $[\F,\Set]$ defined as follows:
\begin{itemize}
\item The functor $\enrich{-}{-}\colon \cat{C}^\op\times\cat{C}
\longrightarrow[\F,\Set]$
maps $A,B\in\cat{C}$ and $[n]\in\F$ to the set 
\[
\enrich{A}{B}_n=\cat{C}(A^n,B).
\]
\item The natural transformation $(j_C\colon I\longrightarrow 
\enrich{C}{C})_{C\in\cat{C}}$ corresponds by the Yoneda lemma
(recall that $I=\F([1],-)$) to the family 
\[
(\overline{j}_C=\id{C}\in \enrich{C}{C}_1)_{C\in\cat{C}}.
\] 
\item The natural transformation $(M_{A,B,C}\colon \enrich{B}{C}\otimes
\enrich{A}{B}\longrightarrow\enrich{A}{C})_{A,B,C\in\cat{C}}$
corresponds by the universality of coends (recall
that $(Y\otimes X)_{n}=\int^{[k]\in\F} Y_k\times (X_n)^k$) to 
the family whose $(A,B,C)$-th component is given by
\[
(\overline{M}_{A,B,C})_{n,k}\colon \enrich{B}{C}_k\times (\enrich{A}{B}_n)^k
\longrightarrow\enrich{A}{B}_n
\]
mapping $(g,f_1,\dots,f_k)\in \cat{C}(B^k,C)\times \cat{C}(A^n,B)^k$
to $g\circ\langle f_1,\dots,f_k\rangle\in\cat{C}(A^n,B)$.
\end{itemize}

Clearly the definition of the clone $\Endcl{A}$ from a set $A$ 
(Definition~\ref{def:endo_clone}) is derived from the above enrichment,
by setting $\cat{C}=\Set$.
Consequently, we restore the classical definitions of model 
(Definition~\ref{def:clone_model})
and homomorphism between models (Definition~\ref{def:clone_mod_hom}) for clones
as instances of Definition~\ref{def:enrich_model}.
\end{example}

\begin{example}
Symmetric operads may be identified with 
monoids in $[\Pcat,\Set]$ with the substitution monoidal structure.
Let $\cat{C}=(\cat{C},I',\otimes')$ 
be a locally small symmetric monoidal category.
We have an enrichment of $\cat{C}$ over $[\Pcat,\Set]$ defined as follows:
\begin{itemize}
\item The functor $\enrich{-}{-}\colon \cat{C}^\op\times\cat{C}
\longrightarrow[\Pcat,\Set]$
maps $A,B\in\cat{C}$ and $[n]\in\Pcat$ to the set 
\[
\enrich{A}{B}_n=\cat{C}(A^{\otimes' n},B),
\]
where $A^{\otimes' n}$ is the monoidal product of $n$ many copies of $A$.
\item The natural transformation $(j_C\colon I\longrightarrow 
\enrich{C}{C})_{C\in\cat{C}}$ corresponds by the Yoneda lemma
(recall that $I=\Pcat([1],-)$) to the family 
\[
(\overline{j}_C=\id{C}\in \enrich{C}{C}_1)_{C\in\cat{C}}.
\] 
\item The natural transformation $(M_{A,B,C}\colon \enrich{B}{C}\otimes
\enrich{A}{B}\longrightarrow\enrich{A}{C})_{A,B,C\in\cat{C}}$
corresponds by the universality of coends (recall 
Definition~\ref{def:subst_monoidal_P}) to 
the family whose $(A,B,C)$-th component is given by
\begin{multline*}
(\overline{M}_{A,B,C})_{n,k,n_1,\dots,n_k}\colon 
\enrich{B}{C}_k\times \Pcat([n_1+\dots+n_k],[n])\\ \times
\enrich{A}{B}_{n_1}\times\dots\times\enrich{A}{B}_{n_k}
\longrightarrow\enrich{A}{B}_n,
\end{multline*}
which is the unique function from the empty set if $n\neq n_1+\dots+n_k$
and, if $n= n_1+\dots+n_k$, maps 
$(g,u,f_1,\dots,f_k)\in \cat{C}(B^{\otimes' k},C)\times 
\Pcat([n_1+\dots+n_k],[n])\times \cat{C}(A^{\otimes' n_1},B)
\times\dots\times \cat{C}(A^{\otimes' n_k},B)$
to $g\circ (f_1\otimes'\dots\otimes' f_k)\circ A^{\otimes' u}$.
\end{itemize}

Via the above enrichment, we restore the classical definitions of model and 
homomorphism between models for symmetric operads.
\end{example}

\begin{example}
Non-symmetric operads may be identified with 
monoids in $[\Ncat,\Set]$ with the substitution monoidal structure.
Let $\cat{C}=(\cat{C},I',\otimes')$ 
be a locally small monoidal category.
We have an enrichment of $\cat{C}$ over $[\Ncat,\Set]$ which is
similar to, and simpler than, the one in the previous example.

This enrichment restores the classical definitions of model and 
homomorphism between models for non-symmetric operads, including
Definition~\ref{def:ns_operad_model}
(take $\cat{C}=(\Set,1,\times)$ for set models
and $\cat{C}=(\mathbf{Ab},\mathbb{Z},\otimes)$ for abelian group models).
\end{example}

\begin{example}
\label{ex:relative_alg}
We may also consider infinitary variants of Example~\ref{ex:clone_enrichment}.
Here we take an extreme. 
Let $\cat{C}$ be a locally small category with all small powers. 
Then we obtain an enrichment of $\cat{C}$ over $[\Set,\Set]$,
the category of endofunctors on $\Set$ with composition as the monoidal product.
\begin{itemize}
\item The functor $\enrich{-}{-}\colon \cat{C}^\op\times 
\cat{C}\longrightarrow[\Set,\Set]$ maps $A,B\in\cat{C}$
and $X\in \Set$ to the set 
\[
\enrich{A}{B}(X)=\cat{C}(A^X,B),
\]
where $A^X$ is the $X$-th power of $A$.
\item The natural transformation $(j_C\colon \id{\Set}\longrightarrow 
\enrich{C}{C})_{C\in\cat{C}}$ corresponds by the Yoneda lemma
(note that $\id{\Set}\cong\Set(1,-)$, where $1$ is a singleton)
to the family
\[
(\overline{j}_C=\id{C}\in\enrich{C}{C}(1))_{C\in\cat{C}}.
\]
\item The natural transformation $(M_{A,B,C}\colon 
\enrich{B}{C}\circ\enrich{A}{B}\longrightarrow\enrich{A}{C})_{A,B,C\in\cat{C}}$
has the $X$-th component ($X\in\Set$)
\[
\enrich{B}{C}\circ\enrich{A}{B}(X)=\cat{C}(B^{\cat{C}(A^X,B)},C)\longrightarrow
\cat{C}(A^X,C)=\enrich{A}{C}(X)
\]
the function induced from the canonical morphism $A^X\longrightarrow 
B^{\cat{C}(A^X,B)}$ in $\cat{C}$.
\end{itemize}

Since monoids in $[\Set,\Set]$ are precisely monads on $\Set$,
this enrichment gives us a definition of model of a monad $\monoid{T}$
on $\Set$ in $\cat{C}$. 
To spell this out, first note that 
for any object $C\in\cat{C}$, the functor $\enrich{C}{C}\colon 
\Set\longrightarrow\Set$ which maps $X\in\Set$ to $\cat{C}(C^X,C)$
acquires the monad structure, giving rise to the monad $\End_{\enrich{-}{-}}(C)$
on $\Set$.
A model of $\monoid{T}$ is then an object $C\in\cat{C}$ together with
a monad morphism $\monoid{T}\longrightarrow\End_{\enrich{-}{-}}(C)$.
This is the definition of \emph{relative algebra} of a monad on $\Set$ by 
Hino, Kobayashi, Hasuo and Jacobs~\cite{Hino}.
As noted in \cite{Hino}, in the case where $\cat{C}=\Set$,
relative algebras of a monad $\monoid{T}$ on $\Set$ agree with Eilenberg--Moore 
algebras of $\monoid{T}$; we shall later show this fact
in Example~\ref{ex:monad_on_cat_w_powers}.
\end{example}

\begin{example}
Let $\cat{S}$ be a large category and consider the monoidal
category $[\cat{S},\cat{S}]$ of endofunctors on $\cat{S}$,
with composition as the monoidal product.
Then an enrichment over $[\cat{S},\cat{S}]$ is the same thing as 
an \emph{$\cat{S}$-parameterised monad} (without strength) in the sense of 
Atkey~\cite[Definition~1]{Atkey_parameterised}, introduced
in the study of computational effects.
\end{example}

\medskip

Having reformulated semantics of notions of algebraic theory
in terms of enrichments, let us 
investigate some of its immediate consequences.

\subsubsection{$\Mod{-}{-}$ as a 2-functor}
It is well-known that given clones $\monoid{T}$ and $\monoid{T'}$,
a clone homomorphism $f\colon \monoid{T}\longrightarrow\monoid{T'}$,
and a locally small category $\cat{C}$ with finite products, 
we have the induced functor 
\[
\Mod{f}{\cat{C}}\colon\Mod{\monoid{T'}}{\cat{C}}\longrightarrow
\Mod{\monoid{T}}{\cat{C}}
\]
between the categories of models.
For instance, we can take $\monoid{T}$ to be the clone for monoids
and $\monoid{T'}$ to be the clone for groups,
with $f\colon \monoid{T}\longrightarrow \monoid{T'}$
the canonical clone map easily obtained from 
the standard presentations of monoids and of groups.
Then $\Mod{f}{\cat{C}}$ is the natural embedding of the 
category of group objects in $\cat{C}$ to the category of monoid objects
in $\cat{C}$; in particular, if we let $\cat{C}=\Set$, we get
the embedding of the category of groups into the category of monoids
(in words, ``groups are a special case of monoids'').

On the other hand, given a clone $\monoid{T}$, locally small 
categories $\cat{C}$ and 
$\cat{C'}$ with finite products, and a functor $G\colon\cat{C}\longrightarrow 
\cat{C'}$ preserving finite products, we obtain a functor 
\[
\Mod{\monoid{T}}{G}\colon\Mod{\monoid{T}}{\cat{C}}\longrightarrow
\Mod{\monoid{T}}{\cat{C'}}.
\]
As a concrete example, let $\monoid{T}$ be the clone for groups,
$\cat{C}=\mathbf{Top}$ (the category of topological spaces),
$\cat{C'}=\Set$ and $G\colon\mathbf{Top}\longrightarrow \Set$
be the functor mapping a topological space to its underlying set.
Then we obtain a functor from the category of topological groups to
the category of groups, which simply forgets the topology.

\medskip

In order to formulate such functoriality of $\Mod{-}{-}$,
we introduce a 2-category of enrichments.

\begin{definition}
\label{def:2-cat_of_enrichments}
Let $\cat{M}=(\cat{M},I,\otimes)$ be a large monoidal category.
The (locally large) 2-category $\Enrich{\cat{M}}$ of enrichments over $\cat{M}$ 
is defined as follows:
\begin{itemize}
\item An object is an enrichment $(\cat{C},\enrich{-}{-},j,M)$ over $\cat{M}$.
\item A 1-cell from $(\cat{C},\enrich{-}{-},j,M)$ to
$(\cat{C'},\enrich{-}{-}',j',M')$ is a functor 
$G\colon\cat{C}\longrightarrow\cat{C'}$ together with a natural transformation
$(g_{A,B}\colon\enrich{A}{B}\longrightarrow\enrich{GA}{GB}')_{A,B\in\cat{C}}$ 
making the 
following diagrams commute for all $A,B,C\in\cat{C}$:
\begin{equation*}
\begin{tikzpicture}[baseline=-\the\dimexpr\fontdimen22\textfont2\relax ]
      \node (TL) at (0,1)  {$I$};
      \node (TR) at (4,1)  {$\enrich{C}{C}$};
      \node (BR) at (4,-1) {$\enrich{GC}{GC}'$};
      \draw[->] (TL) to node[auto,labelsize](T) {$j_C$} (TR);
      \draw[->] (TR) to node[auto,labelsize](R) {$g_{C,C}$} (BR);
      \draw[->]  (TL) to node[auto,labelsize,swap] (L) {$j'_{GC}$} 
      (BR);
\end{tikzpicture} 
\end{equation*}
\begin{equation*}
\begin{tikzpicture}[baseline=-\the\dimexpr\fontdimen22\textfont2\relax ]
      \node (TL) at (0,1)  {$\enrich{B}{C}\otimes\enrich{A}{B}$};
      \node (TR) at (6,1)  {$\enrich{A}{C}$};
      \node (BL) at (0,-1) {$\enrich{GB}{GC}'\otimes\enrich{GA}{GB}'$};
      \node (BR) at (6,-1) {$\enrich{GA}{GC}'.$};
      \draw[->] (TL) to node[auto,labelsize](T) {$M_{A,B,C}$} 
      (TR);
      \draw[->]  (TR) to node[auto,labelsize] {$g_{A,C}$} (BR);
      \draw[->]  (TL) to node[auto,swap,labelsize] {$g_{B,C}\otimes g_{A,B}$} 
      (BL);
      \draw[->]  (BL) to node[auto,labelsize] {$M'_{GA,GB,GC}$} (BR);
\end{tikzpicture} 
\end{equation*}
\item A 2-cell from $(G,g)$ to $(G',g')$, both from 
$(\cat{C},\enrich{-}{-},j,M)$ to $(\cat{C'},\enrich{-}{-}',j',$ $M')$, is a 
natural transformation $\theta\colon G\Longrightarrow G'$
making the following diagram commute for all $A,B\in\cat{C}$:
\begin{equation*}
\begin{tikzpicture}[baseline=-\the\dimexpr\fontdimen22\textfont2\relax ]
      \node (TL) at (0,2)  {$\enrich{A}{B}$};
      \node (TR) at (4,2)  {$\enrich{GA}{GB}'$};
      \node (BL) at (0,0) {$\enrich{G'A}{G'B}'$};
      \node (BR) at (4,0) {$\enrich{GA}{G'B}'.$};
      \draw[->] (TL) to node[auto,labelsize](T) {$g_{A,B}$} 
      (TR);
      \draw[->]  (TR) to node[auto,labelsize] {$\enrich{GA}{\theta_B}'$} (BR);
      \draw[->]  (TL) to node[auto,swap,labelsize] {$g'_{A,B}$} 
      (BL);
      \draw[->]  (BL) to node[auto,labelsize] {$\enrich{\theta_A}{G'B}'$} (BR);
\end{tikzpicture} \qedhere
\end{equation*}
\end{itemize}
\end{definition}

\begin{example}
Let $\FPow$ be the 2-category of locally small categories with chosen finite 
powers,
functors preserving finite powers (in the usual sense\footnote{That is,
we do not require these functors to preserve the chosen finite powers on the 
nose.}) 
and all natural transformations.
We have a canonical 2-functor 
\[
\FPow\longrightarrow\Enrich{[\F,\Set]}
\]
which is fully faithful (see Section~\ref{sec:2-cat_notions} for 
the definition of full faithfulness for 2-functors).

Let $\FProd$ be the 2-category of locally small categories with chosen finite 
products,
functors preserving finite products (in the usual sense) 
and all natural transformations.
We have a canonical 2-functor 
\[
\FProd\longrightarrow\Enrich{[\F,\Set]}
\]
which is locally fully faithful.

Hence we may restore the classical functoriality of $\Mod{\monoid{T}}{-}$
for a clone $\monoid{T}$, recalled above, if we could show that it is 
functorial with respect to morphisms in $\Enrich{[\F,\Set]}$.
\end{example}

We also have canonical (locally faithful) 2-functors
\[
\SymMonCATls\longrightarrow\Enrich{[\Pcat,\Set]},
\]
where the domain is the 2-category of locally small symmetric monoidal 
categories, symmetric lax monoidal functors  and monoidal natural 
transformations, and
\[
\MonCATls\longrightarrow\Enrich{[\Ncat,\Set]},
\]
where the domain is the 2-category of locally small monoidal categories,
lax monoidal functors and monoidal natural transformations.

\medskip
Now the functoriality of $\Mod{-}{-}$
may be expressed by saying that it is a 2-functor
\begin{equation}
\label{eqn:enrich_Mod_bifunctorial}
\Mod{-}{-}\colon \Mon{\cat{M}}^\op\times\Enrich{\cat{M}}\longrightarrow \tCAT
\end{equation}
(when we say that (\ref{eqn:enrich_Mod_bifunctorial}) is a 
2-functor, we are identifying the category $\Mon{\cat{M}}$ with the 
corresponding locally discrete 2-category).
Actually, the 2-functor~(\ref{eqn:enrich_Mod_bifunctorial}) arises
immediately from the structure of the locally large 2-category 
$\Enrich{\cat{M}}$.
Observe that we may identify a monoid object in $\cat{M}$ with 
an enrichment of the terminal category $1$ over $\cat{M}$.
The full sub-2-category of $\Enrich{\cat{M}}$ consisting of 
all enrichments over the (fixed) terminal category $1$ is in fact 
locally discrete, and is isomorphic to $\Mon{\cat{M}}$.
This way we obtain a fully faithful inclusion 2-functor
$\Mon{\cat{M}}\longrightarrow\Enrich{\cat{M}}$.
It is straightforward to see that the appropriate 2-functor 
(\ref{eqn:enrich_Mod_bifunctorial}) is given by the composite
\[
\begin{tikzpicture}[baseline=-\the\dimexpr\fontdimen22\textfont2\relax ]
      \node (1) at (0,0)  {$\Mon{\cat{M}}^\op\times \Enrich{\cat{M}}$};
      \node (2) at (0,-1.5)  {$\Enrich{\cat{M}}^\op\times \Enrich{\cat{M}}$};
      \node (3) at (0,-3) {$\tCAT$,};
      \draw[->] (1) to node[auto,labelsize]{inclusion} (2);
      \draw[->] (2) to node[auto,labelsize] {$\Enrich{\cat{M}}(-,-)$} (3);
\end{tikzpicture}
\]
where $\Enrich{\cat{M}}(-,-)$ is the hom-2-functor for $\Enrich{\cat{M}}$.

\subsubsection{Comparing different notions of algebraic theory}
So far we have been working within a fixed notion of algebraic theory.
We now turn to the question of comparing different notions of algebraic 
theory.

By way of illustration,
let us consider the relationship of 
clones, symmetric operads and non-symmetric operads.
On the ``syntactical'' side, we have inclusions
of algebraic theories
\begin{equation}
\label{eqn:chain_alg_thy}
\{\text{non-sym.~operads}\}\subseteq\{\text{sym.~operads}\}\subseteq
\{\text{clones}\},
\end{equation}
in the sense that every symmetric operad may be derived from 
a regular presentation of an equational theory,
which is at the same time a presentation of an equational theory
and therefore defines a clone, etc.
On the ``semantical'' side, in contrast, we have 
inclusions of (standard) notions of models in the other direction,
namely:
\begin{equation}
\label{eqn:chain_semantics}
\{\text{mon.~cat.}\}
\supseteq\{\text{sym.~mon.~cat.}\}
\supseteq\{\text{cat.~with fin.~prod.}\}.
\end{equation}

Furthermore,
suppose we take the algebraic theory $\monoid{T}$ of monoids 
(which is expressible as a non-symmetric operad)
and the category $\Set$ (which has finite products).
Then we can consider the category of models $\Mod{\monoid{T}}{\Set}$
in three different ways:
either thinking of $\monoid{T}$ as a clone and $\Set$ as a category with finite 
products,
$\monoid{T}$ as a symmetric operad and $\Set$ as a symmetric monoidal category,
or $\monoid{T}$ as a non-symmetric operad and $\Set$ as a monoidal category.
It turns out that the resulting three categories of models are isomorphic to 
each other,
indicating certain compatibility between the three notions of algebraic theory.

\medskip

The key to understand these phenomena in our framework is 
the functoriality of the $\Enrich{-}$ construction. 
That is, we may extend (just like base change of
enriched categories) $\Enrich{-}$ to a 2-functor
\begin{equation}
\label{eqn:Enrich_as_2-functor_1st}
\Enrich{-}\colon \MonCAT\longrightarrow \twoCAT
\end{equation}
from the 2-category $\MonCAT$ 
of large monoidal categories, lax monoidal functors 
and monoidal natural transformations
to the 2-category $\twoCAT$ of huge 2-categories, 2-functors and 
2-natural transformations.
We just describe the action of a lax monoidal functor on an enrichment,
as the rest of the data for the 2-functor (\ref{eqn:Enrich_as_2-functor_1st}) 
follows from that rather routinely.
\begin{definition}
\label{def:action_of_lax_on_enrichment}
Let $\cat{M}=(\cat{M},I_\cat{M},\otimes_\cat{M})$ and 
$\cat{N}=(\cat{N},I_\cat{N},\otimes_\cat{N})$ be large monoidal categories,
$F=(F,f_\cdot,f)\colon\cat{M}\longrightarrow\cat{N}$ be a lax monoidal 
functor,\footnote{In more detail, such a lax monoidal functor 
(also called monoidal functors in e.g.,~\cite{MacLane_CWM})
$(F,f_\cdot,f)$ consists of a functor $F\colon\cat{M}\longrightarrow\cat{N}$,
a morphism $f_\cdot\colon I_\cat{N}\longrightarrow FI_\cat{M}$
and a natural transformation $(f_{X,Y}\colon FY\otimes_\cat{N}FX\longrightarrow 
F(Y\otimes_\cat{M} X))_{X,Y\in\cat{M}}$ satisfying the suitable axioms.} 
$\cat{C}$ be a large category and $\enrich{-}{-}=(\enrich{-}{-},j,M)$ be an 
enrichment of $\cat{C}$ over $\cat{M}$.
We define the enrichment $F_\ast(\enrich{-}{-})=(\enrich{-}{-}',j',M')$ 
of $\cat{C}$ over $\cat{N}$ as follows:
\begin{itemize}
\item The functor 
$\enrich{-}{-}'\colon\cat{C}^\op\times\cat{C}\longrightarrow\cat{N}$
maps $(A,B)\in\cat{C}^\op\times\cat{C}$ to $F\enrich{A}{B}$.
\item The natural transformation $(j'_C\colon 
I_\cat{N}\longrightarrow\enrich{C}{C}')_{C\in\cat{C}}$
is defined by $j'_C=Fj_C\circ f_\cdot$:
\[
\begin{tikzpicture}[baseline=-\the\dimexpr\fontdimen22\textfont2\relax ]
      \node (L) at (0,0)  {$I_\cat{N}$};
      \node (M) at (2,0)  {$FI_\cat{M}$};
      \node (R) at (4.5,0)  {$F\enrich{C}{C}$.};
      \draw[->]  (L) to node[auto,labelsize] {$f_\cdot$} (M);
      \draw[->]  (M) to node[auto,labelsize] {$Fj_C$} (R);
\end{tikzpicture} 
\]
\item The natural transformation $(M'_{A,B,C}\colon 
\enrich{B}{C}'\otimes_\cat{N}\enrich{A}{B}'\longrightarrow\enrich{A}{C}')_{A,
B,C\in\cat{C}}$
is defined by $M'_{A,B,C}=FM_{A,B,C}\circ f_{\enrich{A}{B},\enrich{B}{C}}$:
\[
\begin{tikzpicture}[baseline=-\the\dimexpr\fontdimen22\textfont2\relax ]
      \node (L) at (0,0)  {$F\enrich{B}{C}\otimes_\cat{N}F\enrich{A}{B}$};
      \node (M) at (5.5,0)  {$F(\enrich{B}{C}\otimes_\cat{M}\enrich{A}{B})$};
      \node (R) at (9.5,0)  {$F\enrich{A}{C}$.};
      \draw[->]  (L) to node[auto,labelsize] 
      {$f_{\enrich{A}{B},\enrich{B}{C}}$} (M);
      \draw[->]  (M) to node[auto,labelsize] {$FM_{A,B,C}$} (R);
\end{tikzpicture} \qedhere
\]
\end{itemize}
\end{definition}
As an immediate consequence of the 2-functoriality 
(\ref{eqn:Enrich_as_2-functor_1st}),
it follows that whenever we have a monoidal adjunction (adjunction in 
$\MonCAT$) 
\[
\begin{tikzpicture}[baseline=-\the\dimexpr\fontdimen22\textfont2\relax ]
      \node (L) at (0,0)  {$\cat{M}$};
      \node (R) at (3,0)  {$\cat{N}$,};
      \draw[->,transform canvas={yshift=5pt}]  (L) to node[auto,labelsize] 
      {$L$} (R);
      \draw[<-,transform canvas={yshift=-5pt}]  (L) to 
      node[auto,swap,labelsize] {$R$} (R);
      \node[rotate=-90,labelsize] at (1.5,0)  {$\dashv$};
\end{tikzpicture} 
\] 
we obtain a 2-adjunction
\[
\begin{tikzpicture}[baseline=-\the\dimexpr\fontdimen22\textfont2\relax ]
      \node (L) at (0,0)  {$\Enrich{\cat{M}}$};
      \node (R) at (4.5,0)  {$\Enrich{\cat{N}}$.};
      \draw[->,transform canvas={yshift=5pt}]  (L) to node[auto,labelsize] 
      {$\Enrich{L}$} (R);
      \draw[<-,transform canvas={yshift=-5pt}]  (L) to 
      node[auto,swap,labelsize] {$\Enrich{R}$} 
      (R);
      \node[rotate=-90,labelsize] at (2.25,0)  {$\dashv$};
\end{tikzpicture} 
\] 
Therefore, if we take 
$\monoid{T}\in\Mon{\cat{M}}\subseteq\Enrich{\cat{M}}$
and $(\cat{C},\enrich{-}{-})\in \Enrich{\cat{N}}$ in this situation, 
then 
\begin{multline}
\label{eqn:iso_adj_enrich}
\Enrich{\cat{M}}(\monoid{T},\Enrich{R}(\cat{C},\enrich{-}{-}))\\
\cong\Enrich{\cat{N}}(\Enrich{L}(\monoid{T}),(\cat{C},\enrich{-}{-})).
\end{multline}
Since the action of $\Enrich{-}$ preserves the underlying categories,
we may assume $\Enrich{L}(\monoid{T})\in\Mon{\cat{N}}$.
Therefore 
(\ref{eqn:iso_adj_enrich}) may be seen as an isomorphism between
the category of models of $\monoid{T}$ in $\cat{C}$ with 
respect to $R_\ast(\enrich{-}{-})$
and the category of models of $\Enrich{L}(\monoid{T})$
in $\cat{C}$ with respect to $\enrich{-}{-}$.

The relationship between clones, symmetric operads and non-symmetric operads
mentioned above can be explained in this way.
First note that there is a chain of inclusions
\[
\begin{tikzpicture}[baseline=-\the\dimexpr\fontdimen22\textfont2\relax ]
      \node (L) at (0,0)  {$\Ncat$};
      \node (M) at (2,0)  {$\Pcat$};
      \node (R) at (4,0)  {$\F$.};
      \draw[->]  (L) to node[auto,labelsize] {$J$} (M);
      \draw[->]  (M) to node[auto,labelsize] {$J'$} (R);
\end{tikzpicture} 
\]
Therefore, precomposition and left Kan extensions 
induce a chain of adjunctions
\[
\begin{tikzpicture}[baseline=-\the\dimexpr\fontdimen22\textfont2\relax ]
      \node (L) at (0,0)  {$[\Ncat,\Set]$};
      \node (M) at (3.5,0)  {$[\Pcat,\Set]$};
      \node (R) at (7,0)  {$[\F,\Set]$.};
      \draw[->,transform canvas={yshift=5pt}]  (L) to node[auto,labelsize] 
      {$\Lan_J$} (M);
      \draw[<-,transform canvas={yshift=-5pt}] (L) to node[auto,swap,labelsize] 
      {$[J,\Set]$} (M);
      \draw[->,transform canvas={yshift=5pt}]  (M) to node[auto,labelsize] 
      {$\Lan_{J'}$} (R);
      \draw[<-,transform canvas={yshift=-5pt}] (M) to node[auto,swap,labelsize] 
      {$[J',\Set]$} 
      (R);
      \node[rotate=-90,labelsize] at (1.75,0)  {$\dashv$};
      \node[rotate=-90,labelsize] at (5.25,0)  {$\dashv$};
\end{tikzpicture} 
\] 
It turns out that these adjunctions acquire natural structures of 
monoidal adjunctions.
Hence in our framework, the inclusions 
(\ref{eqn:chain_alg_thy}) are expressed as the functors
\[
\begin{tikzpicture}[baseline=-\the\dimexpr\fontdimen22\textfont2\relax ]
      \node (L) at (0,0)  {$\Mon{[\Ncat,\Set]}$};
      \node (M) at (4.5,0)  {$\Mon{[\Pcat,\Set]}$};
      \node (R) at (9,0)  {$\Mon{[\F,\Set]}$};
      \draw[->]  (L) to node[auto,labelsize] {$\Mon{\Lan_J}$} (M);
      \draw[->]  (M) to node[auto,labelsize] {$\Mon{\Lan_{J'}}$} (R);
\end{tikzpicture} 
\]
between the categories of monoids,
whereas the inclusions (\ref{eqn:chain_semantics}) are restrictions of 
the 2-functors 
\[
\begin{tikzpicture}[baseline=-\the\dimexpr\fontdimen22\textfont2\relax ]
      \node (L) at (0,0)  {$\Enrich{[\Ncat,\Set]}$};
      \node (M) at (5,0)  {$\Enrich{[\Pcat,\Set]}$};
      \node (R) at (10,0)  {$\Enrich{[\F,\Set]}$};
      \draw[<-]  (L) to node[auto,labelsize] {$\Enrich{[J,\Set]}$} (M);
      \draw[<-]  (M) to node[auto,labelsize] {$\Enrich{[{J'},\Set]}$} (R);
\end{tikzpicture} 
\]
between the 2-categories of enrichments.

\subsubsection{Relation to enriched category theory}
Before concluding this section, we shall remark on the relationship between
our notion of {enrichment} and the standard notions in enriched category 
theory~\cite{Kelly:enriched}.
The reader may move on to the next section on oplax actions, since
the results obtained in the following discussion will not be used 
in this thesis, though they explain how our approach relates to an
enriched categorical approach to clones
(= Lawvere theories = finitary monads on $\Set$) 
by Garner~\cite{Garner_Lawvere_Cauchy}.

We have mentioned that an enrichment of $\cat{C}$ over $\cat{M}$
can be equivalently given as an $\cat{M}$-category $\cat{D}$
and an identity-on-objects functor $J\colon\cat{C}\longrightarrow\cat{D}_0$. 
Let us first make the relation of these two formulations precise.
In order to compare them, we introduce a natural 2-category
having the latter as objects.
\begin{definition}
Let $\cat{M}$ be a large monoidal category.
The 2-category $\Enrichtwo{\cat{M}}$ is defined as follows:
\begin{itemize}
\item An object is a triple $(\cat{C},\cat{D},J)$ consisting of a large
category $\cat{C}$, a large $\cat{M}$-category $\cat{D}$ and an 
identity-on-objects
functor $J\colon\cat{C}\longrightarrow \cat{D}_0$.
\item A 1-cell from $(\cat{C},\cat{D},J)$ to $(\cat{C'},\cat{D'},J')$
is given by a functor $G\colon \cat{C}\longrightarrow\cat{C'}$
and an $\cat{M}$-functor $H\colon \cat{D}\longrightarrow\cat{D'}$
such that $H_0\circ J=J'\circ G$.
\item A 2-cell from $(G,H)$ to $(G',H')$, both from $(\cat{C},\cat{D},J)$ to 
$(\cat{C'},\cat{D'},J')$, is given by a natural transformation $\theta\colon 
G\Longrightarrow G'$ and an $\cat{M}$-natural transformation $\phi\colon 
H\Longrightarrow H'$ such that $\phi_0\circ J=J'\circ \theta$.\qedhere
\end{itemize}
\end{definition}
Hence $\Enrichtwo{\cat{M}}$ is a full sub-2-category of 
the comma 2-category defined by the diagram
\[
\begin{tikzpicture}[baseline=-\the\dimexpr\fontdimen22\textfont2\relax ]
      \node (L) at (0,0) {$\tCAT$};
      \node (M) at (2.5,0) {$\tCAT$};
      \node (R) at (5.3,0) {$\entCAT{\cat{M}}$,};
      \draw[->]  (L) to node[auto,labelsize] {$\id{\tCAT}$} (M);
      \draw[->]  (R) to node[auto,swap,labelsize] {$(-)_0$} (M);
\end{tikzpicture}  
\]
where $\entCAT{\cat{M}}$ is the 2-category of large
$\cat{M}$-categories, $\cat{M}$-functors and $\cat{M}$-natural transformations,
and $(-)_0$ is the forgetful 2-functor described in  
\cite[Section~1.3]{Kelly:enriched}.

It is routine to check that the obvious
construction (sketched just after Definition~\ref{def:enrichment}) from 
$(\cat{C},\enrich{-}{-},j,M)\in\Enrich{\cat{M}}$
to $(\cat{C},\cat{D},J)\in\Enrichtwo{\cat{M}}$
extends to an isomorphism of the 2-categories $\Enrich{\cat{M}}$
and $\Enrichtwo{\cat{M}}$.
%
%\begin{proposition}
%Let $\cat{M}$ be a monoidal category.
%The 2-categories $\Enrich{\cat{M}}$ and $\Enrichtwo{\cat{M}}$
%are isomorphic.
%\end{proposition}
%
Therefore we may identify $\Enrich{\cat{M}}$ with $\Enrichtwo{\cat{M}}$
via this isomorphism;
$\Enrichtwo{\cat{M}}$ is better suited to establish connections to 
enriched category theory.

We may embed (fully faithfully) both $\entCAT{\cat{M}}$
and the underlying category $\entCAT{\cat{M}}_0$ of $\entCAT{\cat{M}}$
into $\Enrichtwo{\cat{M}}$.
The embedding
\[
\normalemb\colon
\entCAT{\cat{M}}\longrightarrow \Enrichtwo{\cat{M}}
\]
maps an $\cat{M}$-category $\cat{A}$ to
the normal enrichment $(\cat{A}_0,\cat{A},\id{\cat{A}_0})$ over $\cat{M}$
(recall that an enrichment $(\cat{C},\cat{D},J)$ is called 
\emph{normal} iff $J$ is an isomorphism).
Clearly, an enrichment is normal precisely when it is isomorphic 
to an enrichment of the form $\normalemb \cat{A}$ for some 
$\cat{A}\in\entCAT{\cat{M}}$ (i.e., when it is in the essential image of 
$\normalemb$).
The embedding
\[
\discemb\colon
\entCAT{\cat{M}}_0\longrightarrow\Enrichtwo{\cat{M}}
\]
maps an $\cat{M}$-category $\cat{A}$ to
the enrichment $(\ob{\cat{A}},\cat{A},J)$ over $\cat{M}$ of the set 
$\ob{\cat{D}}$ seen as a discrete category ($J$ is the unique
identity-on-objects functor $\ob{\cat{A}}\longrightarrow\cat{A}_0$).
It is clear from the definition of $\Enrichtwo{\cat{M}}$
that both $\normalemb$ and $\discemb$ are fully faithful 2-functors.
The fully faithful embedding $\Mon{\cat{M}}\longrightarrow \Enrich{\cat{M}}$
mentioned above is a restriction of $\discemb$.

The 2-functor $\normalemb$ admits a left adjoint 2-functor
\[
L\colon \Enrichtwo{\cat{M}}\longrightarrow\entCAT{\cat{M}}
\]
mapping $(\cat{C},\cat{D},J)\in \Enrichtwo{\cat{M}}$ to $\cat{D}\in
\entCAT{\cat{M}}$ and so on.
Therefore for a monoid
$\monoid{T}$ in $\cat{M}$ and a 
\emph{normal} enrichment $(\cat{C},\cat{D},J)$ over $\cat{M}$,
the category of models $\Mod{\monoid{T}}{\cat{C}}
=\Enrich{\cat{M}}(\monoid{T},(\cat{C},\cat{D},J))\cong
\Enrich{\cat{M}}(\monoid{T},\normalemb\cat{D})$
is isomorphic to $\entCAT{\cat{M}}(L\monoid{T},\cat{D})$,
where $L\monoid{T}$ is just a monoid $\monoid{T}$ seen as 
a one-object $\cat{M}$-category.

The enrichments corresponding to the standard notions of model
for clones, symmetric operads and non-symmetric operads
are all normal, hence in order to capture the categories of models
relative to these notions of model, we may work 
entirely within the 2-category $\entCAT{\cat{M}}$,
as already observed (in the case of clones) in 
\cite{Garner_Lawvere_Cauchy}.

\subsection{Notions of model as oplax actions}
\label{subsec:oplax_action}
In order to capture models of monads and generalised operads, 
enrichments do not suffice in general.
A suitable structure is \emph{oplax action}, defined as follows.

\begin{definition}
\label{def:oplax_action}
Let $\cat{M}=(\cat{M},I,\otimes)$ be a large monoidal category.
An \defemph{oplax action of $\cat{M}$} consists of:
\begin{itemize}
\item a large category $\cat{C}$;
\item a functor $\ast \colon \cat{M}\times\cat{C}\longrightarrow\cat{C}$;
\item a natural transformation $(\varepsilon_C\colon I\ast C\longrightarrow 
C)_{C\in\cat{C}}$;
\item a natural transformation 
$(\delta_{X,Y,C}\colon (Y\otimes X)\ast C\longrightarrow Y\ast (X\ast 
C))_{X,Y\in\cat{M},C\in\cat{C}},$\footnote{We have chosen to set 
$\delta_{X,Y,C}\colon (Y\otimes X)\ast C
\longrightarrow Y\ast (X\ast C)$ and not 
$\delta_{X,Y,C}\colon (X\otimes Y)\ast C
\longrightarrow X\ast (Y\ast C)$,
because the former agrees with 
the convention to write composition of morphisms in the
anti-diagrammatic order, which we adopt throughout 
this thesis.}
\end{itemize} 
making the following diagrams commute for all $X,Y,Z\in\cat{M}$
and $C\in\cat{C}$:
\begin{equation*}
\begin{tikzpicture}[baseline=-\the\dimexpr\fontdimen22\textfont2\relax ]
      \node (TL) at (0,1)  {$(I\otimes X)\ast C$};
      \node (TR) at (4,1)  {$I\ast(X\ast C)$};
      \node (BR) at (4,-1) {$X\ast C$};
      \draw[->] (TL) to node[auto,labelsize](T) {$\delta_{X,I,C}$} (TR);
      \draw[->] (TR) to node[auto,labelsize](R) {$\varepsilon_{X\ast C}$} (BR);
      \draw[->]  (TL) to node[auto,labelsize,swap] (L) {$\cong$} (BR);
\end{tikzpicture} 
\quad
\begin{tikzpicture}[baseline=-\the\dimexpr\fontdimen22\textfont2\relax ]
      \node (TL) at (0,1)  {$(X\otimes I)\ast C$};
      \node (TR) at (4,1)  {$X\ast(I\ast C)$};
      \node (BR) at (4,-1) {$X\ast C$};
      \draw[->] (TL) to node[auto,labelsize](T) {$\delta_{I,X,C}$} (TR);
      \draw[->] (TR) to node[auto,labelsize](R) {$X\ast \varepsilon_{C}$} (BR);
      \draw[->]  (TL) to node[auto,labelsize,swap] (L) {$\cong$} (BR);
\end{tikzpicture} 
\end{equation*}
\begin{equation*}
\begin{tikzpicture}[baseline=-\the\dimexpr\fontdimen22\textfont2\relax ]
      \node (TL) at (0,1)  {$((Z\otimes Y)\otimes X)\ast C$};
      \node (TR) at (9,1)  {$(Z\otimes Y)\ast (X\ast C)$};
      \node (BL) at (0,-1) {$(Z\otimes (Y\otimes X))\ast C$};
      \node (BM) at (4.5,-1) {$Z\ast ((Y\otimes X)\ast C)$};
      \node (BR) at (9,-1) {$Z\ast (Y\ast (X\ast C)).$};
      \draw[->] (TL) to node[auto,labelsize](T) {$\delta_{X,Z\otimes Y,C}$} 
      (TR);
      \draw[->]  (TR) to node[auto,labelsize] {$\delta_{Y,Z,X\ast C}$} (BR);
      \draw[->]  (TL) to node[auto,swap,labelsize] {$\cong$} (BL);
      \draw[->]  (BL) to node[auto,labelsize] {$\delta_{Y\otimes X,Z,C}$} (BM);
      \draw[->] (BM) to node[auto,labelsize](B) {$Z\ast \delta_{X,Y,C}$} (BR);
\end{tikzpicture} 
\end{equation*}
We say that $(\cat{C},\ast,\varepsilon,\delta)$ is 
an oplax action of $\cat{M}$, or that $(\ast,\varepsilon,\delta)$ is 
an oplax action of $\cat{M}$ on $\cat{C}$.
\end{definition}

An oplax action $(\ast,\varepsilon,\delta)$ of $\cat{M}$
on $\cat{C}$ is 
%called \defemph{normal} if $\varepsilon$ is a 
%natural isomorphism, and is 
called a \defemph{pseudo action}
(resp.~\defemph{strict action})
if both $\varepsilon$ and $\delta$ are natural isomorphisms
(resp.~identities).

The definition of model we derive from an oplax action is 
the following.

\begin{definition}
\label{def:action_model}
Let $\cat{M}=(\cat{M},I,\otimes)$ be a large monoidal category,
$\monoid{T}=(T,e,m)$ be a monoid object in $\cat{M}$,
$\cat{C}$ be a large category, and $\ast=(\ast,\varepsilon,\delta)$
be an oplax action of $\cat{M}$ on $\cat{C}$.
\begin{enumerate}
\item A \defemph{model of $\monoid{T}$ in $\cat{C}$ with respect to $\ast$}
is a pair $(C,\gamma)$ consisting of an object $C\in\cat{C}$
and a morphism $\gamma\colon T\ast C\longrightarrow C$ in $\cat{C}$
making the following diagrams commute:
\begin{equation*}
\begin{tikzpicture}[baseline=-\the\dimexpr\fontdimen22\textfont2\relax ]
      \node (TL) at (0,2)  {$I\ast C$};
      \node (TR) at (2,2)  {$T\ast C$};
%      \node (BL) at (0,0) {$T$};
      \node (BR) at (2,0) {${C}$};
      \draw[->]  (TL) to node[auto,labelsize] {$e\ast C$} (TR);
      \draw[->]  (TR) to node[auto,labelsize] {$\gamma$} (BR);
      \draw[->] (TL) to node[auto,swap,labelsize](B) {$\varepsilon_C$} (BR);
\end{tikzpicture} 
\qquad
\begin{tikzpicture}[baseline=-\the\dimexpr\fontdimen22\textfont2\relax ]
      \node (TL) at (0,2)  {$(T\otimes T)\ast C$};
      \node (TR) at (5,2)  {$T\ast C$};
      \node (BL) at (0,0) {$T\ast (T\ast C)$};
      \node (BM) at (3,0) {$T\ast C$};
      \node (BR) at (5,0) {${C}.$};
      \draw[->]  (TL) to node[auto,swap,labelsize] {$\delta_{T,T,C}$} (BL);
      \draw[->]  (TL) to node[auto,labelsize] {$m\ast C$} (TR);
      \draw[->]  (BL) to node[auto,labelsize] {$T\ast \gamma$} (BM);
      \draw[->]  (BM) to node[auto,labelsize] {$\gamma$} (BR);
      \draw[->]  (TR) to node[auto,labelsize](B) {$\gamma$} (BR);
\end{tikzpicture} 
\end{equation*}
\item Let $(C,\gamma)$ and $(C',\gamma')$ be models of $\monoid{T}$ in $\cat{C}$
with respect to $\ast$. A \defemph{homomorphism from $(C,\gamma)$ to 
$(C',\gamma')$}
is a morphism $f\colon C\longrightarrow C'$ in $\cat{C}$ making the 
following diagram commute:
\[
\begin{tikzpicture}[baseline=-\the\dimexpr\fontdimen22\textfont2\relax ]
      \node (TL) at (0,2)  {$T\ast C$};
      \node (TR) at (3,2)  {$T\ast C'$};
      \node (BL) at (0,0) {$C$};
      \node (BR) at (3,0) {$C'.$};
      \draw[->]  (TL) to node[auto,swap,labelsize] {$\gamma$} (BL);
      \draw[->]  (TL) to node[auto,labelsize] {$T\ast f$} (TR);
      \draw[->]  (BL) to node[auto,labelsize] {${f}$} (BR);
      \draw[->]  (TR) to node[auto,labelsize](B) {$\gamma'$} (BR);
\end{tikzpicture} 
\] 
\end{enumerate}
We denote the (large) category of models of $\monoid{T}$ in $\cat{C}$ with 
respect to $\ast$ by $\Mod{\monoid{T}}{(\cat{C},\ast)}$.
\end{definition}
The above definition is standard;
see e.g., \cite[Section~2.2]{Baez_Dolan_HDA3}.

\begin{example}
\label{ex:monad_standard_action}
Let $\cat{C}$ be a large category.
Recall that monads on $\cat{C}$ are precisely monoids in the monoidal category
$[\cat{C},\cat{C}]$
whose monoidal product is given by composition.
We have a strict action 
\[
\ast\colon [\cat{C},\cat{C}]\times\cat{C}\longrightarrow\cat{C}
\]
given by evaluation: $(X,C)\longmapsto XC$.

This clearly generates the definitions of 
Eilenberg--Moore algebra and 
homomorphism (Definition~\ref{def:EM_alg}).
\end{example}

\begin{example}
\label{ex:metamodel_S-operad}
Let $\cat{C}$ be a large category with finite limits and 
$\monoid{S}=(S,\eta,\mu)$
be a cartesian monad on $\cat{C}$. 
Recall that under these assumptions the slice category $\cat{C}/S1$
acquires a structure $(I,\otimes)$ of monoidal category, and
an $\monoid{S}$-operad is a monoid in $(\cat{C}/S1,I,\otimes)$.
Models of an $\monoid{S}$-operad and their homomorphisms
(Definition~\ref{def:S-operad_model}) were
introduced by using the pseudo action
\[
\ast\colon(\cat{C}/S1)\times \cat{C}\longrightarrow\cat{C}
\]
in the first place, and therefore are immediately an instance of the above
general definitions.
\end{example}

\subsubsection{The 2-category of oplax actions of $\cat{M}$}
For a monoidal category $\cat{M}$, 
we can define the 2-category of oplax actions of $\cat{M}$
(cf.~Definition~\ref{def:2-cat_of_enrichments}).

\begin{definition}
Let $\cat{M}=(\cat{M},I,\otimes)$ be a large monoidal category.
The (locally large) 2-category $\olAct{\cat{M}}$ of oplax actions of $\cat{M}$
is defined as follows:
\begin{itemize}
\item An object is an oplax action $(\cat{C},\ast,\varepsilon,\delta)$ of 
$\cat{M}$.
\item A 1-cell from $(\cat{C},\ast,\varepsilon,\delta)$ to 
$(\cat{C'},\ast',\varepsilon',\delta')$ is a functor $G\colon \cat{C}
\longrightarrow\cat{C'}$ together with a natural transformation
$(g_{X,C}\colon X\ast' GC\longrightarrow G(X\ast C))_{X\in\cat{M},C\in\cat{C}}$
making the following diagrams commute for all $X,Y\in\cat{M}$ and $C\in\cat{C}$:
\begin{equation*}
\begin{tikzpicture}[baseline=-\the\dimexpr\fontdimen22\textfont2\relax ]
      \node (TL) at (0,1)  {$I\ast' GC$};
      \node (TR) at (4,1)  {$G(I\ast C)$};
      \node (BR) at (4,-1) {$GC$};
      \draw[->] (TL) to node[auto,labelsize](T) {$g_{I,C}$} (TR);
      \draw[->] (TR) to node[auto,labelsize](R) {$G\varepsilon_C$} (BR);
      \draw[->]  (TL) to node[auto,labelsize,swap] (L) {$\varepsilon'_{GC}$} 
      (BR);
\end{tikzpicture} 
\end{equation*}
\begin{equation*}
\begin{tikzpicture}[baseline=-\the\dimexpr\fontdimen22\textfont2\relax ]
      \node (TL) at (0,1)  {$(Y\otimes X)\ast' GC$};
      \node (TR) at (8,1)  {$G((Y\otimes X)\ast C)$};
      \node (BL) at (0,-1) {$Y\ast'(X\ast'GC)$};
      \node (BM) at (4,-1) {$Y\ast'G(X\ast C)$};
      \node (BR) at (8,-1) {$G(Y\ast (X\ast C)).$};
      \draw[->] (TL) to node[auto,labelsize](T) {$g_{Y\otimes X,C}$} 
      (TR);
      \draw[->]  (TR) to node[auto,labelsize] {$G\delta_{X,Y,C}$} (BR);
      \draw[->]  (TL) to node[auto,swap,labelsize] {$\delta'_{X,Y,GC}$} (BL);
      \draw[->]  (BL) to node[auto,labelsize] {$Y\ast' g_{X,C}$} (BM);
      \draw[->] (BM) to node[auto,labelsize](B) {$g_{Y,X\ast C}$} (BR);
\end{tikzpicture} 
\end{equation*}
\item A 2-cell from $(G,g)$ to $(G',g')$, both from 
$(\cat{C},\ast,\varepsilon,\delta)$ to 
$(\cat{C'},\ast',\varepsilon',\delta')$,
is a natural transformation $\theta\colon G\Longrightarrow G'$
making the following diagram commute for all $X\in\cat{M}$
and $C\in\cat{C}$:
\begin{equation*}
\begin{tikzpicture}[baseline=-\the\dimexpr\fontdimen22\textfont2\relax ]
      \node (TL) at (0,2)  {$X\ast' GC$};
      \node (TR) at (4,2)  {$G(X\ast C)$};
      \node (BL) at (0,0) {$X\ast' G'C$};
      \node (BR) at (4,0) {$G'(X\ast C).$};
      \draw[->] (TL) to node[auto,labelsize](T) {$g_{X,C}$} 
      (TR);
      \draw[->]  (TR) to node[auto,labelsize] {$\theta_{X\ast C}$} (BR);
      \draw[->]  (TL) to node[auto,swap,labelsize] {$X\ast' \theta_C$} 
      (BL);
      \draw[->]  (BL) to node[auto,labelsize] {$g'_{X,C}$} (BR);
\end{tikzpicture} \qedhere
\end{equation*}
\end{itemize}
\end{definition}

Similarly as the case of enrichments, we may
extend the $\Mod{-}{-}$ construction into a 2-functor
\[
\Mod{-}{-}\colon \Mon{\cat{M}}^\op\times \olAct{\cat{M}}
\longrightarrow\tCAT.
\]

On the other hand, $\olAct{-}$ extends to a 2-functor 
in an apparently different manner than $\Enrich{-}$.
Namely, it is a 2-functor of type
\[
\olAct{-}\colon (\MonCATol)^\coop\longrightarrow \twoCAT,
\]
where $\MonCATol$ is the 2-category of large monoidal categories,
oplax monoidal functors and monoidal natural transformations.
The apparent discrepancy between functoriality of $\olAct{-}$
and $\Enrich{-}$
will be solved in Section~\ref{subsec:morphism_metatheory}.

We sketch the action of an oplax monoidal functor on an oplax action.
\begin{definition}
\label{def:action_of_oplax_on_oplax_action}
Let $\cat{M}=(\cat{M},I_\cat{M},\otimes_\cat{M})$ and 
$\cat{N}=(\cat{N},I_\cat{N},\otimes_\cat{N})$ be large monoidal categories,
$G=(G,g_\cdot,g)\colon\cat{N}\longrightarrow\cat{M}$ be an oplax monoidal 
functor,\footnote{Such an oplax functor consists of a functor 
$G\colon\cat{N}\longrightarrow\cat{M}$,
a morphism $g_\cdot\colon GI_\cat{N}\longrightarrow I_\cat{M}$
and a natural transformation $(g_{X,Y}\colon G(Y\otimes_\cat{N}X)
\longrightarrow GY\otimes_\cat{M} GX)_{X,Y\in\cat{N}}$ satisfying the suitable 
axioms.} 
$\cat{C}$ be a large category and $\ast=(\ast,\varepsilon,\delta)$ be an 
oplax action of $\cat{M}$ on $\cat{C}$.
We define the oplax action $G^\ast(\ast)=(\ast',\varepsilon',\delta')$ 
of $\cat{N}$ on $\cat{C}$ as follows:
\begin{itemize}
\item The functor 
$\ast'\colon\cat{N}\times\cat{C}\longrightarrow\cat{C}$
maps $(X,C)\in\cat{N}\times\cat{C}$ to $(GX)\ast C$.
\item The natural transformation $(\varepsilon'_C\colon 
I_\cat{N}\ast' C\longrightarrow C)_{C\in\cat{C}}$
is defined by $\varepsilon'_C=\varepsilon_C\circ (g_\cdot\ast C)$:
\[
\begin{tikzpicture}[baseline=-\the\dimexpr\fontdimen22\textfont2\relax ]
      \node (L) at (0,0)  {$GI_\cat{N}\ast C$};
      \node (M) at (2.5,0)  {$I_\cat{M}\ast C$};
      \node (R) at (4.5,0)  {${C}$.};
      \draw[->]  (L) to node[auto,labelsize] {$g_\cdot\ast C$} (M);
      \draw[->]  (M) to node[auto,labelsize] {$\varepsilon_C$} (R);
\end{tikzpicture} 
\]
\item The natural transformation $(\delta'_{X,Y,C}\colon 
(Y\otimes_\cat{N}X)\ast' C\longrightarrow Y\ast'(X\ast' 
C))_{X,Y\in\cat{N},C\in\cat{C}}$
is defined by $\delta'_{X,Y,C}=\delta_{GX,GY,C}\circ (g_{X,Y}\ast C)$:
\[
\begin{tikzpicture}[baseline=-\the\dimexpr\fontdimen22\textfont2\relax ]
      \node (L) at (0,0)  {$G(Y\otimes_\cat{N}X)\ast C$};
      \node (M) at (4.5,0)  {$(GY\otimes_\cat{M}GX)\ast C$};
      \node (R) at (9,0)  {$GY\ast(GX\ast{C})$.};
      \draw[->]  (L) to node[auto,labelsize] 
      {$g_{X,Y}\ast C$} (M);
      \draw[->]  (M) to node[auto,labelsize] {$\delta_{GX,GY,C}$} (R);
\end{tikzpicture} \qedhere
\]
\end{itemize}
\end{definition}

\subsection{The relation between enrichments and oplax actions}
\label{subsec:enrich_action_adjunction}
We have introduced two types of structures---enrichment and oplax action---to
formalise notions of model. 
The former captures the standard notions of model 
for clones, symmetric operads and non-symmetric operads,
whereas the latter captures those for
monads and generalised operads.
We will unify enrichment and oplax action by  
the notion of \emph{metamodel} in Section~\ref{subsec:metamodel},
but before doing so we remark on the relationship 
between them.
Though the results in this section will be subsumed by the theory of metamodels,
we believe that the following direct comparison of enrichments and oplax actions
would be more accessible to some readers.
We also explain why in some good cases we can give definition of model 
both in terms of enrichment and oplax actions;
for instances of this phenomenon in the literature,
see e.g., \cite[Section~3]{Kelly_coherence_lax_dist} and 
\cite[Section~6.4]{Leinster_book}.

%The contents of this section will not be used in the following
%and the reader may skip this and move immediately to 
%Section~\ref{sec:framework_basic}.

\medskip

Let $\cat{M}=(\cat{M},I,\otimes)$ be a large monoidal category and 
$\cat{C}$ be a large category.
The relationship between enrichment and oplax action is summarised in
the adjunction
\begin{equation}
\label{eqn:adj_act_enrich}
\begin{tikzpicture}[baseline=-\the\dimexpr\fontdimen22\textfont2\relax ]
      \node (L) at (0,0)  {$\cat{M}$};
      \node (R) at (3,0)  {$\cat{C}$.};
      \draw[->,transform canvas={yshift=5pt}]  (L) to node[auto,labelsize] 
      {$-\ast C$} (R);
      \draw[<-,transform canvas={yshift=-5pt}]  (L) to 
      node[auto,swap,labelsize] 
      {$\enrich{C}{-}$} (R);
      \node[rotate=-90,labelsize] at (1.5,0)  {$\dashv$};
\end{tikzpicture} 
\end{equation}
In more detail, what we mean is the following.
Suppose that we have an enrichment $(\enrich{-}{-},j,M)$
of $\cat{C}$ over $\cat{M}$.
If, in addition, for each $C\in\cat{C}$ the functor $\enrich{C}{-}$
has a left adjoint as in (\ref{eqn:adj_act_enrich}),
then---by the parameter theorem 
for adjunctions; see \cite[Section~IV.7]{MacLane_CWM}---the left adjoints
canonically extend to a 
bifunctor $\ast\colon \cat{M}\times\cat{C}\longrightarrow\cat{C}$,
and $j$ and $M$ define appropriate natural transformations $\varepsilon$
and $\delta$, giving rise to an oplax action $(\ast,\varepsilon,\delta)$
of $\cat{M}$ on $\cat{C}$.
And vice versa, if we start from an oplax action.

To make this idea into a precise mathematical statement, let us introduce the 
following 2-categories.
\begin{definition}
\label{def:enrichr_olactl}
Let $\cat{M}$ be a large monoidal category. 
\begin{enumerate}
\item Let $\Enrichr{\cat{M}}$ be the full sub-2-category of $\Enrich{\cat{M}}$
consisting of all enrichments $(\cat{C},\enrich{-}{-},j,M)$ such that 
for each $C\in\cat{C}$, $\enrich{C}{-}$ is a right adjoint.
\item Let $\olActl{\cat{M}}$ be the full sub-2-category of $\olAct{\cat{M}}$
consisting of all oplax actions $(\cat{C},\ast,\varepsilon,\delta)$
such that for each $C\in\cat{C}$, $-\ast C$ is a left adjoint.\qedhere
\end{enumerate}
\end{definition}
The above discussion can be summarised into the
statement that the two 2-categories $\Enrichr{\cat{M}}$ and $\olActl{\cat{M}}$
are equivalent.
A direct proof of this equivalence would be essentially routine, 
but seems to involve rather lengthy calculation.
We shall defer a proof to Corollary~\ref{cor:enrichr_olactl_equiv}.

This observation is a variant of well-known categorical folklore.
In the literature, it is usually stated in a slightly more restricted form than 
the above, for example as a correspondence between tensored $\cat{M}$-categories
and closed pseudo actions of 
$\cat{M}$~\cite{Kelly_coherence_lax_dist,Gordon_Power,Lindner_enrich_module,Janelidze_Kelly}.
 
\medskip

Furthermore, the above correspondence is
compatible with the definitions of model (Definitions~\ref{def:enrich_model}
and \ref{def:action_model}).
Suppose that $(\cat{C},\enrich{-}{-},j,M)$ and 
$(\cat{C},\ast,\varepsilon,\delta)$
form a pair of an enrichment over $\cat{M}$ and an oplax action of $\cat{M}$
connected by the adjunctions (\ref{eqn:adj_act_enrich})
(in a way compatible with the natural transformations
$j,M,\varepsilon$ and $\delta$).
Then for any monoid object $\monoid{T}=(T,e,m)$ in $\cat{M}$ and any
object $C\in\cat{C}$,
a morphism 
\[
\chi\colon T\longrightarrow \enrich{C}{C}
\]
is a model of $\monoid{T}$ in $\cat{C}$ with respect to $\enrich{-}{-}$
(Definition~\ref{def:enrich_model})
if and only if its transpose under the adjunction $-\ast C\dashv \enrich{C}{-}$
\[
\gamma\colon T\ast C\longrightarrow C
\]
is a model of $\monoid{T}$ in $\cat{C}$ with respect to $\ast$
(Definition~\ref{def:action_model}),
and similarly for homomorphism between models of $\monoid{T}$.
Hence we obtain an isomorphism of categories 
\[
\Mod{\monoid{T}}{(\cat{C},\enrich{-}{-})}\cong\Mod{\monoid{T}}{(\cat{C},\ast)}
\]
commuting with the forgetful functors into $\cat{C}$.

\begin{comment}
Let $\cat{M}$ be a large monoidal category.
The equivalence of 2-categories 
$\Enrichr{\cat{M}}\longrightarrow\olActl{\cat{M}}$
in Proposition~\ref{prop:enrich_action_adjunction} 
respects the models in the sense that there exists 
a canonical 2-natural isomorphism as in the diagram below.
\begin{equation*}
\begin{tikzpicture}[baseline=-\the\dimexpr\fontdimen22\textfont2\relax ]
      \node (TL) at (0,1)  {$\Mon{\cat{M}}^\op\times\Enrichr{\cat{M}}$};
      \node (TR) at (7,1)  {$\Mon{\cat{M}}^\op\times\olActl{\cat{M}}$};
      \node (BL) at (0,-1) {$\Mon{\cat{M}}^\op\times\Enrich{\cat{M}}$};
      \node (BR) at (7,-1) {$\Mon{\cat{M}}^\op\times\olAct{\cat{M}}$};
      \node (B) at (3.5,-2) {$\tCAT$};
      %
      \draw[->] (TL) to node[auto,labelsize](T) {$\simeq$} 
      (TR);
      \draw[->]  (TR) to node[auto,labelsize] {$\mathrm{inclusion}$} (BR);
      \draw[->]  (TL) to node[auto,swap,labelsize] {$\mathrm{inclusion}$} 
      (BL);
      \draw[->]  (BL) to node[auto,swap,labelsize] {$\Mod{-}{-}$} (B);
      \draw[->]  (BR) to node[auto,labelsize] {$\Mod{-}{-}$} (B);
      %
      \draw[2cell]  (2.5,-0.2) to node[auto,labelsize] {$\cong$} (4.5,-0.2);
\end{tikzpicture} 
\end{equation*}
\end{comment}

Some of the enrichments and oplax actions we have introduced so far
are good enough to obtain the corresponding oplax actions or enrichments,
giving rise to alternative definitions of model.

\begin{example}
\label{ex:monad_on_cat_w_powers}
Let $\cat{C}$ be a locally small category with all small powers.
Recall the strict action
\[
\ast\colon[\cat{C},\cat{C}]\times\cat{C}\longrightarrow\cat{C}
\]
of the monoidal category $[\cat{C},\cat{C}]$ of endofunctors on $\cat{C}$
on $\cat{C}$, used to capture Eilenberg--Moore algebras of monads
on $\cat{C}$.
For any object $C\in\cat{C}$, write by $\name{C}\colon 1\longrightarrow\cat{C}$
the functor from the terminal category $1$ which maps the unique object of $1$
to $C\in\cat{C}$ ($\name{C}$ is sometimes called the \emph{name of $C$}).

By the assumptions on $\cat{C}$,
for any object $A\in\cat{C}$ the functor $-\ast A$ (which may be seen as
the precomposition by $\name{A}\colon 1\longrightarrow\cat{C}$) admits
a right adjoint $\enrich{A}{-}$, which maps any $B\in\cat{C}$ (equivalently,
$\name{B}\colon 1\longrightarrow\cat{C}$) to the right Kan extension 
$\enrich{A}{B}=\Ran_{\name{A}}\name{B}$ of $\name{B}$ along $\name{A}$.
The functor $\Ran_{\name{A}}\name{B}\colon\cat{C}\longrightarrow\cat{C}$
maps $C\in\cat{C}$ to $\Ran_{\name{A}}\name{B}(C)=B^{\cat{C}(C,A)}$.

For any object $C\in\cat{C}$, $\Ran_{\name{C}}\name{C}$
exists and becomes a monad on $\cat{C}$ in a canonical way (the \emph{codensity 
monad of $\name{C}$}).
For any monad $\monoid{T}$ on $\cat{C}$, to give a structure of an 
Eilenberg--Moore algebra on $C\in\cat{C}$ is equivalent to 
give a monad morphism from $\monoid{T}$ to $\Ran_{\name{C}}\name{C}$.
This observation is in e.g., \cite[Section~3]{Kelly_coherence_lax_dist}.

In particular, if we take $\cat{C}=\Set$,
we see that the above enrichment agrees with the one given 
in Example~\ref{ex:relative_alg}.
Hence the notion of relative algebra \cite{Hino} agrees with
that of Eilenberg--Moore algebra in this case. 
\end{example}

\section{Basic concepts}
\label{sec:framework_basic}
In the previous section, we have seen that for each notion of algebraic theory 
there exists a suitable monoidal category $\cat{M}$,
and algebraic theories in that notion of algebraic 
theory corresponds to monoid objects in $\cat{M}$.
We have also observed that suitable categorical structures to give 
definitions of model of algebraic theories (notions of model) 
may be formulated in terms of $\cat{M}$, 
either as enrichment over $\cat{M}$ or as
oplax action of $\cat{M}$.

Motivated by these observations, in this section
we shall define basic concepts of our unified framework for notions of 
algebraic theory. 

\subsection{Metatheories and theories}
\begin{definition}
A \defemph{metatheory} is a large
monoidal category $\cat{M}=(\cat{M},I,\otimes)$.
%
%We say that a metatheory $(\cat{M},I,\otimes)$ 
%is \defemph{locally small} or \defemph{small} if the underlying category 
%$\cat{M}$ is so.
\end{definition}

Metatheories are intended to formalise notions of algebraic theory.
We remark that, in this thesis,
we leave the term \emph{notion of algebraic theory} informal
and will not give any mathematical definitions to it.

\begin{definition}
\label{def:theory}
Let $\cat{M}$ be a metatheory. 
A \defemph{theory in $\cat{M}$} is a monoid object
$\monoid{T}=(T,e,m)$ in $\cat{M}$.

We denote the category of theories in $\cat{M}$ by $\Th{\cat{M}}$,
which we define to be the same as $\Mon{\cat{M}}$,
the category of monoid objects in $\cat{M}$. 
\end{definition}

Theories formalise what we have been calling \emph{algebraic theories}.

\medskip

The above definitions simply renames well-known concepts.
Our hope is that, by using the terms which reflect our intention,
statements and discussions become easier to follow;
think of the terms such as \emph{generalised element} (which is synonymous
to morphism in a category) or \emph{map} (used by some authors to mean 
left adjoint in a bicategory)
which have been used with great benefit in the literature.

\subsection{Metamodels and models}
\label{subsec:metamodel}
In Sections~\ref{subsec:enrichment} and \ref{subsec:oplax_action},
we have seen that the standard notions of model for various notions of 
algebraic theory can be formalised either as enrichments or
as oplax actions. 
With \emph{two} definitions, however, we cannot claim to have 
formalised notions of model in a satisfactory way.
We now unify enrichments and oplax actions by introducing
a more general structure of 
\emph{metamodel} (of a metatheory).
We also derive a definition of models of theories and their homomorphisms
from a metamodel, and show that they generalise the corresponding 
definitions for enrichments and oplax actions.

\medskip

We may approach 
the concept of metamodel of a metatheory $\cat{M}$
in two different ways, one by generalising
enrichments over $\cat{M}$ and the other 
by generalising oplax actions of $\cat{M}$.
Before giving a formal (and neutral) definition of metamodel, we describe
these two perspectives.

\subsubsection{Metamodels as generalised enrichments}
Let us first discuss how a generalisation of enrichments over $\cat{M}$
leads to the notion of metamodel.
For this, we use a construction known as the 
\emph{Day convolution}~\cite{Day_thesis}.
Given any large monoidal category $\cat{M}=(\cat{M},I,\otimes)$,
this construction endows the presheaf category 
$\widehat{\cat{M}}=[\cat{M}^\op,\SET]$
with a (biclosed) monoidal structure $(\DayI,\Dayo)$,
in such a way that the Yoneda embedding 
$\cat{M}\longrightarrow\widehat{\cat{M}}$ canonically becomes strong monoidal.

\begin{definition}[\cite{Day_thesis}]
\label{def:Day_conv}
Let $\cat{M}=(\cat{M},I,\otimes)$ be a large monoidal category.
The \defemph{convolution monoidal structure} 
$(\DayI,\Dayo)$ on the presheaf category $\widehat{\cat{M}}
=[\cat{M}^\op,\SET]$ is defined as follows.
\begin{itemize}
\item The unit object $\DayI$ is the representable functor 
$\cat{M}(-,I)\colon \cat{M}^\op\longrightarrow\SET$.
\item Given $P,Q\in\widehat{\cat{M}}$, their monoidal product
$Q\Dayo P\colon \cat{M}^\op\longrightarrow\SET$ maps $Z\in\cat{M}$
to 
\begin{equation}
\label{eqn:Day_convolution}
(Q\Dayo P)(Z)=\int^{X,Y\in\cat{M}} \cat{M}(Z,Y\otimes X)
\times Q(Y)\times P(X).
\end{equation}
\qedhere
\end{itemize}
\end{definition}

For a metatheory $\cat{M}$, 
a \emph{metamodel of $\cat{M}$} is simply an enrichment over 
$\widehat{\cat{M}}=(\widehat{\cat{M}},\DayI,\Dayo)$\footnote{Although 
we have defined enrichment (Definition~\ref{def:enrichment})
only for large monoidal categories,
the definition does not depend on any size condition
and it is clear what we mean by enrichments over non-large monoidal categories, 
such as $\widehat{\cat{M}}$.}.
Thanks to the Yoneda embedding, it is immediate that 
every enrichment over $\cat{M}$ induces a metamodel of $\cat{M}$.

We can find several uses of $\widehat{\cat{M}}$-categories 
(in the sense of enriched category theory) in the literature.
In particular, \cite[Section~6]{Kelly_et_al_two_sides} and
\cite{Mellies_enriched_adjunction} contain 
discussions on relationship between $\widehat{\cat{M}}$-categories
and various actions of $\cat{M}$.

\subsubsection{Metamodels as generalised oplax actions}
Let us move on to the second perspective on metamodels,
namely as generalised oplax actions.
First note that an oplax action $(\cat{C},\ast,\varepsilon,\delta)$ 
of a large monoidal category $\cat{M}$
can be equivalently given as an oplax monoidal functor 
\begin{equation*}
%\label{eqn:oplax_action_as_oplax_monoidal_functor}
\cat{M}\longrightarrow[\cat{C},\cat{C}]
\end{equation*}
defined by $X\longmapsto X\ast-$,
or as a colax functor
\begin{equation}
\label{eqn:oplax_action_as_oplax_functor}
\Sigma\cat{M}\longrightarrow\tCAT,
\end{equation}
where $\Sigma\cat{M}$ denotes $\cat{M}$ seen as a one-object 
bicategory~\cite{Benabou_bicat}.

To generalise this, 
we use the bicategory $\PROF$ of \emph{profunctors}
(also called \emph{distributors} or 
\emph{bimodules})~\cite{Benabou:dist-at-work,Lawvere_metric}.
The notion of profunctor will recur in this thesis.
\begin{definition}[\cite{Benabou:dist-at-work}]
\label{def:profunctor}
We define the bicategory $\PROF$ as follows.
\begin{itemize}
\item An object is a large category.
\item A 1-cell from $\cat{A}$ to $\cat{B}$ is a \defemph{profunctor from 
$\cat{A}$ to $\cat{B}$}, which we define to be 
a functor 
\[
H\colon \cat{B}^\op\times\cat{A}\longrightarrow\SET.
\]
We write $H\colon \cat{A}\pto\cat{B}$ if $H$ is a profunctor from $\cat{A}$
to $\cat{B}$.
The identity 1-cell on a large category $\cat{C}$
is the hom-functor $\cat{C}(-,-)$.
Given profunctors $H\colon \cat{A}\pto\cat{B}$ and $K\colon \cat{B}\pto\cat{C}$,
their composite 
$K\ptensor H\colon\cat{A}\pto\cat{C}$ maps $(C,A)\in\cat{C}^\op\times\cat{A}$
to 
\begin{equation}
\label{eqn:profunctor_composition}
(K\ptensor H)(C,A)=\int^{B\in\cat{B}} K(C,B)\times H(B,A).
\end{equation}
\item A 2-cell from $H$ to $H'$, both from $\cat{A}$ to $\cat{B}$,
is a natural transformation $\alpha\colon H\Longrightarrow H'
\colon\cat{B}^\op\times\cat{A}\longrightarrow\SET$.\qedhere
\end{itemize}
\end{definition}

It is well-known that 
both $\tCAT$ and $\tCAT^\coop$ canonically
embed into $\PROF$.
Both embeddings are identity-on-objects and locally fully faithful 
pseudofunctors.
The embedding
\[
(-)_\ast\colon\tCAT\longrightarrow\PROF
\]
maps a functor $F\colon\cat{A}\longrightarrow\cat{B}$
to the profunctor $F_\ast\colon\cat{A}\pto\cat{B}$ defined by 
$F_\ast(B,A)=\cat{B}(B,FA)$.
Note that, given functors $F\colon\cat{A}\longrightarrow\cat{B}$
and $G\colon\cat{B}\longrightarrow\cat{C}$,
\begin{align*}
(G_\ast \ptensor F_\ast)(C,A) &=\int^{B\in\cat{B}}\cat{C}(C,GB)\times
\cat{B}(B,FA) \\
&\cong \cat{C}(C,GFA)\\
&=(G\circ F)_\ast(C,A)
\end{align*}
by the Yoneda lemma.
The embedding
\[
(-)^\ast\colon \tCAT^\coop\longrightarrow\PROF
\]
maps a functor $F\colon\cat{A}\longrightarrow\cat{B}$ to 
the profunctor $F^\ast\colon\cat{B}\pto\cat{A}$
with $F^\ast(A,B)=\cat{B}(FA,B)$.
For any functor $F\colon\cat{A}\longrightarrow\cat{B}$,
we have an adjunction $F_\ast\dashv F^\ast$ in $\PROF$.

A \emph{metamodel of $\cat{M}$} is a colax functor 
\[
\Sigma\cat{M}\longrightarrow \PROF^\coop,
\]
or equivalently a lax functor 
\begin{equation}
\label{eqn:metamodel_as_lax_functor}
(\Sigma\cat{M})^\co=\Sigma(\cat{M}^\op)\longrightarrow \PROF^\op.
\end{equation}
Clearly, oplax actions of $\cat{M}$, in the form
(\ref{eqn:oplax_action_as_oplax_functor}),
give rise to metamodels of $\cat{M}$ by postcomposing the pseudofunctor 
$(-)^\ast$.

Let us restate what a lax functor of type (\ref{eqn:metamodel_as_lax_functor})
amounts to, in monoidal categorical terms. 
\begin{definition}
Let $\cat{C}$ be a large category. 
Define the monoidal category $[\cat{C}^\op\times \cat{C},\SET]
=([\cat{C}^\op\times\cat{C},\SET],\cat{C}(-,-),\ptensorrev)$ of
\defemph{endo-profunctors on $\cat{C}$}
to be the endo-hom-category $\PROF^\op(\cat{C},\cat{C})$.
More precisely:
\begin{itemize}
\item The unit object is the hom-functor
$\cat{C}(-,-)\colon \cat{C}^\op\times\cat{C}\longrightarrow\SET$.
\item Given $H,K\colon\cat{C}^\op\times\cat{C}\longrightarrow\SET$,
define their monoidal product $H\ptensorrev K$
to be the functor which maps $(A,C)\in\cat{C}^\op\times\cat{C}$ to
\[
(H\ptensorrev K)(A,C)=\int^{B\in\cat{C}}H(B,C)\times K(A,B).\qedhere
\]
\end{itemize}
\end{definition}
Note that $H\ptensorrev K\cong K\ptensor H$ (i.e., $\ptensorrev$ is 
``$\ptensor$ reversed'').

Using this monoidal structure on $[\cat{C}^\op\times\cat{C},\SET]$, a 
metamodel of $\cat{M}$ in a large category $\cat{C}$
may be written as a lax monoidal functor
\[
\cat{M}^\op\longrightarrow[\cat{C}^\op\times\cat{C},\SET].
\]

\subsubsection{The definition of metamodel}

\begin{definition}
\label{def:metamodel}
Let $\cat{M}=(\cat{M},\otimes, I)$ be a metatheory.
A \defemph{metamodel of $\cat{M}$} consists of:
\begin{itemize}
\item a large category $\cat{C}$;
\item a functor 
$\Phi\colon \cat{M}^\op\times\cat{C}^\op\times\cat{C}\longrightarrow\SET$
(whose action we write as $(X,A,B)\longmapsto \Phi_X(A,B)$);
\item a natural transformation $((\overline{\phi}_\cdot)_C\colon 
1\longrightarrow\Phi_I(C,C))_{C\in\cat{C}}$;
\item a natural transformation 
\[
((\overline{\phi}_{X,Y})_{A,B,C}\colon 
\Phi_Y(B,C)\times\Phi_X(A,B)
\longrightarrow\Phi_{Y\otimes X}(A,C))_{X,Y\in\cat{M},A,B,C\in\cat{C}},
\]
\end{itemize}
making the following diagrams commute 
for all $X,Y,Z\in\cat{M}$ and $A,B,C,D\in\cat{C}$:
\begin{equation*}
\begin{tikzpicture}[baseline=-\the\dimexpr\fontdimen22\textfont2\relax ]
      \node (TL) at (0,1)  {$1\times  \Phi_X(A,B)$};
      \node (TR) at (6,1)  {$\Phi_I(B,B)\times \Phi_X(A,B)$};
      \node (BL) at (0,-1) {$\Phi_X(A,B)$};
      \node (BR) at (6,-1) {$\Phi_{I\otimes X}(A,B)$};
      \draw[->] (TL) to node[auto,labelsize](T) 
      {$(\overline{\phi}_\cdot)_B\times \Phi_X(A,B)$} 
      (TR);
      \draw[->]  (TR) to node[auto,labelsize] 
      {$(\overline{\phi}_{X,I})_{A,B,B}$} (BR);
      \draw[->]  (TL) to node[auto,swap,labelsize] {$\cong$} 
      (BL);
      \draw[->]  (BL) to node[auto,labelsize] {$\cong$} (BR);
\end{tikzpicture} 
\end{equation*}
\begin{equation*}
\begin{tikzpicture}[baseline=-\the\dimexpr\fontdimen22\textfont2\relax ]
      \node (TL) at (0,1)  {$\Phi_X(A,B)\times 1$};
      \node (TR) at (6,1)  {$\Phi_X(A,B)\times\Phi_I(A,A)$};
      \node (BL) at (0,-1) {$\Phi_X(A,B)$};
      \node (BR) at (6,-1) {$\Phi_{X\otimes I}(A,B)$};
      \draw[->] (TL) to node[auto,labelsize](T) 
      {$\Phi_X(A,B)\times(\overline{\phi}_\cdot)_A$} 
      (TR);
      \draw[->]  (TR) to node[auto,labelsize] 
      {$(\overline{\phi}_{I,X})_{A,A,B}$} (BR);
      \draw[->]  (TL) to node[auto,swap,labelsize] {$\cong$} 
      (BL);
      \draw[->]  (BL) to node[auto,labelsize] {$\cong$} (BR);
\end{tikzpicture}
\end{equation*}
\begin{equation*}
\begin{tikzpicture}[baseline=-\the\dimexpr\fontdimen22\textfont2\relax ]
      \node (TL) at (0,2)  {$\big(\Phi_Z(C,D)\times \Phi_Y(B,C)\big)\times 
            \Phi_X(A,B)$};
      \node (TM) at (8,2)  {$\Phi_{Z\otimes Y}(B,D)\times \Phi_X(A,B)$};
      \node (TR) at (8,0)  {$\Phi_{(Z\otimes Y)\otimes X}(A,D)$};
      \node (BL) at (0,0) {$\Phi_Z(C,D)\times \big(\Phi_Y(B,C)\times 
                  \Phi_X(A,B)\big)$};
      \node (BM) at (0,-2) {$\Phi_Z(C,D)\times\Phi_{Y\otimes X}(A,C)$};
      \node (BR) at (8,-2) {$\Phi_{Z\otimes (Y\otimes X)}(A,D).$};
      \draw[->] (TL) to node[auto,labelsize](T) 
      {$(\overline{\phi}_{Y,Z})_{B,C,D}\times\Phi_X(A,B)$} 
      (TM);
      \draw[->] (TM) to node[auto,labelsize](T) 
      {$(\overline{\phi}_{X,Z\otimes Y})_{A,B,D}$} 
      (TR);
      \draw[->]  (TR) to node[auto,labelsize] {$\cong$} (BR);
      \draw[->]  (TL) to node[auto,swap,labelsize] {$\cong$} (BL);
      \draw[->]  (BL) to node[auto,swap,labelsize] {$\Phi_Z(C,D)\times 
      (\overline{\phi}_{X,Y})_{A,B,C}$} 
      (BM);
      \draw[->] (BM) to node[auto,labelsize](B) {$(\overline{\phi}_{Y\otimes 
      X,Z})_{A,C,D}$} (BR);
\end{tikzpicture} 
\end{equation*}
We say that $(\cat{C},\Phi,\overline{\phi}_\cdot,\overline{\phi})$ is a 
metamodel of $\cat{M}$,
or that $(\Phi,\overline{\phi}_\cdot,\overline{\phi})$ is a 
metamodel of $\cat{M}$ in $\cat{C}$.
\end{definition}

The above definition perfectly makes sense even if we replace 
the category $\SET$ of large sets by the category $\Set$
of small sets. 
Indeed, most of the naturally occurring notions of model 
can be captured by these ``small'' metamodels.
However, for later developments it turns out to be more 
convenient to define metamodels as above. 

Note that
we may replace $((\overline{\phi}_\cdot)_C)_{C\in\cat{C}}$ by 
\[
((j_C)_Z\colon \DayI(Z)\longrightarrow\Phi_{Z}(C,C))_{C\in\cat{C},Z\in\cat{M}}
\] 
and $((\overline{\phi}_{X,Y})_{A,B,C})_{X,Y\in\cat{M},A,B,C\in\cat{C}}$ by 
\[
((M_{A,B,C})_Z\colon (\Phi_{(-)}(B,C)\Dayo \Phi_{(-)}(A,B))(Z)
\longrightarrow \Phi_Z(A,C))_{A,B,C\in\cat{M},Z\in\cat{M}}.
\]
The axioms for metamodel then translate to the ones for 
enrichments (over $\widehat{\cat{M}}$).

On the other hand, we may also replace 
$((\overline{\phi}_\cdot)_C)_{C\in\cat{C}}$ by
\[
((\phi_\cdot)_{A,B}\colon\cat{C}(A,B)\longrightarrow\Phi_I(A,B))_{A,B\in\cat{C}}
\]
and $((\overline{\phi}_{X,Y})_{A,B,C})_{X,Y\in\cat{M},A,B,C\in\cat{C}}$ by
\[
((\phi_{X,Y})_{A,C}\colon (\Phi_Y\ptensorrev\Phi_X)(A,C)\longrightarrow
\Phi_{Y\otimes X}(A,C))_{A,C\in\cat{C},X,Y\in\cat{M}}.
\]
The axioms for metamodel then state that 
\[
(\Phi,\phi_\cdot,\phi)\colon(\cat{M}^\op,I,\otimes)\longrightarrow 
([\cat{C}^\op\times\cat{C},\SET],\cat{C}(-,-),\ptensorrev)
\]
is an oplax monoidal functor.

Hence the attempts to generalise enrichments and oplax actions mentioned above
coincide and both give rise to Definition~\ref{def:metamodel}.

\medskip

The definitions of model and homomorphism we derive from a metamodel 
are the following.
\begin{definition}
\label{def:metamodel_model}
Let $\cat{M}=(\cat{M},I,\otimes)$ be a metatheory, 
$\monoid{T}=(T,e,m)$ be a theory in $\cat{M}$, 
$\cat{C}$ be a large category and 
$\Phi=(\Phi,\overline{\phi}_\cdot,\overline{\phi})$ 
be a metamodel of $\cat{M}$ in $\cat{C}$.
\begin{enumerate}
\item A \defemph{model of $\monoid{T}$ in $\cat{C}$ with respect to $\Phi$}
is a pair $(C,\xi)$ consisting of an object $C$ of $\cat{C}$
and an element $\xi\in \Phi_T(C,C)$ such that 
\begin{comment}
$(\Phi_e)_{C,C}(\xi)=(\phi_\cdot)_{C,C}(\id{C})$ and
$(\Phi_m)_{C,C}(\xi)= (\phi_{T,T})_{C,C}([C, \xi, \xi])$ 
(where $[C,\xi,\xi]=\iota_C(\xi,\xi)$):
\[
\begin{tikzpicture}[baseline=-\the\dimexpr\fontdimen22\textfont2\relax ]
      \node(11) at (0,1.5) {$\Phi_T(C,C)$};
      \node(12) at (2,0) {$\Phi_I(C,C)$};
      \node(21) at (4,1.5) {$\cat{C}(C,C)$};
      
      \draw [->]  (11) to node [auto,swap,labelsize]{$(\Phi_e)_{C,C}$} (12);
      \draw [->]  (21) to node [auto,labelsize]{$(\phi_\cdot)_{C,C}$} (12);
\end{tikzpicture}
\quad
\begin{tikzpicture}[baseline=-\the\dimexpr\fontdimen22\textfont2\relax ]
      \node(11) at (0,1.5) {$\Phi_T(C,C)$};
      \node(12) at (2,0) {$\Phi_{T\otimes T}(C,C).$};
      \node(21) at (4,1.5) {$(\Phi_T \ptensor \Phi_T)(C,C)$};
      \node(20) at (4,3) {$\Phi_T(C,C)\times \Phi_T(C,C)$};
      
      \draw [->]  (11) to node [auto,swap,labelsize]{$(\Phi_m)_{C,C}$} (12);
      \draw [->]  (21) to node [auto,labelsize]{$(\phi_{T,T})_{C,C}$} (12);
      \draw [->]  (20) to node [auto,labelsize]{$\iota_C$} (21);
\end{tikzpicture}
\]
\end{comment} 
$(\Phi_e)_{C,C}(\xi)=(\overline{\phi}_\cdot)_{C}(\ast)$ (where $\ast$
is the unique element of $1$)
and $(\Phi_m)_{C,C}(\xi)=(\overline{\phi}_{T,T})_{C,C,C}(\xi,\xi)$:
\[
\begin{tikzpicture}[baseline=-\the\dimexpr\fontdimen22\textfont2\relax ]
      \node(11) at (0,1.5) {$\Phi_T(C,C)$};
      \node(12) at (2,0) {$\Phi_I(C,C)$};
      \node(21) at (4,1.5) {$1$};
      
      \draw [->]  (11) to node [auto,swap,labelsize]{$(\Phi_e)_{C,C}$} (12);
      \draw [->]  (21) to node [auto,labelsize]{$(\overline{\phi}_\cdot)_{C}$} 
      (12);
\end{tikzpicture}
\quad
\begin{tikzpicture}[baseline=-\the\dimexpr\fontdimen22\textfont2\relax ]
      \node(11) at (0,1.5) {$\Phi_T(C,C)$};
      \node(12) at (2,0) {$\Phi_{T\otimes T}(C,C).$};
      \node(21) at (4,1.5) {$\Phi_T(C,C)\times \Phi_T(C,C)$};
      
      \draw [->]  (11) to node [auto,swap,labelsize]{$(\Phi_m)_{C,C}$} (12);
      \draw [->]  (21) to node 
      [auto,labelsize]{$(\overline{\phi}_{T,T})_{C,C,C}$} (12);
\end{tikzpicture}
\] 
\item Let $(C,\xi)$ and $(C',\xi')$ be models of $\monoid{T}$ in $\cat{C}$
with respect to $\Phi$. A \defemph{homomorphism from $(C,\xi)$ to $(C',\xi')$}
is a morphism $f\colon C\longrightarrow C'$ in $\cat{C}$ such that
$\Phi_T(C,f)(\xi)=\Phi_T(f,C')(\xi')$:
\[
\begin{tikzpicture}[baseline=-\the\dimexpr\fontdimen22\textfont2\relax ]
      \node(11) at (0,1.5) {$\Phi_T(C,C)$};
      \node(12) at (2,0) {$\Phi_T(C,C').$};
      \node(21) at (4,1.5) {$\Phi_T(C',C')$};
      
      \draw [->]  (11) to node [auto,swap,labelsize]{$\Phi_T(C,f)$} (12);
      \draw [->]  (21) to node [auto,labelsize]{$\Phi_T(f,C')$} (12);
\end{tikzpicture}
\] 
\end{enumerate}
We denote the (large) category of models of $\monoid{T}$ in $\cat{C}$ with 
respect to 
$\Phi$ by $\Mod{\monoid{T}}{(\cat{C},\Phi)}$.
\end{definition}

\begin{example}
\label{ex:enrichment_as_metamodel}
Let $\cat{M}=(\cat{M},I,\otimes)$ be a metatheory,
$\cat{C}$ be a large category and 
$(\enrich{-}{-},j,M)$ be an enrichment of $\cat{C}$ over $\cat{M}$.
This induces a metamodel $(\Phi,\overline{\phi}_\cdot,\overline{\phi})$ 
of $\cat{M}$ in $\cat{C}$ as follows.
\begin{itemize}
\item The functor 
$\Phi\colon\cat{M}^\op\times\cat{C}^\op\times\cat{C}\longrightarrow\SET$
maps $(X,A,B)\in \cat{M}^\op\times\cat{C}^\op\times\cat{C}$
to 
\[
\Phi_X(A,B)=\cat{M}(X,\enrich{A}{B}).
\]
\item For each $C\in\cat{C}$, $(\overline{\phi}_\cdot)_C\colon 
1\longrightarrow\Phi_I(C,C)$
is the name of $j_C$
(i.e., $(\overline{\phi}_\cdot)_C$ maps the unique element of the singleton 
$1$  to ${j_C}$).
\item For each $A,B,C\in\cat{C}$ and $X,Y\in\cat{M}$, the function
$(\overline{\phi}_{X,Y})_{A,B,C}\colon 
\Phi_Y(B,C)\times\Phi_X(A,B)
\longrightarrow\Phi_{Y\otimes X}(A,C)$
maps $g\colon Y\longrightarrow\enrich{B}{C}$ and $f\colon 
X\longrightarrow\enrich{A}{B}$ to 
\[
\begin{tikzpicture}[baseline=-\the\dimexpr\fontdimen22\textfont2\relax ]
      \node (L) at (0,0)  {$Y\otimes X$};
      \node (M) at (3,0)  {$\enrich{B}{C}\otimes \enrich{A}{B}$};
      \node (R) at (6,0) {$\enrich{A}{C}.$};
      \draw[->] (L) to node[auto,labelsize](T) {$g\otimes f$} 
      (M);
      \draw[->]  (M) to node[auto,labelsize] {$M_{A,B,C}$} (R);
\end{tikzpicture}
\]
\end{itemize}

The definition of model and homomorphism (Definition~\ref{def:enrich_model})
we derive from an enrichment may be seen as 
a special case of the 
corresponding definition (Definition~\ref{def:metamodel_model})
for metamodel.
\end{example}

\begin{example}
\label{ex:oplax_action_as_metamodel}
Let $\cat{M}=(\cat{M},I,\otimes)$ be a metatheory,
$\cat{C}$ be a large category and 
$(\ast,\varepsilon,\delta)$ be an oplax action of $\cat{M}$ on $\cat{C}$.
This induces a metamodel $(\Phi,\overline{\phi}_\cdot,\overline{\phi})$ 
of $\cat{M}$ in $\cat{C}$ as follows.
\begin{itemize}
\item The functor 
$\Phi\colon\cat{M}^\op\times\cat{C}^\op\times\cat{C}\longrightarrow\SET$
maps $(X,A,B)\in \cat{M}^\op\times\cat{C}^\op\times\cat{C}$
to 
\[
\Phi_X(A,B)=\cat{C}(X\ast A,{B}).
\]
\item For each $C\in\cat{C}$, $(\overline{\phi}_\cdot)_C\colon 
1\longrightarrow\Phi_I(C,C)$
is the name of $\varepsilon_C$.
\item For each $A,B,C\in\cat{C}$ and $X,Y\in\cat{M}$, the function
$(\overline{\phi}_{X,Y})_{A,B,C}\colon 
\Phi_Y(B,C)\times\Phi_X(A,B)
\longrightarrow\Phi_{Y\otimes X}(A,C)$
maps $g\colon Y\ast{B}\longrightarrow{C}$ and $f\colon 
X\ast{A}\longrightarrow{B}$ to 
\[
\begin{tikzpicture}[baseline=-\the\dimexpr\fontdimen22\textfont2\relax ]
      \node (L) at (-0.5,0)  {$(Y\otimes X)\ast A$};
      \node (M) at (3,0)  {$Y\ast(X\ast A)$};
      \node (R) at (6,0) {$Y\ast B$};
      \node (RR) at (8,0) {$C.$};
      \draw[->] (L) to node[auto,labelsize](T) {$\delta_{X,Y,A}$} 
      (M);
      \draw[->]  (M) to node[auto,labelsize] {$Y\ast f$} (R);
      \draw[->]  (R) to node[auto,labelsize] {$g$} (RR);
\end{tikzpicture}
\]
\end{itemize}

The definition of model and homomorphism 
(Definition~\ref{def:action_model})
we derive from an oplax action may be seen as a special case of the
corresponding definition (Definition~\ref{def:metamodel_model})
for metamodel.
\end{example}

\subsubsection{The 2-category of metamodels}
Metamodels of a metatheory naturally form a 2-category,
just like enrichments and oplax actions do.

\begin{definition}
Let $\cat{M}=(\cat{M},I,\otimes)$ be a metatheory.
We define the (locally large) 2-category $\MtMod{\cat{M}}$ of 
metamodels of $\cat{M}$ as follows.
\begin{itemize}
\item An object is a metamodel 
$(\cat{C},\Phi,\overline{\phi}_\cdot,\overline{\phi})$ of $\cat{M}$.
\item A 1-cell from $(\cat{C},\Phi,\overline{\phi}_\cdot,\overline{\phi})$ to 
$(\cat{C'},\Phi',\overline{\phi'}_\cdot,\overline{\phi'})$
is a functor $G\colon \cat{C}\longrightarrow \cat{C'}$
together with a natural transformation $(g_{X,A,B}\colon 
\Phi_X(A,B)\longrightarrow \Phi'_X(GA,GB))_{X\in\cat{M},A,B\in\cat{C}}$ 
making the following 
diagrams commute for all 
\begin{comment}
$X,Y\in\cat{M}$ and $A,B\in\cat{C}$:
\begin{equation*}
\begin{tikzpicture}[baseline=-\the\dimexpr\fontdimen22\textfont2\relax ]
      \node (TL) at (0,1)  {$\cat{C}(A,B)$};
      \node (TR) at (4,1)  {$\Phi_I(A,B)$};
      \node (BL) at (0,-1) {$\cat{C'}(GA,GB)$};
      \node (BR) at (4,-1) {$\Phi'_I(GA,GB)$};
      %
      \draw[->] (TL) to node[auto,labelsize](T) {$(\phi_\cdot)_{A,B}$} 
      (TR);
      \draw[->]  (TR) to node[auto,labelsize] {$g_{I,A,B}$} (BR);
      \draw[->]  (TL) to node[auto,swap,labelsize] {$G_{A,B}$} 
      (BL);
      \draw[->]  (BL) to node[auto,labelsize] {$(\phi'_\cdot)_{GA,GB}$} (BR);
      %
\end{tikzpicture} 
\end{equation*}
\begin{equation*}
\begin{tikzpicture}[baseline=-\the\dimexpr\fontdimen22\textfont2\relax ]
      \node (TL) at (0,1)  {$(\Phi_Y\ptensor \Phi_X)(A,B)$};
      \node (TR) at (6,1)  {$\Phi_{Y\otimes X}(A,B)$};
      \node (BL) at (0,-1) {$(\Phi'_Y\ptensor \Phi'_X)(GA,GB)$};
      \node (BR) at (6,-1) {$\Phi'_{Y\otimes X}(GA,GB).$};
      %
      \draw[->] (TL) to node[auto,labelsize](T) {$(\phi_{X,Y})_{A,B}$} 
      (TR);
      \draw[->]  (TR) to node[auto,labelsize] {$g_{Y\otimes X,A,B}$} (BR);
      \draw[->]  (TL) to node[auto,swap,labelsize] {$(g_Y\ptensor g_X)_{A,B}$} 
      (BL);
      \draw[->]  (BL) to node[auto,labelsize] {$(\phi'_{X,Y})_{GA,GB}$} (BR);
      %
\end{tikzpicture} 
\end{equation*}
The condition for $(G,g)$ to be a 1-cell of $\MtMod{\cat{M}}$
is equivalently given by the commutativity of the following diagrams
for all 
\end{comment}
$X,Y\in\cat{M}$ and $A,B,C\in\cat{C}$:
\begin{equation*}
\begin{tikzpicture}[baseline=-\the\dimexpr\fontdimen22\textfont2\relax ]
      \node (TL) at (0,1)  {$1$};
      \node (TR) at (4,1)  {$\Phi_I(C,C)$};
      \node (BR) at (4,-1) {$\Phi'_I(GC,GC)$};
      \draw[->] (TL) to node[auto,labelsize](T) {$(\overline{\phi}_\cdot)_{C}$} 
      (TR);
      \draw[->]  (TR) to node[auto,labelsize] {$g_{I,C,C}$} (BR);
      \draw[->]  (TL) to node[auto,swap,labelsize] 
      {$(\overline{\phi'}_\cdot)_{GC}$} 
      (BR);
\end{tikzpicture} 
\end{equation*}
\begin{equation*}
\begin{tikzpicture}[baseline=-\the\dimexpr\fontdimen22\textfont2\relax ]
      \node (TL) at (0,1)  {$\Phi_Y(B,C)\times \Phi_X(A,B)$};
      \node (TR) at (8,1)  {$\Phi_{Y\otimes X}(A,C)$};
      \node (BL) at (0,-1) {$\Phi'_Y(GB,GC)\times \Phi'_X(GA,GB)$};
      \node (BR) at (8,-1) {$\Phi'_{Y\otimes X}(GA,GC).$};
      \draw[->] (TL) to node[auto,labelsize](T) 
      {$(\overline{\phi}_{X,Y})_{A,B,C}$} 
      (TR);
      \draw[->]  (TR) to node[auto,labelsize] {$g_{Y\otimes X,A,C}$} (BR);
      \draw[->]  (TL) to node[auto,swap,labelsize] {$g_{Y,B,C}\times 
      g_{X,A,B}$} 
      (BL);
      \draw[->]  (BL) to node[auto,labelsize] 
      {$(\overline{\phi'}_{X,Y})_{GA,GB,GC}$} (BR);
\end{tikzpicture} 
\end{equation*}
\begin{comment}
for each $X\in\cat{M}$,
a natural transformation
\[
\begin{tikzpicture}[baseline=-\the\dimexpr\fontdimen22\textfont2\relax ]
      \node (TL) at (0,1.5)  {$\cat{C}$};
      \node (TR) at (2,1.5)  {$\cat{C}$};
      \node (BL) at (0,0) {$\cat{C'}$};
      \node (BR) at (2,0) {$\cat{C'}$};
      %
      \draw[pto] (TL) to node[auto,labelsize](T) {$\Phi_X$} 
      (TR);
      \draw[->]  (TR) to node[auto,labelsize] {$G$} (BR);
      \draw[->]  (TL) to node[auto,swap,labelsize] {$G$} 
      (BL);
      \draw[pto] (BL) to node[auto,swap,labelsize](B) {$\Phi'_X$} (BR);
      %
      \tzsquare{1}{0.75}{$g_X$}
\end{tikzpicture} 
\]
such that for each $f\colon X\longrightarrow X'$ in $\cat{M}$,
$g_{X'} \circ \Phi f = \Phi' f\circ g_X$,
$g_I\circ \phi_\cdot= \phi'_\cdot\circ\id{G}$,
and for each $X,Y\in\cat{M}$,
$\phi'_{X,Y}\circ g_{X\otimes Y} = (g_X\ptensor g_Y)\circ\phi_{X,Y}$.
\end{comment}
\item A 2-cell from $(G,g)$ to $(G',g')$, both from 
$(\cat{C},\Phi,\overline{\phi}_\cdot,\overline{\phi})$ to 
$(\cat{C'},\Phi',\overline{\phi'}_\cdot,\overline{\phi'})$,
is a natural transformation $\theta\colon G\Longrightarrow G'$
making the following diagram commute for all $X\in\cat{M}$
and $A,B\in\cat{C}$:
\begin{equation*}
\begin{tikzpicture}[baseline=-\the\dimexpr\fontdimen22\textfont2\relax ]
      \node (TL) at (0,2)  {$\Phi_X(A,B)$};
      \node (TR) at (6,2)  {$\Phi'_X(GA,GB)$};
      \node (BL) at (0,0) {$\Phi'_X(G'A,G'B)$};
      \node (BR) at (6,0) {$\Phi'_X(GA,G'B).$};
      \draw[->] (TL) to node[auto,labelsize](T) {$g_{X,A,B}$} 
      (TR);
      \draw[->]  (TR) to node[auto,labelsize] {$\Phi'_X(GA,\theta_B)$} (BR);
      \draw[->]  (TL) to node[auto,swap,labelsize] {$g'_{X,A,B}$} 
      (BL);
      \draw[->]  (BL) to node[auto,labelsize] {$\Phi'_X(\theta_A,G'B)$} (BR);
\end{tikzpicture} \qedhere
\end{equation*}
\end{itemize}
\end{definition}

Recall that for a functor (resp.~a 2-functor) 
$F\colon \cat{A}\longrightarrow\cat{B}$,
the \defemph{essential image of $F$} is the full subcategory
(resp.~full sub-2-category) of $\cat{B}$
consisting of all objects $B\in\cat{B}$ such that there exists an object 
$A\in\cat{A}$ and an isomorphism $FA\cong B$.
If $\cat{A}$ is a large category, a contravariant presheaf
$\cat{A}^\op\longrightarrow\SET$ 
(resp.~a covariant presheaf $\cat{A}\longrightarrow\SET$) 
over $\cat{A}$ is 
called \defemph{representable} if and only if it is 
in the essential image of the Yoneda embedding 
$\cat{A}\longrightarrow[\cat{A}^\op,\SET]$
(resp.~$\cat{A}\longrightarrow [\cat{A},\SET]^\op$).

\begin{proposition}
\label{prop:enrichment_as_metamodel}
Let $\cat{M}$ be a metatheory.
The construction given in Example~\ref{ex:enrichment_as_metamodel}
canonically extends to a fully faithful 2-functor
\[
\Enrich{\cat{M}}\longrightarrow\MtMod{\cat{M}}.
\]
A metamodel $(\cat{C},\Phi,\overline{\phi}_\cdot,\overline{\phi})$
of $\cat{M}$ is in the essential image of this 2-functor if and only if 
for each $A,B\in\cat{C}$, the functor 
\[
\Phi_{(-)}(A,B)\colon\cat{M}^\op\longrightarrow\SET
\]
is representable.
\end{proposition}
\begin{proof}
The construction of the 
2-functor $\Enrich{\cat{M}}\longrightarrow\MtMod{\cat{M}}$
is straightforward. 
The rest can also be proved by a standard argument using the Yoneda lemma.
We sketch the argument below.

Let us focus on the characterisation of the essential image.
Suppose that $(\cat{C},\Phi,\overline{\phi}_\cdot,\overline{\phi})$ is a 
metamodel of $\cat{M}$ such that
for each $A,B\in\cat{C}$, the functor $\Phi_{(-)}(A,B)$
is representable.
From such a metamodel we obtain an enrichment $(\enrich{-}{-},j,M)$ of 
$\cat{C}$ over $\cat{M}$ as follows. 
For each $A,B\in\cat{C}$, choose an object $\enrich{A}{B}\in\cat{M}$ and an 
isomorphism 
$\alpha_{A,B}\colon\cat{M}(-,\enrich{A}{B})\longrightarrow\Phi_{(-)}(A,B)$.
By functoriality of $\Phi$, $\enrich{-}{-}$ uniquely extends to a 
functor of type 
$\cat{C}^\op\times\cat{C}\longrightarrow\cat{M}$ while making 
$(\alpha_{A,B})_{A,B\in\cat{C}}$ natural.
For each $C\in\cat{C}$, $(\overline{\phi}_\cdot)_C\colon 1\longrightarrow
\Phi_I(C,C)\cong \cat{M}(I,\enrich{C}{C})$ gives rise to a 
morphism $j_C\colon I\longrightarrow\enrich{C}{C}$ in $\cat{M}$. 
For each $A,B,C\in\cat{M}$, consider the function
\[
\begin{tikzpicture}[baseline=-\the\dimexpr\fontdimen22\textfont2\relax ]
      \node (TT) at (0,4.5)  
      {$\cat{M}(\enrich{B}{C},\enrich{B}{C})\times\cat{M}(\enrich{A}{B},
      \enrich{A}{B})$};
      \node (T) at (0,3)  
      {$\Phi_{\enrich{B}{C}}(B,C)\times\Phi_\enrich{A}{B}(A,B)$};
      \node (L) at (0,1.5)  {$\Phi_{\enrich{B}{C}\otimes\enrich{A}{B}}(A,C)$};
      \node (LL) at (0,0) {$\cat{M}(\enrich{B}{C}\otimes\enrich{A}{B}, 
      \enrich{A}{C}).$};
      \draw[->] (T) to node[auto,labelsize]
      {$(\overline{\phi}_{\enrich{A}{B},\enrich{B}{C}})_{A,B,C}$} 
      (L);
      \draw[->] (TT) to node[auto,labelsize]{$(\alpha_{B,C})_\enrich{B}{C}
      \times(\alpha_{A,B})_{\enrich{A}{B}}$} (T);
      \draw[->] (L) to 
      node[auto,labelsize]{$(\alpha_{A,C})_{\enrich{B}{C}
      \otimes\enrich{A}{B}}^{-1}$}
       (LL);
\end{tikzpicture}
\]
Let the image of $(\id{\enrich{B}{C}},\id{\enrich{A}{B}})$ 
under this function be $M_{A,B,C}\colon 
\enrich{B}{C}\otimes\enrich{A}{B}\longrightarrow\enrich{A}{C}$.
The axioms of metamodel then shows that $(\enrich{-}{-},j,M)$ is an enrichment.

Moreover, if we consider the metamodel induced from this enrichment 
(see Example~\ref{ex:enrichment_as_metamodel}), then it is 
isomorphic to our original 
$(\cat{C},\Phi,\overline{\phi}_\cdot,\overline{\phi})$.
In particular, for each $X,Y\in\cat{M}$ and $A,B,C\in\cat{C}$, 
the function $(\overline{\phi}_{X,Y})_{A,B,C}$ is completely
determined by $M_{A,B,C}$, as in 
Example~\ref{ex:enrichment_as_metamodel}.
To see this, note that for each $f\in\cat{M}(X,\enrich{A}{B})$
and $g\in\cat{M}(Y,\enrich{B}{C})$, the diagram 
\[
\begin{tikzpicture}[baseline=-\the\dimexpr\fontdimen22\textfont2\relax ]
      \node (TT1) at (7.5,4.5)  
      {$\cat{M}(Y,\enrich{B}{C})\times\cat{M}(X,\enrich{A}{B})$};
      \node (T1) at (7.5,3)  
      {$\Phi_Y(B,C)\times\Phi_X(A,B)$};
      \node (L1) at (7.5,1.5)  
      {$\Phi_{Y\otimes X}(A,C)$};
      \node (LL1) at (7.5,0) {$\cat{M}(\enrich{B}{C}\otimes\enrich{A}{B}, 
      \enrich{A}{C})$};
      \node (TT) at (0,4.5)  
      {$\cat{M}(\enrich{B}{C},\enrich{B}{C})\times\cat{M}(\enrich{A}{B},
      \enrich{A}{B})$};
      \node (T) at (0,3)  
      {$\Phi_{\enrich{B}{C}}(B,C)\times\Phi_\enrich{A}{B}(A,B)$};
      \node (L) at (0,1.5)  {$\Phi_{\enrich{B}{C}\otimes\enrich{A}{B}}(A,C)$};
      \node (LL) at (0,0) {$\cat{M}(\enrich{B}{C}\otimes\enrich{A}{B}, 
      \enrich{A}{C})$};
      \draw[->] (T) to node[auto,swap,labelsize]
      {$(\overline{\phi}_{\enrich{A}{B},\enrich{B}{C}})_{A,B,C}$} 
      (L);
      \draw[->] (TT) to 
      node[auto,swap,labelsize]{$\alpha
      \times\alpha$} (T);
      \draw[->] (L) to 
      node[auto,swap,labelsize]{$\alpha^{-1}$}
       (LL);
      \draw[->] (T1) to node[auto,labelsize]
      {$(\overline{\phi}_{X,Y})_{A,B,C}$} 
      (L1);
      \draw[->] (TT1) to 
      node[auto,labelsize]{$\alpha \times\alpha$} (T1);
      \draw[->] (L1) to 
      node[auto,labelsize]{$\alpha^{-1}$}
       (LL1);
      \draw[->] (TT) to node[auto,labelsize]
      {$\cat{M}(g,\id{})\times\cat{M}(f,\id{})$} 
      (TT1);
      \draw[->] (T) to node[auto,labelsize]
      {$\Phi_g(B,C)\times\Phi_f(A,B)$} 
      (T1);
      \draw[->] (L) to node[auto,labelsize]
      {$\Phi_{g\otimes f}(A,C)$} 
      (L1);
      \draw[->] (LL) to node[auto,labelsize]
      {$\cat{M}(g\otimes f,\id{})$} 
      (LL1);
\end{tikzpicture}
\]
commutes. Hence by chasing the element $(\id{\enrich{B}{C}},\id{\enrich{A}{B}})$
in the top left set, we observe that (modulo the isomorphisms $\alpha$)
$(g,f)$ is mapped by $(\overline{\phi}_{X,Y})_{A,B,C}$ to 
$M_{A,B,C}\circ (g\otimes f)$.
\end{proof}

\begin{proposition}
\label{prop:oplax_action_as_metamodel}
Let $\cat{M}$ be a metatheory.
The construction given in Example~\ref{ex:oplax_action_as_metamodel}
canonically extends to a fully faithful 2-functor
\[
\olAct{\cat{M}}\longrightarrow\MtMod{\cat{M}}.
\]
A metamodel $(\cat{C},\Phi,\overline{\phi}_\cdot,\overline{\phi})$ of $\cat{M}$ 
is in the essential image of this 2-functor if and only if for each
$X\in\cat{M}$ and $A\in\cat{C}$, the functor 
\[
\Phi_X(A,-)\colon \cat{C}\longrightarrow\SET
\]
is representable.
\end{proposition}
\begin{proof}
Similar to the proof of Proposition~\ref{prop:enrichment_as_metamodel}.

In particular, given a metamodel 
$(\cat{C},\Phi,\overline{\phi}_\cdot,\overline{\phi})$ of $\cat{M}$ such that
for each $X\in\cat{M}$ and $A\in\cat{C}$, the functor $\Phi_X(A,-)$ is 
representable, we may construct an oplax action $(\ast,\varepsilon,\delta)$ of 
$\cat{C}$ as follows.
For each $X\in\cat{M}$ and $A\in\cat{C}$, choose an object $X\ast A\in\cat{C}$
and an isomorphism $\beta_{X,A}\colon \cat{C}(X\ast A,-) 
\longrightarrow\Phi_X(A,-)$.
We easily obtain a functor 
$\ast\colon\cat{M}\times\cat{C}\longrightarrow\cat{C}$
and a natural transformation $(\varepsilon_C)_{C\in\cat{C}}$. 
To get $\delta$, for each $X,Y\in\cat{M}$ and $A\in\cat{C}$ consider 
the function
\[
\begin{tikzpicture}[baseline=-\the\dimexpr\fontdimen22\textfont2\relax ]
      \node (TT) at (0,4.5)  
      {$\cat{C}(Y\ast (X\ast A),Y\ast (X\ast A))\times\cat{C}
      (X\ast A,X\ast A)$};
      \node (T) at (0,3)  
      {$\Phi_{Y}(X\ast A,Y\ast (X\ast A))\times\Phi_X(A,X\ast A)$};
      \node (L) at (0,1.5)  {$\Phi_{Y\otimes X}(A,Y\ast (X\ast A))$};
      \node (LL) at (0,0) {$\cat{C}((Y\otimes X)\ast A, 
      Y\ast (X\ast A)).$};
      \draw[->] (T) to node[auto,labelsize]
      {$(\overline{\phi}_{X,Y})_{A,X\ast A,Y\ast (X\ast A)}$}
      (L);
      \draw[->] (TT) to node[auto,labelsize]{$(\beta_{Y,X\ast A})_{Y\ast (X\ast 
      A)}
      \times(\beta_{X,A})_{X\ast A}$} (T);
      \draw[->] (L) to 
      node[auto,labelsize]{$(\beta_{Y\otimes X,A})_{Y\ast(X\ast A)}^{-1}$}
       (LL);
\end{tikzpicture}
\]
We define $\delta_{X,Y,A}\colon (Y\otimes X)\ast A\longrightarrow Y\ast(X\ast 
A)$ to be the image of $(\id{Y\ast (X\ast A)},\id{X\ast A})$ under this 
function.

To verify that the metamodel induced from this oplax action (see 
Example~\ref{ex:oplax_action_as_metamodel}) is isomorphic to 
$(\cat{C},\Phi,\overline{\phi}_\cdot,\overline{\phi})$,
essentially we only need to check that $(\overline{\phi}_{X,Y})_{A,B,C}$
for each $X,Y\in\cat{M}$ and $A,B,C\in\cat{C}$ is determined by
$\delta_{X,Y,A}$ as in Example~\ref{ex:oplax_action_as_metamodel}.
Suppressing the isomorphisms $\beta$ from now on,
for each $f\colon X\ast A\longrightarrow B$ and $g\colon Y\ast B\longrightarrow 
C$ consider the following digram:
\[
\begin{tikzpicture}[baseline=-\the\dimexpr\fontdimen22\textfont2\relax ]
      \node (TT1) at (7.5,4)  
      {$\cat{C}((Y\otimes X)\ast A,Y\ast (X\ast A))$};
      \node (T1) at (7.5,2)  
      {$\cat{C}((Y\otimes X)\ast A,C)$};
      \node (L1) at (7.5,0)  
      {$\cat{C}(Y\ast B,C)\times \cat{C}(X\ast A,B).$};
      \node (TT) at (0,4)  
      {$\cat{C}(Y\ast (X\ast A),Y\ast (X\ast A)) \times 
      \cat{C}(X\ast A,X\ast A)$};
      \node (T) at (0,2)  
      {$\cat{C}(Y\ast (X\ast A),C)\times \cat{C}(X\ast A,X\ast A)$};
      \node (L) at (0,0)  {$\cat{C}(Y\ast B,C)\times \cat{C}(X\ast A,X\ast A)$};
      \draw[<-] (T) to node[auto,swap,labelsize]
      {$\cat{C}(Y\ast f,\id{})\times \cat{C}(\id{},\id{})$} 
      (L);
      \draw[->] (TT) to 
      node[auto,swap,labelsize]{$\cat{C}(\id{},g\circ(Y\ast f))\times 
      \cat{C}(\id{},\id{})$} 
      (T);
      \draw[<-] (T1) to node[auto,labelsize]
      {$(\overline{\phi}_{X,Y})_{A,B,C}$} 
      (L1);
      \draw[->] (TT1) to 
      node[auto,labelsize]{$\cat{C}(\id{},g\circ(Y\ast f))$} (T1);
      \draw[->] (TT) to node[auto,labelsize]
      {$(\overline{\phi}_{X,Y})_{A,X\ast A,Y\ast(X\ast A)}$} 
      (TT1);
      \draw[->] (T) to node[auto,labelsize]
      {$(\overline{\phi}_{X,Y})_{A,X\ast A,C}$} 
      (T1);
      \draw[->] (L) to node[auto,labelsize]
      {$\cat{C}(\id{},\id{})\times \cat{C}(\id{},f)$} 
      (L1);
\end{tikzpicture}
\]
The top square commutes by naturality in $C$ of 
$((\overline{\phi}_{X,Y})_{A,B,C})$ and the bottom square commutes by 
(extra) naturality in $B$ of it.
By chasing the appropriate elements as follows
\[
\begin{tikzpicture}[baseline=-\the\dimexpr\fontdimen22\textfont2\relax ]
      \node (TT1) at (6.5,3)  
      {$\delta_{X,Y,A}$};
      \node (T1) at (6.5,1.5)  
      {$g\circ(Y\ast f)\circ \delta_{X,Y,A}$};
      \node (L1) at (6.5,0)  
      {$(g,f),$};
      \node (TT) at (0,3)  
      {$(\id{Y\ast (X\ast A)},\id{X\ast A})$};
      \node (T) at (0,1.5)  
      {$(g\circ (Y\ast f), \id{X\ast A})$};
      \node (L) at (0,0)  {$(g,\id{X\ast A})$};
      \draw[<-|] (T) to node[auto,swap,labelsize]
      {$\cat{C}(Y\ast f,\id{})\times \cat{C}(\id{},\id{})$} 
      (L);
      \draw[|->] (TT) to 
      node[auto,swap,labelsize]{$\cat{C}(\id{},g\circ(Y\ast f))\times 
      \cat{C}(\id{},\id{})$} 
      (T);
      \draw[<-|] (T1) to node[auto,labelsize]
      {$(\overline{\phi}_{X,Y})_{A,B,C}$} 
      (L1);
      \draw[|->] (TT1) to 
      node[auto,labelsize]{$\cat{C}(\id{},g\circ(Y\ast f))$} (T1);
      \draw[|->] (TT) to node[auto,labelsize]
      {$(\overline{\phi}_{X,Y})_{A,X\ast A,Y\ast(X\ast A)}$} 
      (TT1);
      \draw[|->] (T) to node[auto,labelsize]
      {$(\overline{\phi}_{X,Y})_{A,X\ast A,C}$} 
      (T1);
      \draw[|->] (L) to node[auto,labelsize]
      {$\cat{C}(\id{},\id{})\times \cat{C}(\id{},f)$} 
      (L1);
\end{tikzpicture}
\]
we conclude that $(\overline{\phi}_{X,Y})_{A,B,C}(g,f)=g\circ(Y\ast 
f)\circ\delta_{X,Y,A}$, as desired.
\end{proof}

Recall the 2-categories $\Enrichr{\cat{M}}$ and $\olActl{\cat{M}}$
defined in Definition~\ref{def:enrichr_olactl}.

\begin{corollary}
\label{cor:enrichr_olactl_equiv}
Let $\cat{M}$ be a metatheory.
\begin{enumerate}
\item The 2-functors in Proposition~\ref{prop:enrichment_as_metamodel}
and Proposition~\ref{prop:oplax_action_as_metamodel} restrict to 
fully faithful 2-functors 
\[
\Enrichr{\cat{M}}\longrightarrow\MtMod{\cat{M}}\qquad
\olActl{\cat{M}}\longrightarrow\MtMod{\cat{M}}
\]
with the same essential image characterised as follows:
a metamodel $(\cat{C},\Phi,\overline{\phi}_\cdot,\overline{\phi})$ 
of $\cat{M}$ is in the essential image if and only if
for each $X\in\cat{M}$ and $A,B\in\cat{C}$,
the functors 
\[
\Phi_{(-)}(A,B)\colon\cat{M}^\op\longrightarrow\SET\qquad
\Phi_X(A,-)\colon\cat{C}\longrightarrow\SET
\] 
are representable.
\item The 2-categories $\Enrichr{\cat{M}}$ and $\olActl{\cat{M}}$ are 
equivalent.
\end{enumerate}
\end{corollary}
\begin{proof}
The first clause is immediate from the definition of adjunction.
For instance, an enrichment $(\cat{C},\enrich{-}{-},j,M)$ over $\cat{M}$
is in $\Enrichr{\cat{M}}$ if and only if
for each $A\in\cat{C}$, $\enrich{A}{-}$ is a right adjoint,
which in turn is the case if and only if
for each $X\in\cat{M}$ and $A\in\cat{C}$, the functor 
\[
\cat{M}(X,\enrich{A}{-})\colon\cat{C}\longrightarrow\SET
\]
is representable.

The second clause is a direct consequence of the first.
\end{proof}

The reader might have noticed that 
there is another representability condition not covered by 
Propositions~\ref{prop:enrichment_as_metamodel} and 
\ref{prop:oplax_action_as_metamodel}, namely
metamodels $(\cat{C},\Phi,\overline{\phi}_\cdot,\overline{\phi})$
such that for each $X\in\cat{M}$ and $B\in\cat{C}$,
the functor 
\[
\Phi_X(-,B)\colon \cat{C}^\op\longrightarrow\SET
\] 
is representable. They correspond to 
{right lax actions of $\cat{M}^\op$ on $\cat{C}$},
or equivalently, to right oplax actions of $\cat{M}$ on $\cat{C}^\op$.

Extending the definition of enrichment (Definition~\ref{def:enrichment})
and the 2-category of enrichments (Definition~\ref{def:2-cat_of_enrichments})
to huge monoidal categories, we obtain the following.
\begin{proposition}
\label{prop:metatheory_as_enrichment}
Let $\cat{M}$ be a metatheory
and $\widehat{\cat{M}}=([\cat{M}^\op,\SET],\widehat{I},\widehat{\otimes})$
(see Definition~\ref{def:Day_conv}).
The 2-categories $\MtMod{\cat{M}}$ and $\Enrich{\widehat{\cat{M}}}$
are canonically isomorphic.
\end{proposition}

\subsubsection{$\Mod{-}{-}$ as a 2-functor}
Let $\cat{M}$ be a metatheory.
Similarly to the cases of enrichments and oplax actions,
we can view the $\Mod{-}{-}$ construction as a 2-functor
using the 2-category $\MtMod{\cat{M}}$.
In fact, via the inclusion
\begin{equation}
\label{eqn:theory_as_metamodel}
\Th{\cat{M}}=\Mon{\cat{M}}\longrightarrow
\Enrich{\cat{M}}\longrightarrow\MtMod{\cat{M}},
\end{equation}
the 2-functor $\Mod{-}{-}$ is simply given by the following composite:
\[
\begin{tikzpicture}[baseline=-\the\dimexpr\fontdimen22\textfont2\relax ]
      \node (1) at (0,0)  {$\Th{\cat{M}}^\op\times \MtMod{\cat{M}}$};
      \node (2) at (0,-1.5)  {$\MtMod{\cat{M}}^\op\times \MtMod{\cat{M}}$};
      \node (3) at (0,-3) {$\tCAT$,};
      \draw[->] (1) to node[auto,labelsize]{inclusion} (2);
      \draw[->] (2) to node[auto,labelsize] {$\MtMod{\cat{M}}(-,-)$} (3);
\end{tikzpicture}
\]
where $\MtMod{\cat{M}}(-,-)$ is the hom-2-functor for 
the locally large $\MtMod{\cat{M}}$.
The inclusion (\ref{eqn:theory_as_metamodel}) identifies
a theory $\monoid{T}=(T,m,e)$ in $\cat{M}$ with the 
metamodel $(\Phi^{(\monoid{T})}, 
\overline{\phi^{(\monoid{T})}}_\cdot,\overline{\phi^{(\monoid{T})}})$
of $\cat{M}$ in
the terminal category $1$ (whose unique object we denote by $\ast$),
defined as follows:
\begin{itemize}
\item the functor $\Phi^{(\monoid{T})}\colon\cat{M}^\op\times 1^\op\times 1
\longrightarrow\SET$ maps $(X,\ast,\ast)$ to $\cat{M}(X,T)$;
\item the function $(\overline{\phi^{(\monoid{T})}}_\cdot)_\ast\colon 
1\longrightarrow \cat{M}(I,T)$ maps the unique element of $1$ to $e$;
\item for each $X,Y\in\cat{M}$, the function 
$(\overline{\phi^{(\monoid{T})}}_{X,Y})_{\ast,\ast,\ast}\colon
\cat{M}(Y,T)\times \cat{M}(X,T)\longrightarrow\cat{M}(Y\otimes X,T)$
maps $(g,f)$ to $m\circ (g\otimes f)$.
\end{itemize}

\begin{comment}
\begin{proposition}
Let $\cat{M}$ be a metatheory.
The construction $\Mod{-}{-}$ canonically extends to 
a 2-functor from $\Th{\cat{M}}^\op\times \MtMod{\cat{M}}$ to $\tCAT$.
\end{proposition}

\begin{definition}
Let $\cat{M}=(\cat{M},I,\otimes)$ be a metatheory
and $\cat{C}$ be a category.
Define the category $\MtModfib{\cat{C}}{\cat{M}}$ of 
metamodels of $\cat{M}$ in $\cat{C}$ as 
the category of lax monoidal functors from 
$\cat{M}^\op$ to $[\cat{C}^\op\times\cat{C},\SET]$ and 
monoidal natural transformations. That is:
\begin{itemize}
\item An object is a metamodel $(\Phi,\phi_\cdot,\phi)$ of $\cat{M}$
in $\cat{C}$.
\item A morphism from $(\Phi,\phi_\cdot,\phi)$ to 
$(\Phi',\phi'_\cdot,\phi')$
is given by, a family of natural transformations 
$g_X\colon \Phi X\longrightarrow\Phi X'$
for each $X\in\cat{M}$,
such that for each $f\colon X\longrightarrow X'$ in $\cat{M}$,
$g_{X'} \circ \Phi f = \Phi' f\circ g_X$,
$g_I\circ \phi_\cdot= \phi'_\cdot$,
and for each $X,Y\in\cat{M}$,
$\phi'_{X,Y}\circ g_{X\otimes Y} = (g_X\ptensor g_Y)\circ\phi_{X,Y}$.
\end{itemize}
\end{definition}
\end{comment}

\subsection{Morphisms of metatheories}
\label{subsec:morphism_metatheory}
In this section, we introduce a notion of morphism between metatheories.
The main purpose of morphisms of metatheories is to provide a 
uniform method to compare different notions of algebraic theory.
A paradigmatic case of such a comparison
is given in Section~\ref{subsec:enrichment},
where we compare clones, symmetric operads and non-symmetric operads. 
Recall that the crucial observation used there was the fact that 
the $\Enrich{-}$ construction extends to a 2-functor
\begin{equation}
\label{eqn:Enrich_as_2-functor}
\Enrich{-}\colon \MonCAT\longrightarrow\twoCAT.
\end{equation}
Therefore, we want to define morphisms of metatheories with 
respect to which $\MtMod{-}$ behaves (2-)functorially.

On the other hand, recall from Section~\ref{subsec:oplax_action}
that $\olAct{-}$ is a 2-functor of type
\begin{equation}
\label{eqn:olAct_as_2-functor}
\olAct{-}\colon(\MonCATol)^\coop\longrightarrow\twoCAT.
\end{equation}
Since metamodels unify both enrichments and oplax actions, 
we would like to explain both (\ref{eqn:Enrich_as_2-functor})
and (\ref{eqn:olAct_as_2-functor}) by introducing a
sufficiently general notion of morphism of metatheories.

The requirement to unify both $\MonCAT$ and $(\MonCATol^\coop)$
suggests the possibility of using a suitable variant of 
profunctors (Definition~\ref{def:profunctor}),
leading to the following definition.

\begin{definition}
\label{def:morphism_of_metatheories}
Let $\cat{M}=(\cat{M},I_\cat{M},\otimes_\cat{M})$ and 
$\cat{N}=(\cat{N},I_\cat{N},\otimes_\cat{N})$ be metatheories. 
A \defemph{morphism of metatheories from $\cat{M}$ to $\cat{N}$} 
is a lax monoidal functor 
\[
H=(H,h_\cdot,h)\colon\cat{N}^\op\times\cat{M}\longrightarrow\SET.
\]
More precisely, such a morphism consists of:
\begin{itemize}
\item a functor $H\colon\cat{N}^\op\times\cat{M}\longrightarrow\SET$;
\item a function $h_\cdot\colon1\longrightarrow H(I_\cat{N},I_\cat{M})$;
\item a natural transformation $(h_{N,N',M,M'}\colon H(N',M')\times 
H(N,M)\longrightarrow H(N'\otimes_\cat{N}
N,M'\otimes_\cat{M} M))_{N,N'\in\cat{N},M,M'\in\cat{M}}$
\end{itemize}
making the following diagrams commute for each $N,N',N''\in\cat{N}$
and $M,M',M''\in\cat{M}$ (we omit subscripts on $\otimes$):
\begin{equation*}
\begin{tikzpicture}[baseline=-\the\dimexpr\fontdimen22\textfont2\relax ]
      \node (TL) at (0,1)  {$1\times H(N,M)$};
      \node (TR) at (6,1)  {$H(I_\cat{N},I_\cat{M})\times H(N,M)$};
      \node (BL) at (0,-1) {$H(N,M)$};
      \node (BR) at (6,-1) 
      {$H(I_\cat{N}\otimes N,I_\cat{M}\otimes M)$};
      \draw[->] (TL) to node[auto,labelsize](T) 
      {$h_\cdot\times H(N,M)$} 
      (TR);
      \draw[->]  (TR) to node[auto,labelsize] 
      {$h_{N,I_\cat{N},M,I_\cat{M}}$} (BR);
      \draw[->]  (TL) to node[auto,swap,labelsize] {$\cong$} 
      (BL);
      \draw[->]  (BL) to node[auto,labelsize] {$\cong$} (BR);
\end{tikzpicture} 
\end{equation*}
\begin{equation*}
\begin{tikzpicture}[baseline=-\the\dimexpr\fontdimen22\textfont2\relax ]
      \node (TL) at (0,1)  {$H(N,M)\times 1$};
      \node (TR) at (6,1)  {$H(N,M)\times H(I_\cat{N},I_\cat{M})$};
      \node (BL) at (0,-1) {$H(N,M)$};
      \node (BR) at (6,-1) {$H(N\otimes I_\cat{N},M\otimes, 
      I_\cat{M})$};
      \draw[->] (TL) to node[auto,labelsize](T) 
      {$H(N,M)\times h_\cdot$} 
      (TR);
      \draw[->]  (TR) to node[auto,labelsize] 
      {$h_{I_\cat{N},N,I_\cat{M},M}$} (BR);
      \draw[->]  (TL) to node[auto,swap,labelsize] {$\cong$} 
      (BL);
      \draw[->]  (BL) to node[auto,labelsize] {$\cong$} (BR);
\end{tikzpicture}
\end{equation*}
\begin{equation*}
\begin{tikzpicture}[baseline=-\the\dimexpr\fontdimen22\textfont2\relax ]
      \node (TL) at (0,2)  {$\big(H(N'',M'')\times H(N',M')\big)\times 
            H(N,M)$};
      \node (TM) at (8,2)  {$H(N''\otimes N',M''\otimes 
      M')\times H(N,M)$};
      \node (TR) at (8,0)  {$H((N''\otimes N')\otimes  N, 
      (M''\otimes  M')\otimes  M)$};
      \node (BL) at (0,0) {$H(N'',M'')\times \big(H(N',M')\times 
              H(N,M)\big)$};
      \node (BM) at (0,-2) {$H(N'',M'')\times 
      H(N'\otimes N,M'\otimes M)$};
      \node (BR) at (8,-2) {$H(N''\otimes(N'\otimes N),M''\otimes (M'\otimes 
      M)).$};
      \draw[->] (TL) to node[auto,labelsize](T) 
      {$h_{N',N'',M',M''}\times H(N,M)$} 
      (TM);
      \draw[->] (TM) to node[auto,labelsize](T) 
      {$h_{N,N''\otimes N',M,M''\otimes M'}$} 
      (TR);
      \draw[->]  (TR) to node[auto,labelsize] {$\cong$} (BR);
      \draw[->]  (TL) to node[auto,swap,labelsize] {$\cong$} (BL);
      \draw[->]  (BL) to node[auto,swap,labelsize] {$H(N'',M'')\times 
      h_{N,N',M,M'}$} 
      (BM);
      \draw[->] (BM) to node[auto,labelsize](B) {$h_{N'\otimes N,N'',M'\otimes 
      M,M''}$} (BR);
\end{tikzpicture} 
\end{equation*}

We write $H\colon\cat{M}\pto\cat{N}$ if $H$ is a morphism of
metatheories from $\cat{M}$ to $\cat{N}$.
\end{definition}

Morphisms of metatheories are a monoidal version of profunctors,
and indeed they are called \emph{monoidal profunctors} in \cite{Im_Kelly}.
We may identify a morphism $H\colon\cat{M}\pto\cat{N}$
with a lax monoidal functor 
\[
(M\longmapsto H(-,M))\colon\cat{M}\longrightarrow[\cat{N}^\op,\SET],
\]
or equivalently with an oplax monoidal functor 
\[
(N\longmapsto H(N,-))\colon\cat{N}\longrightarrow[\cat{M},\SET]^\op,
\]
where in both cases the codomain is equipped with the convolution monoidal 
structure.

%The reader might have noticed that the definition of morphism of metatheories
%resembles that of metamodel (Definition~\ref{def:metamodel}).
%Indeed, using the notion of \emph{promonoidal category} we may 
%find a common generalisation of these, explaining conceptually 
%why morphisms of metatheories act on metamodels;
%see Appendix~\ref{apx:promon}.

\begin{definition}
We define the bicategory $\MtTH$ of metatheories as follows.
\begin{itemize}
\item An object is a metatheory.
\item A 1-cell from $\cat{M}$ to $\cat{N}$
is a morphism of metatheories $\cat{M}\pto\cat{N}$.
The identity 1-cell on a metatheory $\cat{M}$ is the hom-functor $\cat{M}(-,-)$,
equipped with the evident structure for a morphism of metatheories.
Given morphisms of metatheories 
$(H,h_\cdot,h)\colon \cat{M}\pto\cat{N}$ and 
$(K,k_\cdot,k)\colon \cat{N}\pto\cat{L}$,
their composite is $(K\ptensor H,k_\cdot\ptensor h_\cdot, k \ptensor h)
\colon \cat{M}\pto\cat{L}$ where $K\ptensor H$ is the composition of the 
profunctors $H$ and $K$ (Definition~\ref{def:profunctor}), and 
$k_\cdot \ptensor h_\cdot$ and $k\ptensor h$ are the evident
natural transformations.
\item A 2-cell from $H$ to $H'$, both from $\cat{M}$ to $\cat{N}$,
is a monoidal natural transformation
$\alpha\colon H\Longrightarrow 
H'\colon\cat{N}^\op\times\cat{M}\longrightarrow\SET$. \qedhere
\end{itemize}
\end{definition}

Similarly to the case of profunctors,
we have identity-on-objects fully faithful pseudofunctors 
\[
(-)_\ast\colon\MonCAT\longrightarrow\MtTH
\]
and 
\[
(-)^\ast\colon (\MonCATol)^\coop\longrightarrow\MtTH.
\]
In detail, a lax monoidal functor 
\[
F=(F,f_\cdot,f)\colon \cat{M}\longrightarrow\cat{N}
\]
gives rise to a morphism of metatheories 
\[
F_\ast=(F_\ast,(f_\ast)_\cdot,f_\ast)\colon\cat{M}\pto\cat{N}
\]
with $F_\ast(N,M)=\cat{N}(N,FM)$, $(f_\ast)_\cdot\colon 1\longrightarrow 
\cat{N}(I_\cat{N},FI_\cat{M})$ mapping the unique element of $1$ to 
$f_\cdot\colon I_\cat{N}\longrightarrow FI_\cat{M}$,
and $(f_\ast)_{N,N',M,M'}\colon 
\cat{N}(N',FM')\times\cat{N}(N,FM)\longrightarrow\cat{N}(N'\otimes_\cat{N} N,
F(M'\otimes_\cat{M} M))$
mapping $g'\colon N'\longrightarrow FM'$ and $g\colon N\longrightarrow FM$
to $f_{M,M'}\circ (g'\otimes g)\colon N'\otimes_\cat{N} N\longrightarrow
F(M'\otimes_\cat{M} M)$.
Given an oplax monoidal functor 
\[
F=(F,f_\cdot,f)\colon \cat{M}\longrightarrow\cat{N},
\]
we obtain a morphism of metatheories
\[
F^\ast=(F^\ast, (f^\ast)_\cdot,f^\ast)\colon\cat{N}\pto\cat{M}
\]
analogously.

In particular, a \emph{strong} monoidal functor
\[
F\colon\cat{M}\longrightarrow\cat{N}
\]
gives rise to both $F_\ast\colon \cat{M}\pto\cat{N}$
and $F^\ast\colon\cat{N}\pto\cat{M}$, and 
it is straightforward to see that these form an adjunction $F_\ast\dashv 
F^\ast$ in $\MtTH$.
\medskip

A morphism of metatheories $H\colon\cat{M}\pto\cat{N}$
induces a 2-functor 
\[
\MtMod{H}\colon\MtMod{\cat{M}}\longrightarrow\MtMod{\cat{N}}.
\]
Its action on objects is as follows.
\begin{definition}
\label{def:action_of_morphism_on_metamodel}
Let $\cat{M}=(\cat{M},I_\cat{M},\otimes_\cat{M})$ and 
$\cat{N}=(\cat{N},I_\cat{N},\otimes_\cat{N})$ be metatheories,
$H=(H,h_\cdot,h)\colon\cat{M}\pto\cat{N}$ a morphism of metatheories, 
$\cat{C}$ a large category and 
$\Phi=(\Phi,\overline{\phi}_\cdot,\overline{\phi})$ a metamodel of  $\cat{M}$ 
in $\cat{C}$.
We define the metamodel 
$H(\Phi)=(\Phi',\overline{\phi'}_\cdot,\overline{\phi'})$ 
of $\cat{N}$ on $\cat{C}$ as follows:
\begin{itemize}
\item The functor 
$\Phi'\colon\cat{N}^\op\times\cat{C}^\op\times\cat{C}\longrightarrow\SET$
maps $(N,A,B)\in\cat{N}^\op\times\cat{C}^\op\times\cat{C}$ to the set 
\begin{equation}
\label{eqn:action_of_H_on_metamodels}
\Phi'_N(A,B)=
\int^{M\in\cat{M}}H(N,M)\times \Phi_M(A,B).
\end{equation}
\item The natural transformation $((\overline{\phi'}_\cdot)_C\colon 
1\longrightarrow \Phi'_{I_\cat{N}}(C,C))_{C\in\cat{C}}$
is defined by
mapping the unique element $\ast$ of $1$ to
\begin{multline*}
[I_\cat{M}\in\cat{M}, h_\cdot(\ast)\in H(I_\cat{N},I_\cat{M}), 
(\overline{\phi}_\cdot)_C(\ast)\in
\Phi_{I_\cat{M}}(C,C)]\\
\in\int^{M\in\cat{M}}H(I_\cat{N},M)\times \Phi_M(C,C)
\end{multline*}
for each $C\in\cat{C}$. 
\item The natural transformation 
\[((\overline{\phi'}_{N,N'})_{A,B,C}\colon
\Phi'_{N'}(B,C)\times\Phi'_N(A,B)\longrightarrow
\Phi'_{N'\otimes_\cat{N} N}(A,C))_{N,N'\in\cat{N},A,B,C\in\cat{C}}\]
is defined by mapping a pair consisting of
$[M',x',y']\in\Phi'_{N'}(B,C)$ and $[M,x,y]\in\Phi'_N(A,B)$
to 
\begin{equation*}
[M'\otimes_\cat{M}M, h_{N,N',M,M'}(x',x), 
(\overline{\phi}_{M,M'})_{A,B,C}(y',y)]
\end{equation*}
for each $N,N'\in\cat{N}$ and $A,B,C\in\cat{C}$.\qedhere
\end{itemize}
\end{definition}

The above construction extends routinely, giving rise to a pseudofunctor
\[
\MtMod{-}\colon \MtTH\longrightarrow\twoCAT.
\]

\section{Comparing different notions of algebraic theory}
\label{sec:comparing}
In this section, we shall demonstrate how we can compare 
different notions of algebraic theory via morphisms of metatheories.

We start with a few remarks on simplification of the action
(Definition~\ref{def:action_of_morphism_on_metamodel}) of a morphism of 
metatheories
on metamodels, in certain special cases.
Let $\cat{M}$ and $\cat{N}$ be metatheories, 
\[
H\colon\cat{M}\pto\cat{N}
\]
be a morphism of metatheories,
$\cat{C}$ be a large category and 
$\Phi=(\Phi,\overline{\phi}_\cdot,\overline{\phi})$ be a metamodel of $\cat{M}$
in $\cat{C}$. 

First consider the case where for each $A,B\in\cat{C}$,
the functor $\Phi_{(-)}(A,B)\colon \cat{M}^\op\longrightarrow{\SET}$
is representable. This means that $\Phi$ is in fact (up to an isomorphism) an 
enrichment $\enrich{-}{-}$; see Proposition~\ref{prop:enrichment_as_metamodel}.
In this case, $\Phi_M(A,B)$ may be written as $\cat{M}(M,\enrich{A}{B})$
and hence the formula (\ref{eqn:action_of_H_on_metamodels})
simplifies:
\[
\Phi'_N(A,B)=
\int^{M\in\cat{M}}H(N,M)\times \cat{M}(M,\enrich{A}{B})
\cong H(N,\enrich{A}{B}).
\]
In particular, if moreover $H$ is of the form
\[
F_\ast\colon \cat{M}\pto\cat{N}
\]
for some lax monoidal functor $F\colon\cat{M}\longrightarrow\cat{N}$,
then we have 
\[
\Phi'_N(A,B)\cong \cat{N}(N,F\enrich{A}{B}),
\]
implying that $H(\Phi)=F_\ast(\Phi)$ is again isomorphic to
an enrichment; indeed, this case reduces to $F_\ast(\enrich{-}{-})$
defined in Definition~\ref{def:action_of_lax_on_enrichment}.
Note that, as a special case, for any theory $\monoid{T}$ in $\cat{M}$
(recall that such a theory is identified with a metamodel of $\cat{M}$
in the terminal category $1$),
$F_\ast(\monoid{T})$ is again isomorphic to a theory in $\cat{N}$.
The 2-functor 
\[
\MtMod{F_\ast}\colon\MtMod{\cat{M}}\longrightarrow\MtMod{\cat{N}}
\]
extends the functor 
\[
\Th{F}\colon\Th{\cat{M}}\longrightarrow\Th{\cat{N}}
\]
between the categories of theories induced by $F$, using the well-known fact 
that a lax monoidal functor preserves theories (= monoid objects).

Next consider the case where $H$ is of the form
\[
G^\ast\colon \cat{M}\pto\cat{N}
\]
for some oplax monoidal functor $G\colon\cat{N}\longrightarrow\cat{M}$.
In this case $H(N,M)=\cat{M}(GN,M)$ and the formula 
(\ref{eqn:action_of_H_on_metamodels}) simplifies as follows:
\[
\Phi'_N(A,B)=
\int^{M\in\cat{M}}\cat{M}(GN,M)\times \Phi_M(A,B)
\cong \Phi_{GN}(A,B).
\]
Of course this construction reduces to $G^\ast(\ast)$ defined in 
Definition~\ref{def:action_of_oplax_on_oplax_action} for 
a metamodel induced from an oplax action.

Combining the above observations, suppose now that we have a strong monoidal 
functor
\[
F\colon\cat{M}\longrightarrow\cat{N}
\]
between metatheories $\cat{M}$ and $\cat{N}$.
On the one hand, $F$ induces a functor
\[
\Th{F}\colon\Th{\cat{M}}\longrightarrow\Th{\cat{N}}
\]
between the categories of theories,
which is a restriction of the 2-functor $\MtMod{F_\ast}$.
On the other hand, $F$ induces a 2-functor
\[
\MtMod{F^\ast}\colon \MtMod{\cat{N}}\longrightarrow\MtMod{\cat{M}}
\]
between the 2-categories of metamodels.
The 2-adjointness $\MtMod{F_\ast}\dashv \MtMod{F^\ast}$
yields, for each theory $\monoid{T}$ in $\cat{M}$
and each metamodel $(\cat{C},\Phi)$ of $\cat{N}$,
an isomorphism of categories
\[
\Mod{F_\ast(\monoid{T})}{(\cat{C},\Phi)}\cong
\Mod{\monoid{T}}{(\cat{C},F^\ast(\Phi))}.
\]
Observe that $F_\ast(\monoid{T})=\Th{F}(\monoid{T})$ is the standard action of 
a strong monoidal 
functor on a theory, and $F^\ast(\Phi)$ is, in essence,
simply precomposition by $F$.

\medskip
Now we apply the above argument to some concrete cases.
\begin{example}
Recall from Section~\ref{subsec:enrichment}, where we have compared clones, 
symmetric operads and non-symmetric operads, that there is a chain of lax 
monoidal functors
\[
\begin{tikzpicture}[baseline=-\the\dimexpr\fontdimen22\textfont2\relax ]
      \node (L) at (0,0)  {$[\Ncat,\Set]$};
      \node (M) at (3.5,0)  {$[\Pcat,\Set]$};
      \node (R) at (7,0)  {$[\F,\Set]$.};
      \draw[->]  (L) to node[auto,labelsize] 
      {$\Lan_J$} (M);
      \draw[->]  (M) to node[auto,labelsize] 
      {$\Lan_{J'}$} (R);
\end{tikzpicture} 
\] 
These functors, being left adjoints in $\MonCAT$, are in fact strong monoidal 
\cite{Kelly:doctrinal}.
Theories are mapped as follows, as noted in Section~\ref{subsec:enrichment}:
\[
\begin{tikzpicture}[baseline=-\the\dimexpr\fontdimen22\textfont2\relax ]
      \node (L) at (0,0)  {$\Th{[\Ncat,\Set]}$};
      \node (M) at (4.5,0)  {$\Th{[\Pcat,\Set]}$};
      \node (R) at (9,0)  {$\Th{[\F,\Set]}.$};
      \draw[->]  (L) to node[auto,labelsize] {$\Th{\Lan_J}$} (M);
      \draw[->]  (M) to node[auto,labelsize] {$\Th{\Lan_{J'}}$} (R);
\end{tikzpicture} 
\]
In this case, the suitable 2-functors between 2-categories of metamodels
can be given either as $\MtMod{(\Lan_J)^\ast}$ or $\MtMod{([J,\Set])_\ast}$ 
(and similarly for $J'$),
because $(\Lan_J)^\ast\cong ([J,\Set])_\ast$ in $\MtTH$.
\end{example}
\begin{example}
Let us consider the relationship between clones and monads on $\Set$.
The inclusion functor
\[
J''\colon \F\longrightarrow\Set
\]
induces a functor 
\[
\Lan_{J''}\colon[\F,\Set]\longrightarrow[\Set,\Set],
\]
which naturally acquires the structure of a strong monoidal functor.
The essential image of this functor is precisely the \emph{finitary}
endofunctors on $\Set$, i.e., those endofunctors preserving filtered colimits.
The functor $\Th{\Lan_{J''}}$ maps a clone to a finitary monad on $\Set$,
in accordance with 
the well-known correspondence between clones (= Lawvere theories) and finitary 
monads on $\Set$ \cite{Linton_equational}.
Between the 2-categories of metamodels, we have a 2-functor
\[
\MtMod{(\Lan_{J''})^\ast}\colon\MtMod{[\Set,\Set]}\longrightarrow
\MtMod{[\F,\Set]}.
\]
The standard metamodel of $[\Set,\Set]$ in $\Set$ (corresponding to the 
definition of Eilenberg--Moore algebras) is given by the strict action
described in Example~\ref{ex:monad_standard_action};
in particular, its functor part $\Phi\colon 
[\Set,\Set]^\op\times\Set^\op\times\Set\longrightarrow\SET$ maps $(F,A,B)$ to 
$\Set(FA,B)$. 
The metamodel $(\Lan_{J''})^\ast(\Phi)$ of $[\F,\Set]$ in $\Set$ has the 
functor part $(\Lan_{J''})^\ast(\Phi)\colon[\F,\Set]^\op\times\Set^\op\times
\Set\longrightarrow\SET$ mapping $(X,A,B)$ to 
\begin{align*}
\Set((\Lan_{J''}X)A,B)
&= \Set\left(\int^{[n]\in\F} A^n\times X_n,B\right)\\
&\cong \int_{[n]\in\F} \Set (X_n, \Set(A^n,B))\\
&\cong [\F,\Set] (X, \enrich{A}{B}),
\end{align*}
where $\enrich{A}{B}\in[\F,\Set]$ in the final line is the one in 
Example~\ref{ex:clone_enrichment}.
Hence $\MtMod{(\Lan_{J''})^\ast}$ preserves the standard metamodels 
and this way we restore the well-known observation 
that the classical correspondence 
of clones and finitary monads on $\Set$ preserves semantics.

Note that by combining the previous example we obtain the chain
\[
\begin{tikzpicture}[baseline=-\the\dimexpr\fontdimen22\textfont2\relax ]
      \node (L) at (0,0)  {$[\Ncat,\Set]$};
      \node (M) at (2.5,0)  {$[\Pcat,\Set]$};
      \node (R) at (5,0)  {$[\F,\Set]$};
      \node (RR) at (7.5,0) {$[\Set,\Set]$};
      \draw[->]  (L) to node[auto,labelsize] 
      {$\Lan_J$} (M);
      \draw[->]  (M) to node[auto,labelsize] 
      {$\Lan_{J'}$} (R);
      \draw[->]  (R) to node[auto,labelsize] 
      {$\Lan_{J''}$} (RR);
\end{tikzpicture} 
\] 
of strong monoidal functors, connecting non-symmetric and symmetric operads 
with monads on $\Set$.
\end{example}

\begin{example}
Let $\cat{M}$ be a metatheory, $\cat{C}$ a large category,
and $\ast$ a \emph{pseudo} action of $\cat{M}$ on $\cat{C}$.
We obtain a strong monoidal functor 
\[
F\colon\cat{M}\longrightarrow[\cat{C},\cat{C}]
\]
(where $[\cat{C},\cat{C}]$ is equipped with the composition monoidal structure)
as the transpose of $\ast\colon\cat{M}\times\cat{C}\longrightarrow\cat{C}$.
The functor $\Th{F}$ maps any theory $\monoid{T}=(T,e,m)$ in $\cat{M}$
to the monad $F(\monoid{T})=(T\ast(-), e\ast(-),m\ast(-))$ on $\cat{C}$.
The 2-functor 
$\MtMod{F^\ast}\colon\MtMod{[\cat{C},\cat{C}]}\longrightarrow\MtMod{\cat{M}}$
maps the standard metamodel $\Phi$ of $[\cat{C},\cat{C}]$ in $\cat{C}$
(Example~\ref{ex:monad_standard_action}) to the metamodel 
$F^\ast(\Phi)\colon 
\cat{M}^\op\times\cat{C}^\op\times\cat{C}\longrightarrow\SET$
mapping $(X,A,B)$ to
\[
\cat{C}((FX)A,B)=\cat{C}(X\ast A,B).
\]
Therefore it maps the standard metamodel $\Phi$ to the metamodel induced from 
$\ast$.

As a special case, for a large category $\cat{C}$ with finite limits and a 
cartesian monad $\monoid{S}$ on $\cat{C}$,
the standard metamodel for $\monoid{S}$-operads
(Example~\ref{ex:metamodel_S-operad}) may be related to 
the standard metamodel of monads on $\cat{C}$,
and models of an $\monoid{S}$-operad $\monoid{T}$ may alternatively be defined 
as Eilenberg--Moore algebras of the monad on $\cat{C}$ induced from $\monoid{T}$
(as noted in \cite{Leinster_book}).
\end{example}

We have introduced a notion of morphism between metatheories,
which is more general than both lax monoidal functors and oplax monoidal 
functors (in the opposite direction).
As we pointed out, an adjunction of morphisms between metatheories are rich 
enough to generate isomorphisms of categories of models.
Moreover, such adjunctions abound, as every strong monoidal functor generates 
one.

\section{Related work}
There are a few recent papers~\cite{Curien_operad,Hyland_elts,Avery_thesis}
which develop unified account of various notions of algebraic theory.

The papers by Curien \cite{Curien_operad} and Hyland \cite{Hyland_elts} 
concentrate on clones, symmetric operads and non-symmetric operads, and concern 
primarily the conceptual understanding of the substitution monoidal structures. 
Via the theory of pseudo-distributive laws \cite{Tanaka_Power_psd_dist},
they reduce substitution monoidal structures to certain 2-monads on $\tCat$,
for example the free cartesian category 2-monad in the case of clones.
Their work illuminates the relationship between the notions of algebraic 
theory they treat and their standard metamodels, because the standard 
metamodels arise as Eilenberg--Moore algebras of the 2-monad from which 
the corresponding substitution monoidal structure is induced. 
On the other hand, monads and generalised operads do not seem to be 
captured by their framework.

The framework by Avery~\cite{Avery_thesis} is relative to a well-behaved 
2-category (which he calls a \emph{setting}).
In the basic setting of $\tCAT$
he identifies 
algebraic theories with identity-on-objects functor from a certain category 
$\cat{A}$ of arities, calling them \emph{proto-theories}.
In this case, the relationship to our work may be established by the fact that 
(putting size issues aside)
identity-on-objects functors from $\cat{A}$ correspond to monoid objects in 
$[\cat{A}^\op\times\cat{A},\SET]$ (with the profunctor composition as the 
monoidal structure).
This way we may understand Avery's framework (with respect to the setting 
$\tCAT$)
within ours, although for general 
setting probably we cannot do so.
However we remark that for specific examples of settings treated in 
\cite{Avery_thesis}, it seems that
proto-theories therein can be identified with monoid objects in 
the category of a suitable variant of profunctors.

Avery's framework has an attractive feature that 
it can treat Lawvere theories, PROPs, PROs, symmetric and non-symmetric operads
 by choosing a suitable setting,
without requiring any complicated calculation
(cf.~the definition of substitution monoidal product and 
the relevant enrichments in Section~\ref{sec:framework_prelude}).
Generalised operads do not seem to be captured in Avery's 
framework.

Avery does not consider the questions of 
functoriality that arise at various levels. 
Note that, in contrast, we have defined morphisms of metamodels, of 
metatheories, and so on, which suitably act on the relevant constructions.

\chapter{Structure-semantics adjunctions}
\label{chap:str_sem}

\emph{Structure-semantics adjunctions} are a classical topic in categorical 
algebra.
%~\cite{Lawvere_thesis,Linton_equational,Linton_outline,Dubuc_Kan,
%Isbell_general_functorial_sem,Street_FTM,Avery_thesis}.
They are a family of 
adjunctions parametrised by a metatheory $\cat{M}$
and its metamodel $(\cat{C},\Phi,\overline{\phi}_\cdot,\overline{\phi})$;
if we fix these parameters, the structure-semantics adjunction is 
\emph{ideally} of type
\begin{equation}
\label{eqn:str_sem_idealised}
\begin{tikzpicture}[baseline=-\the\dimexpr\fontdimen22\textfont2\relax ]
      \node(11) at (0,0) 
      {$\Th{\cat{M}}^\op$};
      \node(22) at (4,0) {$\CAT/\cat{C}$,};
  
      \draw [->,transform canvas={yshift=5pt}]  (22) to node 
      [auto,swap,labelsize]{$\Str$} (11);
      \draw [->,transform canvas={yshift=-5pt}]  (11) to node 
      [auto,swap,labelsize]{$\Sem$} (22);
      \path (11) to node[midway](m){} (22); 

      \node at (m) [labelsize,rotate=90] {$\vdash$};
\end{tikzpicture}
\end{equation}
and the functor $\Sem$ is essentially $\Mod{-}{(\cat{C},\Phi)}$.
Various authors have constructed such adjunctions for a variety of 
notions of algebraic theory, most notably for 
clones~\cite{Lawvere_thesis,Linton_equational,
Isbell_general_functorial_sem} and monads~\cite{Dubuc_Kan,Street_FTM}.
There were also some attempts to unify these 
results~\cite{Linton_outline,Avery_thesis}.
See Section~\ref{sec:what_are_str_sem} for the ideas behind such adjunctions.

If we try to work this idea out, however, there turn out to be size-issues or 
other problems, and usually we cannot obtain an adjunction of type 
(\ref{eqn:str_sem_idealised});
we cannot find a suitable functor $\Str$ of that type.
To get an adjunction, various conditions on objects in $\CAT/\cat{C}$
were introduced in the literature in order to single out well-behaved
(usually called \emph{tractable}) objects, 
yielding a restricted version of (\ref{eqn:str_sem_idealised}):
\begin{equation}
\label{eqn:str_sem_tractable}
\begin{tikzpicture}[baseline=-\the\dimexpr\fontdimen22\textfont2\relax ]
      \node(11) at (0,0) 
      {$\Th{\cat{M}}^\op$};
      \node(22) at (4,0) {$(\CAT/\cat{C})_{\mathrm{tr}}$.};
  
      \draw [->,transform canvas={yshift=5pt}]  (22) to node 
      [auto,swap,labelsize]{$\Str$} (11);
      \draw [->,transform canvas={yshift=-5pt}]  (11) to node 
      [auto,swap,labelsize]{$\Sem$} (22);
      \path (11) to node[midway](m){} (22); 

      \node at (m) [labelsize,rotate=90] {$\vdash$};
\end{tikzpicture}
\end{equation}
Here, $(\CAT/\cat{C})_\mathrm{tr}$ is the full-subcategory of 
$\CAT/\cat{C}$ consisting of all {tractable} objects.

In this chapter, we construct a structure-semantics adjunction for 
an arbitrary metatheory and an arbitrary metamodel of it.
Of course, we cannot obtain an adjunction of type (\ref{eqn:str_sem_idealised}),
for the same reasons that have prevented  
other authors from doing so.
However, we shall obtain a modified adjunction by a strategy different from 
theirs (and similar to \cite{Linton_outline,Avery_thesis}):
instead of restricting $\CAT/\cat{C}$,
we {extend} $\Th{\cat{M}}$ to $\Th{\widehat{\cat{M}}}$\footnote{The monoidal 
category $\widehat{\cat{M}}$ is not a metatheory because it is not large.
Extending Definition~\ref{def:theory}, by $\Th{\widehat{\cat{M}}}$ we mean
the category of monoids in $\widehat{\cat{M}}$.}
(where $\widehat{\cat{M}}=[\cat{M}^\op,\SET]$ is equipped with
the convolution monoidal structure),
and obtain an extended version of (\ref{eqn:str_sem_idealised}):
\begin{equation}
\label{eqn:str_sem_extend}
\begin{tikzpicture}[baseline=-\the\dimexpr\fontdimen22\textfont2\relax ]
      \node(11) at (0,0) 
      {$\Th{\widehat{\cat{M}}}^\op$};
      \node(22) at (4,0) {$\CAT/\cat{C}$.};
  
      \draw [->,transform canvas={yshift=5pt}]  (22) to node 
      [auto,swap,labelsize]{$\Str$} (11);
      \draw [->,transform canvas={yshift=-5pt}]  (11) to node 
      [auto,swap,labelsize]{$\Sem$} (22);
      \path (11) to node[midway](m){} (22); 

      \node at (m) [labelsize,rotate=90] {$\vdash$};
\end{tikzpicture}
\end{equation}
We may then obtain known adjunctions of the form (\ref{eqn:str_sem_tractable}),
at least for clones and monads,
by suitably restricting (\ref{eqn:str_sem_extend}).

%This way we avoid the tricky problem of finding a suitable 
%tractability condition for each metatheory.
%In fact, once we establish the adjunction (\ref{eqn:str_sem_extend}),
%it can generate a quite general tractability condition as follows:
%we say that an object $V\colon\cat{A}\longrightarrow\cat{C}$ in 
%$\CAT/\cat{C}$ is \emph{tractable} if $\Str(V)\in\Th{\widehat{\cat{M}}}$
%is actually inside $\Th{\cat{M}}$, which in turn
%reduces to a representability condition in the category $\cat{M}$.

\section{The idea of structure-semantics adjunctions}
\label{sec:what_are_str_sem}
This section is an introduction to the idea of structure-semantics 
adjunctions.
We start with an informal explanation of a duality between 
\emph{sentences} and \emph{structures}~\cite{Lawvere_adjointness},
which may be seen as
a degenerate version of structure-semantics 
adjunctions.
Given any sentence $\phi$ of a suitable type and any structure $A$
of a suitable type, suppose we know whether  
$\phi$ holds in $A$ (written as $A\vDash \phi$) or not.  
Then, from a set of sentences $\Phi$
we may define a set $\mathrm{Mod}(\Phi)$ of structures,
whose elements are called \emph{models of $\Phi$}:
\[
\mathrm{Mod}(\Phi)=\{\,A\mid A\vDash \phi \text{ for all }\phi\in\Phi\,\}.
\]
Conversely, from a set of structures $\mathcal{A}$ we get 
a set $\mathrm{Thm}(\mathcal{A})$ of sentences, whose elements we call
\emph{theorems of $\mathcal{A}$}:
\[
\mathrm{Thm}(\mathcal{A})
=\{\,\phi\mid A\vDash \phi \text{ for all }A\in\mathcal{A}\,\}.
\]
It is straightforward to see that
$\mathrm{Mod}$ and $\mathrm{Thm}$ form a Galois connection:
for any set $\Phi$ of sentences and any set $\mathcal{A}$ of structures,
\[
\Phi\subseteq \mathrm{Thm}(\mathcal{A}) \iff
\mathcal{A}\subseteq \mathrm{Mod}(\Phi)
\]
holds.
%(observe that both sides coincide with the condition that 
%$A\vDash \phi$ for all $\phi\in\Phi$ and all $A\in\mathcal{A}$).
The setting of universal algebra (see Section~\ref{sec:univ_alg})
provides a concrete example.
For a fixed graded set (signature) $\Sigma$,
the notions of $\Sigma$-equation (Definition~\ref{def:Sigma_eq}) 
and $\Sigma$-algebra (Definition~\ref{def:Sigma_alg})
play the roles of sentence and structure respectively,
with the relation $\vDash$ defined as in 
Definition~\ref{def:semantical_consequence_rel}.
\begin{comment}
In this case, the types of $\mathrm{Mod}$ and $\mathrm{Thm}$ are
\begin{equation}
\label{eqn:syntax_semantics_duality_univ_alg}
\begin{tikzpicture}[baseline=-\the\dimexpr\fontdimen22\textfont2\relax ]
      \node(11) at (0,0) 
      {$(\mathcal{P}\{\Sigma\text{-equations}\})^\op$};
      \node(22) at (6,0) {$\mathcal{P}\{\Sigma\text{-algebras}\}$,};
  
      \draw [->,transform canvas={yshift=5pt}]  (22) to node 
      [auto,swap,labelsize]{$\mathrm{Thm}$} (11);
      \draw [->,transform canvas={yshift=-5pt}]  (11) to node 
      [auto,swap,labelsize]{$\mathrm{Mod}$} (22);

      \node at (3.2,0) [labelsize,rotate=90] {$\vdash$};
\end{tikzpicture}
\end{equation}
where $\mathcal{P}X$ denotes the poset of all subsets of $X$.
\end{comment}

In various fields in mathematics, it has been observed that 
behind classical Galois connections
there often hide more profound adjunctions~\cite{Lawvere_adjointness};
the structure-semantics adjunctions are what we may find 
behind the above duality between sentences and structures.
For example, the structure-semantics adjunctions for clones 
refine and unify the dualities for universal algebra 
for arbitrary graded sets $\Sigma$.
Given a \emph{small} category $\cat{C}$ with finite powers, 
the structure-semantics adjunction for clones with respect to $\cat{C}$
may be formulated as an adjunction
\begin{equation}
\label{eqn:str_sem_adj_clone}
\begin{tikzpicture}[baseline=-\the\dimexpr\fontdimen22\textfont2\relax ]
      \node(11) at (0,0) 
      {$\mathbf{Clo}^\op$};
      \node(22) at (4,0) {$\Cat/\cat{C}$,};
  
      \draw [->,transform canvas={yshift=5pt}]  (22) to node 
      [auto,swap,labelsize]{$\Str$} (11);
      \draw [->,transform canvas={yshift=-5pt}]  (11) to node 
      [auto,swap,labelsize]{$\Sem$} (22);
      \path (11) to node[midway](m){} (22); 

      \node at (m) [labelsize,rotate=90] {$\vdash$};
\end{tikzpicture}
\end{equation}
where $\mathbf{Clo}=\Th{[\F,\Set]}$ is the category of clones
and $\Cat/\cat{C}$ is a slice category.
We already know what the functor $\Sem$ does:
it maps a clone $\monoid{T}$ to the 
category $\Mod{\monoid{T}}{\cat{C}}$ of models of $\monoid{T}$ in $\cat{C}$
(with respect to the standard metamodel as in Example~\ref{ex:clone_enrichment})
equipped with the forgetful functor
$U\colon \Mod{\monoid{T}}{\cat{C}}\longrightarrow\cat{C}$.
The functor $\Str$, in this case, maps any functor 
$V\colon\cat{A}\longrightarrow\cat{C}$ with small domain $\cat{A}$
to the clone whose underlying graded set is given by 
$([\cat{A},\cat{C}]((-)^n\circ V,V))_{n\in\N}$,
where $(-)^n\colon \cat{C}\longrightarrow\cat{C}$ is the functor taking 
$n$-th powers, and whose clone operations canonically induced from powers in 
$\cat{C}$.

An object of $\Cat/\cat{C}$, say $V\colon\cat{A}\longrightarrow\cat{C}$,
may be seen as specifying an additional structure (of a very general type) 
on objects in $\cat{C}$, by viewing $\cat{A}$ as the category of 
$\cat{C}$-objects with the additional structure and $V$ as the associated 
forgetful functor.
The functor $\Str$ extracts a clone from $V$,  
giving the best approximation of this additional structure by 
structures expressible by clones.

We remark that if we take a \emph{locally small} category $\cat{C}$,
as is often the case of interest (e.g., $\cat{C}=\Set$),
then in general we cannot have an adjunction
\begin{equation*}
\begin{tikzpicture}[baseline=-\the\dimexpr\fontdimen22\textfont2\relax ]
      \node(11) at (0,0) 
      {$\mathbf{Clo}^\op$};
      \node(22) at (4,0) {$\CAT/\cat{C}$.};
  
      \draw [->,transform canvas={yshift=5pt}]  (22) to node 
      [auto,swap,labelsize]{$\Str$} (11);
      \draw [->,transform canvas={yshift=-5pt}]  (11) to node 
      [auto,swap,labelsize]{$\Sem$} (22);
      \path (11) to node[midway](m){} (22); 

      \node at (m) [labelsize,rotate=90] {$\vdash$};
\end{tikzpicture}
\end{equation*}
The above construction fails because
for an object $V\colon\cat{A}\longrightarrow\cat{C}$ in $\CAT/\cat{C}$
and a natural number $n$,
the set $[\cat{A},\cat{C}] ((-)^n\circ V,V)$ may not be small.
Indeed, a functor $V\colon \cat{A}\longrightarrow\cat{C}$ is called 
\emph{tractable} in 
\cite{Linton_equational} precisely when the sets of the form
$[\cat{A},\cat{C}] ((-)^n\circ V,V)$ are small.
We obtain an adjunction if we restrict $\CAT/\cat{C}$ to its full subcategory 
consisting of all 
tractable functors.

\section{The structure and semantics functors}
Let $\cat{M}=(\cat{M},I,\otimes)$ be a metatheory, $\cat{C}$ be a large 
category, and $\Phi=(\Phi,\overline{\phi}_\cdot,\overline{\phi})$ be a 
metamodel of $\cat{M}$ in $\cat{C}$.
The metamodel $\Phi$ enables us to define, for each theory $\monoid{T}
\in\Th{\cat{M}}$, the category of models $\Mod{\monoid{T}}{(\cat{C},\Phi)}$
together with the forgetful functor $U\colon \Mod{\monoid{T}}{(\cat{C},\Phi)}
\longrightarrow\cat{C}$.
This construction is functorial, and gives rise to a functor
\[
\Th{\cat{M}}^\op\longrightarrow\CAT/\cat{C}.
\]
However, as we have remarked in Proposition~\ref{prop:metatheory_as_enrichment},
a metamodel of $\cat{M}$ in $\cat{C}$ corresponds to 
an enrichment of $\cat{C}$ over $\widehat{\cat{M}}$;
hence using $\Phi$ we can actually give 
the definition of models for any theory (i.e., monoid object) in 
$\widehat{\cat{M}}$.
Therefore the previous functor can be extended to 
\begin{equation}
\label{eqn:semantics_functor}
\Sem\colon \Th{\widehat{\cat{M}}}^\op\longrightarrow\CAT/\cat{C}.
\end{equation}
The category $\Th{\widehat{\cat{M}}}$ is 
isomorphic to the category of lax monoidal functors of type 
$\cat{M}^\op\longrightarrow\SET$ and monoidal natural transformations between 
them.
Indeed, an object $(P,e,m)$ of $\Th{\widehat{\cat{M}}}$
consists of:
\begin{itemize}
\item a functor $P\colon\cat{M}^\op\longrightarrow\SET$;
\item a natural transformation $(e_X\colon \DayI (X)\longrightarrow 
P(X))_{X\in\cat{M}}$;
\item a natural transformation $(m_X\colon (P\Dayo P)(X)\longrightarrow 
P(X))_{X\in\cat{M}}$
\end{itemize}
satisfying the monoid axioms, and such a data is equivalent to 
\begin{itemize}
\item a functor $P\colon\cat{M}^\op\longrightarrow\SET$;
\item a function $\overline{e}\colon 1\longrightarrow P(I)$;
\item a natural transformation $(\overline{m}_{X,Y}\colon P(Y)\times P(X)
\longrightarrow P(Y\otimes X))_{X,Y\in\cat{M}}$
\end{itemize}
satisfying the axioms for $(P,\overline{e},\overline{m})$ to be a
lax monoidal functor $\cat{M}^\op\longrightarrow\SET$.
We shall use these two descriptions of objects of the category 
$\Th{\widehat{\cat{M}}}$ 
interchangeably.

Let us describe the action of the functor $\Sem$ concretely. 
For any $\monoid{P}=(P,e,m)\in\Th{\widehat{\cat{M}}}$, we define 
the category $\Mod{\monoid{P}}{(\cat{C},\Phi)}$ as follows:
\begin{itemize}
\item An object is a pair consisting of an object $C\in\cat{C}$ and 
a natural transformation 
\[
(\xi_X\colon P(X)\longrightarrow \Phi_X(C,C))_{C\in\cat{M}}
\]
making the following diagrams commute for each $X,Y\in\cat{M}$:
\begin{equation}
\label{eqn:model_of_P}
\begin{tikzpicture}[baseline=-\the\dimexpr\fontdimen22\textfont2\relax ]
      \node (TL) at (0,1)  {$1$};
      \node (TR) at (2,1)  {$P(I)$};
      \node (BR) at (2,-1) {$\Phi_I(C,C)$};
      \draw[->] (TL) to node[auto,labelsize](T) {$\overline{e}$} (TR);
      \draw[->]  (TR) to node[auto,labelsize] {$\xi_I$} (BR);
      \draw[->]  (TL) to node[auto,swap,labelsize] 
      {$(\overline{\phi}_\cdot)_{C}$} 
      (BR);
\end{tikzpicture} 
\quad
\begin{tikzpicture}[baseline=-\the\dimexpr\fontdimen22\textfont2\relax ]
      \node (TL) at (0,1)  {$P(Y)\times P(X)$};
      \node (TR) at (4.5,1)  {$P({Y\otimes X})$};
      \node (BL) at (0,-1) {$\Phi_Y(C,C)\times \Phi_X(C,C)$};
      \node (BR) at (4.5,-1) {$\Phi_{Y\otimes X}(C,C).$};
      \draw[->] (TL) to node[auto,labelsize](T) 
      {$\overline{m}_{X,Y}$} 
      (TR);
      \draw[->]  (TR) to node[auto,labelsize] {$\xi_{Y\otimes X}$} (BR);
      \draw[->]  (TL) to node[auto,swap,labelsize] {$\xi_Y\times \xi_X$} 
      (BL);
      \draw[->]  (BL) to node[auto,labelsize] 
      {$(\overline{\phi}_{X,Y})_{C,C,C}$} (BR);
\end{tikzpicture} 
\end{equation}
\item A morphism from $(C,\xi)$ to $(C',\xi')$ is a morphism $f\colon 
C\longrightarrow C'$ in $\cat{C}$ making the following diagram commute for 
each $X\in\cat{M}$:
\begin{equation}
\label{eqn:P_model_homomorphism}
\begin{tikzpicture}[baseline=-\the\dimexpr\fontdimen22\textfont2\relax ]
      \node (TL) at (0,1)  {$P(X)$};
      \node (TR) at (4,1)  {$\Phi_X(C,C)$};
      \node (BL) at (0,-1) {$\Phi_X(C',C')$};
      \node (BR) at (4,-1) {$\Phi_{X}(C,C').$};
      \draw[->] (TL) to node[auto,labelsize](T) {$\xi_X$} (TR);
      \draw[->]  (TR) to node[auto,labelsize] {$\Phi_X(C,f)$} (BR);
      \draw[->]  (TL) to node[auto,swap,labelsize] {$\xi'_X$} (BL);
      \draw[->]  (BL) to node[auto,labelsize] {$\Phi_X(f,C')$} (BR);
\end{tikzpicture} 
\end{equation}
\end{itemize}
There exists an evident forgetful functor 
$U\colon\Mod{\monoid{P}}{(\cat{C},\Phi)}\longrightarrow\cat{C}$ mapping 
$(C,\xi)$ to $C$ and $f$ to $f$;
the functor $\Sem$ maps $\monoid{P}$ to $U$.

We have a canonical fully faithful functor
\begin{equation*}
J\colon\Th{\cat{M}}\longrightarrow\Th{\widehat{\cat{M}}}
\end{equation*}
mapping $(T,e,m)\in\Th{\cat{M}}$ to the functor $\cat{M}(-,T)$
with the evident monoid structure induced from $e$ and $m$.
An object $(P,e,m)\in\Th{\widehat{\cat{M}}}$ is in the essential 
image of $J$
if and only if $P\colon\cat{M}^\op\longrightarrow\SET$ is representable.

Let us describe the left adjoint $\Str$ to (\ref{eqn:semantics_functor}).
Given an object $V\colon\cat{A}\longrightarrow\cat{C}$ of $\CAT/\cat{C}$, 
we define $\Str(V)=(P^{(V)},e^{(V)},m^{(V)})\in\Th{\widehat{\cat{M}}}$ as 
follows:
\begin{itemize}
\item The functor $P^{(V)}\colon \cat{M}^\op\longrightarrow\SET$
maps $X\in\cat{M}$ to 
\begin{equation}
\label{eqn:structure_of_V}
P^{(V)}(X)=\int_{A\in\cat{A}}\Phi_X(VA,VA).
\end{equation}
\item The function $\overline{e^{(V)}}\colon 1\longrightarrow P^{(V)}(I)$
maps the unique element of $1$ to 
$((\overline{\phi}_\cdot)_{VA} (\ast))_{A\in\cat{A}}$ $\in P^{(V)}(I)$.
\item The $(X,Y)$-th component of the natural transformation 
\[(\overline{m^{(V)}}_{X,Y}\colon P^{(V)}(Y)\times P^{(V)}(X)\longrightarrow 
P^{(V)}(Y\otimes X))_{X,Y\in\cat{M}}\]
maps $((y_A)_{A\in\cat{A}},(x_A)_{A\in\cat{A}})$
%\in P^{(V)}(Y)\times P^{(V)}(X)$
to $((\overline{\phi}_{X,Y})_{VA,VA,VA}(y_A,x_A))_{A\in\cat{A}}$.
%\in P^{(V)}(Y\otimes X)$.
\end{itemize}
The monoid axioms for $(P^{(V)},{e^{(V)}},{m^{(V)}})$ 
follow easily from the axioms for metamodels,
and $\Str$ routinely extends to a functor of type 
$\CAT/\cat{C}\longrightarrow\Th{\widehat{\cat{M}}}^\op$. 

\begin{thm}
\label{thm:str_sem_small}
Let $\cat{M}$ be a metatheory, $\cat{C}$ be a large category and 
$\Phi=(\Phi,\overline{\phi}_\cdot,\overline{\phi})$ 
be a metamodel of $\cat{M}$ in $\cat{C}$.
The functors $\Sem$ and $\Str$ defined above
form an adjunction:
\[
\begin{tikzpicture}[baseline=-\the\dimexpr\fontdimen22\textfont2\relax ]
      \node(11) at (0,0) 
      {$\Th{\widehat{\cat{M}}}^\op$};
      \node(22) at (4,0) {$\CAT/\cat{C}$.};
  
      \draw [->,transform canvas={yshift=5pt}]  (22) to node 
      [auto,swap,labelsize]{$\Str$} (11);
      \draw [->,transform canvas={yshift=-5pt}]  (11) to node 
      [auto,swap,labelsize]{$\Sem$} (22);
      \path (11) to node[midway](m){} (22); 

      \node at (m) [labelsize,rotate=90] {$\vdash$};
\end{tikzpicture}
\] 
\end{thm}
\begin{proof}
We show that there are bijections
\[\Th{\widehat{\cat{M}}}(\monoid{P},\Str(V))\cong
(\CAT/\cat{C})(V,\Sem(\monoid{P}))\]
natural in $\monoid{P}=(P,e,m)\in\Th{\widehat{\cat{M}}}$ and $(V\colon 
\cat{A}\longrightarrow\cat{C})\in\CAT/\cat{C}$.

In fact, we show that the following three types of data naturally 
correspond to each other.
\begin{enumerate}
\item A morphism $\alpha\colon \monoid{P}\longrightarrow\Str(V)$ in 
$\Th{\widehat{\cat{M}}}$; that is, a natural transformation 
\[
(\alpha_X\colon P(X)\longrightarrow P^{(V)}(X))_{X\in\cat{M}}
\]
making the suitable diagrams commute.
\item A natural transformation \[
(\xi_{A,X}\colon P(X)\longrightarrow \Phi_X(VA,VA))_{X\in\cat{M},
A\in\cat{A}}
\]
making the following diagrams commute for each $A\in\cat{A}$ and 
$X,Y\in\cat{M}$:
\begin{equation*}
\begin{tikzpicture}[baseline=-\the\dimexpr\fontdimen22\textfont2\relax ]
      \node (TL) at (0,1)  {$1$};
      \node (TR) at (2,1)  {$P(I)$};
      \node (BR) at (2,-1) {$\Phi_I(VA,VA)$};
      \draw[->] (TL) to node[auto,labelsize](T) {$\overline{e}$} (TR);
      \draw[->]  (TR) to node[auto,labelsize] {$\xi_{A,I}$} (BR);
      \draw[->]  (TL) to node[auto,swap,labelsize] 
      {$(\overline{\phi}_\cdot)_{VA}$} 
      (BR);
\end{tikzpicture} 
\quad
\begin{tikzpicture}[baseline=-\the\dimexpr\fontdimen22\textfont2\relax ]
      \node (TL) at (0,1)  {$P(Y)\times P(X)$};
      \node (TR) at (5.5,1)  {$P({Y\otimes X})$};
      \node (BL) at (0,-1) {$\Phi_Y(VA,VA)\times \Phi_X(VA,VA)$};
      \node (BR) at (5.5,-1) {$\Phi_{Y\otimes X}(VA,VA).$};
      \draw[->] (TL) to node[auto,labelsize](T) 
      {$\overline{m}_{X,Y}$} 
      (TR);
      \draw[->]  (TR) to node[auto,labelsize] {$\xi_{A,Y\otimes X}$} (BR);
      \draw[->]  (TL) to node[auto,swap,labelsize] {$\xi_{A,Y}\times 
      \xi_{A,X}$} 
      (BL);
      \draw[->]  (BL) to node[auto,labelsize] 
      {$(\overline{\phi}_{X,Y})_{VA,VA,VA}$} (BR);
\end{tikzpicture} 
\end{equation*}
\item A morphism $F\colon V\longrightarrow\Sem(\monoid{P})$ in $\CAT/\cat{C}$;
that is, a functor 
$F\colon\cat{A}\longrightarrow\Mod{\monoid{P}}{(\cat{C},\Phi)}$
such that $U\circ F=V$ 
($U\colon\Mod{\monoid{P}}{(\cat{C},\Phi)}\longrightarrow\cat{C}$ is the 
forgetful functor).
\end{enumerate}

The correspondence between 1 and 2 is by the universality of ends
(see (\ref{eqn:structure_of_V})). 
To give $\xi$ as in 2 \emph{without} requiring naturality
in $A\in\cat{A}$, is equivalent to give a 
function $\ob{F}\colon 
\ob{\cat{A}}\longrightarrow\ob{\Mod{\monoid{P}}{(\cat{C},\Phi)}}$
such that $\ob{U}\circ \ob{F}=\ob{V}$ (see (\ref{eqn:model_of_P})). 
To say that $\xi$ is natural also in $A\in\cat{A}$
is equivalent to saying that $\ob{F}$ extends to a functor 
$F\colon \cat{A}\longrightarrow\Mod{\monoid{P}}{(\cat{C},\Phi)}$
by mapping each morphism $f$ in $\cat{A}$ to $Vf$.
\end{proof}

\section{The classical cases}
We conclude this chapter by showing that we can restore the known
structure-semantics adjunctions for clones and monads, by restricting
our version of structure-semantics adjunctions 
(Theorem~\ref{thm:str_sem_small}).

In both cases of clones and monads, we shall consider the diagram
\[
\begin{tikzpicture}[baseline=-\the\dimexpr\fontdimen22\textfont2\relax ]
      \node(11T) at (0,2) 
      {$\Th{\cat{\widehat{M}}}^\op$};
      \node(22T) at (4,2) {$\CAT/\cat{C}$};
      \draw [->,transform canvas={yshift=5pt}]  (22T) to node 
      [auto,swap,labelsize]{$\Str$} (11T);
      \draw [->,transform canvas={yshift=-5pt}]  (11T) to node 
      [auto,swap,labelsize]{$\Sem$} (22T);
      \path (11T) to node[midway](mT){} (22T); 
      \node at (mT) [labelsize,rotate=90] {$\vdash$};
      \node(11) at (0,0) 
      {$\Th{\cat{M}}^\op$};
      \node(22) at (4,0) {$(\CAT/\cat{C})_{\mathrm{tr}}$};
      \draw [->,transform canvas={yshift=5pt}]  (22) to node 
      [auto,swap,labelsize]{$\Str'$} (11);
      \draw [->,transform canvas={yshift=-5pt}]  (11) to node 
      [auto,swap,labelsize]{$\Sem'$} (22);
      \path (11) to node[midway](m){} (22); 
      \node at (m) [labelsize,rotate=90] {$\vdash$};
      \draw[->] (11) to node[auto,labelsize] {$J$} (11T);
      \draw[->] (22) to node[auto,swap,labelsize] {$K$} (22T);
\end{tikzpicture}
\]
in which the top adjunction is the one we have constructed in the previous 
section,
the bottom adjunction is a classical structure-semantics adjunction,
and $J$ and $K$ are the canonical fully faithful functors
(the precise definition of $(\CAT/\cat{C})_\mathrm{tr}$ will be given below).
We shall prove that the two squares, one involving $\Str$ and $\Str'$,
the other involving $\Sem$ and $\Sem'$, commute,
showing that $\Str'$ (resp.~$\Sem'$) arises as a restriction of 
$\Str$ (resp.~$\Sem$).

First, that $K\circ \Sem'\cong \Sem\circ J$ holds is straightforward,
and this is true as soon as $\Sem'$ maps any $\monoid{T}\in\Th{\cat{M}}$ 
to the forgetful functor 
$U\colon\Mod{\monoid{T}}{(\cat{C},\Phi)}\longrightarrow\cat{C}$.
Indeed, for any theory $\monoid{T}=(T,e,m)$ in $\cat{M}$,
$J(\monoid{T})\in\Th{\widehat{\cat{M}}}$ has the underlying object 
$\cat{M}(-,T)\in\widehat{\cat{M}}$, and the description of 
$\Mod{J(\monoid{T})}{(\cat{C},\Phi)}$ in the previous section coincides with
$\Mod{\monoid{T}}{(\cat{C},\Phi)}$ by the Yoneda lemma.

Let us check that $J\circ \Str'\cong \Str\circ K$ holds.\footnote{This does not 
seem to follow formally from $K\circ \Sem'\cong \Sem\circ J$,
even if we take into consideration the fact that $J$ and $K$ are fully 
faithful.}
For this, we have to review the classical structure functors
and the tractability conditions.

We begin with the case of clones as treated in \cite{Linton_equational},
which we have already sketched in Section~\ref{sec:what_are_str_sem}.
Let $\cat{C}$ be a locally small category with finite powers and 
consider the standard metamodel $\Phi$ of $[\F,\Set]$ in $\cat{C}$
(derived from the enrichment $\enrich{-}{-}$ in 
Example~\ref{ex:clone_enrichment}). 
An object $V\colon\cat{A}\longrightarrow\cat{C}\in\CAT/\cat{C}$ is called 
\defemph{tractable} if and only if for any natural number $n$,
the set $[\cat{A},\cat{C}]((-)^n\circ V,V)$ is small.
Given a tractable $V$, $\Str'(V)\in\Th{[\F,\Set]}$ has the underlying 
functor $|\Str'(V)|$ mapping $[n]\in\F$ to $[\cat{A},\cat{C}]((-)^n\circ V,V)$.
On the other hand, our formula (\ref{eqn:structure_of_V})
reduces as follows:
\begin{align*}
P^{(V)}(X)&=\int_{A\in\cat{A}}\Phi_X(VA,VA)\\
&= \int_{A\in\cat{A}} [\F,\Set](X, \enrich{VA}{VA})\\
&\cong \int_{A\in\cat{A},[n]\in\F} \Set(X_n, \cat{C}((VA)^n,VA))\\
&\cong \int_{[n]\in\F} \Set\bigg(X_n, 
\int_{A\in\cat{A}}\cat{C}((VA)^n,VA)\bigg)\\
&\cong \int_{[n]\in\F} \Set(X_n, [\cat{A},\cat{C}]((-)^n\circ V,V))\\
&\cong [\F,\Set](X, |\Str'(V)|). 
\end{align*}
It is routine from this to see that $J\circ \Str'\cong \Str\circ K$ holds.

Finally, for monads, we take as a classical structure-semantics adjunction 
the one in \cite[Section~II.~1]{Dubuc_Kan}.
Let $\cat{C}$ be a large category and consider the standard metamodel $\Phi$
of $[\cat{C},\cat{C}]$ in $\cat{C}$ (derived from the standard strict action 
$\ast$ in Example~\ref{ex:monad_standard_action}).
An object $V\colon \cat{A}\longrightarrow\cat{C}\in\CAT/\cat{C}$
is called \defemph{tractable} if and only if the right Kan extension  
$\Ran_V V$ of $V$ along itself exists.\footnote{In fact, in 
\cite[p.~68]{Dubuc_Kan} 
Dubuc defines tractability as a slightly stronger condition. 
However, the condition we have introduced above is the one which is used for 
the construction of structure-semantics adjunctions in \cite{Dubuc_Kan}.}
\begin{comment}
; that is,
for each object $C\in\cat{C}$, the limit
\begin{equation}
\label{eqn:Ran_1}
\lim_{(A\in\cat{A},f\colon C\longrightarrow VA)\in (C\downarrow V)} VA
\end{equation}
indexed by the comma category $(C\downarrow V)$ exists.
This limit may also be written as the end
\begin{equation}
\label{eqn:Ran_2}
\int_{A\in\cat{A}} (VA)^{\cat{C}(C,VA)}
\end{equation}
in $\cat{C}$, provided that the power $(VA)^{\cat{C}(VA)}$ exists in $\cat{A}$
for each $A\in\cat{A}$.
Anyway, in that case, the object $(\Ran_V V)C$ is given by  
(\ref{eqn:Ran_1}), or equivalently, by (\ref{eqn:Ran_2}).
\end{comment}
It is known that a functor of the form
$\Ran_V V$ acquires a canonical monad structure,
and the resulting monad is called the {codensity monad of $V$}. 
For a tractable $V$, $\Str'(V)$ is defined to be the codensity monad of $V$.
Now let us return to our formula (\ref{eqn:structure_of_V}):
\begin{align*}
P^{(V)}(X)&=\int_{A\in\cat{A}}\Phi_X(VA,VA)\\
&= \int_{A\in\cat{A}} \cat{C}(XVA,{VA})\\
&\cong [\cat{A},\cat{C}](X\circ V,V)\\
&\cong [\cat{C},\cat{C}](X,\Ran_V V). 
\end{align*}
Again we see that $J\circ \Str'\cong \Str\circ K$ holds.

\chapter{Categories of models as double limits}
\label{chap:double_lim}
Let $\cat{M}$ be a metatheory, $\monoid{T}$ be a theory in $\cat{M}$,
$\cat{C}$ be a large category and $\Phi$ be a metamodel of $\cat{M}$ in 
$\cat{C}$.
Given these data, in Definition~\ref{def:metamodel_model} we have defined---in
a concrete manner---the category $\Mod{\monoid{T}}{(\cat{C},\Phi)}$ of models 
of $\monoid{T}$ in $\cat{C}$ with respect to $\Phi$,
equipped with the evident forgetful functor 
$U\colon\Mod{\monoid{T}}{(\cat{C},\Phi)}\longrightarrow\cat{C}$.

In this chapter, we give an abstract characterisation of the categories of 
models.
%\textbf{Motivations to seek such a characterisation. Street used the
%universal property to show various facts about the Eilenberg--Moore 
%categories, but can we derive something??}
A similar result is known for the {Eilenberg--Moore category}
of a monad; Street \cite{Street_FTM}
has proved that it can be abstractly characterised as the lax limit in 
$\tCAT$ of a certain diagram canonically constructed from the original monad.
We prove that the categories of models in our framework can also be 
characterised by a certain universal property.
A suitable language to express this universal property is
that of \emph{pseudo double categories}~\cite{GP1},
reviewed in Section~\ref{sec:pseudo_double_cat}.
We show that the category $\Mod{\monoid{T}}{(\cat{C},\Phi)}$, together with 
the forgetful functor $U$ and some other natural data, form a double limit 
in the pseudo double category $\PPROF$ of large categories,
profunctors, functors and natural transformations.

\section{The universality of Eilenberg--Moore categories}
\label{sec:EM_cat_as_2-lim}
In this section we review the 2-categorical characterisation in 
\cite{Street_FTM} of the Eilenberg--Moore category of a monad on a large 
category, in elementary terms.\footnote{The main point of the paper  
\cite{Street_FTM} is the introduction of the notion of Eilenberg--Moore object
in an \emph{arbitrary 2-category} $\tcat{B}$ via a universal property and show 
that, if exists, it satisfies certain formal properties of Eilenberg--Moore 
categories.
However, for our purpose, it suffices to consider the simple case 
$\tcat{B}=\tCAT$ only. It is left as future work to investigate whether 
we can develop a similar ``formal theory'' from the double-categorical 
universal property of categories of models.}
Let $\cat{C}$ be a large category and $\monoid{T}=(T,\eta,\mu)$ be a monad on 
$\cat{C}$.
The Eilenberg--Moore category $\cat{C}^\monoid{T}$ of $\monoid{T}$ 
is equipped with a canonical forgetful functor $U\colon 
\cat{C}^\monoid{T}\longrightarrow\cat{C}$ mapping an Eilenberg--Moore algebra 
$(C,\gamma)$ of $\monoid{T}$ to its underlying object $C$.
Moreover, there exists a canonical natural transformation $u\colon T\circ 
U\Longrightarrow U$, i.e., of type 
\[
\begin{tikzpicture}[baseline=-\the\dimexpr\fontdimen22\textfont2\relax ]
      \node (TL) at (0,0.75)  {$\cat{C}^\monoid{T}$};
      \node (TR) at (2,0.75)  {$\cat{C}^\monoid{T}$};
      \node (BL) at (0,-0.75) {$\cat{C}$};
      \node (BR) at (2,-0.75) {$\cat{C}.$};
      \draw[->]  (TL) to node[auto,swap,labelsize] {$U$} (BL);
      \draw[->]  (TL) to node[auto,labelsize] {$\id{\cat{C}^\monoid{T}}$} (TR);
      \draw[->]  (BL) to node[auto,swap,labelsize] {$T$} (BR);
      \draw[->]  (TR) to node[auto,labelsize](B) {$U$} (BR);
      \tzsquareup{1}{0}{$u$}
\end{tikzpicture} 
\] 
We are depicting $u$ in a square rather than in a triangle for later comparison 
with similar diagrams in a pseudo double category.
For each $(C,\gamma)\in\cat{C}^\monoid{T}$, 
the $(C,\gamma)$-th component of $u$ is simply $\gamma\colon TC\longrightarrow 
C$.
We claim that the data $(\cat{C}^\monoid{T}, U,u)$ is characterised by a
certain universal property.

To state this universal property, let us define 
a \defemph{left $\monoid{T}$-module} to be 
a triple $(\cat{A},V,v)$ consisting of a large category $\cat{A}$,
a functor $V\colon \cat{A}\longrightarrow\cat{C}$ and a natural transformation 
$v\colon T\circ V\Longrightarrow V$,
such that the following equations hold:
\[
\begin{tikzpicture}[baseline=-\the\dimexpr\fontdimen22\textfont2\relax ]
      \node (TL) at (0,0.75)  {$\cat{A}$};
      \node (TR) at (2,0.75)  {$\cat{A}$};
      \node (BL) at (0,-0.75) {$\cat{C}$};
      \node (BR) at (2,-0.75) {$\cat{C}$};
      \draw[->] (TL) to node[auto,labelsize](T) {$\id{\cat{A}}$} 
      (TR);
      \draw[->]  (TR) to node[auto,labelsize] {$V$} (BR);
      \draw[->]  (TL) to node[auto,swap,labelsize] {$V$} 
      (BL);
      \draw[->, bend left=30] (BL) to node[auto,labelsize](B) 
      {$T$} 
      (BR);
      \draw[->, bend right=30] (BL) to node[auto,swap,labelsize](B) 
      {$\id{\cat{C}}$} 
      (BR);
      \tzsquareup{1}{0.25}{$v$}
      \tzsquareup{1}{-0.75}{$\eta$}
\end{tikzpicture}
\quad=\quad
\begin{tikzpicture}[baseline=-\the\dimexpr\fontdimen22\textfont2\relax ]
      \node (TL) at (0,0.75)  {$\cat{A}$};
      \node (TR) at (2,0.75)  {$\cat{A}$};
      \node (BL) at (0,-0.75) {$\cat{C}$};
      \node (BR) at (2,-0.75) {$\cat{C}$};
      \draw[->] (TL) to node[auto,labelsize](T) {$\id{\cat{A}}$} 
      (TR);
      \draw[->]  (TR) to node[auto,labelsize] {$V$} (BR);
      \draw[->]  (TL) to node[auto,swap,labelsize] {$V$} 
      (BL);
      \draw[->] (BL) to node[auto,swap,labelsize](B) {$\id{\cat{C}}$} 
      (BR);
      \tzsquareup{1}{0}{$\id{V}$}
\end{tikzpicture}
\]
\[
\begin{tikzpicture}[baseline=-\the\dimexpr\fontdimen22\textfont2\relax ]
      \node (TL) at (0,0.75)  {$\cat{A}$};
      \node (TR) at (2,0.75)  {$\cat{A}$};
      \node (BL) at (0,-0.75) {$\cat{C}$};
      \node (BR) at (2,-0.75) {$\cat{C}$};
      \draw[->] (TL) to node[auto,labelsize](T) {$\id{\cat{A}}$} 
      (TR);
      \draw[->]  (TR) to node[auto,labelsize] {$V$} (BR);
      \draw[->]  (TL) to node[auto,swap,labelsize] {$V$} 
      (BL);
      \draw[->, bend left=30] (BL) to node[auto,labelsize](B) 
      {$T$} 
      (BR);
      \draw[->, bend right=30] (BL) to node[auto,swap,labelsize](B) 
      {$T\circ T$} 
      (BR);
      \tzsquareup{1}{0.25}{$v$}
      \tzsquareup{1}{-0.75}{$\mu$}
\end{tikzpicture}
\quad=\quad
\begin{tikzpicture}[baseline=-\the\dimexpr\fontdimen22\textfont2\relax ]
      \node (TL) at (0,0.75)  {$\cat{A}$};
      \node (TR) at (2,0.75)  {$\cat{A}$};
      \node (BL) at (0,-0.75) {$\cat{C}$};
      \node (BR) at (2,-0.75) {$\cat{C}$};
      \node (TRR) at (4,0.75)  {$\cat{A}$};
      \node (BRR) at (4,-0.75) {$\cat{C}.$};
      \draw[->] (TL) to node[auto,labelsize](T) {$\id{\cat{A}}$} 
      (TR);
      \draw[->] (TR) to node[auto,labelsize](T2) {$\id{\cat{A}}$} 
      (TRR);
      \draw[->]  (TR) to node[auto,labelsize] {$V$} (BR);
      \draw[->]  (TRR) to node[auto,labelsize] {$V$} (BRR);
      \draw[->]  (TL) to node[auto,swap,labelsize] {$V$} 
      (BL);
      \draw[->] (BL) to node[auto,swap,labelsize](B) {$T$} 
      (BR);
      \draw[->] (BR) to node[auto,swap,labelsize](B) {$T$} 
      (BRR);
      \tzsquareup{1}{0}{$v$}
      \tzsquareup{3}{0}{$v$}
\end{tikzpicture}
\]

The triple $(\cat{C}^\monoid{T}, U,u)$ is then a \defemph{universal left 
$\monoid{T}$-module}, meaning that it satisfies the following:
\begin{enumerate}
\item it is a left $\monoid{T}$-module;
\item for any left $\monoid{T}$-module $(\cat{A},V,v)$,
there exists a unique functor $K\colon \cat{A}\longrightarrow\cat{C}^\monoid{T}$
such that 
\[
\begin{tikzpicture}[baseline=-\the\dimexpr\fontdimen22\textfont2\relax ]
      \node (TL) at (0,0.75)  {$\cat{A}$};
      \node (TR) at (2,0.75)  {$\cat{A}$};
      \node (BL) at (0,-0.75) {$\cat{C}$};
      \node (BR) at (2,-0.75) {$\cat{C}$};
      \draw[->] (TL) to node[auto,labelsize](T) {$\id{\cat{A}}$} 
      (TR);
      \draw[->]  (TR) to node[auto,labelsize] {$V$} (BR);
      \draw[->]  (TL) to node[auto,swap,labelsize] {$V$} 
      (BL);
      \draw[->] (BL) to node[auto,swap,labelsize](B) {$T$} (BR);
      \tzsquareup{1}{0}{$v$}
\end{tikzpicture}  
\quad=\quad
\begin{tikzpicture}[baseline=-\the\dimexpr\fontdimen22\textfont2\relax ]
      \node (TL) at (0,1.5)  {$\cat{A}$};
      \node (TR) at (2,1.5)  {$\cat{A}$};
      \node (BL) at (0,0) {$\cat{C}^\monoid{T}$};
      \node (BR) at (2,0) {$\cat{C}^\monoid{T}$};
      \node (BBL) at (0,-1.5) {$\cat{C}$};
      \node (BBR) at (2,-1.5) {$\cat{C}$};
      \draw[->] (TL) to node[auto,labelsize](T) {$\id{\cat{A}}$} 
      (TR);
      \draw[->]  (TR) to node[auto,labelsize] {$K$} (BR);
      \draw[->]  (TL) to node[auto,swap,labelsize] {$K$} 
      (BL);
      \draw[->] (BL) to node[auto,labelsize](B) {$\id{\cat{C}^\monoid{T}}$} 
      (BR);
      \draw[->]  (BR) to node[auto,labelsize] {$U$} (BBR);
      \draw[->]  (BL) to node[auto,swap,labelsize] {$U$} (BBL);
      \draw[->] (BBL) to node[auto,swap,labelsize](B) {$T$} (BBR);
      \tzsquareup{1}{0.75}{$\id{K}$}
      \tzsquareup{1}{-0.75}{$u$} 
\end{tikzpicture} 
\]
holds;
\item for any pair of left $\monoid{T}$-modules $(\cat{A},V,v)$ and 
$(\cat{A},V',v')$ on a common large category $\cat{A}$
and any natural transformation $\theta\colon V\Longrightarrow V'$ such that 
\[
\begin{tikzpicture}[baseline=-\the\dimexpr\fontdimen22\textfont2\relax ]
      \node (TL) at (0,0.75)  {$\cat{A}$};
      \node (TR) at (2,0.75)  {$\cat{A}$};
      \node (BL) at (0,-0.75) {$\cat{C}$};
      \node (BR) at (2,-0.75) {$\cat{C}$};
      \draw[->] (TL) to node[auto,labelsize](T) {$\id{\cat{A}}$} 
      (TR);
      \draw[->,bend right=30]  (TR) to node[auto,swap,labelsize] {$V$} (BR);
      \draw[->,bend left=30]  (TR) to node[auto,labelsize] {$V'$} (BR);
      \draw[->]  (TL) to node[auto,swap,labelsize] {$V$} 
      (BL);
      \draw[->] (BL) to node[auto,swap,labelsize](B) {$T$} 
      (BR);
      \draw[2cell] (1.75,0) to node[auto,labelsize] {$\theta$} (2.25,0);
      \tzsquareupswap{1}{0}{$v$}
\end{tikzpicture}
\quad=\quad
\begin{tikzpicture}[baseline=-\the\dimexpr\fontdimen22\textfont2\relax ]
      \node (TL) at (0,0.75)  {$\cat{A}$};
      \node (TR) at (2,0.75)  {$\cat{A}$};
      \node (BL) at (0,-0.75) {$\cat{C}$};
      \node (BR) at (2,-0.75) {$\cat{C}$};
      \draw[->] (TL) to node[auto,labelsize](T) {$\id{\cat{A}}$} 
      (TR);
      \draw[->]  (TR) to node[auto,labelsize] {$V'$} (BR);
      \draw[->,bend right=30]  (TL) to node[auto,swap,labelsize] {$V$} (BL);
      \draw[->,bend left=30]  (TL) to node[auto,labelsize] {$V'$} 
      (BL);
      \draw[->] (BL) to node[auto,swap,labelsize](B) {$T$} 
      (BR);
      \draw[2cell] (-0.25,0) to node[auto,labelsize] {$\theta$} (0.25,0);
      \tzsquareup{1}{0}{$v'$}
\end{tikzpicture}
\]
holds,
there exists a unique natural transformation $\sigma\colon K\Longrightarrow K'$
such that $\theta=U\circ \sigma$, where $K\colon 
\cat{A}\longrightarrow\cat{C}^\monoid{T}$
and $K'\colon\cat{A}\longrightarrow\cat{C}^\monoid{T}$ are the functors 
corresponding to $(\cat{A},V,v)$ and 
$(\cat{A},V',v')$ respectively.
\end{enumerate}
In more conceptual terms, this means that we have a family of isomorphisms of 
categories 
\[
\tCAT(\cat{A},\cat{C}^\monoid{T})\cong
\tCAT(\cat{A},\cat{C})^{\tCAT(\cat{A},\monoid{T})}
\]
natural in $\cat{A}\in\tCAT$,
where the right hand side denotes the Eilenberg--Moore category of the 
monad $\tCAT(\cat{A},\monoid{T})$; note that 
$\tCAT(\cat{A},-)$ is a 2-functor and therefore preserves monads.

It is straightforward to verify the above three statements on 
$(\cat{C}^\monoid{T},U,u)$.
That $(\cat{C}^\monoid{T},U,u)$ is a left $\monoid{T}$-module follows  
from the definition of Eilenberg--Moore algebras.
Given a left $\monoid{T}$-module $(\cat{A},V,v)$, for any object $A\in\cat{A}$
the pair $(VA,v_A)$ is an Eilenberg--Moore algebra of $\monoid{T}$.
Hence the required functor $K\colon \cat{A}\longrightarrow\cat{C}^\monoid{T}$ 
can be defined by mapping an object $A\in\cat{A}$
to $(VA,v_A)$ and a morphism $f$ in $\cat{A}$ to $Vf$.
The final clause can be proved similarly. In fact, 
this automatically follows from the second clause since $\tCAT$ 
admits tensor products (= cartesian products) with the arrow category; see
\cite{Kelly:elementary}.

As with any universal characterisation, the above property characterises 
the triple $(\cat{C}^\monoid{T},U,u)$ uniquely up to
unique isomorphisms.
One can also express this universal property in terms of the standard 
2-categorical limit notions, such as lax limit or weighted 2-limit 
\cite{Street_limits}.

\section{Pseudo double categories}
\label{sec:pseudo_double_cat}
We shall see that our category of models admit a similar characterisation,
in a different setting:
instead of the 2-category $\tCAT$, we will work within
the \emph{pseudo double category} $\PPROF$.
The notion of pseudo double category is due to Grandis and Par{\'e}~\cite{GP1},
and it generalises the classical notion of double category \cite{Ehresmann}
in a way similar to the generalisation of 2-categories to bicategories,
or to the generalisation of strict monoidal categories to monoidal categories.
In this section we briefly review pseudo double categories, and introduce the 
pseudo double category $\PPROF$.

Let us begin with an informal explanation of double categories.
A double category consists of \defemph{objects} $A$, \defemph{vertical 
morphisms}
$f\colon A\longrightarrow A^\prime$,
\defemph{horizontal morphisms} $X\colon A\pto B$
and \defemph{squares}
\[
\begin{tikzpicture}[baseline=-\the\dimexpr\fontdimen22\textfont2\relax ]
      \node (TL) at (0,1.5)  {$A$};
      \node (TR) at (2,1.5)  {$B$};
      \node (BL) at (0,0) {$A^\prime$};
      \node (BR) at (2,0) {$B^\prime$,};
      \draw[pto] (TL) to node[auto,labelsize](T) {$X$} (TR);
      \draw[->]  (TR) to node[auto,labelsize] {$g$} (BR);
      \draw[->]  (TL) to node[auto,swap,labelsize] {$f$} (BL);
      \draw[pto] (BL) to node[auto,swap,labelsize](B) {$X^\prime$} (BR);
      \tzsquare{1}{0.75}{$\alpha$}
\end{tikzpicture}
\]
together with several identity and composition operations,
namely: 
\begin{itemize}
\item for each object $A$ we have the \defemph{vertical identity morphism} 
$\id{A}\colon A\longrightarrow A$;
\item for each composable pair of vertical morphisms $f\colon A\longrightarrow 
A'$ and $f'\colon A'\longrightarrow A''$ we have the \defemph{vertical 
composition}
$f'\circ f\colon A\longrightarrow A''$;
\item for each horizontal morphism $X\colon A\pto B$ we have the 
\defemph{vertical identity square}
\[
\begin{tikzpicture}[baseline=-\the\dimexpr\fontdimen22\textfont2\relax ]
      \node (TL) at (0,1.5)  {$A$};
      \node (TR) at (2,1.5)  {$B$};
      \node (BL) at (0,0) {$A$};
      \node (BR) at (2,0) {$B$;};
      \draw[pto] (TL) to node[auto,labelsize](T) {$X$} (TR);
      \draw[->]  (TR) to node[auto,labelsize] {$\id{B}$} (BR);
      \draw[->]  (TL) to node[auto,swap,labelsize] {$\id{A}$} (BL);
      \draw[pto] (BL) to node[auto,swap,labelsize](B) {$X$} (BR);
      \tzsquare{1}{0.75}{$\id{X}$}
\end{tikzpicture}
\]
\item for each vertically composable pair of squares 
\[
\begin{tikzpicture}[baseline=-\the\dimexpr\fontdimen22\textfont2\relax ]
      \node (TL) at (0,0.75)  {$A$};
      \node (TR) at (2,0.75)  {$B$};
      \node (BL) at (0,-0.75) {$A^\prime$};
      \node (BR) at (2,-0.75) {$B^\prime$};
      \draw[pto] (TL) to node[auto,labelsize](T) {$X$} (TR);
      \draw[->]  (TR) to node[auto,labelsize] {$g$} (BR);
      \draw[->]  (TL) to node[auto,swap,labelsize] {$f$} (BL);
      \draw[pto] (BL) to node[auto,swap,labelsize](B) {$X^\prime$} (BR);
      \tzsquare{1}{0}{$\alpha$}
\end{tikzpicture}
\text{ and }
\begin{tikzpicture}[baseline=-\the\dimexpr\fontdimen22\textfont2\relax ]
      \node (TL) at (0,0.75)  {$A'$};
      \node (TR) at (2,0.75)  {$B'$};
      \node (BL) at (0,-0.75) {$A''$};
      \node (BR) at (2,-0.75) {$B''$};
      \draw[pto] (TL) to node[auto,labelsize](T) {$X'$} (TR);
      \draw[->]  (TR) to node[auto,labelsize] {$g'$} (BR);
      \draw[->]  (TL) to node[auto,swap,labelsize] {$f'$} (BL);
      \draw[pto] (BL) to node[auto,swap,labelsize](B) {$X''$} (BR);
      \tzsquare{1}{0}{$\alpha'$}
\end{tikzpicture}
\]
we have the \defemph{vertical composition}
\[
\begin{tikzpicture}[baseline=-\the\dimexpr\fontdimen22\textfont2\relax ]
      \node (TL) at (0,1.5)  {$A$};
      \node (TR) at (2,1.5)  {$B$};
      \node (BL) at (0,0) {$A''$};
      \node (BR) at (2,0) {$B''$;};
      \draw[pto] (TL) to node[auto,labelsize](T) {$X$} (TR);
      \draw[->]  (TR) to node[auto,labelsize] {$g'\circ g$} (BR);
      \draw[->]  (TL) to node[auto,swap,labelsize] {$f'\circ f$} (BL);
      \draw[pto] (BL) to node[auto,swap,labelsize](B) {$X''$} (BR);
      \tzsquare{1}{0.75}{$\alpha'\circ \alpha$}
\end{tikzpicture}
\]
\end{itemize}
and symmetrically:
\begin{itemize}
\item for each object $A$ we have the \defemph{horizontal identity morphism} 
$I_{A}\colon A\pto A$;
\item for each composable pair of horizontal morphisms $X\colon A\pto B$ and 
$Y\colon B\pto C$ we have the \defemph{horizontal composition}
$Y\otimes X\colon A\pto  C$;
\item for each vertical morphism $f\colon A\longrightarrow A'$ we have the 
\defemph{horizontal identity square}
\[
\begin{tikzpicture}[baseline=-\the\dimexpr\fontdimen22\textfont2\relax ]
      \node (TL) at (0,1.5)  {$A$};
      \node (TR) at (2,1.5)  {$A$};
      \node (BL) at (0,0) {$A'$};
      \node (BR) at (2,0) {$A'$;};
      \draw[pto] (TL) to node[auto,labelsize](T) {$I_A$} (TR);
      \draw[->]  (TR) to node[auto,labelsize] {$f$} (BR);
      \draw[->]  (TL) to node[auto,swap,labelsize] {$f$} (BL);
      \draw[pto] (BL) to node[auto,swap,labelsize](B) {$I_{A'}$} (BR);
      \tzsquare{1}{0.75}{$I_f$}
\end{tikzpicture}
\]
\item for each horizontally composable pair of squares 
\[
\begin{tikzpicture}[baseline=-\the\dimexpr\fontdimen22\textfont2\relax ]
      \node (TL) at (0,0.75)  {$A$};
      \node (TR) at (2,0.75)  {$B$};
      \node (BL) at (0,-0.75) {$A^\prime$};
      \node (BR) at (2,-0.75) {$B^\prime$};
      \draw[pto] (TL) to node[auto,labelsize](T) {$X$} (TR);
      \draw[->]  (TR) to node[auto,labelsize] {$g$} (BR);
      \draw[->]  (TL) to node[auto,swap,labelsize] {$f$} (BL);
      \draw[pto] (BL) to node[auto,swap,labelsize](B) {$X^\prime$} (BR);
      \tzsquare{1}{0}{$\alpha$}
\end{tikzpicture}
\text{ and }
\begin{tikzpicture}[baseline=-\the\dimexpr\fontdimen22\textfont2\relax ]
      \node (TL) at (0,0.75)  {$B$};
      \node (TR) at (2,0.75)  {$C$};
      \node (BL) at (0,-0.75) {$B'$};
      \node (BR) at (2,-0.75) {$C'$};
      \draw[pto] (TL) to node[auto,labelsize](T) {$Y$} (TR);
      \draw[->]  (TR) to node[auto,labelsize] {$h$} (BR);
      \draw[->]  (TL) to node[auto,swap,labelsize] {$g$} (BL);
      \draw[pto] (BL) to node[auto,swap,labelsize](B) {$Y'$} (BR);
      \tzsquare{1}{0}{$\beta$}
\end{tikzpicture}
\]
we have the \defemph{horizontal composition}
\[
\begin{tikzpicture}[baseline=-\the\dimexpr\fontdimen22\textfont2\relax ]
      \node (TL) at (0,1.5)  {$A$};
      \node (TR) at (2,1.5)  {$C$};
      \node (BL) at (0,0) {$A'$};
      \node (BR) at (2,0) {$C'$.};
      \draw[pto] (TL) to node[auto,labelsize](T) {$Y\otimes X$} (TR);
      \draw[->]  (TR) to node[auto,labelsize] {$h$} (BR);
      \draw[->]  (TL) to node[auto,swap,labelsize] {$f$} (BL);
      \draw[pto] (BL) to node[auto,swap,labelsize](B) {$Y'\otimes X'$} (BR);
      \tzsquare{1}{0.75}{$\beta\otimes \alpha$}
\end{tikzpicture}
\]
\end{itemize}
These identity and composition operations are required to satisfy several 
axioms, such as the unit and associativity axioms for vertical 
(resp.~horizontal) 
identity and composition, as well as the axiom 
$\id{I_A}=I_\id{A}$ for each object $A$ and 
the \emph{interchange law}, saying that whenever we have a configuration 
of squares as in 
\[
\begin{tikzpicture}[baseline=-\the\dimexpr\fontdimen22\textfont2\relax ]
      \node (TL) at (0,3)  {$\bullet$};
      \node (TM) at (2,3)  {$\bullet$};
      \node (TR) at (4,3)  {$\bullet$};
      \node (ML) at (0,1.5)  {$\bullet$};
      \node (MM) at (2,1.5)  {$\bullet$};
      \node (MR) at (4,1.5)  {$\bullet$};
      \node (BL) at (0,0)  {$\bullet$};
      \node (BM) at (2,0)  {$\bullet$};
      \node (BR) at (4,0)  {$\bullet$,};
      \draw[pto] (TL) to (TM);
      \draw[pto] (TM) to (TR);
      \draw[pto] (ML) to (MM);
      \draw[pto] (MM) to (MR);
      \draw[pto] (BL) to (BM);
      \draw[pto] (BM) to (BR);
      \draw[->]  (TR) to (MR);
      \draw[->]  (MR) to (BR);
      \draw[->]  (TM) to (MM);
      \draw[->]  (MM) to (BM);
      \draw[->]  (TL) to (ML);
      \draw[->]  (ML) to (BL);
      \tzsquare{1}{2.25}{$\alpha$}
      \tzsquare{1}{0.75}{$\alpha'$}
      \tzsquare{3}{2.25}{$\beta$}
      \tzsquare{3}{0.75}{$\beta'$}
\end{tikzpicture}
\]
$(\beta'\otimes \alpha')\circ(\beta\otimes \alpha)=(\beta'\circ 
\beta)\otimes(\alpha'\circ\alpha)$ holds.

\medskip

Some naturally arising double-category-like structure, 
including $\PPROF$, are 
such that whose vertical morphisms are homomorphism-like (e.g., functors)
and whose horizontal morphisms are bimodule-like (e.g., profunctors);
see \cite[Section~1]{Shulman_framed} for a discussion
on these two kinds of morphisms.
However, a problem crops up from the bimodule-like horizontal
morphisms: 
in general, their composition is not unital nor associative on the nose.
Therefore such structures fail to form  
(strict) double categories, but instead form 
\emph{pseudo} (or \emph{weak}) \emph{double 
categories}~\cite{GP1,Leinster_book,Garner_thesis,Shulman_framed},
in which horizontal composition is allowed to 
be unital and associative up to suitable isomorphism.\footnote{In the 
literature, definitions of pseudo double category differ as to 
whether to weaken horizontal compositions or vertical compositions.
We follow  \cite{Garner_thesis,Shulman_framed} and weaken horizontal 
compositions, but note that 
the original paper~\cite{GP1} weakens vertical compositions.}
%Note that in a pseudo double category,
%vertical composition is still strictly 
%unital and associative.

\begin{definition}[\cite{GP1}]
\label{def:psd_dbl_cat}
A \defemph{pseudo double category} $\dcat{D}$ consists of the following data.
\begin{description}
\item[(DD1)] A category $\dcat{D}_0$, whose objects are called
\defemph{objects} of $\dcat{D}$ and whose 
morphisms \defemph{vertical morphisms} of $\dcat{D}$.
\item[(DD2)] A category $\dcat{D}_1$, whose objects are called
\defemph{horizontal morphisms} of $\dcat{D}$ and whose 
morphisms \defemph{squares} of $\dcat{D}$.
\item[(DD3)] Functors 
\begin{align*}
s,t&\colon \dcat{D}_1\longrightarrow\dcat{D}_0,\\
I&\colon \dcat{D}_0\longrightarrow \dcat{D}_1,\\
\otimes&\colon\dcat{D}_2\longrightarrow\dcat{D}_1,
\end{align*}
where $\dcat{D}_2$ is the pullback
\[
\begin{tikzpicture}[baseline=-\the\dimexpr\fontdimen22\textfont2\relax ]
      \node (TL) at (0,1.5)  {$\dcat{D}_2$};
      \node (TR) at (1.5,1.5)  {$\dcat{D}_1$};
      \node (BL) at (0,0) {$\dcat{D}_1$};
      \node (BR) at (1.5,0) {$\dcat{D}_0$};
      \draw[->] (TL) to node[auto,labelsize](T) {$\pi_2$} (TR);
      \draw[->] (TR) to node[auto,labelsize] {$t$} (BR);
      \draw[->] (TL) to node[auto,swap,labelsize] {$\pi_1$} (BL);
      \draw[->] (BL) to node[auto,swap,labelsize](B) {$s$} (BR);
      \draw (TL)++(0.2,-0.6)-- ++ (0.4,0) -- ++(0,0.4);
\end{tikzpicture}
\]
of categories.
\item[(DD4)] Natural isomorphisms with components
\begin{align*}
\isoa_{X,Y,Z}&\colon (Z\otimes Y)\otimes X \longrightarrow Z\otimes(Y\otimes 
X),\\
\isol_X&\colon I_B\otimes X\longrightarrow X, \\
\isor_X&\colon X\otimes I_A\longrightarrow X
\end{align*}
in $\dcat{D}_1$, where $(Z,Y,X)\in \dcat{D}_3$
which is the pullback
\[
\begin{tikzpicture}[baseline=-\the\dimexpr\fontdimen22\textfont2\relax ]
      \node (TL) at (0,1.5)  {$\dcat{D}_3$};
      \node (TR) at (1.5,1.5)  {$\dcat{D}_1$};
      \node (BL) at (0,0) {$\dcat{D}_2$};
      \node (BR) at (1.5,0) {$\dcat{D}_0$};
      \draw[->] (TL) to node[auto,labelsize](T) {} (TR);
      \draw[->] (TR) to node[auto,labelsize] {$t$} (BR);
      \draw[->] (TL) to node[auto,swap,labelsize] {} (BL);
      \draw[->] (BL) to node[auto,swap,labelsize](B) {$s\circ \pi_2$} (BR);
      \draw (TL)++(0.2,-0.6)-- ++ (0.4,0) -- ++(0,0.4);
\end{tikzpicture}
\]
of categories and $X\in \dcat{D}_1$ with $s(X)=A$ and $t(X)=B$.
\end{description}
These data are subject to the following axioms.
\begin{description}
\item[(DA1)] The diagrams
\[
\begin{tikzpicture}[baseline=-\the\dimexpr\fontdimen22\textfont2\relax ]
      \node (TL) at (0,1.5)  {$\dcat{D}_0$};
      \node (BL) at (0,0) {$\dcat{D}_1$};
      \node (BR) at (1.5,0) {$\dcat{D}_0$};
      \node (BLL) at (-1.5,0) {$\dcat{D}_0$};
      \draw[->] (TL) to node[auto,labelsize](T) {$\id{}$} (BR);
      \draw[->] (TL) to node[auto,labelsize] {$I$} (BL);
      \draw[->] (BL) to node[auto,swap,labelsize](B) {$t$} (BR);
      \draw[->] (TL) to node[auto,swap,labelsize](T) {$\id{}$} (BLL);
      \draw[->] (BL) to node[auto,labelsize](B) {$s$} (BLL);
\end{tikzpicture}
\ \qquad
\begin{tikzpicture}[baseline=-\the\dimexpr\fontdimen22\textfont2\relax ]
      \node (TL) at (0,1.5)  {$\dcat{D}_2$};
      \node (BL) at (0,0) {$\dcat{D}_1$};
      \node (BR) at (1.5,0) {$\dcat{D}_0$};
      \node (BLL) at (-1.5,0) {$\dcat{D}_0$};
      \draw[->] (TL) to node[auto,labelsize](T) {$t\circ \pi_1$} (BR);
      \draw[->] (TL) to node[auto,labelsize] {$\otimes$} (BL);
      \draw[->] (BL) to node[auto,swap,labelsize](B) {$t$} (BR);
      \draw[->] (TL) to node[auto,swap,labelsize](T) {$s\circ \pi_2$} (BLL);
      \draw[->] (BL) to node[auto,labelsize](B) {$s$} (BLL);
\end{tikzpicture}
\]
commute (on the nose).
\item[(DA2)] The morphisms $s(\isoa_{X,Y,Z}),s(\isol_X)$ and 
$s(\isor_X)$
are equal to $\id{A}$ for all $(Z,Y,X)$ $\in\dcat{D}_3$ 
and $X\in\dcat{D}_1$, where $A=s(X)$.
Similarly for $t$.
\item[(DA3)] The coherence axioms (triangle and pentagon) for $\isoa,\isol$ and 
$\isor$.\qedhere
\end{description}
\end{definition}
See \cite[Section~2.1]{Garner_thesis}
for the full details of the definition.
Although Definition~\ref{def:psd_dbl_cat} might look quite different from the  
aforementioned informal description of double categories at the first sight,
in fact it is not, and 
the only difference is the existence of isomorphisms $\isoa,\isol$ and $\isor$
instead of equalities. 
Perhaps it is worth remarking that the functors $s$ and $t$ are meant to assign 
the (horizontal) \emph{sources} and \emph{targets},
so given the diagram
\[
\begin{tikzpicture}[baseline=-\the\dimexpr\fontdimen22\textfont2\relax ]
      \node (TL) at (0,1.5)  {$A$};
      \node (TR) at (2,1.5)  {$B$};
      \node (BL) at (0,0) {$A^\prime$};
      \node (BR) at (2,0) {$B^\prime$};
      \draw[pto] (TL) to node[auto,labelsize](T) {$X$} (TR);
      \draw[->]  (TR) to node[auto,labelsize] {$g$} (BR);
      \draw[->]  (TL) to node[auto,swap,labelsize] {$f$} (BL);
      \draw[pto] (BL) to node[auto,swap,labelsize](B) {$X^\prime$} (BR);
      \tzsquare{1}{0.75}{$\alpha$}
\end{tikzpicture}
\]
in $\dcat{D}$, we read as:
$A$ is the domain of $f$ in $\dcat{D}_0$, $X'$ is the codomain 
of $\alpha$ in $\dcat{D}_1$, $A=s(X)$, $g=t(\alpha)$, and so on. 

We write the isomorphisms $\isoa,\isol$ and $\isor$ as 
\[
\begin{tikzpicture}[baseline=-\the\dimexpr\fontdimen22\textfont2\relax ]
      \node (L) at (0,0)  {$A$};
      \node (T) at (1.5,0.5)  {$B$};
      \node (B) at (1.5,-0.5) {$C$};
      \node (R) at (3,0) {$D$};
      \draw[pto, bend left=10] (L) to node[auto,labelsize]{$X$} (T);
      \draw[pto, bend left=10] (T) to node[auto,labelsize]{$Z\otimes Y$} (R);
      \draw[pto, bend left=-10] (L) to node[auto,swap,labelsize]{$Y\otimes X$} 
      (B);
      \draw[pto, bend left=-10] (B) to node[auto,swap,labelsize]{$Z$} (R);
      \tzsquare{1.5}{0}{$\cong$}
\end{tikzpicture}
\quad
\begin{tikzpicture}[baseline=-\the\dimexpr\fontdimen22\textfont2\relax ]
      \node (L) at (0,0)  {$A$};
      \node (T) at (1.5,0.5)  {$B$};
      \node (R) at (3,0) {$B$};
      \draw[pto, bend left=10] (L) to node[auto,labelsize]{$X$} (T);
      \draw[pto, bend left=10] (T) to node[auto,labelsize]{$I_B$} (R);
      \draw[pto, bend left=-30] (L) to node[auto,swap,labelsize]{$X$} 
      (R);
      \tzsquare{1.5}{0}{$\cong$}
\end{tikzpicture}
\quad
\begin{tikzpicture}[baseline=-\the\dimexpr\fontdimen22\textfont2\relax ]
      \node (L) at (0,0)  {$A$};
      \node (T) at (1.5,0.5)  {$A$};
      \node (R) at (3,0) {$B$};
      \node at (3.05,0) {$\phantom{B}.$};
      \draw[pto, bend left=10] (L) to node[auto,labelsize]{$I_A$} (T);
      \draw[pto, bend left=10] (T) to node[auto,labelsize]{$X$} (R);
      \draw[pto, bend left=-30] (L) to node[auto,swap,labelsize]{$X$} 
      (R);
      \tzsquare{1.5}{0}{$\cong$}
\end{tikzpicture}
\]
The suppression of the vertical morphisms in the above diagrams is 
justified by (DA2).
Similarly we also denote inverses and composites of 
$\isoa,\isol$ and $\isor$ by unnamed double arrows labelled with $\cong$.

\begin{example}[\cite{GP1}]
\label{ex:H_construction}
Let $\tcat{B}$ be a bicategory.
This induces a pseudo double category $\dcat{H}\tcat{B}$, given as follows:
\begin{itemize}
\item an object of $\dcat{H}\tcat{B}$ is an object of $\tcat{B}$;
\item all vertical morphisms of $\dcat{H}\tcat{B}$ are vertical identity 
morphisms;
\item a horizontal morphism of $\dcat{H}\tcat{B}$ is a 1-cell of $\tcat{B}$;
\item a square of $\dcat{H}\tcat{B}$ is a 2-cell of $\tcat{B}$.
\end{itemize} 
The isomorphisms $\isoa,\isol$ and $\isor$ of $\dcat{H}\tcat{B}$ is given by
the corresponding iso-2-cells of $\tcat{B}$.

Conversely, for any pseudo double category $\dcat{D}$,
we obtain a bicategory $\tcat{H}\dcat{D}$ given as follows:
\begin{itemize}
\item an object of $\tcat{H}\dcat{D}$ is an object of $\dcat{D}$;
\item a 1-cell of $\tcat{H}\dcat{D}$ is a horizontal morphism of $\dcat{D}$;
\item a 2-cell of $\tcat{H}\dcat{D}$ is a square in $\dcat{D}$ whose 
horizontal source and target are both vertical identity morphisms.\qedhere
\end{itemize}
\end{example}

Let us introduce the pseudo double category $\PPROF$.
\begin{definition}[{\cite[Section~3.1]{GP1}}]
We define the pseudo double category $\PPROF$ as follows.
\begin{itemize}
\item An object is a large category.
\item A vertical morphism from $\cat{A}$ to $\cat{A'}$ is a functor 
$F\colon \cat{A}\longrightarrow\cat{A'}$.
\item A horizontal morphism from $\cat{A}$ to $\cat{B}$ is 
a profunctor $H\colon\cat{A}\pto \cat{B}$, i.e.,
a functor $H\colon \cat{B}^\op\times\cat{A}\longrightarrow\SET$. 
Horizontal identities and horizontal compositions are the same as 
in Definition~\ref{def:profunctor}.
\item A square as in 
\[
\begin{tikzpicture}[baseline=-\the\dimexpr\fontdimen22\textfont2\relax ]
      \node (TL) at (0,1.5)  {$\cat{A}$};
      \node (TR) at (2,1.5)  {$\cat{B}$};
      \node (BL) at (0,0) {$\cat{A^\prime}$};
      \node (BR) at (2,0) {$\cat{B^\prime}$};
      \draw[pto] (TL) to node[auto,labelsize](T) {$H$} (TR);
      \draw[->]  (TR) to node[auto,labelsize] {$G$} (BR);
      \draw[->]  (TL) to node[auto,swap,labelsize] {$F$} (BL);
      \draw[pto] (BL) to node[auto,swap,labelsize](B) {$H^\prime$} (BR);
      \tzsquare{1}{0.75}{$\alpha$}
\end{tikzpicture}
\] 
is a natural transformation 
\[
\alpha=(\alpha_{B,A}\colon H(B,A)\longrightarrow 
H'(GB,FA))_{B\in\cat{B},A\in\cat{A}},
\]
that is, of type
\[
\begin{tikzpicture}[baseline=-\the\dimexpr\fontdimen22\textfont2\relax ]
      \node (L) at (0,0)  {$\cat{B}^\op\times\cat{A}$};
      \node (T) at (2.5,-0.7)  {$\cat{B'}^\op\times\cat{A'}$};
      \node (R) at (5,0) {$\SET$.};
      \draw[->, bend left=-5] (L) to node[auto,swap,labelsize]{$G^\op\times F$} 
      (T);
      \draw[->, bend left=-5] (T) to node[auto,swap,labelsize]{$H'$} (R);
      \draw[->, bend left=20] (L) to node[auto,labelsize]{$H$} 
      (R);
      \tzsquare{2.5}{0}{$\alpha$}
\end{tikzpicture}
\]
\end{itemize}
It is straightforward to define various compositions of these morphisms and 
squares. 
The isomorphisms $\isoa,\isol$ and $\isor$ are the same as those in
the bicategory $\PROF$.
(Indeed, using the construction introduced in Example~\ref{ex:H_construction}, 
$\PROF=\tcat{H}\PPROF$.)
\end{definition}

Given a pseudo double category $\dcat{D}$,
denote by $\dcat{D}^\op$, $\dcat{D}^\co$, and $\dcat{D}^\coop$
the pseudo double categories obtained from $\dcat{D}$ by reversing
the horizontal direction (swapping $s$ and $t$), reversing the 
vertical direction (taking the opposites of $\dcat{D}_0$ and $\dcat{D}_1$)
and reversing both the horizontal and vertical directions,
respectively.\footnote{For a double category $\dcat{D}$ we also have the 
\emph{transpose} $\dcat{D}^\mathrm{tr}$, obtained from $\dcat{D}$ by swapping 
the horizontal and vertical directions.
However, in pseudo double categories the horizontal direction and the vertical 
direction are not symmetric and we no longer have this duality for them.}
In the following we shall mainly work within $\PPROF^\op$,
though most of the diagrams are symmetric in the horizontal direction and 
this makes little difference.
(In fact, the pseudo double category defined in \cite[Section~3.1]{GP1} 
amounts to our $\PPROF^\op$, because our convention on the direction of 
profunctors differs from theirs.)

\section{The universality of categories of models}

Let $\cat{M}$ be a metatheory, $\monoid{T}=(T,e,m)$ be a theory in $\cat{M}$,
$\cat{C}$ be a large category, and $\Phi=(\Phi, 
\overline{\phi}_\cdot,\overline{\phi})$ be a metamodel of $\cat{M}$ in 
$\cat{C}$.
Recall from Section~\ref{subsec:metamodel} that in the data of 
the metamodel 
$(\Phi,\overline{\phi}_\cdot,\overline{\phi})$,
the natural transformations 
\[
((\overline{\phi}_\cdot)_C\colon 1\longrightarrow \Phi_I(C,C))_{C\in\cat{C}}
\]
and 
\[
((\overline{\phi}_{X,Y})_{A,B,C}\colon \Phi_Y(B,C)\times 
\Phi_X(A,B)\longrightarrow \Phi_{Y\otimes 
X}(A,C))_{X,Y\in\cat{M},A,B,C\in\cat{C}},
\]
may be replaced by
the natural transformations 
\[
((\phi_\cdot)_{A,B}\colon\cat{C}(A,B)\longrightarrow 
\Phi_I(A,B))_{A,B\in\cat{C}}
\]
and
\[
((\phi_{X,Y})_{A,B}\colon (\Phi_Y \ptensorrev 
\Phi_X)(A,B)\longrightarrow\Phi_{Y\otimes 
X}(A,B))_{X,Y\in\cat{M},A,B\in\cat{C}},
\]
respectively.
In this chapter we shall mainly use the expression of 
metamodel via the data $(\Phi,\phi_\cdot,\phi)$.
The category of models $\Mod{\monoid{T}}{(\cat{C},\Phi)}$,
henceforth abbreviated as $\Mod{\monoid{T}}{\cat{C}}$,
defined in Definition~\ref{def:metamodel_model} admits a canonical 
forgetful functor $U\colon 
\Mod{\monoid{T}}{\cat{C}}\longrightarrow\cat{C}$
and a natural transformation (a square in $\PPROF^\op$) $u$ as in
\[
\begin{tikzpicture}[baseline=-\the\dimexpr\fontdimen22\textfont2\relax ]
      \node (TL) at (0,0.75)  {$\Mod{\monoid{T}}{\cat{C}}$};
      \node (TR) at (5,0.75)  {$\Mod{\monoid{T}}{\cat{C}}$};
      \node (BL) at (0,-0.75) {$\cat{C}$};
      \node (BR) at (5,-0.75) {$\cat{C}.$};
      \draw[->]  (TL) to node[auto,swap,labelsize] {$U$} (BL);
      \draw[pto]  (TL) to node[auto,labelsize] 
      {$\Mod{\monoid{T}}{\cat{C}}(-,-)$} (TR);
      \draw[pto]  (BL) to node[auto,swap,labelsize] {$\Phi_T$} (BR);
      \draw[->]  (TR) to node[auto,labelsize](B) {$U$} (BR);
      \tzsquare{2.5}{0}{$u$}
\end{tikzpicture} 
\] 
Concretely, $u$ is a natural transformation 
\[
(u_{(C,\xi),(C',\xi')}\colon 
\Mod{\monoid{T}}{\cat{C}}((C,\xi),(C',\xi'))\longrightarrow
\Phi_T(C,C'))_{(C,\xi),(C',\xi')\in\Mod{\monoid{T}}{\cat{C}}}
\]
whose $((C,\xi),(C',\xi'))$-th component 
maps each morphism $f\colon (C,\xi)\longrightarrow (C',\xi')$ in 
$\Mod{\monoid{T}}{\cat{C}}$ to the element 
$\Phi_T(C,f)(\xi)=\Phi_T(f,C')(\xi')\in \Phi_T(C,C')$.
Alternatively, by the Yoneda lemma, $u$ may be equivalently given as 
a natural transformation 
\[
(\overline{u}_{(C,\xi)}\colon 
1\longrightarrow
\Phi_T(C,C))_{(C,\xi)\in\Mod{\monoid{T}}{\cat{C}}}
\]
whose $(C,\xi)$-th component maps the unique element of $1$ to 
$\xi\in\Phi_T(C,C)$.

We claim that the triple $(\Mod{\monoid{T}}{\cat{C}},U,u)$ 
has a certain universal property.
\begin{definition}
\label{def:double_cone}
Define a \defemph{vertical double cone over $\Phi(\monoid{T})$}
to be a triple $(\cat{A},V,v)$ consisting of a large category $\cat{A}$,
a functor $V\colon\cat{A}\longrightarrow\cat{C}$, and a square $v$ in 
$\PPROF^\op$ of type 
\[
\begin{tikzpicture}[baseline=-\the\dimexpr\fontdimen22\textfont2\relax ]
      \node (TL) at (0,0.75)  {$\cat{A}$};
      \node (TR) at (2,0.75)  {$\cat{A}$};
      \node (BL) at (0,-0.75) {$\cat{C}$};
      \node (BR) at (2,-0.75) {$\cat{C},$};
      \draw[->]  (TL) to node[auto,swap,labelsize] {$V$} (BL);
      \draw[pto]  (TL) to node[auto,labelsize] 
      {$\cat{A}(-,-)$} (TR);
      \draw[pto]  (BL) to node[auto,swap,labelsize] {$\Phi_T$} (BR);
      \draw[->]  (TR) to node[auto,labelsize](B) {$V$} (BR);
      \tzsquare{1}{0}{$v$}
\end{tikzpicture} 
\]
satisfying the following equations:
\begin{equation}
\label{eqn:cone_e}
\begin{tikzpicture}[baseline=-\the\dimexpr\fontdimen22\textfont2\relax ]
      \node (TL) at (0,0.75)  {$\cat{A}$};
      \node (TR) at (2,0.75)  {$\cat{A}$};
      \node (BL) at (0,-0.75) {$\cat{C}$};
      \node (BR) at (2,-0.75) {$\cat{C}$};
      \draw[pto] (TL) to node[auto,labelsize](T) {$\cat{A}(-,-)$} 
      (TR);
      \draw[->]  (TR) to node[auto,labelsize] {$V$} (BR);
      \draw[->]  (TL) to node[auto,swap,labelsize] {$V$} 
      (BL);
      \draw[pto, bend left=30] (BL) to node[auto,labelsize](B) 
      {$\Phi_T$} 
      (BR);
      \draw[pto, bend right=30] (BL) to node[auto,swap,labelsize](B) 
      {$\Phi_I$} 
      (BR);
      \tzsquare{1}{0.25}{$v$}
      \tzsquare{1}{-0.75}{$\Phi_e$}
\end{tikzpicture}
\quad=\quad
\begin{tikzpicture}[baseline=-\the\dimexpr\fontdimen22\textfont2\relax ]
      \node (TL) at (0,0.75)  {$\cat{A}$};
      \node (TR) at (2,0.75)  {$\cat{A}$};
      \node (BL) at (0,-0.75) {$\cat{C}$};
      \node (BR) at (2,-0.75) {$\cat{C}$};
      \draw[pto] (TL) to node[auto,labelsize](T) {$\cat{A}(-,-)$} 
      (TR);
      \draw[->]  (TR) to node[auto,labelsize] {$V$} (BR);
      \draw[->]  (TL) to node[auto,swap,labelsize] {$V$} 
      (BL);
      \draw[pto, bend left=30] (BL) to node[auto,labelsize](B) {$\cat{C}(-,-)$} 
      (BR);
      \draw[pto, bend right=30] (BL) to node[auto,swap,labelsize](B) 
      {$\Phi_I$} 
      (BR);
      \tzsquare{1}{0.25}{$I_{V}$}
      \tzsquare{1}{-0.75}{$\phi_\cdot$}
\end{tikzpicture}
\end{equation}
\begin{equation}
\label{eqn:cone_m}
\begin{tikzpicture}[baseline=-\the\dimexpr\fontdimen22\textfont2\relax ]
      \node (TL) at (0,0.75)  {$\cat{A}$};
      \node (TR) at (2,0.75)  {$\cat{A}$};
      \node (BL) at (0,-0.75) {$\cat{C}$};
      \node (BR) at (2,-0.75) {$\cat{C}$};
      \draw[pto] (TL) to node[auto,labelsize](T) {$\cat{A}(-,-)$} 
      (TR);
      \draw[->]  (TR) to node[auto,labelsize] {$V$} (BR);
      \draw[->]  (TL) to node[auto,swap,labelsize] {$V$} 
      (BL);
      \draw[pto, bend left=30] (BL) to node[auto,labelsize](B) 
      {$\Phi_T$} 
      (BR);
      \draw[pto, bend right=30] (BL) to node[auto,swap,labelsize](B) 
      {$\Phi_{T\otimes T}$} 
      (BR);
      \tzsquare{1}{0.25}{$v$}
      \tzsquare{1}{-0.75}{$\Phi_m$}
\end{tikzpicture}
\quad=\quad
\begin{tikzpicture}[baseline=-\the\dimexpr\fontdimen22\textfont2\relax ]
      \node (TL) at (0,0.75)  {$\cat{A}$};
      \node (TR) at (2,0.75)  {$\cat{A}$};
      \node (BL) at (0,-0.75) {$\cat{C}$};
      \node (BR) at (2,-0.75) {$\cat{C}$};
      \node (TRR) at (4,0.75)  {$\cat{A}$};
      \node (BRR) at (4,-0.75) {$\cat{C}$};
      \node at (4.05,-0.75) {$\phantom{\cat{C}}$.};
      \draw[pto] (TL) to node[auto,labelsize](T) {$\cat{A}(-,-)$} 
      (TR);
      \draw[pto] (TR) to node[auto,labelsize](T2) {$\cat{A}(-,-)$} 
      (TRR);
      \draw[->]  (TR) to node[auto,labelsize] {$V$} (BR);
      \draw[->]  (TRR) to node[auto,labelsize] {$V$} (BRR);
      \draw[->]  (TL) to node[auto,swap,labelsize] {$V$} 
      (BL);
      \draw[pto] (BL) to node[auto,labelsize](B) {$\Phi_T$} 
      (BR);
      \draw[pto] (BR) to node[auto,labelsize](B) {$\Phi_T$} 
      (BRR);
      \draw[pto, bend right=30] (BL) to node[auto,swap,labelsize](B) 
      {$\Phi_{T\otimes T}$} 
      (BRR);
      \draw[pto, bend left=30] (TL) to node[auto,labelsize](B) 
      {$\cat{A}(-,-)$} 
      (TRR);
      \tzsquare{1}{0}{$v$}
      \tzsquare{3}{0}{$v$}
      \tzsquare{2}{-1.1}{$\phi_{T,T}$}
      \tzsquare{2}{1.1}{$\cong$}
\end{tikzpicture}
\end{equation}
\end{definition}
Using this notion, we can state the universal property of 
the triple $(\Mod{\monoid{T}}{\cat{C}},$ $U,u)$, just as in the case of 
Eilenberg--Moore categories.

\begin{thm}
\label{thm:double_categorical_univ_property}
Let $\cat{M}=(\cat{M},I,\otimes)$ be a metatheory, $\monoid{T}=(T,e,m)$ be a 
theory in $\cat{M}$, $\cat{C}$ be a large category, and 
$\Phi=(\Phi,\phi_\cdot,\phi)$ be a metamodel of $\cat{M}$ in $\cat{C}$.
The triple $(\Mod{\monoid{T}}{\cat{C}},U,u)$ defined above is 
a universal vertical double cone over $\Phi(\monoid{T})$,
namely:\begin{enumerate}
\item it is a vertical double cone over $\Phi(\monoid{T})$;
\item for any vertical double cone $(\cat{A},V,v)$ over $\Phi(\monoid{T})$,
there exists a unique functor $K\colon\cat{A}\longrightarrow
\Mod{\monoid{T}}{\cat{C}}$ such that 
\[
\begin{tikzpicture}[baseline=-\the\dimexpr\fontdimen22\textfont2\relax ]
      \node (TL) at (0,0.75)  {$\cat{A}$};
      \node (TR) at (2,0.75)  {$\cat{A}$};
      \node (BL) at (0,-0.75) {$\cat{C}$};
      \node (BR) at (2,-0.75) {$\cat{C}$};
      \draw[pto] (TL) to node[auto,labelsize](T) {$\cat{A}(-,-)$} 
      (TR);
      \draw[->]  (TR) to node[auto,labelsize] {$V$} (BR);
      \draw[->]  (TL) to node[auto,swap,labelsize] {$V$} 
      (BL);
      \draw[pto] (BL) to node[auto,swap,labelsize](B) {$\Phi_T$} (BR);
      \tzsquare{1}{0}{$v$}
\end{tikzpicture}  
\quad=\quad
\begin{tikzpicture}[baseline=-\the\dimexpr\fontdimen22\textfont2\relax ]
      \node (TL) at (0,1.5)  {$\cat{A}$};
      \node (TR) at (5,1.5)  {$\cat{A}$};
      \node (BL) at (0,0) {$\Mod{\monoid{T}}{\cat{C}}$};
      \node (BR) at (5,0) {$\Mod{\monoid{T}}{\cat{C}}$};
      \node (BBL) at (0,-1.5) {$\cat{C}$};
      \node (BBR) at (5,-1.5) {$\cat{C}$};
      \draw[pto] (TL) to node[auto,labelsize](T) {${\cat{A}}(-,-)$} 
      (TR);
      \draw[->]  (TR) to node[auto,labelsize] {$K$} (BR);
      \draw[->]  (TL) to node[auto,swap,labelsize] {$K$} 
      (BL);
      \draw[pto] (BL) to node[auto,labelsize](B) 
      {$\Mod{\monoid{T}}{\cat{C}}(-,-)$} 
      (BR);
      \draw[->]  (BR) to node[auto,labelsize] {$U$} (BBR);
      \draw[->]  (BL) to node[auto,swap,labelsize] {$U$} (BBL);
      \draw[pto] (BBL) to node[auto,swap,labelsize](B) {$\Phi_T$} (BBR);
      \tzsquare{2.5}{0.75}{$I_{K}$}
      \tzsquare{2.5}{-0.75}{$u$} 
\end{tikzpicture} 
\]
holds;
\item for any pair of vertical cones $(\cat{A},V,v)$ and $(\cat{A'},V',v')$
over $\Phi(\monoid{T})$, any horizontal morphism $H\colon \cat{A}\pto \cat{A'}$
in $\PPROF^\op$ and any square 
\[
\begin{tikzpicture}[baseline=-\the\dimexpr\fontdimen22\textfont2\relax ]
      \node (TL) at (0,0.75)  {$\cat{A}$};
      \node (TR) at (2,0.75)  {$\cat{A'}$};
      \node (BL) at (0,-0.75) {$\cat{C}$};
      \node (BR) at (2,-0.75) {$\cat{C}$};
      \draw[pto] (TL) to node[auto,labelsize](T) {$H$} (TR);
      \draw[->]  (TR) to node[auto,labelsize] {$V'$} (BR);
      \draw[->]  (TL) to node[auto,swap,labelsize] {$V$} (BL);
      \draw[pto] (BL) to node[auto,swap,labelsize](B) {$\cat{C}(-,-)$} (BR);
      \tzsquare{1}{0}{$\theta$}
\end{tikzpicture}  
\]
in $\PPROF^\op$ such that 
\begin{equation}
\label{eqn:H_theta}
\begin{tikzpicture}[baseline=-\the\dimexpr\fontdimen22\textfont2\relax ]
      \node (TL) at (0,0.75)  {$\cat{A}$};
      \node (TM) at (2,0.75)  {$\cat{A}$};
      \node (TR) at (4,0.75)  {$\cat{A'}$};
      \node (BL) at (0,-0.75) {$\cat{C}$};
      \node (BM) at (2,-0.75) {$\cat{C}$};
      \node (BR) at (4,-0.75) {$\cat{C}$};
      \draw[pto, bend left=30] (TL) to node[auto,labelsize](T) {$H$} (TR);
      \draw[pto] (TL) to node[auto,labelsize](T) {$\cat{A}(-,-)$} (TM);
      \draw[pto] (TM) to node[auto,labelsize](T) {$H$} (TR);
      \draw[->]  (TL) to node[auto,swap,labelsize] {$V$} (BL);
      \draw[->]  (TM) to node[auto,labelsize] {$V$} (BM);
      \draw[->]  (TR) to node[auto,labelsize] {$V'$} (BR);
      \draw[pto] (BL) to node[auto,swap,labelsize](B) {$\Phi_T$} (BM);
      \draw[pto] (BM) to node[auto,swap,labelsize](B) {$\cat{C}(-,-)$} (BR);
      \draw[pto, bend left=-30] (BL) to node[auto,swap,labelsize](T) {$\Phi_T$} 
      (BR);
      \tzsquare{2}{1.1}{$\cong$}
      \tzsquare{1}{0}{$v$}
      \tzsquare{3}{0}{$\theta$}
      \tzsquare{2}{-1.1}{$\cong$}
\end{tikzpicture}
\ =\ 
\begin{tikzpicture}[baseline=-\the\dimexpr\fontdimen22\textfont2\relax ]
      \node (TL) at (0,0.75)  {$\cat{A}$};
      \node (TM) at (2,0.75)  {$\cat{A'}$};
      \node (TR) at (4,0.75)  {$\cat{A'}$};
      \node (BL) at (0,-0.75) {$\cat{C}$};
      \node (BM) at (2,-0.75) {$\cat{C}$};
      \node (BR) at (4,-0.75) {$\cat{C}$};
      \draw[pto, bend left=30] (TL) to node[auto,labelsize](T) {$H$} (TR);
      \draw[pto] (TL) to node[auto,labelsize](T) {$H$} (TM);
      \draw[pto] (TM) to node[auto,labelsize](T) {$\cat{A'}(-,-)$} (TR);
      \draw[->]  (TL) to node[auto,swap,labelsize] {$V$} (BL);
      \draw[->]  (TM) to node[auto,labelsize] {$V'$} (BM);
      \draw[->]  (TR) to node[auto,labelsize] {$V'$} (BR);
      \draw[pto] (BL) to node[auto,swap,labelsize](B) {$\cat{C}(-,-)$} (BM);
      \draw[pto] (BM) to node[auto,swap,labelsize](B) {$\Phi_T$} (BR);
      \draw[pto, bend left=-30] (BL) to node[auto,swap,labelsize](T) {$\Phi_T$} 
      (BR);
      \tzsquare{2}{1.1}{$\cong$}
      \tzsquare{1}{0}{$\theta$}
      \tzsquare{3}{0}{$v'$}
      \tzsquare{2}{-1.1}{$\cong$}
\end{tikzpicture}  
\end{equation}
holds, there exists a unique square 
\[
\begin{tikzpicture}[baseline=-\the\dimexpr\fontdimen22\textfont2\relax ]
      \node (TL) at (0,0.75)  {$\cat{A}$};
      \node (TR) at (5,0.75)  {$\cat{A'}$};
      \node (BL) at (0,-0.75) {$\Mod{\monoid{T}}{\cat{C}}$};
      \node (BR) at (5,-0.75) {$\Mod{\monoid{T}}{\cat{C}}$};
      \draw[pto] (TL) to node[auto,labelsize](T) {$H$} (TR);
      \draw[->]  (TR) to node[auto,labelsize] {$K'$} (BR);
      \draw[->]  (TL) to node[auto,swap,labelsize] {$K$} (BL);
      \draw[pto] (BL) to node[auto,swap,labelsize](B) 
      {$\Mod{\monoid{T}}{\cat{C}}(-,-)$} (BR);
      \tzsquare{2.5}{0}{$\sigma$}
\end{tikzpicture}  
\]
in $\PPROF^\op$ such that 
\[
\begin{tikzpicture}[baseline=-\the\dimexpr\fontdimen22\textfont2\relax ]
      \node (TL) at (0,0.75)  {$\cat{A}$};
      \node (TR) at (2,0.75)  {$\cat{A'}$};
      \node (BL) at (0,-0.75) {$\cat{C}$};
      \node (BR) at (2,-0.75) {$\cat{C}$};
      \draw[pto] (TL) to node[auto,labelsize](T) {$H$} (TR);
      \draw[->]  (TR) to node[auto,labelsize] {$V'$} (BR);
      \draw[->]  (TL) to node[auto,swap,labelsize] {$V$} (BL);
      \draw[pto] (BL) to node[auto,swap,labelsize](B) {$\cat{C}(-,-)$} (BR);
      \tzsquare{1}{0}{$\theta$}
\end{tikzpicture}  
\quad =\quad
\begin{tikzpicture}[baseline=-\the\dimexpr\fontdimen22\textfont2\relax ]
      \node (TL) at (0,1.5)  {$\cat{A}$};
      \node (TR) at (5,1.5)  {$\cat{A'}$};
      \node (ML) at (0,0) {$\Mod{\monoid{T}}{\cat{C}}$};
      \node (MR) at (5,0) {$\Mod{\monoid{T}}{\cat{C}}$};
      \node (BL) at (0,-1.5) {${\cat{C}}$};
      \node (BR) at (5,-1.5) {${\cat{C}}$};
      \draw[pto] (TL) to node[auto,labelsize](T) {$H$} (TR);
      \draw[->]  (TR) to node[auto,labelsize] {$K'$} (MR);
      \draw[->]  (MR) to node[auto,labelsize] {$U$} (BR);
      \draw[->]  (ML) to node[auto,swap,labelsize] {$U$} (BL);
      \draw[->]  (TL) to node[auto,swap,labelsize] {$K$} (ML);
      \draw[pto] (ML) to node[auto,swap,labelsize](B) 
      {$\Mod{\monoid{T}}{\cat{C}}(-,-)$} (MR);
      \draw[pto] (BL) to node[auto,swap,labelsize]{${\cat{C}}(-,-)$} (BR);
      \tzsquare{2.5}{0.75}{$\sigma$}
      \tzsquare{2.5}{-0.75}{$I_{U}$}
\end{tikzpicture}  
\]
holds, where $K$ and $K'$ are the functors corresponding to $(\cat{A},V,v)$ and 
$(\cat{A'},V',v')$ respectively.
\end{enumerate}
\end{thm}
The above statements are taken from the definition of double limit
\cite[Section~4.2]{GP1}.
\begin{proof}[Proof of Theorem~\ref{thm:double_categorical_univ_property}]
First, that $(\Mod{\monoid{T}}{\cat{C}},U,u)$ is a vertical double cone over 
$\Phi(\monoid{T})$ follows directly from the definition of model of $\monoid{T}$
in $\cat{C}$ with respect to $\Phi$ (Definition~\ref{def:metamodel_model}).

Given a vertical double cone $(\cat{A},V,v)$ over $\Phi(\monoid{T})$,
for each object $A\in\cat{A}$, the pair $(VA,v_{A,A}(\id{A}))$ is a
$\monoid{T}$-model in $\cat{C}$ with respect to $\Phi$,
and for each morphism $f\colon A\longrightarrow A'$ in $\cat{A}$,
the morphism $Vf$ is a $\monoid{T}$-model homomorphism from 
$(VA,v_{A,A}(\id{A}))$ to $(VA',v_{A',A'}(\id{A'}))$.
The functor $K\colon \cat{A}\longrightarrow \Mod{\monoid{T}}{\cat{C}}$
can therefore be given as $KA=(VA,v_{A,A}(\id{A}))$ and $Kf=Vf$.
The uniqueness is clear. 

Finally, given $H$ and $\theta$ as in the third clause, the equation 
(\ref{eqn:H_theta}) says that for each $A\in\cat{A}$, $A'\in\cat{A'}$ 
and $x\in H(A,A')$, the 
morphism $\theta_{A,A'}(x)\colon VA\longrightarrow V'A'$
in $\cat{C}$ satisfies $\Phi_T(VA,\theta_{A,A'}(x)) (v_{A,A}(\id{A}))= 
\Phi_T(\theta_{A,A'}(x),VA')(v'_{A',A'}(\id{A'}))$;
in other words, that $\theta_{A,A'}(x)$ is a $\monoid{T}$-model homomorphism
from $KA$ to $K'A'$.
The square $\sigma$ can then be given as the natural transformation 
with $\sigma_{A,A'}(x)=\theta_{A,A'}(x)$.
\end{proof}

\section{Relation to double limits}
In this final section of this chapter, we sketch how the double categorical 
universal property (Theorem~\ref{thm:double_categorical_univ_property}) of 
categories of models in our framework can be expressed 
via the notion of double limit \cite{GP1}, connecting our characterisation 
to a well-established notion.
A short outline of this reduction is as follows.
\begin{enumerate}
\item A theory $\monoid{T}$ in a metatheory $\cat{M}$ may be equivalently given 
as a strong monoidal functor $\monoid{T}\colon\Simp\longrightarrow\cat{M}$,
where $\Simp$ is the augmented simplex category with monoidal structure given 
by ordinal sum; see Definition~\ref{def:simp}.
\item A metamodel $\Phi$ of a metatheory $\cat{M}$
may be identified with a lax double functor 
$\Phi\colon\dcat{H}\Sigma(\cat{M}^\op)\longrightarrow\PPROF^\op$,
where $\Sigma$ turns a monoidal category to the corresponding one-object 
bicategory and $\dcat{H}$ turns a bicategory to the corresponding vertically 
discrete pseudo double category (see Example~\ref{ex:H_construction}).
%Explicitly, $\dcat{H}\Sigma (\cat{M}^\op)$ has one object $\ast$, one vertical 
%morphism 
%(the vertical identity of $\ast$), its horizontal morphisms correspond to 
%objects of $\cat{M}^\op$, and its squares correspond to morphisms of 
%$\cat{M}^\op$.
%If the metamodel $\Phi$ is in a large category $\cat{C}$,
%then the lax double functor 
%$\dcat{H}\Sigma(\cat{M}^\op)\longrightarrow\PPROF^\op$ maps the unique object 
%$\ast$ of $\dcat{H}\Sigma(\cat{M}^\op)$ to $\cat{C}\in\PPROF^\op$.
\item Therefore given a theory $\monoid{T}$ and a metamodel $\Phi$ (in 
$\cat{C}$) of a metatheory $\cat{M}$, we obtain a lax double functor 
$\Phi(\monoid{T})\colon \dcat{H}\Sigma(\Simp^\op)\longrightarrow\PPROF^\op$
as the following composition:
\[
\begin{tikzpicture}[baseline=-\the\dimexpr\fontdimen22\textfont2\relax ]
      \node (L) at (0,0) {$\dcat{H}\Sigma(\Simp^\op)$};
      \node (M) at (3,0) {$\dcat{H}\Sigma(\cat{M}^\op)$};
      \node (R) at (6,0) {$\PPROF^\op$.};
      \draw[->] (L) to node[auto,labelsize]{$\dcat{H}\Sigma(\monoid{T}^\op)$} 
      (M);
      \draw[->] (M) to node[auto,labelsize] {$\Phi$} (R);
\end{tikzpicture}  
\]
Theorem~\ref{thm:double_categorical_univ_property} may then be interpreted as 
establishing that $\Mod{\monoid{T}}{\cat{C}}$ is (the apex of) the double limit 
of $\Phi(\monoid{T})$ in the sense of \cite{GP1}.
\end{enumerate}
We remark that the 2-categorical universal property of Eilenberg--Moore 
categories (Section~\ref{sec:EM_cat_as_2-lim}) can also be interpreted as 
establishing $\cat{C}^\monoid{T}$ as (the apex of) the
lax limit of the 2-functor of type $\Sigma\Simp\longrightarrow\tCAT$ 
corresponding to a monad $\monoid{T}$ on a large category $\cat{C}$; 
see~\cite{Street_limits}.
The following reduction is essentially routine and rather peripheral, so those 
readers contented with the above outline may safely skip the rest of this 
section.

\medskip

We start from the first step, namely a well-known observation (see e.g., 
\cite[Section~VII.~5]{MacLane_CWM})
that monoid objects (= theories) may be identified with
strong monoidal functors out of $\Simp$.
\begin{definition}
\label{def:simp}
We define the \defemph{augmented simplex category} (also known as the
\defemph{algebraists' simplex category})
$\Simp$ as follows.
\begin{itemize}
\item Objects are all finite ordinals $\ord{n}=\{\,{0}<{1}<\dots 
<{n-1} \,\}$,
including the empty ordinal $\ord{0}=\{\ \}$.
\item Morphisms are all monotone functions.
\end{itemize}
Note that a morphism in $\Simp$ is mono (resp.~epi)
iff it is an injective (resp.~surjective) monotone function.

This category has a natural monoidal structure, given as follows.
\begin{itemize}
\item The unit object is $\ord{0}$.
\item The monoidal product $+\colon 
\Simp\times\Simp\longrightarrow\Simp$ maps a pair of objects 
$\ord{n}$ and $\ord{m}$ in $\Simp$ to $\ord{n+m}$, and maps a pair of morphisms 
$u\colon \ord{n}\longrightarrow \ord{n'}$ and 
$v\colon \ord{m}\longrightarrow\ord{m'}$ in $\Simp$ to 
$u+v\colon \ord{n+m}\longrightarrow\ord{n'+m'}$ defined as
\[
(u+v) ({i}) =\begin{cases}
u(i) & \text{if }i\leq n-1\\
n'+v(i) & \text{if }i\geq n.
\end{cases}
\]
\end{itemize}
In the following, whenever we talk about a monoidal structure on $\Simp$,
we always mean this (strict) monoidal structure $(\ord{0},+)$.
\end{definition}

The morphisms in the category $\Simp$ are generated by 
certain simple morphisms.
For each $\ord{n}\in\Simp$ and $i\in\{0,\dots, n\}$, 
let $\delta^{(n)}_i\colon \ord{n}\longrightarrow\ord{n+1}$ be the 
unique injective monotone function 
whose image does not contain $i\in\ord{n+1}$, and for each
$\ord{n}\in\Simp$ and $i\in \{0,\dots,n-1\}$, let 
$\sigma^{(n)}_i\colon\ord{n}\longrightarrow\ord{n-1}$ be the unique 
surjective monotone function such that 
$\sigma^{(n)}_i(i)=\sigma^{(n)}_i(i+1)=i$.
Morphisms of the form $\delta^{(n)}_i$ are called \defemph{face maps}
and those of the form $\sigma^{(n)}_i$ \defemph{degeneracy maps}.
It is easy to see that every monomorphism in $\Simp$
can be expressed as a composition of face maps\footnote{An identity morphism in 
$\Simp$ is interpreted as the result of 0-ary composition of morphisms.}, and 
every epimorphism in $\Simp$ as a composition of degeneracy maps.
Furthermore, an arbitrary morphism in $\Simp$ can be 
written uniquely as the composition of an epimorphism followed by a 
monomorphism (the {image factorisation}). 
Hence all morphisms in $\Simp$ can be written as a composition of 
face and degeneracy maps.
This means that an arbitrary \emph{functor} from $\Simp$ to a category 
is completely determined by its images of all objects in $\Simp$
and face and degeneracy maps.
Conversely, such an assignment of the images of objects and face and degeneracy 
maps extends to a functor if and only if it satisfies the well-known
\emph{simplicial identities}; see  \cite[Section~VII.~5]{MacLane_CWM}.

Moreover, if we take into account the monoidal structure of $\Simp$,
we can further cut down the generating data.
Clearly, every object in $\Simp$ is written as the monoidal product of finitely 
many copies of $\ord{1}$.
Consider the unique morphism ${!_\ord{0}}\colon \ord{0}\longrightarrow 
\ord{1}$ in $\Simp$. 
Every face map $\delta^{(n)}_i\colon \ord{n}\longrightarrow\ord{n+1}$
can be written as $\id{\ord{i}}+{!_\ord{0}}+\id{\ord{n-i}}$ using this morphism
and the monoidal product.
Similarly, using the unique morphism ${!_\ord{2}}\colon 
\ord{2}\longrightarrow\ord{1}$, every degeneracy map $\sigma^{(n)}_i\colon 
\ord{n}\longrightarrow\ord{n-1}$ can be written as $\id{\ord{i}}+{!_{\ord{2}}}+
\id{\ord{n-i-2}}$.
Hence every \emph{strict monoidal functor} of type 
$F\colon\Simp\longrightarrow\cat{M}$ to a strict monoidal category 
$(\cat{M},I,\otimes)$
is completely determined by the object $T=F(\ord{1})\in\cat{M}$
and the morphisms $e=F({!_\ord{0}})\colon I\longrightarrow T$ and 
$m=F({!_\ord{2}})\colon T\otimes T\longrightarrow T$ in $\cat{M}$.
It turns out that, conversely, such a data $(T,e,m)$ defines a strict monoidal 
functor if and only if $(T,e,m)$ is a monoid object in $\cat{M}$.

The following proposition is a mild variant of this.

\begin{proposition}
Let $\cat{M}=(\cat{M},I,\otimes)$ be a metatheory.
There is an equivalence of categories 
between the category $\MonCATst(\Simp,\cat{M})$ of all strong monoidal functors
$\Simp\longrightarrow\cat{M}$ and monoidal natural transformations,
and the category $\Th{\cat{M}}$.
\end{proposition}
\begin{proof}
Recall that a strong monoidal functor $(F,f_\cdot,f)\colon 
\Simp\longrightarrow\cat{M}$ consists of a functor $F\colon 
\Simp\longrightarrow\cat{M}$, an isomorphism $f_\cdot\colon I\longrightarrow 
F(\ord{0})$ and a natural isomorphism 
\[
f=(f_{\ord{m},\ord{n}}\colon F(\ord{n})\otimes F(\ord{m})\longrightarrow 
F(\ord{n+m}))_{\ord{m},\ord{n}\in\Simp}
\]
satisfying the suitable axioms. The functor
\[
\MonCATst(\Simp,\cat{M})\longrightarrow \Th{\cat{M}}
\]
mapping an object $(F,f_\cdot,f)\in\MonCATst(\Simp,\cat{M})$ to $(F(\ord{1}),
F({!_\ord{0}})\circ f_\cdot, F({!_{\ord{2}}})\circ 
f_{\ord{1},\ord{1}})$
and a morphism $\phi\colon (F,f_\cdot,f)\longrightarrow (G,g_\cdot,g)$
in $\MonCATst(\Simp,\cat{M})$ to $\phi_\ord{1}$ 
is well-defined and is an equivalence of categories.
\end{proof}

\medskip

The second step, that a metamodel of a metatheory $\cat{M}$
corresponds to a lax double functor of type 
$\dcat{H}\Sigma(\cat{M}^\op)\longrightarrow\PPROF^\op$, is straightforward.
Rather than introducing a general definition of lax double functor (for this, 
see \cite[Section~7.2]{GP1}), we shall use the following fact: for any 
bicategory $\tcat{B}$ and any pseudo double category $\dcat{D}$, 
lax double functors of type $\dcat{H}\tcat{B}\longrightarrow\dcat{D}$ 
bijectively correspond to lax functors of type 
$\tcat{B}\longrightarrow\tcat{H}\dcat{D}$ in a canonical way.
Hence it suffices to see that a metamodel of $\cat{M}$ corresponds to 
a lax functor of type 
$\Sigma(\cat{M}^\op)\longrightarrow\tcat{H}(\PPROF^\op)=\PROF^\op$,
which we have already remarked in Section~\ref{subsec:metamodel}. 

\medskip

As a sketch for the final step, we show that a vertical double cone over the 
lax double functor $\Phi(\monoid{T})$
in the sense of \cite[Section~4.1, 7.3]{GP1} is indeed equivalent to
a triple $(\cat{A},V,v)$ defined in Definition~\ref{def:double_cone}.
Specialising the original definition,
a vertical double cone over $\Phi(\monoid{T})\colon 
\dcat{H}\Sigma(\cat{M}^\op)\longrightarrow \PPROF^\op$ 
consists of the following data:
\begin{description}
\item[(DCD1)] A category $\cat{A}$.
\item[(DCD2)] A functor $V\colon \cat{A}\longrightarrow \cat{C}$.
\item[(DCD3')] For each $\ord{n}\in\Simp$, a square in $\PPROF^\op$
\[
\begin{tikzpicture}[baseline=-\the\dimexpr\fontdimen22\textfont2\relax ]
      \node (TL) at (0,1.5)  {$\cat{A}$};
      \node (TR) at (3,1.5)  {$\cat{A}$};
      \node (BL) at (0,0) {$\cat{C}$};
      \node (BR) at (3,0) {$\cat{C}$.};
      \draw[pto] (TL) to node[auto,labelsize](T) {$\cat{A}(-,-)$} 
      (TR);
      \draw[->]  (TR) to node[auto,labelsize] {$V$} (BR);
      \draw[->]  (TL) to node[auto,swap,labelsize] {$V$} 
      (BL);
      \draw[pto] (BL) to node[auto,swap,labelsize](B) 
      {$\Phi_{T^{\otimes n}}(=\Phi(\monoid{T})_\ord{n})$} (BR);
      \tzsquare{1.5}{0.75}{$v_\ord{n}$}
\end{tikzpicture} 
\]
\end{description}
satisfying the following axioms:
\begin{description}
\item[(DCA1')] 
\[
\begin{tikzpicture}[baseline=-\the\dimexpr\fontdimen22\textfont2\relax ]
      \node (TL) at (0,0.75)  {$\cat{A}$};
      \node (TR) at (2,0.75)  {$\cat{A}$};
      \node (BL) at (0,-0.75) {$\cat{C}$};
      \node (BR) at (2,-0.75) {$\cat{C}$};
      \draw[pto] (TL) to node[auto,labelsize](T) {$\cat{A}(-,-)$} 
      (TR);
      \draw[->]  (TR) to node[auto,labelsize] {$V$} (BR);
      \draw[->]  (TL) to node[auto,swap,labelsize] {$V$} 
      (BL);
      \draw[pto] (BL) to node[auto,swap,labelsize](B) 
      {$\Phi_I$} (BR);
      \tzsquare{1}{0}{$v_\ord{0}$}
\end{tikzpicture} 
\quad =\quad
\begin{tikzpicture}[baseline=-\the\dimexpr\fontdimen22\textfont2\relax ]
      \node (TL) at (0,0.75)  {$\cat{A}$};
      \node (TR) at (2,0.75)  {$\cat{A}$};
      \node (BL) at (0,-0.75) {$\cat{C}$};
      \node (BR) at (2,-0.75) {$\cat{C}$};
      \draw[pto] (TL) to node[auto,labelsize](T) {$\cat{A}(-,-)$} 
      (TR);
      \draw[->]  (TR) to node[auto,labelsize] {$V$} (BR);
      \draw[->]  (TL) to node[auto,swap,labelsize] {$V$} 
      (BL);
      \draw[pto, bend left=30] (BL) to node[auto,labelsize](B) {$\cat{C}(-,-)$} 
      (BR);
      \draw[pto, bend right=30] (BL) to node[auto,swap,labelsize](B) 
      {$\Phi_I$} 
      (BR);
      \tzsquare{1}{0.25}{$\id{V}$}
      \tzsquare{1}{-0.75}{$\phi_\cdot$}
\end{tikzpicture}
\]
\item[(DCA2')] For each pair of objects $\ord{n},\ord{m}\in\Simp$,
\[
\begin{tikzpicture}[baseline=-\the\dimexpr\fontdimen22\textfont2\relax ]
      \node (TL) at (0,0.75)  {$\cat{A}$};
      \node (TR) at (2,0.75)  {$\cat{A}$};
      \node (BL) at (0,-0.75) {$\cat{C}$};
      \node (BR) at (2,-0.75) {$\cat{C}$};
      \draw[pto] (TL) to node[auto,labelsize](T) {$\cat{A}(-,-)$} 
      (TR);
      \draw[->]  (TR) to node[auto,labelsize] {$V$} (BR);
      \draw[->]  (TL) to node[auto,swap,labelsize] {$V$} 
      (BL);
      \draw[pto] (BL) to node[auto,swap,labelsize](B) 
      {$\Phi_{T^{\otimes (m+n)}}$} (BR);
      \tzsquare{1}{0}{$v_\ord{m+n}$}
\end{tikzpicture} 
\quad =\quad
\begin{tikzpicture}[baseline=-\the\dimexpr\fontdimen22\textfont2\relax ]
      \node (TL) at (0,0.75)  {$\cat{A}$};
      \node (TR) at (2,0.75)  {$\cat{A}$};
      \node (BL) at (0,-0.75) {$\cat{C}$};
      \node (BR) at (2,-0.75) {$\cat{C}$};
      \node (TRR) at (4,0.75)  {$\cat{A}$};
      \node (BRR) at (4,-0.75) {$\cat{C}$};
      \draw[pto] (TL) to node[auto,labelsize](T) {$\cat{A}(-,-)$} 
      (TR);
      \draw[pto] (TR) to node[auto,labelsize](T2) {$\cat{A}(-,-)$} 
      (TRR);
      \draw[->]  (TR) to node[auto,labelsize] {$V$} (BR);
      \draw[->]  (TRR) to node[auto,labelsize] {$V$} (BRR);
      \draw[->]  (TL) to node[auto,swap,labelsize] {$V$} 
      (BL);
      \draw[pto] (BL) to node[auto,labelsize](B) {$\Phi_{T^{\otimes m}}$} 
      (BR);
      \draw[pto] (BR) to node[auto,labelsize](B) {$\Phi_{T^{\otimes n}}$} 
      (BRR);
      \draw[pto, bend right=30] (BL) to node[auto,swap,labelsize](B) 
      {$\Phi_{T^{\otimes (m+n)}}$} 
      (BRR);
      \draw[pto, bend left=30] (TL) to node[auto,labelsize](B) 
      {$\cat{A}(-,-)$} 
      (TRR);
      \tzsquare{1}{0}{$v_\ord{m}$}
      \tzsquare{3}{0}{$v_\ord{n}$}
      \tzsquare{2}{-1.1}{$\phi_{T^{\otimes n},T^{\otimes m}}$}
      \tzsquare{2}{1.1}{$\cong$}
\end{tikzpicture}
\]
\item[(DCA3')] For each morphism $u\colon \ord{n}\longrightarrow\ord{n'}$
in $\Simp$, 
\[
\begin{tikzpicture}[baseline=-\the\dimexpr\fontdimen22\textfont2\relax ]
      \node (TL) at (0,0.75)  {$\cat{A}$};
      \node (TR) at (2,0.75)  {$\cat{A}$};
      \node (BL) at (0,-0.75) {$\cat{C}$};
      \node (BR) at (2,-0.75) {$\cat{C}$};
      \draw[pto] (TL) to node[auto,labelsize](T) {$\cat{A}(-,-)$} 
      (TR);
      \draw[->]  (TR) to node[auto,labelsize] {$V$} (BR);
      \draw[->]  (TL) to node[auto,swap,labelsize] {$V$} 
      (BL);
      \draw[pto] (BL) to node[auto,swap,labelsize](B) {$\Phi_{T^{\otimes n}}$} 
      (BR);
      \tzsquare{1}{0}{$v_\ord{n}$}
\end{tikzpicture} 
\quad =\quad
\begin{tikzpicture}[baseline=-\the\dimexpr\fontdimen22\textfont2\relax ]
      \node (TL) at (0,0.75)  {$\cat{A}$};
      \node (TR) at (2,0.75)  {$\cat{A}$};
      \node (BL) at (0,-0.75) {$\cat{C}$};
      \node (BR) at (2,-0.75) {$\cat{C}.$};
      \draw[pto] (TL) to node[auto,labelsize](T) {$\cat{A}(-,-)$} 
      (TR);
      \draw[->]  (TR) to node[auto,labelsize] {$V$} (BR);
      \draw[->]  (TL) to node[auto,swap,labelsize] {$V$} 
      (BL);
      \draw[pto, bend left=30] (BL) to node[auto,labelsize](B) 
      {$\Phi_{T^{\otimes n'}}$} 
      (BR);
      \draw[pto, bend right=30] (BL) to node[auto,swap,labelsize](B) 
      {$\Phi_{T^{\otimes n}}$} 
      (BR);
      \tzsquare{1}{0.25}{$v_\ord{n'}$}
      \tzsquare{0.7}{-0.75}{$\Phi(\monoid{T})_{u}$}
\end{tikzpicture}
\]
\end{description}
By (DCA1') and (DCA2'), $v_\ord{1}$ determines all $v_\ord{n}$.
Also, it suffices to check the condition (DCA3') with respect to all 
face and degeneracy maps.
In fact, it suffices to check (DCA3') only with respect to
two maps,
namely ${!_\ord{0}}\colon \ord{0}\longrightarrow\ord{1}$ and 
${!_\ord{2}}\colon \ord{2}\longrightarrow \ord{1}$.
This is because, as noted above, any face map $\delta^{(n)}_i\colon 
\ord{n}\longrightarrow\ord{n+1}$
can be written as $\id{\ord{i}}+{!_\ord{0}}+\id{\ord{n-i}}$
and any degeneracy map $\sigma^{(n)}_i\colon\ord{n}\longrightarrow\ord{n-1}$
as $\id{\ord{i}}+{!_\ord{2}}+\id{\ord{n-i-2}}$.
Therefore if 
\[
\begin{tikzpicture}[baseline=-\the\dimexpr\fontdimen22\textfont2\relax ]
      \node (TL) at (0,0.75)  {$\cat{A}$};
      \node (TR) at (2,0.75)  {$\cat{A}$};
      \node (BL) at (0,-0.75) {$\cat{C}$};
      \node (BR) at (2,-0.75) {$\cat{C}$};
      \draw[pto] (TL) to node[auto,labelsize](T) {$\cat{A}(-,-)$} 
      (TR);
      \draw[->]  (TR) to node[auto,labelsize] {$V$} (BR);
      \draw[->]  (TL) to node[auto,swap,labelsize] {$V$} 
      (BL);
      \draw[pto] (BL) to node[auto,swap,labelsize](B) {$\Phi_I$} (BR);
      \tzsquare{1}{0}{$v_\ord{0}$}
\end{tikzpicture} 
\quad =\quad
\begin{tikzpicture}[baseline=-\the\dimexpr\fontdimen22\textfont2\relax ]
      \node (TL) at (0,0.75)  {$\cat{A}$};
      \node (TR) at (2,0.75)  {$\cat{A}$};
      \node (BL) at (0,-0.75) {$\cat{C}$};
      \node (BR) at (2,-0.75) {$\cat{C}$};
      \draw[pto] (TL) to node[auto,labelsize](T) {$\cat{A}(-,-)$} 
      (TR);
      \draw[->]  (TR) to node[auto,labelsize] {$V$} (BR);
      \draw[->]  (TL) to node[auto,swap,labelsize] {$V$} 
      (BL);
      \draw[pto, bend left=30] (BL) to node[auto,labelsize](B) 
      {$\Phi_T$} 
      (BR);
      \draw[pto, bend right=30] (BL) to node[auto,swap,labelsize](B) 
      {$\Phi_I$} 
      (BR);
      \tzsquare{1}{0.25}{$v_\ord{1}$}
      \tzsquare{1}{-0.75}{$\Phi_e$}
\end{tikzpicture}
\]
holds (we have used $\Phi(\monoid{T})_{!_{\ord{0}}}=\Phi_e$, where 
$e\colon I\longrightarrow T$ is given by the theory $\monoid{T}=(T,e,m)$), 
then for 
$\delta^{(n)}_i=\id{\ord{i}}+{!_\ord{0}}+\id{\ord{n-i}}\colon 
\ord{n}\longrightarrow\ord{n+1}$,
\begingroup
\allowdisplaybreaks
\begin{align*}
\begin{tikzpicture}[baseline=-\the\dimexpr\fontdimen22\textfont2\relax ]
      \node (TL) at (0,0.75)  {$\cat{A}$};
      \node (TR) at (2,0.75)  {$\cat{A}$};
      \node (BL) at (0,-0.75) {$\cat{C}$};
      \node (BR) at (2,-0.75) {$\cat{C}$};
      \draw[pto] (TL) to node[auto,labelsize](T) {$\cat{A}(-,-)$} 
      (TR);
      \draw[->]  (TR) to node[auto,labelsize] {$V$} (BR);
      \draw[->]  (TL) to node[auto,swap,labelsize] {$V$} 
      (BL);
      \draw[pto] (BL) to node[auto,swap,labelsize](B) {$\Phi_{T^{\otimes n}}$} 
      (BR);
      \tzsquare{1}{0}{$v_\ord{n}$}
\end{tikzpicture} 
\quad &=\quad
\begin{tikzpicture}[baseline=-\the\dimexpr\fontdimen22\textfont2\relax ]
      \node (TL) at (0,0.75)  {$\cat{A}$};
      \node (TR) at (2,0.75)  {$\cat{A}$};
      \node (BL) at (0,-0.75) {$\cat{C}$};
      \node (BR) at (2,-0.75) {$\cat{C}$};
      \node (TRR) at (4,0.75)  {$\cat{A}$};
      \node (BRR) at (4,-0.75) {$\cat{C}$};
      \node (TRRR) at (6,0.75)  {$\cat{A}$};
      \node (BRRR) at (6,-0.75) {$\cat{C}$};
      \draw[pto] (TL) to node[auto,labelsize](T) {$\cat{A}(-,-)$} 
      (TR);
      \draw[pto] (TR) to node[auto,labelsize](T2) {$\cat{A}(-,-)$} 
      (TRR);
      \draw[pto] (TRR) to node[auto,labelsize](T2) {$\cat{A}(-,-)$} 
      (TRRR);
      \draw[->]  (TR) to node[auto,labelsize] {$V$} (BR);
      \draw[->]  (TRR) to node[auto,labelsize] {$V$} (BRR);
      \draw[->]  (TRRR) to node[auto,labelsize] {$V$} (BRRR);
      \draw[->]  (TL) to node[auto,swap,labelsize] {$V$} 
      (BL);
      \draw[pto, bend left=30] (TL) to node[auto,labelsize](B) 
      {$\cat{A}(-,-)$} 
      (TRRR);
      \draw[pto] (BL) to node[auto,labelsize](B) {$\Phi_{T^{\otimes (n-i)}}$} 
      (BR);
      \draw[pto] (BR) to node[auto,labelsize](B) {$\Phi_I$} 
      (BRR);
      \draw[pto] (BRR) to node[auto,labelsize](B) {$\Phi_{T^{\otimes i}}$} 
      (BRRR);
      \draw[pto, bend right=30] (BL) to node[auto,swap,labelsize](B) 
      {$\Phi_{T^{\otimes n}}$} 
      (BRRR);
      \tzsquare{1}{0}{$v_\ord{n-i}$}
      \tzsquare{3}{0}{$v_\ord{0}$}
      \tzsquare{5}{0}{$v_\ord{i}$}
      \tzsquare{3}{-1.3}{$\phi_{T^{\otimes i},I,T^{\otimes (n-i)}}$}
      \tzsquare{3}{1.3}{$\cong$}
\end{tikzpicture}\\
 &=\quad
\begin{tikzpicture}[baseline=-\the\dimexpr\fontdimen22\textfont2\relax ]
      \node (TL) at (0,0.75)  {$\cat{A}$};
      \node (TR) at (2,0.75)  {$\cat{A}$};
      \node (BL) at (0,-0.75) {$\cat{C}$};
      \node (BR) at (2,-0.75) {$\cat{C}$};
      \node (TRR) at (4,0.75)  {$\cat{A}$};
      \node (BRR) at (4,-0.75) {$\cat{C}$};
      \node (TRRR) at (6,0.75)  {$\cat{A}$};
      \node (BRRR) at (6,-0.75) {$\cat{C}$};
      \draw[pto] (TL) to node[auto,labelsize](T) {$\cat{A}(-,-)$} 
      (TR);
      \draw[pto] (TR) to node[auto,labelsize](T2) {$\cat{A}(-,-)$} 
      (TRR);
      \draw[pto] (TRR) to node[auto,labelsize](T2) {$\cat{A}(-,-)$} 
      (TRRR);
      \draw[->]  (TR) to node[auto,labelsize] {$V$} (BR);
      \draw[->]  (TRR) to node[auto,labelsize] {$V$} (BRR);
      \draw[->]  (TRRR) to node[auto,labelsize] {$V$} (BRRR);
      \draw[->]  (TL) to node[auto,swap,labelsize] {$V$} 
      (BL);
      \draw[pto, bend left=30] (TL) to node[auto,labelsize](B) 
      {$\cat{A}(-,-)$} 
      (TRRR);
      \draw[pto] (BL) to node[auto,labelsize](B) {$\Phi_{T^{\otimes (n-i)}}$} 
      (BR);
      \draw[pto,bend left=30] (BR) to node[auto,labelsize](B) {$\Phi_T$} 
      (BRR);
      \draw[pto,bend right=30] (BR) to node[auto,pos=0.4,swap,labelsize](B) 
      {$\Phi_I$} 
      (BRR);
      \draw[pto] (BRR) to node[auto,labelsize](B) {$\Phi_{T^{\otimes i}}$} 
      (BRRR);
      \draw[pto, bend right=30] (BL) to node[auto,swap,labelsize](B) 
      {$\Phi_{T^{\otimes n}}$} 
      (BRRR);
      \tzsquare{1}{0}{$v_\ord{n-i}$} 
      \tzsquare{3}{0.25}{$v_\ord{1}$}
      \tzsquare{5}{0}{$v_\ord{i}$}
      \tzsquare{3}{-1.375}{$\phi_{T^{\otimes i},I,T^{\otimes (n-i)}}$}
      \tzsquare{3}{-0.75}{$\Phi_e$}
      \tzsquare{3}{1.3}{$\cong$}
\end{tikzpicture}\\
&=\quad
\begin{tikzpicture}[baseline=-\the\dimexpr\fontdimen22\textfont2\relax ]
      \node (TL) at (0,0.75)  {$\cat{A}$};
      \node (TR) at (2,0.75)  {$\cat{A}$};
      \node (BL) at (0,-0.75) {$\cat{C}$};
      \node (BR) at (2,-0.75) {$\cat{C}$};
      \node (TRR) at (4,0.75)  {$\cat{A}$};
      \node (BRR) at (4,-0.75) {$\cat{C}$};
      \node (TRRR) at (6,0.75)  {$\cat{A}$};
      \node (BRRR) at (6,-0.75) {$\cat{C}$};
      \draw[pto] (TL) to node[auto,labelsize](T) {$\cat{A}(-,-)$} 
      (TR);
      \draw[pto] (TR) to node[auto,labelsize](T2) {$\cat{A}(-,-)$} 
      (TRR);
      \draw[pto] (TRR) to node[auto,labelsize](T2) {$\cat{A}(-,-)$} 
      (TRRR);
      \draw[->]  (TR) to node[auto,labelsize] {$V$} (BR);
      \draw[->]  (TRR) to node[auto,labelsize] {$V$} (BRR);
      \draw[->]  (TRRR) to node[auto,labelsize] {$V$} (BRRR);
      \draw[->]  (TL) to node[auto,swap,labelsize] {$V$} 
      (BL);
      \draw[pto, bend left=30] (TL) to node[auto,labelsize](B) 
      {$\cat{A}(-,-)$} 
      (TRRR);
      \draw[pto] (BL) to node[auto,labelsize](B) {$\Phi_{T^{\otimes (n-i)}}$} 
      (BR);
      \draw[pto] (BR) to node[auto,labelsize](B) {$\Phi_T$} 
      (BRR);
      \draw[pto] (BRR) to node[auto,labelsize](B) {$\Phi_{T^{\otimes i}}$} 
      (BRRR);
      \draw[pto, bend right=25] (BL) to node[auto,pos=0.45,swap,labelsize](B) 
      {$\Phi_{T^{\otimes (n+1)}}$} 
      (BRRR);
      \draw[pto, bend right=50] (BL) to node[auto,swap,labelsize](B) 
      {$\Phi_{T^{\otimes n}}$} 
      (BRRR);
      \tzsquare{1}{0}{$v_\ord{n-i}$}
      \tzsquare{3}{0}{$v_\ord{1}$}
      \tzsquare{5}{0}{$v_\ord{i}$}
      \tzsquare{3}{-1.15}{$\phi_{T^{\otimes i},T,T^{\otimes (n-i)}}$}
      \tzsquare{3}{-1.9}{$\Phi_{\delta^{(n)}_i}$}
      \tzsquare{3}{1.3}{$\cong$}
\end{tikzpicture}\\
 &=\quad
\begin{tikzpicture}[baseline=-\the\dimexpr\fontdimen22\textfont2\relax ]
      \node (TL) at (0,0.75)  {$\cat{A}$};
      \node (TR) at (2,0.75)  {$\cat{A}$};
      \node (BL) at (0,-0.75) {$\cat{C}$};
      \node (BR) at (2,-0.75) {$\cat{C}$};
      \draw[pto] (TL) to node[auto,labelsize](T) {$\cat{A}(-,-)$} 
      (TR);
      \draw[->]  (TR) to node[auto,labelsize] {$V$} (BR);
      \draw[->]  (TL) to node[auto,swap,labelsize] {$V$} 
      (BL);
      \draw[pto, bend left=30] (BL) to node[auto,labelsize](B) 
      {$\Phi_{T^{\otimes (n+1)}}$} 
      (BR);
      \draw[pto, bend right=30] (BL) to node[auto,swap,labelsize](B) 
      {$\Psi_{T^{\otimes n}}$} 
      (BR);
      \tzsquare{1}{0.25}{$v_\ord{n+1}$}
      \tzsquare{1}{-0.75}{$\Phi_{\delta^{(n)}_i}$}
\end{tikzpicture}
\end{align*}
\endgroup
(where $\phi$ with three subscripts denote suitable composites of $\phi_{X,Y}$),
and similarly for the degeneracy maps.

Therefore, a vertical double cone for $\Phi(\monoid{T})$ is given equivalently 
as the data
(DCD1), (DCD2) together with: 
\begin{description}
\item[(DCD3)] a square in $\PPROF^\op$
\[
\begin{tikzpicture}[baseline=-\the\dimexpr\fontdimen22\textfont2\relax ]
      \node (TL) at (0,1.5)  {$\cat{A}$};
      \node (TR) at (2,1.5)  {$\cat{A}$};
      \node (BL) at (0,0) {$\cat{C}$};
      \node (BR) at (2,0) {$\cat{C},$};
      \draw[pto] (TL) to node[auto,labelsize](T) {$\cat{A}(-,-)$} 
      (TR);
      \draw[->]  (TR) to node[auto,labelsize] {$V$} (BR);
      \draw[->]  (TL) to node[auto,swap,labelsize] {$V$} 
      (BL);
      \draw[pto] (BL) to node[auto,swap,labelsize](B) {$\Phi_T$} (BR);
      \tzsquare{1}{0.75}{$v$}
\end{tikzpicture} 
\]
\end{description}
satisfying the equations (\ref{eqn:cone_e}) and (\ref{eqn:cone_m}).
This coincides with Definition~\ref{def:double_cone}.

Arguing similarly, we obtain the following corollary of 
Theorem~\ref{thm:double_categorical_univ_property}.

\begin{corollary}
\label{cor:Mod_as_dbl_lim}
Let $\cat{M}$ be a metatheory, $\monoid{T}$ be a theory in $\cat{M}$,
$\cat{C}$ be a large category and $\Phi$ be a metamodel of $\cat{M}$ in 
$\cat{C}$.
The category $\Mod{\monoid{T}}{(\cat{C},\Phi)}$ of models of $\monoid{T}$
in $\cat{C}$ with respect to $\Phi$is the apex of 
the double limit of the lax double functor $\Phi(\monoid{T})$.
\end{corollary}

\part{Weak $n$-dimensional $\cat{V}$-categories}
\chapter{Extensive categories}
\label{chap:extensive}
From this chapter on we shall turn to the study of weak $n$-categories.
In this chapter, we introduce \emph{extensive categories},
a central notion in our study of weak $n$-categories, and prove useful lemmas 
for them.

The results in this section have been published in 
\cite{CFP1,CFP2}.

\section{The definition and examples}
Extensive categories were first introduced by 
Lawvere~\cite{Lawvere_thoughts,Lawvere_space_quantity}
and their basic properties established by Carboni, Lack and 
Walters~\cite{Carboni_Lack_Walters} and by Cockett~\cite{Cockett}.
Roughly speaking, an extensive category is a category with well-behaved 
coproducts.\footnote{The original notion of extensive category requires 
well-behaved \emph{finite} coproducts, but what we shall use below is an
infinitary variant of this, requiring well-behaved \emph{small} coproducts;
such a notion is previously used in e.g., \cite[Section~4]{Centazzo_Vitale}.
In this thesis, the term ``extensive category'' always refer to this infinitary 
variant as defined in Definition~\ref{def:extensive_cat}.}

Let $\cat{V}$ be a large category with all small coproducts, $I$ be a small
set and $(X_i)_{i\in I}$ be an $I$-indexed family 
of objects of $\cat{V}$. We have the functor
\begin{equation}
\label{eqn:ext_coprod}
\coprod\colon \prod_{i\in I}(\cat{V}/X_i)\longrightarrow 
\cat{V}/\coprod_{i\in I}X_i
\end{equation}
which maps $(f_i\colon A_i\longrightarrow X_i)_{i\in I}$
to $(\coprod_{i\in I}f_i\colon\coprod_{i\in I}A_i\longrightarrow
\coprod_{i\in I}X_i)$.
\begin{definition}[\cite{Centazzo_Vitale}, 
cf.~\cite{Carboni_Lack_Walters,Cockett}]
\label{def:extensive_cat}
A large category $\cat{V}$ is \defemph{extensive}
if and only if it admits all small coproducts and for any small set $I$ and 
$I$-indexed family $(X_i)_{i\in I}$ of objects of $\cat{V}$,
the functor $\coprod$ in (\ref{eqn:ext_coprod}) is an equivalence of 
categories.
\end{definition}

Our leading examples of extensive categories are $\Set$ and the
category $\omega$-$\Cpo$ of (small) posets with sups of $\omega$-chains
and monotone functions preserving sups of $\omega$-chains,
together with, for any extensive category $\cat{V}$ with finite
limits, the 
categories 
$\enGph{\cat{V}}^{(n)}$ and
$\enCat{\cat{V}}^{(n)}$, which are defined recursively. In order to
define the former, we first need to define the category of
$\cat{V}$-graphs.

\begin{definition}[\cite{Wolff}]\label{Vgraph}
Let $\cat{V}$ be a large category.
\begin{enumerate}
\item 
A \defemph{small $\cat{V}$-graph} $G$
consists of a small set $\ob{G}$ together with,
for each $x,y\in\ob{G}$,
an object $G(x,y)\in \cat{V}$.
\item 
A \defemph{morphism of $\cat{V}$-graphs} from $G$ to $G'$
is a function $f\colon \ob{G}\longrightarrow\ob{G'}$
together with, for each $x,y\in\ob{G}$,
a morphism $f_{x,y}\colon G(x,y)\longrightarrow G'(fx,fy)$ in $\cat{V}$. 
\qedhere
\end{enumerate}
\end{definition}
Clearly, a $\Set$-graph is nothing but a directed multigraph.

We denote the category of all small $\cat{V}$-graphs and morphisms by
$\enGph{\cat{V}}$. The construction $\enGph{(-)}$ routinely extends
to an endo-2-functor on the 2-category $\tCAT$ of large
categories.

\begin{definition}\label{nVgraph}
For any natural number $n$ and any large category $\cat{V}$, the category
$\enGph{\cat{V}}^{(n)}$ is defined as follows:
\[
\enGph{\cat{V}}^{(0)}=\cat{V};\qquad \enGph{\cat{V}}^{(n+1)}=
\enGph{(\enGph{\cat{V}}^{(n)})}.
\]
An object of $\enGph{\cat{V}}^{(n)}$ is called an \defemph{$n$-dimensional 
$\cat{V}$-graph}.
%\begin{enumerate}
%\item $\enGph{0}=\Set$;
%\item $\enGph{(n+1)}=\enGph{(\enGph{n})}$.
%\end{enumerate}
\end{definition}

\begin{definition}\label{nVcat}
For each natural number $n$ and any large category $\cat{V}$ with finite
products, the category $\enCat{\cat{V}}^{(n)}$ 
is defined as follows (using the cartesian structure for 
enrichment):
\[
\enCat{\cat{V}}^{(0)}=\cat{V};\qquad \enCat{\cat{V}}^{(n+1)}=
\enCat{(\enCat{\cat{V}}^{(n)})}.
\]
An object of $\enCat{\cat{V}}^{(n)}$ is called a \defemph{strict 
$n$-dimensional $\cat{V}$-category} (to avoid confusion with \emph{weak} 
$n$-dimensional $\cat{V}$-category which we are trying to define).
%\begin{enumerate}
%\item $\enCat{0}=\Set$;
%\item $\enCat{(n+1)}=\enCat{(\enCat{n})}$.
%\end{enumerate}
\end{definition}

From now on, whenever we mention enriched categories,
we always mean enrichment with respect to the cartesian structure.
When $\cat{V} = \Set$, we abbreviate $\enGph{\cat{V}}^{(n)}$ by $\enGph{n}$
(whose object we call an \defemph{$n$-graph}),
and we abbreviate $\enCat{\cat{V}}^{(n)}$  by $\enCat{n}$
(whose object we call a \defemph{strict $n$-category}).

\medskip

We now show that
if $\cat{V}$ is an extensive category with finite limits, then so are
$\enGph{\cat{V}}$ and $\enCat{\cat{V}}$.
Actually, to ensure that $\enGph{\cat{V}}$ and $\enCat{\cat{V}}$
are extensive, 
the much weaker requirement of $\cat{V}$ having a strict initial object
suffices.
Recall that an initial object $0$ in a category
is called \defemph{strict} if every morphism going into $0$ is
an isomorphism.
Every extensive category has a strict initial object; consider the case 
$I=\emptyset$ in (\ref{eqn:ext_coprod}).

\begin{proposition}\label{prop:VGph_ext}
If $\cat{V}$ is a large category with a strict initial object $0$, then
$\enGph{\cat{V}}$ is extensive.
\end{proposition}
\begin{proof}
The coproduct of a family $(G_i)_{i\in I}$ 
of $\cat{V}$-graphs is given by
$\ob{\coprod_{i\in I}G_i}=\coprod_{i\in I}\ob{G_i}$
and 
\[
(\coprod_{i\in I}G_i) ((i,x), (i',x'))=
\begin{cases}
G_i(x,x')\quad &\text{if }i=i',\\
0&\text{otherwise.}
\end{cases}
\]
It is easy to see that
the functor $\coprod\colon \prod_{i\in I}(\enGph{\cat{V}}/G_i)\longrightarrow
\enGph{\cat{V}}/(\coprod_{i\in I}G_i)$
(as in~(\ref{eqn:ext_coprod})) is full and faithful.
For any object $(f\colon H\longrightarrow \coprod_{i\in I}G_i)$
in $\enGph{\cat{V}}/{(\coprod_{i\in I} G_i)}$,
define an object $(f_i\colon H_i\longrightarrow G_i)_{i\in I}\in \prod_{i\in 
I}(\cat{V}/G_i)$
by the pullbacks
of $f$ along the coprojections $\sigma_i\colon G_i\longrightarrow \coprod_{i\in 
I}G_i$;
note that these pullbacks always exist, 
and $H_i$ are just the suitable ``full sub'' $\cat{V}$-graphs
of $H$.
Since $0$ is strict, 
$(\coprod_{i\in I}f_i \colon \coprod_{i\in I}H_i\longrightarrow 
\coprod_{i\in 
I}G_i)$
is isomorphic to $f$. Hence $\coprod$ is also essentially surjective.
\end{proof}

\begin{proposition}\label{prop:VCat_ext}
If $\cat{V}$ is a large category with a strict initial object and finite 
products, then $\enCat{\cat{V}}$ is extensive.
\end{proposition}
\begin{proof}
Coproducts in $\enCat{\cat{V}}$ are formed just as in $\enGph{\cat{V}}$;
namely, given a family $(\cat{C}_i)_{i\in I}$ of $\cat{V}$-categories, 
we have $\ob{\coprod_{i\in I} \cat{C}_i}=\coprod_{i\in I} \ob{\cat{C}_i}$
and 
\[
(\coprod_{i\in I}\cat{C}_i) ((i,x), (i',x'))=
\begin{cases}
\cat{C}_i(x,x')\quad &\text{if }i=i',\\
0&\text{otherwise.}
\end{cases}
\]
Note that to define a composition law for $\coprod_{i\in I}\cat{C}_i$, we use 
the fact that for a category $\cat{V}$ with a strict initial object $0$,
$0\times B\cong 0$ for every object $B\in\cat{V}$.\footnote{In fact, for a 
category $\cat{V}$ with an initial object $0$ and finite products, $0$ is 
strict if and only if $0\times B\cong 0$ for 
every $B\in \cat{V}$.}
The rest of the proof is identical to that of Proposition~\ref{prop:VGph_ext}.
\end{proof}

When $\cat{V}$ has finite limits, then so do $\enGph{\cat{V}}$ and 
$\enCat{\cat{V}}$.
Given a finite category $\cat{I}$ and a functor $F\colon \cat{I}\longrightarrow
\enGph{\cat{V}}$, the limit $\lim F$ of $F$ can be constructed 
as follows.\footnote{This construction is valid for limits indexed 
by an arbitrary small category $\cat{I}$, provided that $\cat{V}$ has all 
$\cat{I}$-indexed limits.}
The set of objects 
is $\ob{\lim F}=\lim (\mathrm{ob}\circ {F})$, where 
$\mathrm{ob}\colon\enGph{\cat{V}}\longrightarrow\Set$ is the functor 
mapping a $\cat{V}$-graph to its set of objects.
Explicitly, an object of  $\lim F$ is an $\ob{\cat{I}}$-indexed family 
$(a_i)_{i\in\cat{I}}$ where $a_i$ is an object of the $\cat{V}$-graph ${Fi}$ and
such that for any morphism  $u\colon i\longrightarrow j$ in $\cat{I}$,
$(Fu)(a_i)=a_j$ holds.
Given any pair of objects $a=(a_i)_{i\in{\cat{I}}},b=(b_i)_{i\in{\cat{I}}}
\in \ob{\lim F}$, we obtain a functor 
$F_{a,b}\colon\cat{I}\longrightarrow\cat{V}$
by mapping an object $i\in\cat{I}$ to $(Fi)(a_i,b_i)$ and a morphism $u\colon 
i\longrightarrow j$ in $\cat{I}$ to $(Fu)_{a_i,b_i}$; observe that 
$(Fu)(a_i)=a_j$ and $(Fu)(b_i)=b_j$ hold and we indeed have 
a morphism $(Fu)_{a_i,b_i}\colon (Fi)(a_i,b_i)\longrightarrow (Fj)(a_j,b_j)$
in $\cat{V}$. 
The object $(\lim F)(a,b)$ is given by $\lim F_{a,b}$.

Finite limits in $\enCat{\cat{V}}$ may be constructed similarly, noting that 
limits commute with products; we remind the reader that $\enCat{\cat{V}}$
is defined using the cartesian structure of $\cat{V}$.

Below we record the case of pullbacks,
as they will play an important role later.
\begin{proposition}\label{prop:pb_in_V_gph_and_V_cat}
Let $\cat{V}$ have finite limits.
A commutative square
\begin{equation*}%\label{eqn:pb_in_V_gph}
\begin{tikzpicture}[baseline=-\the\dimexpr\fontdimen22\textfont2\relax ]
      \node(0) at (0,1) {$P$};
      \node(1) at (2,1) {$B$};
      \node(2) at (0,-1) {$A$};
      \node(3) at (2,-1) {$X$};
      
      \draw [->] 
            (0) to node (t)[auto,labelsize] {$k$} 
            (1);
      \draw [->] 
            (1) to node (r)[auto,labelsize] {$g$} 
            (3);
      \draw [->] 
            (0) to node (l)[auto,swap,labelsize] {$h$} 
            (2);
      \draw [->] 
            (2) to node (b)[auto,swap,labelsize] {$f$} 
            (3);  
\end{tikzpicture}
\end{equation*}
in $\enGph{\cat{V}}$ or in $\enCat{\cat{V}}$ 
is a pullback if and only if the square
\[
\begin{tikzpicture}[baseline=-\the\dimexpr\fontdimen22\textfont2\relax ]
      \node(0) at (0,1) {$\ob{P}$};
      \node(1) at (3,1) {$\ob{B}$};
      \node(2) at (0,-1) {$\ob{A}$};
      \node(3) at (3,-1) {$\ob{X}$};
      
      \draw [->] 
            (0) to node (t)[auto,labelsize] {${k}$} 
            (1);
      \draw [->] 
            (1) to node (r)[auto,labelsize] {${g}$} 
            (3);
      \draw [->] 
            (0) to node (l)[auto,swap,labelsize] {${h}$} 
            (2);
      \draw [->] 
            (2) to node (b)[auto,swap,labelsize] {${f}$} 
            (3);  
            
%     \draw (0.2,0.4) -- (0.6,0.4) -- (0.6,0.8);
\end{tikzpicture}
\]
is a pullback in $\Set$, 
and for any pair 
$p_1,p_2\in\ob{P}$, writing $a_i=h(p_i)$, $b_i=k(p_i)$ and 
$x_i=f(a_i)=g(b_i)$ for $i=1,2$, the square
\begin{equation*}
\begin{tikzpicture}[baseline=-\the\dimexpr\fontdimen22\textfont2\relax ]
      \node(0) at (0,1) {${P}(p_1,p_2)$};
      \node(1) at (3,1) {${B}(b_1,b_2)$};
      \node(2) at (0,-1) {${A}(a_1,a_2)$};
      \node(3) at (3,-1) {${X}(x_1,x_2)$};
      
      \draw [->] 
            (0) to node (t)[auto,labelsize] {$k_{p_1,p_2}$} 
            (1);
      \draw [->] 
            (1) to node (r)[auto,labelsize] {$g_{b_1,b_2}$} 
            (3);
      \draw [->] 
            (0) to node (l)[auto,swap,labelsize] {$h_{p_1,p_2}$} 
            (2);
      \draw [->] 
            (2) to node (b)[auto,swap,labelsize] {$f_{a_1,a_2}$} 
            (3);  

%     \draw (0.2,0.4) -- (0.6,0.4) -- (0.6,0.8);
\end{tikzpicture}
\end{equation*}
is a  pullback in $\cat{V}$.
\end{proposition}

Combining the above observation with Propositions~\ref{prop:VGph_ext} and 
\ref{prop:VCat_ext}, we immediately obtain the following result.
\begin{corollary}
If $\cat{V}$ is an extensive category with finite limits,
then so are $\enGph{\cat{V}}^{(n)}$ and $\enCat{\cat{V}}^{(n)}$,
for each natural number $n$.
\end{corollary}

%{\bf Motivate extensive categories by the following examples.}
%\begin{definition}
%For each natural number $n$, define the category $\enGph{n}$ 
%of (small) \defemph{$n$-graphs} recursively as follows:
%\[
%\enGph{0}=\Set;\qquad \enGph{(n+1)}=\enGph{(\enGph{n})}.
%\]
%\begin{enumerate}
%\item $\enGph{0}=\Set$;
%\item $\enGph{(n+1)}=\enGph{(\enGph{n})}$.
%\end{enumerate}
%\end{definition}

%\begin{definition}
%For each natural number $n$, define the category $\enCat{n}$ 
%of (small) \defemph{strict $n$-categories} recursively as follows:
%\[
%\enCat{0}=\Set;\qquad \enCat{(n+1)}=\enCat{(\enCat{n})}.
%\]
%\begin{enumerate}
%\item $\enCat{0}=\Set$;
%\item $\enCat{(n+1)}=\enCat{(\enCat{n})}$.
%\end{enumerate}
%\end{definition}

\section{Properties of coproducts in an extensive category}
We need several results about behaviour of coproducts in extensive categories
later, so in this section we collect such results.

The first proposition gives a characterisation of extensive categories.
\begin{proposition}[{\cite[Section~4.2, Exercise~1]{Centazzo_Vitale}, 
cf.~\cite[Proposition~2.2]{Carboni_Lack_Walters}}]\label{prop:ext_criterion}
A category $\cat{V}$ with small coproducts is extensive if and only if
it has all pullbacks along coprojections associated with small coproducts,
and for any small set $I$, $I$-indexed family $(X_i)_{i\in I}$ of objects of 
$\cat{V}$,
morphism $f\colon A\longrightarrow \coprod_{i\in I}X_i$
in $\cat{V}$, and $I$-indexed family of commutative squares 
\begin{equation}\label{eqn:square}
\begin{tikzpicture}[baseline=-\the\dimexpr\fontdimen22\textfont2\relax ]
      \node(0) at (0,1) {$A_i$};
      \node(1) at (2,1) {$A$};
      \node(2) at (0,-1) {$X_i$};
      \node(3) at (2,-1) {$\coprod_{i\in I}X_i$};
      
      \draw [->] 
            (0) to node (t)[auto,labelsize] {$\tau_i$} 
            (1);
      \draw [->] 
            (1) to node (r)[auto,labelsize] {$f$} 
            (3);
      \draw [->] 
            (0) to node (l)[auto,swap,labelsize] {$f_i$} 
            (2);
      \draw [->] 
            (2) to node (b)[auto,swap,labelsize] {$\sigma_i$} 
            (3);  
\end{tikzpicture}
\end{equation}
in $\cat{V}$ (in which $\sigma_i$ is the $i$-th coprojection),
each square (\ref{eqn:square}) is a pullback square if and only if
$(\tau_i)_{i\in I}$ defines a coproduct (that is, $A=\coprod_{i\in I} A_i$
with $\tau_i$ the $i$-th coprojection).
\end{proposition}
\begin{proof}
If $\cat{V}$ has small coproducts, then the functor (\ref{eqn:ext_coprod})
has a right adjoint if and only if all pullbacks along $\sigma_i$ exists in
$\cat{V}$, and in that case the right adjoint 
\[
\langle \sigma_i^\ast\rangle_{i\in I}\colon\cat{V}/\coprod_{i\in 
I}X_i\longrightarrow\prod_{i\in I}(\cat{V}/X_i) 
\]
has the $i$-th component $\sigma_i^\ast\colon \cat{V}/\coprod_{i\in 
I}X_i\longrightarrow\cat{V}/X_i$
mapping $(f\colon A\longrightarrow \coprod_{i\in I}X_i)\in 
\cat{V}/\coprod_{i\in I} X_i$ to $(\sigma_i^\ast f\colon \sigma_i^\ast 
A\longrightarrow X_i)\in\cat{V}/X_i$, defined by the pullback
\[
\begin{tikzpicture}[baseline=-\the\dimexpr\fontdimen22\textfont2\relax ]
      \node(0) at (0,1) {$\sigma_i^\ast A$};
      \node(1) at (2,1) {$A$};
      \node(2) at (0,-1) {$X_i$};
      \node(3) at (2,-1) {$\coprod_{i\in I}X_i$};
      
      \draw [->] 
            (0) to node (t)[auto,labelsize] {} 
            (1);
      \draw [->] 
            (1) to node (r)[auto,labelsize] {$f$} 
            (3);
      \draw [->] 
            (0) to node (l)[auto,swap,labelsize] {$\sigma_i^\ast f$} 
            (2);
      \draw [->] 
            (2) to node (b)[auto,swap,labelsize] {$\sigma_i$} 
            (3);  

     \draw (0.2,0.4) -- (0.6,0.4) -- (0.6,0.8);
\end{tikzpicture}
\]
in $\cat{V}$.

In general, a functor is an equivalence of categories if and only if it has 
a right adjoint and the associated unit and counit are natural isomorphisms.
Applying this fact to the functors of the form (\ref{eqn:ext_coprod}),
we obtain the desired result.
\end{proof}

\begin{proposition}\label{prop:ext_coprod_of_pbs}
Let $\cat{V}$ be an extensive category. 
For any small set $I$ and $I$-indexed family of pullback squares in $\cat{V}$
as on the left of the following diagram, 
the square as on the right is a pullback.
\begin{equation*}%\label{eqn:pb_square}
\begin{tikzpicture}[baseline=-\the\dimexpr\fontdimen22\textfont2\relax ]
      \node(0) at (0,1) {$P_i$};
      \node(1) at (2,1) {$B_i$};
      \node(2) at (0,-1) {$A_i$};
      \node(3) at (2,-1) {$X_i$};
      
      \draw [->] 
            (0) to node (t)[auto,labelsize] {$q_i$} 
            (1);
      \draw [->] 
            (1) to node (r)[auto,labelsize] {$g_i$} 
            (3);
      \draw [->] 
            (0) to node (l)[auto,swap,labelsize] {$p_i$} 
            (2);
      \draw [->] 
            (2) to node (b)[auto,swap,labelsize] {$f_i$} 
            (3);  

     \draw (0.2,0.4) -- (0.6,0.4) -- (0.6,0.8);
\end{tikzpicture}
\qquad\qquad
\begin{tikzpicture}[baseline=-\the\dimexpr\fontdimen22\textfont2\relax ]
      \node(0) at (0,1) {$\coprod_{i\in I}P_i$};
      \node(1) at (3,1) {$\coprod_{i\in I}B_i$};
      \node(2) at (0,-1) {$\coprod_{i\in I}A_i$};
      \node(3) at (3,-1) {$\coprod_{i\in I}X_i$};
      
      \draw [->] 
            (0) to node (t)[auto,labelsize] {$\coprod_{i\in I}q_i$} 
            (1);
      \draw [->] 
            (1) to node (r)[auto,labelsize] {$\coprod_{i\in I}g_i$} 
            (3);
      \draw [->] 
            (0) to node (l)[auto,swap,labelsize] {$\coprod_{i\in I}p_i$} 
            (2);
      \draw [->] 
            (2) to node (b)[auto,swap,labelsize] {$\coprod_{i\in I}f_i$} 
            (3);  

     \draw (0.2,0.4) -- (0.6,0.4) -- (0.6,0.8);
\end{tikzpicture}
\end{equation*}
\end{proposition}
\begin{proof}
By the definition of extensivity, the functor 
$
\coprod\colon \prod_{i\in I} (\cat{V}/X_i) \longrightarrow 
\cat{V}/(\coprod_{i\in I} X_i)
$
is an equivalence of categories and, in particular, it preserves binary 
products. 
%Regarding $(f_i\colon A_i\longrightarrow X_i)_{i\in I}$ and 
%$(g_i\colon B_i\longrightarrow X_i)_{i\in I}$ as objects of $\prod_{i\in 
%I}(\cat{V}/X_i)$,
%the claim of the proposition reduces to
%preservation of their product by $\coprod$.
\end{proof}

\begin{proposition}\label{prop:extensive_prod}
Let $\cat{V}$ be an extensive category with finite products. 
For any $B\in\cat{V}$, the functor $(-)\times B\colon 
\cat{V}\longrightarrow\cat{V}$ preserves small coproducts.
\end{proposition}
\begin{proof}
In any category, a square as on the left of the following diagram
is always a pullback.
Hence for any object $B\in\cat{V}$, 
small set $I$, and $I$-indexed family $(X_i)_{i\in I}$
of objects of $\cat{V}$, for each $i\in I$ the square as on the right
is a pullback.
\[
\begin{tikzpicture}[baseline=-\the\dimexpr\fontdimen22\textfont2\relax ]
      \node(0) at (0,1) {$A\times B$};
      \node(1) at (2,1) {$C\times B$};
      \node(2) at (0,-1) {$A$};
      \node(3) at (2,-1) {$C$};
      
      \draw [->] 
            (0) to node (t)[auto,labelsize] {$h\times B$} 
            (1);
      \draw [->] 
            (1) to node (r)[auto,labelsize] {$\pi_1$} 
            (3);
      \draw [->] 
            (0) to node (l)[auto,swap,labelsize] {$\pi_1$} 
            (2);
      \draw [->] 
            (2) to node (b)[auto,swap,labelsize] {$h$} 
            (3);  
\end{tikzpicture}
\qquad\qquad
\begin{tikzpicture}[baseline=-\the\dimexpr\fontdimen22\textfont2\relax ]
      \node(0) at (0,1) {$X_i\times B$};
      \node(1) at (3,1) {$(\coprod_{i\in I}X_i)\times B$};
      \node(2) at (0,-1) {$X_i$};
      \node(3) at (3,-1) {$\coprod_{i\in I}X_i$};
      
      \draw [->] 
            (0) to node (t)[auto,labelsize] {$\sigma_i\times B$} 
            (1);
      \draw [->] 
            (1) to node (r)[auto,labelsize] {$\pi_1$} 
            (3);
      \draw [->] 
            (0) to node (l)[auto,swap,labelsize] {$\pi_1$} 
            (2);
      \draw [->] 
            (2) to node (b)[auto,swap,labelsize] {$\sigma_i$} 
            (3);  
\end{tikzpicture}
\]
Therefore by Proposition~\ref{prop:ext_criterion},
$(\coprod_{i\in I}X_i)\times B \cong \coprod_{i\in I}(X_i\times B)$.
\end{proof}

\begin{proposition}\label{prop:extensive_slice}
Let $\cat{V}$ be an extensive category. For any object $Y\in\cat{V}$,
the slice category $\cat{V}/Y$ is again extensive.
\end{proposition}
\begin{proof}
Clearly $\cat{V}/Y$ has small coproducts given by  $\coprod_{i\in I}(f_i\colon 
X_i\longrightarrow Y)
= ([f_i]_{i\in I}\colon $ $\coprod_{i\in I} X_i\longrightarrow Y)$.
Also note that for any object $(f\colon X\longrightarrow Y)$ of $\cat{V}/Y$,
the canonical functor $(\cat{V}/Y)/f\longrightarrow \cat{V}/X$ 
which maps $(h\colon (g\colon A\longrightarrow Y)\longrightarrow f)\in 
(\cat{V}/Y)/f$ to
$(h\colon A\longrightarrow X)\in \cat{V}/X$ is an isomorphism of categories.
For any small set $I$ and $I$-indexed family $(f_i\colon X_i\longrightarrow 
Y)_{i\in 
I}$
of objects of $\cat{V}/Y$, the diagram 
\[
\begin{tikzpicture}[baseline=-\the\dimexpr\fontdimen22\textfont2\relax ]
      \node(0) at (0,1.5) {$\prod_{i\in I}((\cat{V}/Y)/f_i)$};
      \node(1) at (5,1.5) {$(\cat{V}/Y)/[f_i]_{i\in I}$};
      \node(2) at (0,0) {$\prod_{i\in I}(\cat{V}/X_i)$};
      \node(3) at (5,0) {$\cat{V}/(\coprod_{i\in I}X_i)$};
      
      \draw [->] 
            (0) to node (t)[auto,labelsize] {$\coprod$} 
            (1);
      \draw [->] 
            (1) to node (r)[auto,labelsize] {$\cong$} 
            (3);
      \draw [->] 
            (0) to node (l)[auto,swap,labelsize] {$\cong$} 
            (2);
      \draw [->] 
            (2) to node (b)[auto,labelsize] {$\coprod$} 
            (3);  
\end{tikzpicture}
\]
(in which the vertical arrows are the canonical isomorphisms mentioned above)
commutes. 
Since the lower $\coprod$ is an equivalence by the assumption, so is the upper 
one.
\end{proof}

\begin{corollary}\label{cor:ext_pb_coprod}
Let $\cat{V}$ be an extensive category with pullbacks. 
\begin{enumerate}
\item For any morphism $g\colon B\longrightarrow X$ in $\cat{V}$, 
small set $I$, and $I$-indexed family of pullback squares in $\cat{V}$
as on the left of the following diagram, 
the square as on the right is a pullback.
\begin{equation*}
\begin{tikzpicture}[baseline=-\the\dimexpr\fontdimen22\textfont2\relax ]
      \node(0) at (0,1) {$P_i$};
      \node(1) at (2,1) {$B$};
      \node(2) at (0,-1) {$A_i$};
      \node(3) at (2,-1) {$X$};
      
      \draw [->] 
            (0) to node (t)[auto,labelsize] {$q_i$} 
            (1);
      \draw [->] 
            (1) to node (r)[auto,labelsize] {$g$} 
            (3);
      \draw [->] 
            (0) to node (l)[auto,swap,labelsize] {$p_i$} 
            (2);
      \draw [->] 
            (2) to node (b)[auto,swap,labelsize] {$f_i$} 
            (3);  

     \draw (0.2,0.4) -- (0.6,0.4) -- (0.6,0.8);
\end{tikzpicture}
\qquad\qquad
\begin{tikzpicture}[baseline=-\the\dimexpr\fontdimen22\textfont2\relax ]
      \node(0) at (0,1) {$\coprod_{i\in I}P_i$};
      \node(1) at (3,1) {$B$};
      \node(2) at (0,-1) {$\coprod_{i\in I}A_i$};
      \node(3) at (3,-1) {$X$};
      
      \draw [->] 
            (0) to node (t)[auto,labelsize] {$[q_i]_{i\in I}$} 
            (1);
      \draw [->] 
            (1) to node (r)[auto,labelsize] {$g$} 
            (3);
      \draw [->] 
            (0) to node (l)[auto,swap,labelsize] {$\coprod_{i\in I}p_i$} 
            (2);
      \draw [->] 
            (2) to node (b)[auto,swap,labelsize] {$[f_i]_{i\in I}$} 
            (3);  

     \draw (0.2,0.4) -- (0.6,0.4) -- (0.6,0.8);
\end{tikzpicture}
\end{equation*}
\item For any object $X\in\cat{V}$, small set $I$, 
$I$-indexed family of morphisms $(f_i\colon A_i\longrightarrow X)_{i\in I}$ 
in $\cat{V}$, small set $J$, $J$-indexed family of morphisms
$(g_j\colon B_j\longrightarrow X)_{j\in J}$ in $\cat{V}$,
and $(I\times J)$-indexed family of pullback squares in $\cat{V}$
as on the left of the following diagram, 
the square as on the right is a pullback.
\begin{equation*}
\begin{tikzpicture}[baseline=-\the\dimexpr\fontdimen22\textfont2\relax ]
      \node(0) at (0,1) {$P_{i,j}$};
      \node(1) at (2,1) {$B_j$};
      \node(2) at (0,-1) {$A_i$};
      \node(3) at (2,-1) {$X$};
      
      \draw [->] 
            (0) to node (t)[auto,labelsize] {$q_{i,j}$} 
            (1);
      \draw [->] 
            (1) to node (r)[auto,labelsize] {$g_j$} 
            (3);
      \draw [->] 
            (0) to node (l)[auto,swap,labelsize] {$p_{i,j}$} 
            (2);
      \draw [->] 
            (2) to node (b)[auto,swap,labelsize] {$f_i$} 
            (3);  

     \draw (0.2,0.4) -- (0.6,0.4) -- (0.6,0.8);
\end{tikzpicture}
\qquad\qquad
\begin{tikzpicture}[baseline=-\the\dimexpr\fontdimen22\textfont2\relax ]
      \node(0) at (0,1) {$\coprod_{i\in I,j\in J}P_{i,j}$};
      \node(1) at (4,1) {$\coprod_{j\in J}B_j$};
      \node(2) at (0,-1) {$\coprod_{i\in I}A_i$};
      \node(3) at (4,-1) {$X$};
      
      \draw [->] 
            (0) to node (t)[auto,labelsize] {$\coprod_{j\in J}([q_{i,j}]_{i\in 
            I})$} 
            (1);
      \draw [->] 
            (1) to node (r)[auto,labelsize] {$[g_j]_{j\in J}$} 
            (3);
      \draw [->] 
            (0) to node (l)[auto,swap,labelsize] {$\coprod_{i\in 
            I}([p_{i,j}]_{j\in J})$} 
            (2);
      \draw [->] 
            (2) to node (b)[auto,swap,labelsize] {$[f_i]_{i\in I}$} 
            (3);  

     \draw (0.2,0.4) -- (0.6,0.4) -- (0.6,0.8);
\end{tikzpicture}
\end{equation*}
\end{enumerate}
\end{corollary}
\begin{proof}
\begin{enumerate}
\item By the assumption, the slice category $\cat{V}/X$ has 
finite products $\times_X$ (given by pullbacks in $\cat{V}$),
and is extensive (Proposition~\ref{prop:extensive_slice}).
Hence by Proposition~\ref{prop:extensive_prod}, 
binary product by $(g\colon B\longrightarrow X)\in \cat{V}/X$ 
preserves small coproducts, that is, $(\coprod_{i\in I} f_i)\times_X g \cong 
\coprod_{i\in I}(f_i\times_X g)$.
\item Using the first clause iteratively, we obtain
$(\coprod_{i\in I}f_i)\times_X(\coprod_{j\in J}g_j)\cong \coprod_{i\in I,j\in 
J}$ $(f_i\times_X g_j)$. \qedhere
\end{enumerate}
\end{proof}

\chapter{The free strict $n$-dimensional $\cat{V}$-category monad on 
$\enGph{\cat{V}}^{(n)}$}
\label{chap:free_strict_n_cat_monad}
The construction of the free category $FG$ generated by a ($\Set$-)graph $G$ 
is well-known:
the set of objects of $FG$ is the same as that of $G$, and a morphism in 
$FG$ is a (directed) \emph{path} in $G$ (see 
Section~\ref{sec:free_Vcat_monad}).
This construction is the left adjoint to the  
forgetful functor $U\colon\Cat\longrightarrow\mathbf{Gph}=\enGph{\Set}$,
and gives rise to a monad $\monoid{T}$ on $\mathbf{Gph}$,
the \emph{free category monad}.
This monad and its higher dimensional analogues, the free strict $n$-category 
monad $\monoid{T}^{(n)}$ on $\enGph{n}$ for each natural number $n$, play a 
crucial role 
in the Batanin--Leinster approach to weak $n$-categories, because 
they turn out to be \emph{cartesian} monads and therefore we may consider 
\emph{$\monoid{T}^{(n)}$-operads}.
The structure of weak $n$-category is expressed via a certain 
$\monoid{T}^{(n)}$-operad.

In this chapter, we show a generalisation of these facts; 
rather than starting from the category $\Set$, we start from an arbitrary 
extensive category $\cat{V}$ with finite limits, and show that we have the 
\emph{free strict $n$-dimensional $\cat{V}$-category monad} $\monoid{T}^{(n)}$ 
on 
$\enGph{\cat{V}}^{(n)}$, and that it is cartesian. 

The results in this section have been published in 
\cite{CFP2}.

\section{The free $\cat{V}$-category monad}
\label{sec:free_Vcat_monad}

In this section we deal with the one-dimensional case; that is,
we define the free $\cat{V}$-category monad on $\enGph{\cat{V}}$
and show it is cartesian. 

Let us start with reviewing the construction of free categories over graphs.
Suppose that $G=(\ob{G},(G(x,y))_{x,y\in\ob{G}})$ is an object of 
$\mathbf{Gph}$, i.e., 
a directed multigraph.
For $x,y\in\ob{G}$, a \defemph{path} in $G$ from $x$ to $y$
is a sequence 
\[
(w_0, f_1,w_1,f_2,\dots, f_n,w_n)
\]
where $n$ is a natural number called the \defemph{length} of the path, 
$w_i\in\ob{G}$ and $f_i\in G(w_{i-1},w_i)$ such that $w_0=x$ and $w_n=y$:
\begin{equation*}
\begin{tikzpicture}[baseline=-\the\dimexpr\fontdimen22\textfont2\relax ]
      \node            (0) {$x=w_0$};
      \node[right=of 0](1) {$w_1$};
      \node[right=of 1](d) {$\cdots$};
      \node[right=of d](n) {$w_n=y$.};
      
      \draw [->] (0) to node [auto,labelsize] {$f_1$} (1);
      \draw [->] (1) to node [auto,labelsize] {$f_2$} (d);
      \draw [->] (d) to node [auto,labelsize] {$f_n$} (n);
\end{tikzpicture}
\end{equation*}
The set of all paths in $G$ from $x$ to $y$ is therefore given by
\begin{equation}
\label{eqn:paths_in_graph}
\coprod_{n\in\N} \coprod_{\substack{w_0,\dots,w_n\in 
\ob{G}\\w_0=x,w_n=y}}G(w_{n-1},w_{n})\times\cdots\times G(w_0,w_1)
\end{equation}
(we have written $G(w_{n-1},w_n)\times \dots\times G(w_0,w_1)$ instead of 
$G(w_0,w_1)\times\dots\times G(w_{n-1},w_n)$ because the former agrees 
with our convention to write compositions in a category in the 
anti-diagrammatic order).

The free category $FG$ over $G$ has the same objects as $G$, and its 
hom-set $(FG)(x,y)$ is given by (\ref{eqn:paths_in_graph}).
Note that the set of all paths in $G$ from $x$ to $y$ of length $n$ 
is given as
\begin{equation}
\label{eqn:paths_of_length_n}
(FG)(x,y)_n=\coprod_{\substack{w_0,\dots,w_n\in 
\ob{G}\\w_0=x,w_n=y}}G(w_{n-1},w_n)\times\dots\times G(w_0,w_1),
\end{equation}
and using this, we may rewrite (\ref{eqn:paths_in_graph}) as 
\[
(FG)(x,y)=\coprod_{n\in\NN}(FG)(x,y)_n.
\]
The identities in $FG$ are given by the paths of length 0
(note that $(FG)(x,y)_0$ is a singleton if $x=y$ and is
empty otherwise),
and compositions in $FG$ are given by the evident compositions of paths.
 
The following construction is a straightforward generalisation of the 
above ``path'' construction for free categories over graphs.
\begin{proposition}\label{prop:free_V_cat}
If $\cat{V}$ has finite products and small coproducts, and if 
for any $B\in\cat{V}$ the functor $(-)\times B\colon
\cat{V}\longrightarrow \cat{V}$ preserves small coproducts, then the
forgetful functor
$U\colon\enCat{\cat{V}}\longrightarrow\enGph{\cat{V}}$ admits a left
adjoint $F$.
\end{proposition}
\begin{proof}
Given a $\cat{V}$-graph $G=(\ob{G},(G(x,y))_{x,y\in\ob{G}})$, the free 
$\cat{V}$-category $FG$ on $G$ 
has the same objects as $G$ and the hom-object given by
\[
(FG)(x,y)=\coprod_{n\in\NN}\coprod_{\substack{w_0,\dots,w_n\in 
\ob{G}\\w_0=x,w_n=y}}G(w_{n-1},w_{n})\times\cdots\times G(w_0,w_1)
\]
for all $x,y\in \ob{FG}=\ob{G}$.
To spell out the identity elements and composition laws in $FG$, let us write 
\[
(FG)(x,y)_n=\coprod_{\substack{w_0,\dots,w_n\in 
\ob{G}\\w_0=x,w_n=y}}G(w_{n-1},w_{n})\times\cdots\times G(w_0,w_1) 
\]
for all $x,y\in\ob{FG}$ and $n\in\NN$.

Note that $(FG)(x,y)_0$ is the terminal object $1$ of $\cat{V}$ if $x=y$
(otherwise, it is the initial object $0$ of $\cat{V}$).
Hence the identity element on $x\in\ob{FG}$ can be given as 
\[
\begin{tikzpicture}[baseline=-\the\dimexpr\fontdimen22\textfont2\relax ]
      \node            (0) {$1$};
      \node[right=of 0](1) {$(FG)(x,x)_0$};
      \node[right=of 1](d) 
      {$\coprod_{n\in\NN}(FG)(x,x)_n=(FG)(x,y)$,};
      
      \draw [->] (0) to node [auto,labelsize] {$\cong$} (1);
      \draw [->] (1) to node [auto,labelsize] {$\sigma_0$} (d);
\end{tikzpicture}
\]
where $\sigma_0$ denotes the $0$-th coprojection.
Given any triple $x,y,z\in\ob{FG}$ of objects, by the assumption we have 
\[
(FG)(y,z)\times (FG)(x,y)\cong \coprod_{k,l\in\NN}(FG)(y,z)_l\times (FG)(x,y)_k.
\]
Using the assumption once again, we see that $(FG)(y,z)_l\times (FG)(x,y)_k$
is isomorphic to 
\[
\coprod_{\substack{u_0,\dots,u_k,v_0,\dots,v_l\in \ob{G}\\
u_0=x,u_k=v_0=y,v_l=z}}G(v_{l-1},v_l)\times\dots\times G(v_0,v_1)\times
G(u_{k-1},u_k)\times\dots\times G(u_0,u_1),
\]
and therefore naturally embeds into $(FG)(x,z)_{k+l}$.
The universality of coproducts induce the composition laws for $FG$ from these 
embeddings.
\end{proof}

Examples of categories $\cat{V}$ satisfying the assumptions of 
Proposition~\ref{prop:free_V_cat} include cartesian closed categories with 
small coproducts (in this case, Proposition~\ref{prop:free_V_cat} appears in 
\cite[Proposition~2.2]{Wolff}) and extensive categories with 
finite products (by Proposition~\ref{prop:extensive_prod}).

For any extensive category $\cat{V}$ with finite limits, 
the \defemph{free $\cat{V}$-category 
monad} $\monoid{T}=(T,\eta,\mu)$ is the monad on $\enGph{\cat{V}}$ generated by 
the adjunction $F\dashv U$ in Proposition~\ref{prop:free_V_cat}.
The rest of this section is devoted to a proof of the fact that $\monoid{T}$ is 
cartesian. 
We show this by inspecting the adjunction $F\dashv U$
rather than the monad $\monoid{T}$ itself, because we will use  
certain properties of $F\dashv U$ in an inductive argument 
in the next section. 

\medskip

As a preliminary for the proof of the next proposition, 
let us examine the action of the functor 
$F\colon\enGph{\cat{V}}\longrightarrow\enCat{\cat{V}}$ on morphisms.
Suppose that $f\colon G\longrightarrow H$ is a morphism in $\enGph{\cat{V}}$.
The $\cat{V}$-functor $Ff\colon FG\longrightarrow FH$ is given as follows.
Its action on objects is the same as $f$.
Given $x,y\in\ob{G}$, the morphism $(Ff)_{x,y}\colon(FG)(x,y)\longrightarrow 
(FH)(fx,fy)$ is induced by the universality of coproducts, as the unique 
morphism making the following diagram commute
for all $n\in\NN$ and $w_0,\dots,w_n\in\ob{G}$ such that $w_0=x$ and $w_n=y$:
\begin{equation*}
%\label{eqn:Ff_xy}
\begin{tikzpicture}[baseline=-\the\dimexpr\fontdimen22\textfont2\relax ]
      \node(0) at (0,2.5) {$G(w_{n-1},w_n)\times\dots\times G(w_0,w_1)$};
      \node(1) at (6,1) 
      {$\displaystyle\coprod_{n\in\N}\coprod_{\substack{w_0,\dots,w_n\in\ob{G}\\
      w_0=x,w_n=y}}G(w_{n-1},w_n)\times\dots\times G(w_0,w_1)$};
      \node(2) at (0,-1) {$H(fw_{n-1},fw_n)\times\dots\times H(fw_0,fw_1)$};
      \node(3) at (6,-2.5) 
      {$\displaystyle\coprod_{n\in\N}\coprod_{\substack{v_0,\dots,v_n\in\ob{H}\\
      v_0=fx,v_n=fy}}H(v_{n-1},v_n)\times\dots\times H(v_0,v_1),$};
      
      \draw [->,rounded corners] (0)--(6,2.5)--(1);
      \node [labelsize] at (6,2.7) {$\sigma_{(n,w_0,\dots,w_n)}$};
      \draw [->] (1) to node (r)[auto,labelsize] {$(Ff)_{x,y}$} (3);
      \draw [->] (0) to node (l)[auto,swap,labelsize] 
            {$f_{w_{n-1},w_n}\times\dots\times f_{w_0,w_1}$} (2);
      \draw [->,rounded corners] (2)--(0,-2.5)--(3); 
      \node [labelsize] at (0,-2.7) {$\sigma_{(n,fw_0,\dots,fw_n)}$};
\end{tikzpicture}
\end{equation*}
where $\sigma$ denotes the appropriate coprojections.
Note that the morphism $(Ff)_{x,y}$ may be written as 
\begin{equation}
\label{eqn:Ffxy_via_Ffxyn}
(Ff)_{x,y}=\coprod_{n\in\NN}(Ff)_{x,y,n},
\end{equation}
where $(Ff)_{x,y,n}\colon (FG)(x,y)_n\longrightarrow(FH)(fx,fy)_n$ is 
characterised by the condition that the diagram 
\begin{equation*}
%\label{eqn:Ff_xyn}
\begin{tikzpicture}[baseline=-\the\dimexpr\fontdimen22\textfont2\relax ]
      \node(0) at (0,2.5) {$G(w_{n-1},w_n)\times\dots\times G(w_0,w_1)$};
      \node(1) at (6,1) 
      {$\displaystyle\coprod_{\substack{w_0,\dots,w_n\in\ob{G}\\
      w_0=x,w_n=y}}G(w_{n-1},w_n)\times\dots\times G(w_0,w_1)$};
      \node(2) at (0,-1) {$H(fw_{n-1},fw_n)\times\dots\times H(fw_0,fw_1)$};
      \node(3) at (6,-2.5) 
      {$\displaystyle\coprod_{\substack{v_0,\dots,v_n\in\ob{H}\\
      v_0=fx,v_n=fy}}H(v_{n-1},v_n)\times\dots\times H(v_0,v_1)$};
      
      \draw [->,rounded corners] (0)--(6,2.5)--(1);
      \node [labelsize] at (6,2.7) {$\sigma_{(w_0,\dots,w_n)}$};
      \draw [->] (1) to node (r)[auto,labelsize] {$(Ff)_{x,y,n}$} (3);
      \draw [->] (0) to node (l)[auto,swap,labelsize] 
            {$f_{w_{n-1},w_n}\times\dots\times f_{w_0,w_1}$} (2);
      \draw [->,rounded corners] (2)--(0,-2.5)--(3); 
      \node [labelsize] at (0,-2.7) {$\sigma_{(fw_0,\dots,fw_n)}$};
\end{tikzpicture}
\end{equation*}
commutes, and this morphism $(Ff)_{x,y,n}$ may in turn be rewritten, using 
\begin{multline*}
\coprod_{\substack{w_0,\dots,w_n\in\ob{G}\\ 
w_0=x,w_n=y}}G(w_{n-1},w_n)\times\dots\times G(w_0,w_1) \\
\cong\coprod_{\substack{v_0,\dots,v_n\in\ob{H}\\ v_0=fx,v_n=fy}}
\coprod_{\substack{w_0,\dots,w_n\in\ob{G}\\ w_0=x,w_n=y\\ 
f(w_i)=v_i}}G(w_{n-1},w_n)\times\dots\times G(w_0,w_1),
\end{multline*}
as 
\begin{equation}
\label{eqn:Ffxyn_via_Ffxynvv}
(Ff)_{x,y,n}=\coprod_{\substack{v_0,\dots,v_n\in \ob{H}\\ 
v_0=fx,v_n=fy}}(Ff)_{x,y,n,v_0,\dots,v_n},
\end{equation}
where 
\begin{equation}
\label{eqn:Ffxyvnn_as_f}
(Ff)_{x,y,n,v_0,\dots,v_n}=[f_{w_{n-1},w_n}\times\dots\times 
f_{w_0,w_1}]_{\substack{w_0,\dots,w_n\in\ob{G}\\
 w_0=x,w_n=y\\ f(w_i)=v_i}}.
\end{equation}
\begin{comment}
\label{eqn:Ff_xynvv}
\begin{tikzpicture}[baseline=-\the\dimexpr\fontdimen22\textfont2\relax ]
      \node(0) at (0,1) {$G(w_{n-1},w_n)\times\dots\times G(w_0,w_1)$};
      \node(1) at (6,2.5) 
      {$\displaystyle\coprod_{\substack{w_0,\dots,w_n\in\ob{G}\\
      w_0=x,w_n=y\\ f(w_i)=v_i}}G(w_{n-1},w_n)\times\dots\times G(w_0,w_1)$};
      \node(2) at (0,-1) {$H(fw_{n-1},fw_n)\times\dots\times H(fw_0,fw_1).$};
      \node(3) at (6,-2.5) 
      {$H(v_{n-1},v_n)\times\dots\times H(v_0,v_1)$};
      
      \draw [->,rounded corners] (0)--(0,2.5)--(1);
      \node [labelsize] at (0,2.7) {$\sigma_{(w_0,\dots,w_n)}$};
      \draw [->] (1) to node (r)[auto,labelsize] {$(Ff)_{x,y,n}$} (3);
      \draw [->] (0) to node (l)[auto,swap,labelsize] 
            {$f_{w_{n-1},w_n}\times\dots\times f_{w_0,w_1}$} (2);
      \draw [->,rounded corners] (2)--(0,-2.5)--(3); 
      \node [labelsize] at (0,-2.7) {$\id{}$};
\end{tikzpicture}
\end{comment}

\begin{proposition}\label{prop:free_V_cat_pres_pb}
If $\cat{V}$ is an extensive category with finite limits, then the functor
$F\colon\enGph{\cat{V}}\longrightarrow\enCat{\cat{V}}$ given in
Proposition~\ref{prop:free_V_cat} preserves pullbacks.
\end{proposition}
\begin{proof}
%This follows from Proposition~\ref{prop:pb_in_V_gph_and_V_cat},
%Proposition~\ref{prop:ext_coprod_of_pbs}, and 
%Corollary~\ref{cor:ext_pb_coprod} (ii).
Suppose we have a pullback
\begin{equation*}%\label{eqn:pb_in_V_gph}
\begin{tikzpicture}[baseline=-\the\dimexpr\fontdimen22\textfont2\relax ]
      \node(0) at (0,1) {$P$};
      \node(1) at (2,1) {$B$};
      \node(2) at (0,-1) {$A$};
      \node(3) at (2,-1) {$X$};
      
      \draw [->] 
            (0) to node (t)[auto,labelsize] {$k$} 
            (1);
      \draw [->] 
            (1) to node (r)[auto,labelsize] {$g$} 
            (3);
      \draw [->] 
            (0) to node (l)[auto,swap,labelsize] {$h$} 
            (2);
      \draw [->] 
            (2) to node (b)[auto,swap,labelsize] {$f$} 
            (3);  
     \draw (0.2,0.4) -- (0.6,0.4) -- (0.6,0.8);
\end{tikzpicture}
\end{equation*}
in $\enGph{\cat{V}}$.
Since $F$ does nothing on the set of objects, by 
Proposition~\ref{prop:pb_in_V_gph_and_V_cat} it suffices to show 
that for any pair $p=(a,b),p'=(a',b')\in\ob{P}$ with $f(a)=g(b)=x$
and $f(a')=g(b')=x'$, the square 
\begin{equation*}
\begin{tikzpicture}[baseline=-\the\dimexpr\fontdimen22\textfont2\relax ]
      \node(0) at (0,1) {$(FP)(p,p')$};
      \node(1) at (4,1) {$(FB)(b,b')$};
      \node(2) at (0,-1) {$(FA)(a,a')$};
      \node(3) at (4,-1) {$(FX)(x,x')$};
      
      \draw [->] 
            (0) to node (t)[auto,labelsize] {$(Fk)_{p,p'}$} 
            (1);
      \draw [->] 
            (1) to node (r)[auto,labelsize] {$(Fg)_{b,b'}$} 
            (3);
      \draw [->] 
            (0) to node (l)[auto,swap,labelsize] {$(Fh)_{p,p'}$} 
            (2);
      \draw [->] 
            (2) to node (b)[auto,swap,labelsize] {$(Ff)_{a,a'}$} 
            (3);  

%     \draw (0.2,0.4) -- (0.6,0.4) -- (0.6,0.8);
\end{tikzpicture}
\end{equation*}
is a pullback in $\cat{V}$.
Recall that 
\[
(FP)(p,p')= \coprod_{n\in\NN}(FP)(x,y)_n= 
\coprod_{n\in\N}\coprod_{\substack{p_0,\dots,p_n\in\ob{P}\\
p_0=p,p_n=p'}}P(p_{n-1},p_n)\times\dots\times P(p_0,p_1),
\]
and similarly for other objects in the above diagram.
Decomposing the morphisms by (\ref{eqn:Ffxy_via_Ffxyn}), we may apply
Proposition~\ref{prop:ext_coprod_of_pbs} and now it suffices 
to show that for each $n\in\N$, the square 
\begin{equation*}
\begin{tikzpicture}[baseline=-\the\dimexpr\fontdimen22\textfont2\relax ]
      \node(0) at (0,1) {$(FP)(p,p')_n$};
      \node(1) at (4,1) {$(FB)(b,b')_n$};
      \node(2) at (0,-1){$(FA)(a,a')_n$};
      \node(3) at (4,-1){$(FX)(x,x')_n$};
      
      \draw [->] 
            (0) to node (t)[auto,labelsize] {$(Fk)_{p,p',n}$} 
            (1);
      \draw [->] 
            (1) to node (r)[auto,labelsize] {$(Fg)_{b,b',n}$} 
            (3);
      \draw [->] 
            (0) to node (l)[auto,swap,labelsize] {$(Fh)_{p,p',n}$} 
            (2);
      \draw [->] 
            (2) to node (b)[auto,swap,labelsize] {$(Ff)_{a,a',n}$} 
            (3);  

%     \draw (0.2,0.4) -- (0.6,0.4) -- (0.6,0.8);
\end{tikzpicture}
\end{equation*}
is a pullback.
Decomposing the morphisms by (\ref{eqn:Ffxyn_via_Ffxynvv}) and applying 
Proposition~\ref{prop:ext_coprod_of_pbs} once again, we see that it suffices to 
show that for each $n\in\N$ and $x_0,\dots,x_n\in\ob{X}$ with $x_0=x$ and 
$x_n=x'$, the square
\[
\begin{tikzpicture}[baseline=-\the\dimexpr\fontdimen22\textfont2\relax ]
      \node(0) at (0,2.5) {$\displaystyle 
      \coprod_{\substack{p_0,\dots,p_n\in\ob{P}\\
            p_0=p,p_n=p'\\f\circ h(p_i)=x_i}}P(p_{n-1},p_n)\times\dots\times 
            P(p_0,p_1)$};
      \node(1) at (6,1) 
      {$\displaystyle \coprod_{\substack{b_0,\dots,b_n\in\ob{B}\\
                  b_0=b,b_n=b'\\g(b_i)=x_i}}B(b_{n-1},b_n)\times\dots\times 
                  B(b_0,b_1)$};
      \node(2) at (0,-1) {$\displaystyle 
      \coprod_{\substack{a_0,\dots,a_n\in\ob{A}\\
                  a_0=a,a_n=a'\\f(a_i)=x_i}}A(a_{n-1},a_n)\times\dots\times 
                  A(a_0,a_1)$};
      \node(3) at (6,-2.5) 
      {$X(x_{n-1},x_n)\times\dots\times X(x_0,x_1)$};
      
      \draw [->,rounded corners] (0)--(6,2.5)--(1);
      \node [labelsize] at (6,3.2) 
      {$\displaystyle\coprod_{\substack{b_0,\dots,b_n\in\ob{B}\\
      b_0=b,b_n=b'\\g(b_i)=x_i}} (Fk)_{p,p',n,b_0,\dots,b_n}$};
      \draw [->] (1) to node (r)[auto,labelsize] 
      {$(Fg)_{b,b',n,x_0,\dots,x_n}$} (3);
      \draw [->] (0) to node (l)[auto,swap,labelsize] 
      {$\displaystyle\coprod_{\substack{a_0,\dots,a_n\in\ob{A}\\
      a_0=a,a_n=a'\\f(a_i)=x_i}} (Fh)_{p,p',n,a_0,\dots,a_n}$} (2);
      \draw [->,rounded corners] (2)--(0,-2.5)--(3); 
      \node [labelsize] at (0,-2.7) {$(Ff)_{a,a',n,x_0,\dots,x_n}$};
\end{tikzpicture}
\]
is a pullback.
Writing the indexing sets of the coproducts appearing in the above diagram as 
\begin{align*}
I&=\{\,(a_0,\dots,a_n)\mid a_i\in\ob{A}, a_0=a,a_n=a',f(a_i)=x_i\,\},\\
J&=\{\,(b_0,\dots,b_n)\mid b_i\in\ob{B}, b_0=b,b_n=b',g(b_i)=x_i\,\},\\
K&=\{\,(p_0,\dots,p_n)\mid p_i\in\ob{P}, p_0=p,p_n=p',f\circ h(p_i)=x_i\,\},
\end{align*}
we have $I\times J\cong K$ by the description of $\ob{P}$ as a pullback.
Using (\ref{eqn:Ffxyvnn_as_f}) and the second clause of
Corollary~\ref{cor:ext_pb_coprod}, 
it suffices to show that for any $n,x_i,a_i,b_i,p_i$ with $p_i=(a_i,b_i)$,
$f(a_i)=g(b_i)=x_i$, the square
\begin{equation*}
\begin{tikzpicture}[baseline=-\the\dimexpr\fontdimen22\textfont2\relax ]
      \node(0) at (0,1) {$P(p_{n-1},p_n)\times\dots\times 
      P(p_0,p_1)$};
      \node(1) at (8,1) {$B(b_{n-1},b_n)\times\dots\times 
            B(b_0,b_1)$};
      \node(2) at (0,-1) {$A(a_{n-1},a_n)\times\dots\times 
            A(a_0,a_1)$};
      \node(3) at (8,-1) {$ X(x_{n-1},x_n)\times\dots\times X(x_0,x_1)$};
      
      \draw [->] 
            (0) to node (t)[auto,labelsize] 
            {$k_{p_{n-1},p_n}\times\dots\times k_{p_0,p_1}$} 
            (1);
      \draw [->] 
            (1) to node (r)[midway,fill=white,labelsize] 
            {$g_{b_{n-1},b_n}\times\dots\times g_{b_0,b_1}$} 
            (3);
      \draw [->] 
            (0) to node (l)[midway,fill=white,labelsize] 
            {$h_{p_{n-1},p_n}\times\dots\times h_{p_0,p_1}$} 
            (2);
      \draw [->] 
            (2) to node (b)[auto,swap,labelsize] 
            {$f_{a_{n-1},a_n}\times\dots\times f_{a_0,a_1}$} 
            (3);  
\end{tikzpicture}
\end{equation*}
is a pullback.
This follows from the fact that each $P(p_i,p_{i+1})$
is the pullback of $A(a_i,a_{i+1})$ and $B(b_i,b_{i+1})$
over $X(x_i,x_{i+1})$, as pullbacks commute with products.
\end{proof}

\begin{proposition}\label{prop:forgetful_V_cat_coprod}
If $\cat{V}$ has a strict initial object $0$ and finite products,
%such that for any object $B\in\cat{V}$, $0\times B\cong 0$,
then the categories $\enGph{\cat{V}}$ and $\enCat{\cat{V}}$
admit small coproducts and the forgetful functor
$U\colon \enCat{\cat{V}}\longrightarrow\enGph{\cat{V}}$
preserves small coproducts.
\end{proposition}
\begin{proof}
In both $\enGph{\cat{V}}$ and $\enCat{\cat{V}}$,
small coproducts are given by taking disjoint union of objects and
setting the hom-objects between objects from different components to be $0$
(see the proofs of Propositions~\ref{prop:VGph_ext} and \ref{prop:VCat_ext}). 
\end{proof}

\begin{proposition}\label{prop:free_V_cat_unit}
If $\cat{V}$ is an extensive category with finite limits, then
the unit $\eta\colon \id{\enGph{\cat{V}}}\Longrightarrow UF$
of the adjunction $F\dashv U$ in Proposition~\ref{prop:free_V_cat}
is cartesian.
\end{proposition}
\begin{proof}
By Proposition~\ref{prop:pb_in_V_gph_and_V_cat}, it 
suffices to show that for any morphism 
$f\colon G\longrightarrow H$ of $\cat{V}$-graphs,
and $x,y\in\ob{G}$,
the square
\begin{equation*}
\begin{tikzpicture}[baseline=-\the\dimexpr\fontdimen22\textfont2\relax ]
      \node(0) at (0,1) {$G(x,y)$};
      \node(1) at (4,1) {$(FG)(x,y)$};
      \node(2) at (0,-1) {$H(fx,fy)$};
      \node(3) at (4,-1) {$(FH)(fx,fy)$};
      
      \draw [->] 
            (0) to node (t)[auto,labelsize] {$\sigma_{(1,x,y)}$} 
            (1);
      \draw [->] 
            (1) to node (r)[auto,labelsize] {$(Ff)_{x,y}$} 
            (3);
      \draw [->] 
            (0) to node (l)[auto,swap,labelsize] {$f_{x,y}$} 
            (2);
      \draw [->] 
            (2) to node (b)[auto,swap,labelsize] {$\sigma_{(1,fx,fy)}$} 
            (3);  
\end{tikzpicture}
\end{equation*}
is a pullback in $\cat{V}$. 
This follows from 
Proposition~\ref{prop:ext_criterion}
(note that if we rewrite $(Ff)_{x,y}$ via (\ref{eqn:Ffxy_via_Ffxyn}) and 
(\ref{eqn:Ffxyn_via_Ffxynvv}), we have $(Ff)_{x,y,1,fx,fy}=f_{x,y}$).
\end{proof}

We only need the case $m=1$ of the following proposition in order to show
that $\monoid{T}$ is cartesian; the full generality of this stronger version 
will be needed in the next section.
\begin{proposition}\label{prop:free_V_cat_counit}
If $\cat{V}$ is an extensive category with finite limits,
then for each natural number $m$, the natural transformation 
\begin{equation*}
\begin{tikzpicture}[baseline=-\the\dimexpr\fontdimen22\textfont2\relax ]
      \node(0) at (0,0) {$(\enGph{\cat{V}})^m$};
      \node(1) at (3,0) {$(\enCat{\cat{V}})^m$};
      \node(2) at (6,0) {$\enCat{\cat{V}}$};
      \node(3) at (8,1) {$\enGph{\cat{V}}$};
      \node(4) at (10,0) {$\enCat{\cat{V}},$};
      
      \draw [->] 
            (0) to node (t)[auto,labelsize] {$F^m$} 
            (1);
      \draw [->] 
            (1) to node (r)[auto,labelsize] {$\prod$} 
            (2);
      \draw [->,bend left=20] 
            (2) to node (l)[auto,labelsize] {$U$} 
            (3);
      \draw [->,bend left=20] 
            (3) to node (b)[auto,labelsize] {$F$} 
            (4);  
      \draw [->] 
            (2) to node (b)[auto,swap,labelsize] {$\id{\enCat{\cat{V}}}$} 
            (4);  

	\draw [2cell] (8,0.9) to node[auto,labelsize] {$\varepsilon$}  
	(8,0.1);          
%     \draw (0.2,0.4) -- (0.6,0.4) -- (0.6,0.8);
\end{tikzpicture}
\end{equation*}
where $\varepsilon$ is the counit of the adjunction $F\dashv U$
in Proposition~\ref{prop:free_V_cat}
and $\prod$ is the $m$-ary product functor,
is cartesian.
\end{proposition}
\begin{proof}
Let $f=(f^{(1)},\dots, f^{(m)})\colon (G^{(1)},\dots,G^{(m)})\longrightarrow
(H^{(1)},\dots,H^{(m)})$
be a morphism in $(\enGph{\cat{V}})^m$. 
Our aim is to show that the square 
\begin{equation*}
\begin{tikzpicture}[baseline=-\the\dimexpr\fontdimen22\textfont2\relax ]
      \node(0) at (0,1) {$FU(FG^{(1)}\times \dots \times FG^{(m)})$};
      \node(1) at (7,1) {$FG^{(1)}\times \dots \times FG^{(m)}$};
      \node(2) at (0,-1) {$FU(FH^{(1)}\times \dots \times FH^{(m)})$};
      \node(3) at (7,-1) {$FH^{(1)}\times \dots \times FH^{(m)}$};
      
      \draw [->] 
            (0) to node (t)[auto,labelsize] {${\varepsilon_{FG^{(1)}\times 
            \dots \times FG^{(m)}}}$} 
            (1);
      \draw [->] 
            (1) to node (r)[auto,labelsize] {${Ff^{(1)}\times \dots \times 
            Ff^{(m)}}$} 
            (3);
      \draw [->] 
            (0) to node (l)[auto,swap,labelsize] {${FU(Ff^{(1)}\times \dots 
            \times Ff^{(m)})}$} 
            (2);
      \draw [->] 
            (2) to node (b)[auto,swap,labelsize] 
            {${\varepsilon_{FH^{(1)}\times \dots \times FH^{(m)}}}$} 
            (3);  
            
%     \draw (0.2,0.4) -- (0.6,0.4) -- (0.6,0.8);
\end{tikzpicture}
\end{equation*}
in $\enCat{\cat{V}}$ is a pullback.
By Proposition~\ref{prop:pb_in_V_gph_and_V_cat} it suffices 
to show that for every pair of objects 
$x=(x^{(1)},\dots,x^{(m)}),y=(y^{(1)},\dots,y^{(m)})\in\ob{FU(FG^{(1)}\times 
\dots \times FG^{(m)})}=\ob{G^{(1)}}\times \dots \times \ob{G^{(m)}}$, the 
square
\begin{equation}
\label{eqn:homs_complicated}
\begin{tikzpicture}[baseline=-\the\dimexpr\fontdimen22\textfont2\relax ]
      \node(0) at (0,1) 
      {$\big(FU\prod_{i=1}^{m}(FG^{(i)})\big)(x,y)$};
      \node(1) at (6.5,1) {$\prod_{i=1}^{m}(FG^{(i)})(x^{(i)},y^{(i)})$};
      \node(2) at (0,-1) 
      {$\big(FU\prod_{i=1}^{m}(FH^{(i)})\big)(fx,fy)$};
      \node(3) at (6.5,-1) 
      {$\prod_{i=1}^{m}(FH^{(i)})(f^{(i)}x^{(i)},f^{(i)}y^{(i)})$};
      
      \draw [->] 
            (0) to node (t)[auto,labelsize] {$$} 
            (1);
      \draw [->] 
            (1) to node (r)[auto,labelsize] {$$} 
            (3);
      \draw [->] 
            (0) to node (l)[auto,swap,labelsize] {$$} 
            (2);
      \draw [->] 
            (2) to node (b)[auto,swap,labelsize] 
            {$$} 
            (3);  
            
%     \draw (0.2,0.4) -- (0.6,0.4) -- (0.6,0.8);
\end{tikzpicture}
\end{equation}
in $\cat{V}$ is a pullback.
Using Proposition~\ref{prop:extensive_prod}, we may rewrite the bottom right 
object $\prod_{i=1}^{m}(FH^{(i)})(f^{(i)}x^{(i)},f^{(i)}y^{(i)})$
as a coproduct of products. Precisely, we define an (indexing) set $I$ to be 
\begin{multline*}
I=\{\,(n_1,v^{(1)}_0,\dots, v^{(1)}_{n_1},\dots, n_m, 
v^{(m)}_0,\dots,v^{(m)}_{n_m})\\
\mid n_i\in \NN, v^{(i)}_j\in\ob{H^{(i)}}, v^{(i)}_0=f^{(i)}x^{(i)},
v^{(i)}_{n_i}=f^{(i)}y^{(i)} \,\}.
\end{multline*}
Then $\prod_{i=1}^{m}(FH^{(i)})(f^{(i)}x^{(i)},f^{(i)}y^{(i)})$ is isomorphic to
\begin{multline}
\label{eqn:large_hom_in_proof}
\coprod_{I} H^{(1)}(v^{(1)}_{n_1-1},v^{(1)}_{n_1})\times\dots\times 
H^{(1)}(v^{(1)}_{0},v^{(1)}_{1})\\ \times\dots\times 
H^{(m)}(v^{(m)}_{n_m-1},v^{(m)}_{n_m})\times\dots\times 
H^{(m)}(v^{(m)}_{0},v^{(m)}_{1}).
\end{multline}
We may now decompose the diagram (\ref{eqn:homs_complicated}) into coproducts 
over $I$ and apply Proposition~\ref{prop:ext_coprod_of_pbs}.
Fix an element $(n_1,v^{(1)}_0,\dots, v^{(1)}_{n_1},\dots, n_m, 
v^{(m)}_0,\dots,v^{(m)}_{n_m})\in I$,
and introduce new indexing sets 
\begin{align*}
J&=\{\,(k,j^{(1)}_1,\dots,j^{(1)}_{k-1},\dots, j^{(m)}_1,\dots,j^{(m)}_{k-1})\\
&\qquad\qquad\qquad\mid k, j^{(i)}_l\in \NN, 0\leq j^{(i)}_1\leq\dots\leq 
j^{(i)}_{k-1}\leq 
n_i\,\},\\
K&=\{\, (w^{(1)}_0,\dots,w^{(1)}_{n_1},\dots,w^{(m)}_0,\dots,w^{(m)}_{n_m})\\
&\qquad\qquad\qquad\mid w^{(i)}_j\in\ob{G^{(i)}},w^{(i)}_0=x^{(i)}, 
w^{(i)}_{n_i}=y^{(i)}, 
f^{(i)}w^{(i)}_j=v^{(i)}_j\,\}.
\end{align*}
It suffices to show that the square 
\[
\begin{tikzpicture}[baseline=-\the\dimexpr\fontdimen22\textfont2\relax ]
      \node(0) at (0,2.5) {$\displaystyle 
      \coprod_J \coprod_K
      G^{(1)}(w^{(1)}_{n_1-1},w^{(1)}_{n_1})\times\dots\times 
      G^{(m)}(w^{(m)}_{0},w^{(m)}_{1})$};
      \node(1) at (6,1) 
      {$\displaystyle 
      \coprod_K
      G^{(1)}(w^{(1)}_{n_1-1},w^{(1)}_{n_1})\times\dots\times 
      G^{(m)}(w^{(m)}_{0},w^{(m)}_{1})$};
      \node(2) at (0,-1) {$\displaystyle \coprod_J
      H^{(1)}(v^{(1)}_{n_1-1},v^{(1)}_{n_1})\times\dots\times 
      H^{(m)}(v^{(m)}_{0},v^{(m)}_{1})$};
      \node(3) at (6,-2.5) 
      {$H^{(1)}(v^{(1)}_{n_1-1},v^{(1)}_{n_1})\times\dots\times 
      H^{(m)}(v^{(m)}_{0},v^{(m)}_{1})$};
      
      \draw [->,rounded corners] (0)--(6,2.5)--(1);
      \node [labelsize] at (6,2.7) 
      {$\coprod_K [\id{}]_J$};
      \draw [->] (1) to node (r)[midway,fill=white,labelsize,xshift={30pt}] 
      {$[f^{(1)}_{w^{(1)}_{n_1-1},w^{(1)}_{n_1}}\times \dots \times 
      f^{(m)}_{w^{(m)}_{0},w^{(m)}_{1}}]_K$} (3);
      \draw [->] (0) to node (l)[midway,fill=white,labelsize,xshift={-30pt}] 
      {$\coprod_J[f^{(1)}_{w^{(1)}_{n_1-1},w^{(1)}_{n_1}}\times \dots \times 
            f^{(m)}_{w^{(m)}_{0},w^{(m)}_{1}}]_K$} (2);
      \draw [->,rounded corners] (2)--(0,-2.5)--(3); 
      \node [labelsize] at (0,-2.7) {$[\id{}]_{J}$};
\end{tikzpicture}
\]
is a pullback,
which follows from the second clause of Corollary~\ref{cor:ext_pb_coprod}.
\end{proof}

\begin{thm}
\label{thm:free_V_cat_monad_cart}
If $\cat{V}$ is an extensive category with finite limits, then the 
free $\cat{V}$-category monad $\monoid{T}=(T,\eta,\mu)$ on $\enGph{\cat{V}}$ is 
cartesian.
\end{thm}
\begin{proof}
The functor $T=UF$ preserves pullbacks since both $U$ (because it is a right 
adjoint) and $F$ (by Proposition~\ref{prop:free_V_cat_pres_pb}) do.
The unit $\eta$ is cartesian by Proposition~\ref{prop:free_V_cat_unit}.
The multiplication $\mu=U\varepsilon F$ is cartesian because $\varepsilon F$ 
is so by Proposition~\ref{prop:free_V_cat_counit} (take $m=1$), 
and $U$ preserves pullbacks.
\end{proof}

\section{The free strict $n$-dimensional $\cat{V}$-category monad
}\label{sec:n-gph_n-cat}
In this section we %define inductively
%the categories of $n$-graphs and of strict $n$-categories,
%together with 
show that the forgetful functor from the category of strict
$n$-dimensional $\cat{V}$-categories to that of 
$n$-dimensional $\cat{V}$-graphs has a left adjoint. We
assume throughout that $\cat{V}$ is extensive and has finite limits.
It follows that $\enGph{\cat{V}}$ and $\enCat{\cat{V}}$ are likewise
(by Propositions~\ref{prop:VGph_ext}, \ref{prop:VCat_ext} and 
\ref{prop:pb_in_V_gph_and_V_cat}), and so, by 
induction, for each natural number $n$,
the categories $\enGph{\cat{V}}^{(n)}$ and $\enCat{\cat{V}}^{(n)}$ are
also extensive with finite limits. 

Recall that, by Propositions~\ref{prop:extensive_prod}
and~\ref{prop:free_V_cat}, the forgetful functor 
\[U\colon
\enCat{(\enCat{\cat{V}}^{(n)})}\longrightarrow\enGph{(\enCat{\cat{V}}^{(n)})}\]
admits a left adjoint $F$.

\begin{definition}
For each natural number $n$, we define an adjunction
$F^{(n)}\dashv U^{(n)}\colon
\enCat{\cat{V}}^{(n)}\longrightarrow\enGph{\cat{V}}^{(n)}$
recursively as follows:
\begin{enumerate}
\item $F^{(0)}=U^{(0)}=\id{\cat{V}}$;
\item $F^{(n+1)}\dashv U^{(n+1)}$ is the composite:
\[
\begin{tikzpicture}[baseline=-\the\dimexpr\fontdimen22\textfont2\relax ]
%      \node(0) at (0,4.9) {$\enCat{\cat{V}}^{(n+1)}$};
      \node(1) at (10,0) {$\enCat{(\enCat{\cat{V}}^{(n)})}.$};
      \node(2) at (5,0) {$\enGph{(\enCat{\cat{V}}^{(n)})}$};
      \node(3) at (0,0) {$\enGph{(\enGph{\cat{V}}^{(n)})}$};
%      \node(4) at (0,0.1) {$\enGph{\cat{V}}^{(n+1)}.$};
      
%      \node[rotate=90] at (0.03,4.45) {$=$}; 
      \draw [transform canvas={yshift=0.5em},<-] 
            (1) to node [auto,swap,labelsize] {$F$} 
            (2);
	\node [labelsize,rotate=-90]at (2.5,0) {$\dashv$};
      \draw [transform canvas={yshift=-0.5em},->] 
            (1) to node [auto,labelsize] {$U$} 
            (2);
      \draw [transform canvas={yshift=0.5em},<-] 
            (2) to node [auto,swap,labelsize] {$\enGph{F^{(n)}}$} 
            (3);
	\node[labelsize,rotate=-90] at (7.5,0) {$\dashv$};
      \draw [transform canvas={yshift=-0.5em},->] 
            (2) to node [auto,labelsize] {$\enGph{U^{(n)}}$} 
            (3);
%      \node[rotate=90] at (0.03,0.55) {$=$}; 
\end{tikzpicture}
\] 
\iffalse \begin{comment}
\[
\begin{tikzpicture}[baseline=-\the\dimexpr\fontdimen22\textfont2\relax ]
      \node(0) at (0,4.9) {$\enCat{\cat{V}}^{(n+1)}$};
      \node(1) at (0,4) {$\enCat{(\enCat{\cat{V}}^{(n)})}$};
      \node(2) at (0,2.5) {$\enGph{(\enCat{\cat{V}}^{(n)})}$};
      \node(3) at (0,1) {$\enGph{(\enGph{\cat{V}}^{(n)})}$};
      \node(4) at (0,0.1) {$\enGph{\cat{V}}^{(n+1)}.$};
      
      \node[rotate=90] at (0.03,4.45) {$=$}; 
      \draw [transform canvas={xshift=-0.5em},<-] 
            (1) to node [auto,swap,labelsize] {$F$} 
            (2);
	\node [labelsize]at (0,3.25) {$\dashv$};
      \draw [transform canvas={xshift=0.5em},->] 
            (1) to node [auto,labelsize] {$U$} 
            (2);
      \draw [transform canvas={xshift=-0.5em},<-] 
            (2) to node [auto,swap,labelsize] {$\enGph{F^{(n)}}$} 
            (3);
	\node[labelsize] at (0,1.75) {$\dashv$};
      \draw [transform canvas={xshift=0.5em},->] 
            (2) to node [auto,labelsize] {$\enGph{U^{(n)}}$} 
            (3);
      \node[rotate=90] at (0.03,0.55) {$=$}; 
\end{tikzpicture}
\] 
\end{comment} \fi
\end{enumerate}
\end{definition}

The adjunction $F^{(n)}\dashv U^{(n)}$ induces a monad $\monoid{T}^{(n)}=
(T^{(n)},\eta^{(n)},\mu^{(n)})$ on $\enGph{\cat{V}}^{(n)}$.
We call $\monoid{T}^{(n)}$ the \defemph{free strict $n$-dimensional 
$\cat{V}$-category monad}, and
now show that it is cartesian.

\begin{proposition}\label{prop:free_n_V_cat_pres_pb}
For each natural number $n$, %the functor 
$F^{(n)}\colon \enGph{\cat{V}}^{(n)}\longrightarrow\enCat{\cat{V}}^{(n)}$
preserves pullbacks. 
\end{proposition}
\begin{proof}
%We prove this by induction. 
For $n=0$, the assertion is trivial.
Proceeding inductively, 
if $F^{(n)}$ preserves pullbacks,
so does $\enGph{F^{(n)}}$ by Proposition~\ref{prop:pb_in_V_gph_and_V_cat}.
The functor $F\colon \enGph{(\enCat{\cat{V}}^{(n)})}\longrightarrow 
\enCat{(\enCat{\cat{V}}^{(n)})}$
preserves pullbacks by Proposition~\ref{prop:free_V_cat_pres_pb}.
\end{proof}

\begin{proposition}\label{prop:forgetful_n_pres_coprod}
For each natural number $n$, %the functor 
$U^{(n)}\colon \enCat{\cat{V}}^{(n)}\longrightarrow\enGph{\cat{V}}^{(n)}$
preserves small coproducts. 
\end{proposition}
\begin{proof}
For $n=0$, the assertion is trivial. Proceeding inductively, if
$U^{(n)}$ preserves small coproducts, it preserves initial
objects, and so the functor $\enGph{U^{(n)}}$ preserves
small coproducts. The functor $U\colon
\enCat{(\enCat{\cat{V}}^{(n)})}\longrightarrow
\enGph{(\enCat{\cat{V}}^{(n)})}$ also preserves small coproducts by
Proposition~\ref{prop:forgetful_V_cat_coprod}.
\end{proof}

\begin{proposition}\label{prop:free_n_V_cat_unit_cart}
For each natural number $n$, the unit $\eta^{(n)}\colon
\id{\enGph{\cat{V}}^{(n)}} \Longrightarrow U^{(n)}F^{(n)}$
of the adjunction $F^{(n)}\dashv U^{(n)}$ is cartesian.
\end{proposition}
\begin{proof}
Observe that adjunctions whose units are cartesian are closed under
composition.  
Proceeding inductively, if 
$\eta^{(n)}$ is cartesian, so is $\enGph{\eta^{(n)}}$
by Proposition~\ref{prop:pb_in_V_gph_and_V_cat}.
The unit of the adjunction $F\dashv U\colon
\enCat{(\enCat{\cat{V}}^{(n)})}\longrightarrow
\enGph{(\enCat{\cat{V}}^{(n)})}$ is cartesian by
Proposition~\ref{prop:free_V_cat_unit}. 
\end{proof}

\begin{proposition}\label{prop:free_n_V_cat_counit}
For each pair of natural numbers $n$ and $m$, the natural transformation 
\begin{equation*}
\begin{tikzpicture}[baseline=-\the\dimexpr\fontdimen22\textfont2\relax ]
      \node(0) at (-0.5,0) {$(\enGph{\cat{V}}^{(n)})^m$};
      \node(1) at (3,0) {$(\enCat{\cat{V}}^{(n)})^m$};
      \node(2) at (6,0) {$\enCat{\cat{V}}^{(n)}$};
      \node(3) at (8,1) {$\enGph{\cat{V}}^{(n)}$};
      \node(4) at (10,0) {$\enCat{\cat{V}}^{(n)},$};
      
      \draw [->] 
            (0) to node (t)[auto,labelsize] {$(F^{(n)})^m$} 
            (1);
      \draw [->] 
            (1) to node (r)[auto,labelsize] {$\prod$} 
            (2);
      \draw [->,bend left=20] 
            (2) to node (l)[auto,labelsize] {$U^{(n)}$} 
            (3);
      \draw [->,bend left=20] 
            (3) to node (b)[auto,labelsize] {$F^{(n)}$} 
            (4);  
      \draw [->] 
            (2) to node (b)[auto,swap,labelsize] {$\id{\enCat{\cat{V}}^{(n)}}$} 
            (4);  

	\draw [2cell] (8,0.8) to node[auto,labelsize] {$\varepsilon^{(n)}$}  
	(8,0.1);          
%     \draw (0.2,0.4) -- (0.6,0.4) -- (0.6,0.8);
\end{tikzpicture}
\end{equation*}
where $\varepsilon^{(n)}$ is the counit of the adjunction $F^{(n)}\dashv 
U^{(n)}$
and $\prod$ is the $m$-ary product functor,
is cartesian.
\end{proposition}
\begin{proof}
By induction on $n$. Suppose the claim is true for $n=k$ and for all $m$. 
For brevity, we will write the adjunction $\enGph{F^{(k)}}\dashv\enGph{U^{(k)}}$
as $F'\dashv U'$, and whose counit $\enGph{\varepsilon^{(k)}}$ as 
$\varepsilon'$.
We aim to show that for every
morphism 
$(f^{(1)},\dots, f^{(m)})\colon (G^{(1)},\dots,G^{(m)})\longrightarrow
(H^{(1)},\dots,H^{(m)})$
in $(\enGph{\cat{V}}^{(k+1)})^m$,  
the outer rectangle in the diagram
\begin{equation*}
\begin{tikzpicture}[baseline=-\the\dimexpr\fontdimen22\textfont2\relax ]
      \node(0) at (0,1) {$FF'U'U(\prod_{i=1}^{m}FF'G^{(i)})$};
      \node(1) at (5.8,1) {$FU(\prod_{i=1}^{m}FF'G^{(i)})$};
      \node(2) at (0,-1) {$FF'U'U(\prod_{i=1}^{m}FF'H^{(i)})$};
      \node(3) at (5.8,-1) {$FU(\prod_{i=1}^{m}FF'H^{(i)})$};
      \node(4) at (10.5,1) {$\prod_{i=1}^{m}FF'G^{(i)}$};
      \node(5) at (10.5,-1) {$\prod_{i=1}^{m}FF'H^{(i)}$};
      
      \draw [->] 
            (0) to node (t)[auto,labelsize] 
            {${F\varepsilon'_{U(\prod_{i=1}^{m}FF'G^{(i)})}}$} 
            (1);
      \draw [->] 
            (1) to node (r)[midway,fill=white,labelsize] 
            {${FU(\prod_{i=1}^{m}FF'f^{(i)})}$} 
            (3);
      \draw [->] 
            (0) to node (l)[midway,fill=white,labelsize] 
            {${FF'U'U(\prod_{i=1}^{m}FF'f^{(i)})}$} 
            (2);
      \draw [->] 
            (2) to node (b)[auto,swap,labelsize] 
            {${F\varepsilon'_{U(\prod_{i=1}^{m}FF'H^{(i)})}}$} 
            (3);  
      \draw [->] 
            (1) to node [auto,labelsize] 
            {${\varepsilon_{\prod_{i=1}^{m}FF'G^{(i)}}}$} 
            (4);
      \draw [->] 
            (4) to node [midway,fill=white,labelsize] 
            {${\prod_{i=1}^{m}FF'f^{(i)}}$} 
            (5);
      \draw [->] 
            (3) to node [auto,swap,labelsize] 
            {${\varepsilon_{\prod_{i=1}^{m}FF'H^{(i)}}}$} 
            (5);
\end{tikzpicture}
\end{equation*}
in $\enCat{\cat{V}}^{(k+1)}$ is a pullback.
The right square is a pullback by Proposition~\ref{prop:free_V_cat_counit},
so we shall show that the left square is also a pullback.
Since $F$ preserves pullbacks by Proposition~\ref{prop:free_V_cat_pres_pb}, 
it suffices to show that the square
\begin{equation*}
\begin{tikzpicture}[baseline=-\the\dimexpr\fontdimen22\textfont2\relax ]
      \node(0) at (0,1) {$F'U'U(\prod_{i=1}^{m}FF'G^{(i)})$};
      \node(1) at (6,1) {$U(\prod_{i=1}^{m}FF'G^{(i)})$};
      \node(2) at (0,-1) {$F'U'U(\prod_{i=1}^{m}FF'H^{(i)})$};
      \node(3) at (6,-1) {$U(\prod_{i=1}^{m}FF'H^{(i)})$};
      
      \draw [->] 
            (0) to node (t)[auto,labelsize] 
            {${\varepsilon'_{U(\prod_{i=1}^{m}FF'G^{(i)})}}$} 
            (1);
      \draw [->] 
            (1) to node (r)[auto,labelsize] {${U(\prod_{i=1}^{m}FF'f^{(i)})}$} 
            (3);
      \draw [->] 
            (0) to node (l)[auto,swap,labelsize] 
            {${F'U'U(\prod_{i=1}^{m}FF'f^{(i)})}$} 
            (2);
      \draw [->] 
            (2) to node (b)[auto,swap,labelsize] 
            {${\varepsilon'_{U(\prod_{i=1}^{m}FF'H^{(i)})}}$} 
            (3);   
            
%     \draw (0.2,0.4) -- (0.6,0.4) -- (0.6,0.8);
\end{tikzpicture}
\end{equation*}
in $\enGph{(\enCat{\cat{V}}^{(k)})}$ is a pullback.
By Proposition~\ref{prop:pb_in_V_gph_and_V_cat},
it suffices to show that for every pair of objects 
$(x^{(1)},\dots,x^{(m)}),(y^{(1)},\dots,y^{(m)})\in
\ob{F'U'U(\prod_{i=1}^{m}FF'G^{(i)})}=\ob{G^{(1)}}\times \dots \times 
\ob{G^{(m)}}$, 
the square
\begin{equation}
\label{eqn:large_hom_1}
\begin{tikzpicture}[baseline=-\the\dimexpr\fontdimen22\textfont2\relax ]
      \node(0) at (0,1) 
      {$F^{(k)}U^{(k)}(\prod_{i=1}^{m}(FF'G^{(i)})(x^{(i)},y^{(i)}))$};
      \node(1) at (7.5,1) {$\prod_{i=1}^{m}(FF'G^{(i)})(x^{(i)},y^{(i)})$};
      \node(2) at (0,-1) 
      {$F^{(k)}U^{(k)}(\prod_{i=1}^{m}(FF'H^{(i)})(f^{(i)}x^{(i)},f^{(i)}y^{(i)}))$};
      \node(3) at (7.5,-1) 
      {$\prod_{i=1}^{m}(FF'H^{(i)})(f^{(i)}x^{(i)},f^{(i)}y^{(i)})$};
      
      \draw [->] 
            (0) to node (t)[auto,labelsize] 
            {${}$} 
            (1);
      \draw [->] 
            (1) to node (r)[auto,labelsize] {${}$} 
            (3);
      \draw [->] 
            (0) to node (l)[auto,swap,labelsize] 
            {$$} 
            (2);
      \draw [->] 
            (2) to node (b)[auto,swap,labelsize] 
            {${}$} 
            (3);   
            
%     \draw (0.2,0.4) -- (0.6,0.4) -- (0.6,0.8);
\end{tikzpicture}
\end{equation}
in $\enCat{\cat{V}}^{(k)}$ is a pullback.
The bottom right object may be rewritten, using the set 
\begin{multline*}
I=\{\,(n_1,v^{(1)}_0,\dots, v^{(1)}_{n_1},\dots, n_m, 
v^{(m)}_0,\dots,v^{(m)}_{n_m})\\
\mid n_i\in \NN, v^{(i)}_j\in\ob{H^{(i)}}, v^{(i)}_0=f^{(i)}x^{(i)},
v^{(i)}_{n_i}=f^{(i)}y^{(i)} \,\},
\end{multline*}
as
\begin{multline*}
\coprod_{I} F^{(k)}\big(H^{(1)}(v^{(1)}_{n_1-1},v^{(1)}_{n_1})\big)
\times\dots\times 
F^{(k)}\big(H^{(1)}(v^{(1)}_{0},v^{(1)}_{1})\big)\\ \times\dots\times 
F^{(k)}\big(H^{(m)}(v^{(m)}_{n_m-1},v^{(m)}_{n_m})\big)\times\dots\times 
F^{(k)}\big(H^{(m)}(v^{(m)}_{0},v^{(m)}_{1})\big);
\end{multline*}
cf.~(\ref{eqn:large_hom_in_proof}).
Because both $F^{(k)}$ and $U^{(k)}$ (by 
Proposition~\ref{prop:forgetful_n_pres_coprod}) preserve small 
coproducts, we may decompose (\ref{eqn:large_hom_1})
as the coproduct over the set $I$ and apply
Proposition~\ref{prop:ext_coprod_of_pbs}. 
Fix an element $(n_1,v^{(1)}_0,\dots, v^{(1)}_{n_1},\dots, n_m, 
v^{(m)}_0,\dots,v^{(m)}_{n_m})\in I$ and 
introduce the set 
\begin{multline*}
K=\{\, (w^{(1)}_0,\dots,w^{(1)}_{n_1},\dots,w^{(m)}_0,\dots,w^{(m)}_{n_m})\\
\mid w^{(i)}_j\in\ob{G^{(i)}},w^{(i)}_0=x^{(i)}, 
w^{(i)}_{n_i}=y^{(i)}, 
f^{(i)}w^{(i)}_j=v^{(i)}_j\,\}.
\end{multline*}
It suffices to show that the square 
\[
\begin{tikzpicture}[baseline=-\the\dimexpr\fontdimen22\textfont2\relax ]
      \node(0)[align=left] at (0,3) {$\displaystyle 
      \coprod_K
      F^{(k)}U^{(k)}F^{(k)}G^{(1)}(w^{(1)}_{n_1-1},w^{(1)}_{n_1})\times\cdots$\\
      $\qquad
      \cdots\times 
      F^{(k)}U^{(k)}F^{(k)}G^{(m)}(w^{(m)}_{0},w^{(m)}_{1})$};
      
      \node(1)[align=left] at (7,1) 
      {$\displaystyle 
      \coprod_K
      F^{(k)}G^{(1)}(w^{(1)}_{n_1-1},w^{(1)}_{n_1})\times\cdots$\\
      $\qquad\cdots\times 
      F^{(k)}G^{(m)}(w^{(m)}_{0},w^{(m)}_{1})$};
      
      \node(2)[align=left] at (0,-1) {$
      F^{(k)}U^{(k)}F^{(k)}H^{(1)}(v^{(1)}_{n_1-1},v^{(1)}_{n_1})\times\cdots$\\
      $\qquad
      \cdots\times 
      F^{(k)}U^{(k)}F^{(k)}H^{(m)}(v^{(m)}_{0},v^{(m)}_{1})$};
      
      \node(3)[align=left] at (7,-3) 
      {$F^{(k)}H^{(1)}(v^{(1)}_{n_1-1},v^{(1)}_{n_1})\times\cdots$\\
      $\qquad\cdots\times 
      F^{(k)}H^{(m)}(v^{(m)}_{0},v^{(m)}_{1})$};
      
      \draw [->,rounded corners] (0)--(7,3)--(1);
      \draw [->] (1) to  (3);
      \draw [->] (0) to (2);
      \draw [->,rounded corners] (2)--(0,-3)--(3); 
      %\node [labelsize] at (0,-2.7) {$[\id{}]_{J}$};
\end{tikzpicture}
\]
is a pullback. 
This follows from the first clause of
Corollary~\ref{cor:ext_pb_coprod}, and the induction hypothesis.
\end{proof}

\begin{thm}
\label{thm:free_nV_cat_monad_cart}
For each natural number $n$, the free strict $n$-dimensional $\cat{V}$-category 
monad $\monoid{T}^{(n)}$ is cartesian.
\end{thm}
\begin{proof}
This follows from Proposition~\ref{prop:free_n_V_cat_pres_pb}, 
Proposition~\ref{prop:free_n_V_cat_unit_cart},
and Proposition~\ref{prop:free_n_V_cat_counit} (take $m=1$);
cf.~the proof of Theorem~\ref{thm:free_V_cat_monad_cart}.
\end{proof}

\chapter{The definition of weak $n$-dimensional $\cat{V}$-category}
\label{chap:def_weak_n_V_cat}
Building upon the results of the previous chapters, 
in this chapter we define weak $n$-dimensional $\cat{V}$-category 
for each natural number $n$ and \emph{locally presentable} extensive 
category $\cat{V}$.\footnote{Locally presentable categories are both complete 
and cocomplete, so we do not have to write the condition that $\cat{V}$ admits 
all finite limits separately.}
Local presentability is a certain size condition on a category,
and we need to assume this in the final step of 
the definition.
Our definition follows and enriches that of Leinster~\cite{Leinster_book},
which in turn was inspired by Batanin's work~\cite{Batanin_98}.

Leinster's definition of weak $n$-category may be summarised as follows.
Consider the free strict $n$-category monad $\monoid{T}^{(n)}$
on the category $\enGph{n}$ of $n$-graphs;
this is the case $\cat{V}=\Set$ of the monad $\monoid{T}^{(n)}$ studied in the 
previous chapter.
As we have already seen in Theorem~\ref{thm:free_nV_cat_monad_cart}
in the enriched setting, this monad is cartesian,
hence we may consider $\monoid{T}^{(n)}$-operads.
Now, Leinster has introduced the notion of \emph{contraction} on morphisms 
in $\enGph{n}$. 
Recall that a $\monoid{T}^{(n)}$-operad is a monoid object in the slice 
category $\enGph{n}/T^{(n)}1$.
By defining a contraction on a $\monoid{T}^{(n)}$-operad to be a contraction on
its underlying object in $\enGph{n}/T^{(n)}1$, we may also talk about 
$\monoid{T}^{(n)}$-operads with contractions.
Let $\monoid{L}^{(n)}$ be the \emph{initial $\monoid{T}^{(n)}$-operad with a 
contraction}.
Leinster defines weak $n$-categories to be the models of $\monoid{L}^{(n)}$.

In this chapter, we will carry out the enriched version of the above 
development.
Leinster's original formulation of contraction depends heavily on set-theoretic 
manipulations, so we shall use Garner's reformulation~\cite{Garner_homotopy}
of contractions in more categorical terms.
We define contractions on morphisms in $\enGph{\cat{V}}^{(n)}$,
and then on $\monoid{T}^{(n)}$-operads.
We show the existence of the initial $\monoid{T}^{(n)}$-operad with a 
contraction $\monoid{L}^{(n)}$ using our new assumption that $\cat{V}$ is 
locally 
presentable,
and finally define weak $n$-dimensional $\cat{V}$-categories to be the models 
of $\monoid{L}^{(n)}$.

The results in this section have been published in 
\cite{CFP2}.

\section{Contractions}
In this section we describe the notion of \emph{contraction},
introduced by Leinster~\cite{Leinster_book},
and generalise it to the enriched setting.
We follow Garner~\cite{Garner_homotopy} and 
define contraction as a choice of certain diagonal fillers.
The following definition is an example of the construction 
described in \cite[Proposition~3.8]{Garner_understanding}.

\begin{definition}\label{def:contraction}
Let $\cat{C}$ be a category, ${J}$ a set, and 
$\sfam{F}=(f_j\colon A_j\longrightarrow B_j)_{j\in {J}}$ a
${J}$-indexed family of morphisms in $\cat{C}$.
\begin{enumerate}
\item A \defemph{contraction} (with respect to $\sfam{F}$) 
on a morphism $g\colon C\longrightarrow D$
in $\cat{C}$ is a $J$-indexed family of functions $(\kappa_j)_{j\in J}$
such that for each $j\in J$, $\kappa_j$
assigns to every 
pair of morphisms $(h,k)$ in $\cat{V}$ which
makes the perimeter of (\ref{eqn:diagonal_filler}) commute,
a morphism $\kappa_j(h,k)$ making the whole diagram 
(\ref{eqn:diagonal_filler})
commute.
\begin{equation}\label{eqn:diagonal_filler}
\begin{tikzpicture}[baseline=-\the\dimexpr\fontdimen22\textfont2\relax ]
      \node(0) at (0,1) {$A_j$};
      \node(1) at (2,1) {$C$};
      \node(2) at (0,-1) {$B_j$};
      \node(3) at (2,-1) {$D$};
      
      \draw [->] 
            (0) to node (t)[auto,labelsize] {$h$} 
            (1);
      \draw [->] 
            (1) to node (r)[auto,labelsize] {$g$} 
            (3);
      \draw [->] 
            (0) to node (l)[auto,swap,labelsize] {$f_j$} 
            (2);
      \draw [->] 
            (2) to node (b)[auto,swap,labelsize] {$k$} 
            (3);  
      \draw [->] 
            (2) to node [midway,fill=white,labelsize] {$\kappa_j(h,k)$} (1);
\end{tikzpicture}
\end{equation}
\item A \defemph{map of morphisms with contractions} from $(g\colon 
C\longrightarrow D, (\kappa)_{j\in J})$
to $(g'\colon C'\longrightarrow D',(\kappa'_j)_{j\in J})$
is a map of morphisms $(u\colon C\longrightarrow C',v\colon D\longrightarrow 
D')$
from $g$ to $g'$ which commutes with contractions:
for each $j\in J$ and $(h,k)$ in the domain of $\kappa_j$,
$u\circ \kappa_j(h,k)=\kappa'_j(u\circ h, v\circ k)$.
\begin{equation*}
\begin{tikzpicture}[baseline=-\the\dimexpr\fontdimen22\textfont2\relax ]
      \node(0) at (0,2) {$A_j$};
      \node(1) at (2,2) {$C$};
      \node(2) at (0,0) {$B_j$};
      \node(3) at (2,0) {$D$};
      \node(4) at (4,2) {$C'$};
      \node(5) at (4,0) {$D'$};

      \draw [->] 
            (0) to node (t)[auto,labelsize] {$h$} 
            (1);
      \draw [->] 
            (1) to node (r)[auto,near end,labelsize] {$g$} 
            (3);
      \draw [->] 
            (0) to node (l)[auto,swap,labelsize] {$f_j$} 
            (2);
      \draw [->] 
            (2) to node (b)[auto,swap,labelsize] {$k$} 
            (3);  
      \draw [->] 
            (1) to node (b)[auto,labelsize] {$u$} 
            (4);  
      \draw [->] 
            (3) to node (b)[auto,swap,labelsize] {$v$} 
            (5);  
      \draw [->] 
            (4) to node (b)[auto,labelsize] {$g'$} 
            (5);  

      \draw [->] 
            (2) to node [midway,fill=white,labelsize] {$\kappa_j(h,k)$} (1);
      \draw [->] 
            (2) to node [near end,fill=white,labelsize] {$\kappa'_j(u 
            h,v k)$} (4);
\end{tikzpicture}
\end{equation*}
\end{enumerate}
We write the category of morphisms in $\cat{C}$ with contractions (with respect 
to $\mathcal{F}$) as $\Contr{\mathcal{F}}$.
Note that we have the evident forgetful functor 
$V\colon\Contr{\mathcal{F}}\longrightarrow\cat{C}^{\ord{2}}$
where $\ord{2}$ denotes the arrow category (i.e., the ordinal $\ord{2}$ seen as 
a category) and $\cat{C}^\ord{2}=[\ord{2},\cat{C}]$ is the functor category.
\end{definition}
In other words, for each $j\in J$, $\kappa_j$ is a section
of the function $\rho_j$ below, induced by the universality of pullback.
\begin{equation*}
\begin{tikzpicture}[baseline=-\the\dimexpr\fontdimen22\textfont2\relax ]
      \node(v) at (-2,2) {$\cat{C}(B_j,C)$};
      \node(0) at (0,1) {$P_j$};
      \node(1) at (3,1) {$\cat{C}(B_j,D)$};
      \node(2) at (0,-1) {$\cat{C}(A_j,C)$};
      \node(3) at (3,-1) {$\cat{C}(A_j,D)$};
      
      \draw [->] 
            (0) to node (t)[auto,labelsize] {} 
            (1);
      \draw [->] 
            (1) to node (r)[auto,labelsize] {$\cat{C}(f_j,D)$} 
            (3);
      \draw [->] 
            (0) to node (l)[auto,swap,labelsize] {} 
            (2);
      \draw [->] 
            (2) to node (b)[auto,swap,labelsize] {$\cat{C}(A_j,g)$} 
            (3);  

      \draw [->,bend left=10] 
            (v) to node [auto,labelsize] {$\cat{C}(B_j,g)$} 
            (1);  
      \draw [->,bend right=20] 
            (v) to node (b)[auto,swap,labelsize] {$\cat{C}(f_j,C)$} 
            (2);  

      \draw [->,dashed] 
            (v) to node (b)[auto,labelsize] {$\rho_j$} 
            (0); 

     \draw (0.2,0.4) -- (0.6,0.4) -- (0.6,0.8);
\end{tikzpicture}
\end{equation*}

As observed in \cite{Garner_homotopy},
Leinster's notion of contraction, for each natural number $n$,
is a special case of Definition~\ref{def:contraction}
where $\cat{C}=\enGph{n}$ and $\mathcal{F}$ is a certain family
$\mathcal{F}^{(n)}=(f^{(n)}_0,\dots,
f^{(n)}_{n+1})$ 
consisting of $n+2$ morphisms in $\enGph{n}$.
Before giving a precise definition, we try to give an intuitive 
idea of them by drawing a suggestive picture.
For example, when $n=2$ the family can be drawn as
\[
\mathcal{F}^{(2)}=\left(
\begin{tikzpicture}[baseline=-\the\dimexpr\fontdimen22\textfont2\relax ]
      \node(11) at (0,1) {$\bigg( \quad\bigg) $};
      \node(21) at (0,-1) {$\bigg(\bullet\bigg)$};
      
      \draw [->] (0, 0.4) to node[auto,labelsize]{$f^{(2)}_0$} (0,-0.4);
\end{tikzpicture}, 
\begin{tikzpicture}[baseline=-\the\dimexpr\fontdimen22\textfont2\relax ]
      \node(11) at (0,1) {$\bigg( \bullet$};
      \node(12) at (1.5,1) {$\bullet \bigg)$};
      \node(21) at (0,-1) {$\bigg( \bullet$};
      \node(22) at (1.5,-1) {$\bullet \bigg)$};
      
      \draw [->]  (21) to(22);      
      
      \draw [->] (0.75, 0.4) to  node[auto,labelsize]{$f^{(2)}_1$} (0.75,-0.4);
\end{tikzpicture}, 
\begin{tikzpicture}[baseline=-\the\dimexpr\fontdimen22\textfont2\relax ]
      \node(11) at (0,1) {$\bigg( \bullet$};
      \node(12) at (1.5,1) {$\bullet \bigg)$};
      \node(21) at (0,-1) {$\bigg( \bullet$};
      \node(22) at (1.5,-1) {$\bullet \bigg)$};
      
      \draw [->,bend left=30]  (11) to node (1u) {} (12);
      \draw [->,bend right=30] (11) to node (1b) {} (12);
      \draw [->,bend left=30]  (21) to node (2u) {} (22);      
      \draw [->,bend right=30] (21) to node (2b) {} (22); 
      
      \draw [->] (2u) to (2b);
      
      \draw [->] (0.75, 0.4) to  node[auto,labelsize]{$f^{(2)}_2$} (0.75,-0.4);
\end{tikzpicture}, 
\begin{tikzpicture}[baseline=-\the\dimexpr\fontdimen22\textfont2\relax ]
      \node(11) at (0,1) {$\bigg( \bullet$};
      \node(12) at (1.5,1) {$\bullet \bigg)$};
      \node(21) at (0,-1) {$\bigg( \bullet$};
      \node(22) at (1.5,-1) {$\bullet \bigg)$};
      
      \draw [->,bend left=30]  (11) to node (1u) {} (12);
      \draw [->,bend right=30] (11) to node (1b) {} (12);
      \draw [->,bend left=30]  (21) to node (2u) {} (22);      
      \draw [->,bend right=30] (21) to node (2b) {} (22); 
      
      \draw [->,transform canvas={xshift=-0.4em}] (1u) to (1b);
      \draw [->,transform canvas={xshift=0.4em}]  (1u) to (1b);
      \draw [->] (2u) to (2b);
      
      \draw [->] (0.75, 0.4) to  node[auto,labelsize]{$f^{(2)}_3$} (0.75,-0.4);
\end{tikzpicture}
\right).
\]
Just in case it is not clear how to read the above picture, let us 
explain one object. The picture
\[
\begin{tikzpicture}[baseline=-\the\dimexpr\fontdimen22\textfont2\relax ]
      \node(21) at (0,0) {$\bigg( \bullet$};
      \node(22) at (1.5,0) {$\bullet \bigg)$};
      
      \draw [->,bend left=30]  (21) to node (2u) {} (22);      
      \draw [->,bend right=30] (21) to node (2b) {} (22); 
      
      \draw [->] (2u) to (2b);
\end{tikzpicture}
\]
denotes the 2-graph $G$ with two objects ($\ob{G}=\{s,t\}$, represented by the 
black dots), such that the 1-graphs
$G(s,s),G(t,s)$ and $G(t,t)$ have no objects, and 
$G(s,t)$ is the 1-graph with two objects ($\ob{G(s,t)}=\{x,y\}$, represented by 
the two horizontal arrows between the black dots) such that 
 $(G(s,t)) (x,x)=(G(s,t)) (y,x)=(G(s,t)) (y,y)=\emptyset$ and
$(G(s,t)) (x,y)=\{z\}$ (the vertical arrow).

The morphisms $f^{(2)}_0,f^{(2)}_1$ and $f^{(2)}_2$ are 
monomorphisms, and $f^{(2)}_3$ is an epimorphism in $\enGph{n}$.
The idea is that an element of $\mathcal{F}^{(n)}$ is
``the inclusion of the boundary of a ball'',
although $f^{(n)}_{n+1}$ is no longer a monomorphism
due to lack of cells of dimension greater than $n$.

To give a recursive definition of $\mathcal{F}^{(n)}$ in the enriched setting, 
we start with auxiliary definitions.
For any category $\cat{V}'$ with an initial object $0$, define the
\defemph{suspension} functor 
$\Sigma\colon\cat{V}'\longrightarrow \enGph{\cat{V}'}$
which maps $X\in\cat{V}$ to 
\[
\Sigma X= (\{s,t\}, (\Sigma X(i,j))_{i,j\in\{s,t\}})
\]
given by $\Sigma X(s,t)=X$, $\Sigma X(i,j)=0$ if $(i,j)\neq(s,t)$;
cf.~\cite[Section~9.3]{Leinster_book}.
Also define the \defemph{discrete $\cat{V}'$-graph} 
functor $D\colon \Set\longrightarrow\enGph{\cat{V}'}$
which maps a set $I$ to $DI=(I,(0)_{i,j\in I})$. The functor $D$ is the 
left adjoint of $\ob{-}\colon \enGph{\cat{V}'}\longrightarrow \Set$.

\begin{definition}\label{defn:generating_cofibrations}
Let $\cat{V}$ be a category with a terminal object and finite coproducts.
For each natural number $n$, define a family
$\mathcal{F}^{(n)}=(f^{(n)}_0,\dots,f^{(n)}_{n+1})$
of morphisms in $\enGph{\cat{V}}^{(n)}$ recursively as follows.
\begin{enumerate}
\item $f^{(0)}_0\colon 0\longrightarrow 1$ and 
$f^{(0)}_1\colon 1+1\longrightarrow 1$ are the unique morphisms in $\cat{V}$
into the terminal object $1$.
\item $f^{(n)}_0\colon D\emptyset \longrightarrow D\{\ast\}$,
where $\emptyset$ and $\{\ast\}$ are the empty set and a singleton
respectively,
is the unique morphism in $\enGph{\cat{V}}^{(n)}$ 
out of the initial object $D\emptyset$,
and for each $i\in \{1,\dots, n+1\}$,
$f^{(n)}_i=\Sigma f^{(n-1)}_{i-1}$.\qedhere
\end{enumerate}
\end{definition}

For each object $X\in \enGph{\cat{V}}^{(n)}$,
define the category $\Contr{\mathcal{F}^{(n)}}_X$ of morphisms into $X$ with 
contractions (with respect to $\mathcal{F}^{(n)}$) as
the following pullback of categories:
\begin{equation}
\label{eqn:contr}
\begin{tikzpicture}[baseline=-\the\dimexpr\fontdimen22\textfont2\relax ]
      \node(0) at (0,1) {$\Contr{\mathcal{F}^{(n)}}_X$};
      \node(1) at (4,1) {$\Contr{\mathcal{F}^{(n)}}$};
      \node(2) at (0,-1) {$\enGph{\cat{V}}^{(n)}/X$};
      \node(3) at (4,-1) {$(\enGph{\cat{V}}^{(n)})^\ord{2},$};
      
      \draw [->] 
            (0) to node (t)[auto,labelsize] {} 
            (1);
      \draw [->] 
            (1) to node (r)[auto,labelsize] {$V$} 
            (3);
      \draw [->] 
            (0) to node (l)[auto,swap,labelsize] {$V_X$} 
            (2);
      \draw [->]
            (2) to node (b)[auto,swap,labelsize] {} 
            (3);  

     \draw (0.2,0.4) -- (0.6,0.4) -- (0.6,0.8);
\end{tikzpicture}
\end{equation}
where $\enGph{\cat{V}}^{(n)}/X\longrightarrow(\enGph{\cat{V}}^{(n)})^\ord{2}$
is the inclusion functor.
Explicitly, the category $\Contr{\mathcal{F}^{(n)}}_X$ is given as follows.
\begin{itemize}
\item An object is a morphism $g$ in $\enGph{\cat{V}}^{(n)}$ with
a contraction as in Definition~\ref{def:contraction}
such that the codomain of $g$ is $X$.
\item A morphism is a map of morphisms with contractions $(u,v)$ as 
in Definition~\ref{def:contraction} such that $v=\id{X}$.
\end{itemize}
We will in particular be concerned with the case where
$X=T^{(n)}1$.

\medskip
Now we can describe our definition of weak $n$-dimensional 
$\cat{V}$-category in more detail. 
We have already mentioned at the beginning of this chapter that we  
define a {weak $n$-dimensional $\cat{V}$-category} to be a model
of a certain $\monoid{T}^{(n)}$-operad $\monoid{L}^{(n)}$, characterised as 
the \emph{initial $\monoid{T}^{(n)}$-operad with a contraction}.
Let us define what this means in more precise terms.
We define the category $\OC{\monoid{T}^{(n)}}$ of $\monoid{T}^{(n)}$-operads
with contractions to be the following pullback of categories:
\begin{equation}
\label{eqn:pb_OC}
\begin{tikzpicture}[baseline=-\the\dimexpr\fontdimen22\textfont2\relax ]
      \node(0) at (0,1) {$\OC{\monoid{T}^{(n)}}$};
      \node(1) at (4,1) {$\Contr{\mathcal{F}^{(n)}}_{T^{(n)}1}$};
      \node(2) at (0,-1) {$\Operad{\monoid{T}^{(n)}}$};
      \node(3) at (4,-1) {$\enGph{\cat{V}}^{(n)}/T^{(n)}1,$};
      
      \draw [->] 
            (0) to node (t)[auto,labelsize] {} 
            (1);
      \draw [->] 
            (1) to node (r)[auto,labelsize] {$V_{T^{(n)}1}$} 
            (3);
      \draw [->] 
            (0) to node (l)[auto,swap,labelsize] {} 
            (2);
      \draw [->]
            (2) to node (b)[auto,swap,labelsize] {$W$} 
            (3);  

     \draw (0.2,0.4) -- (0.6,0.4) -- (0.6,0.8);
\end{tikzpicture}
\end{equation}
where the functor $V_{T^{(n)}1}$ is the appropriate instance of 
(\ref{eqn:contr}) and 
$W$ forgets the $\monoid{T}^{(n)}$-operad structure 
(recall that $\Operad{\monoid{T}^{(n)}}=\Mon{\enGph{\cat{V}}^{(n)}/T^{(n)}1}$).
Provided that the category $\OC{\monoid{T}^{(n)}}$ has an initial object 
$((\arity{L^{(n)}}\colon L^{(n)}\longrightarrow T^{(n)}1),m,$ $e,\kappa)$,
by the initial $\monoid{T}^{(n)}$-operad with contraction we mean
its underlying $\monoid{T}^{(n)}$-operad $\monoid{L}^{(n)}=
((\arity{L^{(n)}}\colon L^{(n)}\longrightarrow T^{(n)}1),m,e)$
(forgetting the contraction $\kappa$).

Thus the remaining step in our definition of weak $n$-dimensional 
$\cat{V}$-category is to show that the category $\OC{\monoid{T}^{(n)}}$ indeed 
has an initial object.
This can be shown, under the additional assumption that $\cat{V}$ is locally 
presentable.

\section{Local presentability and algebraic weak factorisation systems}
We first provide a minimal introduction to locally presentable categories;
see \cite[Chapter~1]{Adamek_Rosicky} for more details.

A cardinal $\alpha$ is called \defemph{regular} if  for any set $I$ 
and $I$-indexed family of sets $(x_i)_{i\in I}$, 
$|I|<\alpha$ and $|x_i|<\alpha$ for all $i\in I$ imply $|\coprod_{i\in I}x_i| 
<\alpha$.
We shall only talk about small regular cardinals.

From now on, let $\alpha$ be a (small) regular cardinal.
A small poset $\cat{I}$ is said to be \defemph{$\alpha$-directed} if 
any subset of $\cat{I}$ whose cardinality is less than $\alpha$
admits an upper bound in $\cat{I}$.
For any category $\cat{C}$,
an \defemph{$\alpha$-directed diagram} is a functor 
$\cat{I}\longrightarrow\cat{C}$ from an $\alpha$-directed poset 
$\cat{I}$ (seen as a category).
By an \defemph{$\alpha$-directed colimit} we mean the colimit of an 
$\alpha$-directed diagram.

Suppose that $\cat{C}$ and $\cat{D}$ are locally small categories 
admitting all $\alpha$-directed colimits
(i.e., admitting all colimits indexed by small $\alpha$-directed posets).
A functor $\cat{C}\longrightarrow\cat{D}$ is said to be 
\defemph{$\alpha$-accessible} if it preserves all $\alpha$-directed colimits.
An object $C\in\cat{C}$ is called \defemph{$\alpha$-presentable}
if the functor $\cat{C}(C,-)\colon\cat{C}\longrightarrow\Set$ is 
$\alpha$-accessible.

A locally small category $\cat{C}$ is called \defemph{locally 
$\alpha$-presentable} if it is cocomplete and 
there exists a small full subcategory 
$\cat{C}_\alpha\subseteq \cat{C}$ such that (i) all objects in $\cat{C}_\alpha$
are $\alpha$-presentable, and (ii) any object in $\cat{C}$ can be expressed as 
an $\alpha$-directed colimit of objects in $\cat{C}_\alpha$.

Finally, a locally small category $\cat{C}$ is called \defemph{locally 
presentable} if there exists a (small) regular cardinal $\alpha$ such that 
$\cat{C}$ is locally $\alpha$-presentable.
A functor $F$ between cocomplete categories is called \defemph{accessible} if 
there exists a (small) regular cardinal $\alpha$ such that $F$ is 
$\alpha$-accessible.

It is known that $\Set$ is locally $\aleph_0$-presentable (also called 
\emph{locally finitely presentable}), and $\omega$-$\Cpo$ is locally 
$\aleph_1$-presentable (see \cite[Example~1.18]{Adamek_Rosicky}).
It is also known that whenever $\cat{V}$ is locally presentable, so is 
$\enGph{\cat{V}}$ (\cite[Proposition~4.4]{Kelly_Lack_loc_presentable}).

\medskip

Among others, local presentability is used as a standard condition on 
categories in order to ensure 
that certain transfinite constructions to converge \cite{Kelly_unified}.
An example of such constructions relevant to our purpose is 
Garner's version \cite{Garner_understanding} of 
the \emph{small object argument} originally developed by Quillen \cite{Quillen}.
We have the following result, easily deducible from 
\cite[Proposition~16]{Bourke_Garner}.

\begin{proposition}
\label{prop:V_X_monadic_accessible}
Let $\cat{V}$ be a locally presentable category. 
Then for each $n\in\NN$ and $X\in \enGph{\cat{V}}^{(n)}$,
the functor $V_X\colon 
\Contr{\mathcal{F}^{(n)}}_{X}\longrightarrow 
\enGph{\cat{V}}^{(n)}/X$
is monadic and accessible.
\end{proposition}

\section{Weak $n$-dimensional $\cat{V}$-categories}\label{sec:weak}
Finally we prove that $\OC{\monoid{T}^{(n)}}$ actually
has an initial object, for any category $\cat{V}$
which is locally presentable and extensive. 

\begin{thm}
If $\cat{V}$ is a locally presentable and extensive category,
then for any natural number $n$
the category $\OC{T^{(n)}}$ has an initial object.
\end{thm}
\begin{proof}
We shall follow the argument in \cite[Appendix~G]{Leinster_book} 
(where $\cat{V}=\Set$ and $n=\omega$) and 
%\cite[Theorem~27.1]{Kelly_unified}.
show that $\enGph{\cat{V}}^{(n)}/T^{(n)}1$ is locally presentable
(hence is both complete and cocomplete), and that $W$ and $V_{T^{(n)}1}$
are monadic and accessible. 
Then by \cite[Theorem~27.1]{Kelly_unified} it follows that the 
forgetful functor from $\OC{\monoid{T}^{(n)}}$ to 
$\enGph{\cat{V}}^{(n)}/T^{(n)}1$
(the composite of functors in (\ref{eqn:pb_OC}))
is also monadic, thus in particular
$\OC{\monoid{T}^{(n)}}$ has an initial object, given by the 
free algebra over the initial object in $\enGph{\cat{V}}^{(n)}/T^{(n)}1$.

Because $\enGph{\cat{V}}^{(n)}$ is locally 
presentable,
so is $\enGph{\cat{V}}^{(n)}/T^{(n)}1$, being its slice.
The functor $W$ is monadic because it is the forgetful functor from a 
category of monoids and admits a left adjoint $G$
(which, incidentally, is of a particularly simple form 
$GP=\coprod_{n\in\N}P^{\otimes n}$
thanks to Proposition~\ref{prop:forgetful_n_pres_coprod}).
It is routine to show that $W$ is accessible.
The functor $V_{T^{(n)}1}$ is monadic and accessible by 
Proposition~\ref{prop:V_X_monadic_accessible}.
\end{proof}

The condition of $\cat{V}$ being locally presentable and extensive
is an axiomatic reason why Batanin and Leinster's approach works. 
Of course the category $\Set$ satisfies this condition, but 
in their work this fact is used only implicitly, often in the 
form of concrete set-theoretic manipulation. 

\begin{definition}
Let $\cat{V}$ be a locally presentable extensive category and $n$ be a natural
number.
A \defemph{weak $n$-dimensional $\cat{V}$-category} is 
a model of the initial $\monoid{T}^{(n)}$-operad with contraction,
where $\monoid{T}^{(n)}$ is the free strict $n$-dimensional $\cat{V}$-category
monad on $\enGph{\cat{V}}^{(n)}$.
\end{definition}

We remark that when $\cat{V}(1,-)$ is not conservative, it might be more
appropriate to replace 1 of Definition~\ref{defn:generating_cofibrations}
by the family of morphisms $0\longrightarrow X$ and $X+X \longrightarrow X$
(codiagonal) where $X$ ranges over a small set of strong generators of $\cat{V}$
(exists if $\cat{V}$ is locally presentable).
We thank an anonymous reviewer of \cite{CFP2} for pointing this out. 
Even if we alter Definition~\ref{defn:generating_cofibrations} this way,
all arguments so far hold unchanged.

\begin{example}
If we let $\cat{V}=\Set$ and $n=2$,
then weak $2$-categories (weak $2$-dimensional $\Set$-categories) 
are equivalent to \emph{unbiased bicategories},
which are a variant of bicategories
equipped with for each natural number $m$, an $m$-ary horizontal
composition operation. 
See \cite[Section~9.4]{Leinster_book} for details.
\end{example}

\begin{example}
If we let $\cat{V}=\omega\mhyphen\mathbf{Cpo}$ and $n=2$, then 
weak $2$-dimensional $\omega\mhyphen\mathbf{Cpo}$-categories
are the unbiased version of $\omega\mhyphen\mathbf{Cpo}$-enriched
bicategories as in \cite{Power_Tanaka}.
\end{example}

\chapter{Conclusion}

\section{Summary}
In this thesis, we have investigated aspects of algebraic structure.
In the first part, we have developed a unified framework for various 
notions of algebraic theory.
In the second part, we focused on a particular algebraic structure,
weak $n$-categories {\`a} la Batanin and Leinster, and generalised the known
definition by allowing enrichment over any extensive and locally presentable 
category.

\medskip

Our unified framework for notions of algebraic theory is based on a number of 
more or less independent observations
made by many researchers over years, which we have summarised in  
Section~\ref{sec:framework_prelude}.
The concepts of metatheory and theory, 
being identical to (large) monoidal category and monoid object, are of course 
well-known.
As for these, the novelty is not in the concepts themselves but in our attitude 
to 
identify them with notion of algebraic theory and algebraic theory respectively.
To the best of our knowledge, no one seems to have proposed such identification.

We have supported this rather bold proposal by 
modelling the semantical aspect of notions of algebraic theory as well
in our framework.
Here, in order to unify enrichments and oplax actions, which have been 
observed to underlie notions of model, we have introduced a new concept of 
metamodel.
Although one can reduce metamodels (of $\cat{M}$ in $\cat{C}$)
to combinations of known concepts, such as enrichment\footnote{To be precise, 
the concept of enrichment (Definition~\ref{def:enrichment}) also seems to have 
been newly introduced in this thesis, though it is fairly similar to the 
well-known concept of $\cat{M}$-category~\cite{Kelly:enriched}.} 
of $\cat{C}$ over 
$\widehat{\cat{M}}=[\cat{M}^\op,\SET]$ or as a lax monoidal functor 
$\cat{M}^\op\longrightarrow[\cat{C}^\op\times\cat{C},\SET]$,
they do not seem to have been studied extensively so far, let alone 
in connection to notions of algebraic theory.
The fact that we can give a definition of model relative to a metamodel in a 
way compatible with those relative to an enrichment or an oplax action,
though not particularly difficult to show, seems to testify to the inherent 
coherence underlying various notions of algebraic theory.

We have also introduced morphism between metatheories.
An appropriate notion of morphism turned out to be more general than 
the ones usually considered, namely lax, oplax or strong monoidal functors;
it is a monoidal version of profunctors.
If the morphisms come in an adjoint pair, 
then (by the pseudo-functoriality of $\MtMod{-}$) we obtain a 
2-adjunction between the 2-categories of metamodels.
Because our morphisms between metatheories are quite general, it is not 
difficult to obtain an adjoint pair of them;
any strong monoidal functor generates an adjoint pair.
In this case, we immediately obtain isomorphisms of categories of models
in different notions of algebraic theory,
by a purely formal categorical argument (see Section~\ref{sec:comparing}).

Within our framework,
we have also obtained a general structure-semantics adjointness result 
(Chapter~\ref{chap:str_sem}) and a double categorical universal 
characterisation of categories of models (Chapter~\ref{chap:double_lim}).
The former result
supports our claim that the framework is appropriate for notions of 
algebraic theory, by incorporating the topic which has been studied extensively
in the categorical algebra community.
The latter result may be taken as an evidence of the naturality or 
canonicity of our framework,
as it gives an abstract characterisation of categories of models arising in
our framework, generalising the characterisation of 
Eilenberg--Moore categories by Street~\cite{Street_FTM} in a natural direction.
In addition, we believe that it provides a non-trivial
example of double limits, which are a newly 
introduced notion \cite{GP1} and seem to be in need of examples.

\medskip

Our generalisation of Batanin and Leinster's definition of 
weak $n$-category clarifies the structure of their original definition, by 
pointing out the fact that the categorical properties of extensivity and local 
presentability play a key role in the definition.

We have established in Chapter~\ref{chap:extensive}
a number of properties on (infinitary) extensive categories.
Since these properties are not very hard to show, we 
expect that they are either known to or immediately recognisable by
the experts,
but we have not been able to find a suitable reference.
The papers \cite{Carboni_Lack_Walters,Cockett} are excellent sources of 
information,
but they only treat finitary extensive categories.

In Chapter~\ref{chap:free_strict_n_cat_monad} we have shown by induction on $n$ 
that the free strict $n$-dimensional $\cat{V}$-category monad $\monoid{T}^{(n)}$
on $\enGph{\cat{V}}^{(n)}$ is cartesian.
Our inductive argument is more delicate than one might first imagine,
and we had to choose properties more general than 
is strictly necessary for our goal
(see e.g., Proposition~\ref{prop:free_n_V_cat_counit}).
The proofs fully exploit the properties of extensive categories established in 
Chapter~\ref{chap:extensive}.

Our definition of weak $n$-dimensional $\cat{V}$-category for any 
extensive and locally presentable category $\cat{V}$ is given in 
Chapter~\ref{chap:def_weak_n_V_cat}.
Here, in order to generalise Leinster's notion of contraction,
we have applied Garner's theory of algebraic weak factorisation systems 
\cite{Garner_understanding}.

\section{Future work}
As future work, we would like to further investigate 
various aspects of our unified 
framework for notions of algebraic theory.
One natural open problem is to characterise 
the categories of models arising in our framework---or rather, the associated 
forgetful functors---by their intrinsic properties.
For the case of monad, the corresponding result is various 
\emph{monadicity theorems} (such as Beck's 
theorem~\cite[Section~VI.~7]{MacLane_CWM}),
characterising the \emph{monadic functors}, i.e., those functors 
isomorphic to the forgetful functors from Eilenberg--Moore categories.
The forgetful functors arising in our framework are more general than the 
monadic functors; for example, they need not admit left adjoints,
as is the case for the forgetful functor 
$\mathbf{FinGrp}\longrightarrow\mathbf{FinSet}$
from the category of finite groups to the category of finite sets
(this functor arises if we consider the metatheory $[\F,\Set]$ for clones,
the clone of groups and the standard metamodel of $[\F,\Set]$
in the category $\mathbf{FinSet}$ with finite powers).
However, they are far from being arbitrary.
For example, it is immediate from the definition of categories of models 
(Definition~\ref{def:metamodel_model}) that 
such functors are faithful and conservative. 
We would like to identify what additional condition on a functor is 
enough to ensure that it arises (up to an isomorphism)
as the forgetful functor associated with a category of models in our framework.
Such a result would help us to better understand the generality of 
our framework.

We would also like to incorporate more examples of notions of algebraic theory
into our framework.
We have already listed some possible examples
in Section~\ref{sec:other_ex}.
As for PROs and PROPs, we expect that 
monoidal and symmetric monoidal versions of profunctors
(cf.~morphisms of metatheories in Definition~\ref{def:morphism_of_metatheories})
would be useful.
For example, a PRO is defined as a strict
monoidal category together with an identity-on-objects 
strict monoidal functor from 
$\Ncat$, the free strict monoidal category generated by one object.
By considering the monoidal category of monoidal endo-profunctors on $\Ncat$,
we would obtain PROs as monoids therein.
%by using the result of 
%Garner~\cite{Garner_thesis}
%expressing polycategories (and in particular, PROs and PROPs) as 
%monads in a certain bicategory (hence monoids in a suitable monoidal 
%category), they can be captured by our framework without 
%serious difficulty.
As for multi-sorted algebraic theories, 
we think that the best way to model them 
is to identify them with pseudo double categories,
in such a way that objects, 
vertical morphisms, horizontal morphisms and squares correspond to 
sorts, translations between sorts, functional signatures (with designated 
input/output sorts)
and translations of functional signatures, respectively.
This view is compatible with our current framework, because  
pseudo double categories with one object and one vertical morphism 
correspond to monoidal categories.
In fact, the pseudo double categories suitable for multli-sorted clones, 
symmetric operads, non-symmetric operads and generalised operads 
are already studied in \cite{Cruttwell_Shulman};
this paper would lay foundations for the syntactic aspect of the multi-sorted
version of our framework.

Our framework shows that whenever we have a monoidal category,
we can regard it as a notion of algebraic theory.
This observation provides a novel, particularly simple way to define 
new notions of algebraic theory.
Need for new notions of algebraic theory would arise, for example,
in study of computational effects. 
The monad and Lawvere theory approaches to computational effects
(see Section~\ref{sec_alg_str_math_cs}) have
captured different aspects of computational effects, and 
the characteristic features of these notions of algebraic theory
are reflected in their major applications: 
the simplicity of monad makes it into a popular design pattern in 
functional programming~\cite{Wadler}, and the modularity of 
Lawvere theory neatly explains how to model
combinations of effects~\cite{Hyland_Plotkin_Power}.
One naturally expects that suitable notions of algebraic theory 
would be useful in capturing other aspects of computational effects.
Here we mention one such possibility:
the quantitative aspect as measured by \emph{effect systems} 
\cite{Lucassen_Gifford}.
A categorical semantics of effect systems is given via the notion of 
\emph{graded monad} \cite{Katsumata},
which is a monad in a suitable 2-category \cite{FKM}
and hence a monoid in a monoidal category,
but a suitable notion of \emph{graded Lawvere theory}
is yet to be defined.

Another future work is to apply our framework to the study of higher 
dimensional categories.
As we have mentioned in the introduction, currently there are many definitions 
of weak $n$-category and a conceptual understanding of the relationship 
between these definitions is in need.
An obstruction to the direct comparison 
is the fact that different (algebraic)
definitions of weak $n$-category are given in terms of algebraic theories 
belonging to different notion of algebraic theory,
such as generalised operads, symmetric operads and monads; 
cf.~\cite{Leinster_survey}.
We expect that our unified framework may overcome this difficulty thanks to its
generality, incorporating a wide range of notions of algebraic theory.

Finally we mention that there are also a lot to be done around
Batanin and Leinster's weak $n$-categories.
In Leinster's definition, weak $n$-categories are defined as models of a 
certain $\monoid{T}^{(n)}$-operad $\monoid{L}^{(n)}$.
However, if we consider homomorphisms in the usual sense between models of 
$\monoid{L}^{(n)}$, then 
these correspond to \emph{strict $n$-functors} and the more natural
\emph{weak $n$-functors} are not treated in \cite{Leinster_book}.
Batanin gives a definition of weak $n$-functor in 
\cite[Definition~8.8]{Batanin_98}, and it would be interesting 
to adapt that definition to Leisnter's version of weak $n$-categories,
and to enrich it over an extensive and locally presentable category 
$\cat{V}$ in order to clarify the structure of the definition.
We believe that a substantial theory of weak $n$-categories would 
have applications in computer science as well, 
for instance by suggesting new semantically motivated axioms 
to homotopy type theory.

%\appendix
%\include{chap_promon}
%\include{chap_avery}

%-------------------
\bibliographystyle{plain} % 参考文献
\bibliography{myref} %
%-------------------
\end{document}